\documentclass[a4paper,11pt]{amsart}
\usepackage{fixltx2e}
\usepackage{amsbsy,amscd,amsfonts,amsmath,amssymb,amstext,amsthm}
\usepackage{mathalfa,mathrsfs}
\usepackage{gensymb}
\usepackage{geometry,tikz}
\usepackage{dsfont,euscript}
\usepackage{hyperref,url}
\usepackage{graphics,graphicx}
\usepackage{color,xcolor}
\usepackage{epsf,pdfsync,psfrag}
\usepackage{esint,verbatim,xspace}
\usepackage{enumerate,enumitem}
\usepackage{eurosym,latexsym}
\usepackage{bibtopic,textcomp}
\usepackage[latin1]{inputenc}
\newcommand{\N}{\mathbb{N}}

\newcommand{\R}{\mathbb{R}}
\newcommand{\C}{\mathbb{C}}

\newcommand{\PP}{\mathbb{P}}

\newcommand{\cC}{\mathcal{C}}

\newcommand{\cL}{\mathcal{L}}

\newcommand{\fkh}{\mathfrak{h}}

\newcommand{\Ric}{\operatorname{Ric}}

\newcommand{\del}{\partial}
\newcommand{\db}{\bar{\partial}}
\newcommand{\deldb}{\partial \bar{\partial}}
\newcommand{\ideldb}{\sqrt{-1} \partial \bar{\partial}}

\newcommand{\impl}{\implies}

\newcommand{\Vol}{\operatorname{Vol}}

\newcommand{\sfc}{\mathsf{C}}

\newcommand{\ttp}{\mathtt{p}}

\newcommand{\ttf}{\mathtt{f}}
\newcommand{\tth}{\mathtt{h}}
\newcommand{\ttg}{\mathtt{g}}

\newcommand{\ttd}{\mathtt{d}}

\newcommand{\prj}{\mathbb{P} \left(L \oplus \mathcal{O}\right)}
\newcommand{\so}{S_0}
\newcommand{\soo}{S_\infty}
\newcommand{\xsosoo}{X \smallsetminus \left(S_0 \cup S_\infty\right)}
\newcommand{\xsoo}{X \smallsetminus S_\infty}
\newcommand{\xso}{X \smallsetminus S_0}
\newcommand{\bo}{\beta_0}
\newcommand{\boo}{\beta_\infty}
\newcommand{\pcsoo}{2 \pi \left(\mathsf{C} + m S_\infty\right)}
\newcommand{\kcsoo}{k \left(\mathsf{C} + m S_\infty\right)}
\newcommand{\acsoo}{a \mathsf{C} + b S_\infty}

\newcommand{\cp}{\mathbb{C} \mathbb{P}}

\newcommand{\stzero}{\left\lbrace 0 \right\rbrace}
\newcommand{\stone}{\left\lbrace 1 \right\rbrace}
\newcommand{\clom}{\left[0, m\right]}
\newcommand{\opom}{\left(0, m\right)}
\newcommand{\cllml}{\left[1, m + 1\right]}
\newcommand{\oplml}{\left(1, m + 1\right)}

\newcommand{\rms}{Riemann surface}
\newcommand{\rmn}{Riemannian}

\newcommand{\hmn}{Hermitian}

\newcommand{\hmnm}{Hermitian metric}

\newcommand{\kr}{K\"ahler}
\newcommand{\krmf}{K\"ahler manifold}
\newcommand{\krm}{K\"ahler metric}
\newcommand{\krf}{K\"ahler form}
\newcommand{\krc}{K\"ahler current}
\newcommand{\krcl}{K\"ahler class}
\newcommand{\krp}{K\"ahler potential}
\newcommand{\hbch}{Hirzebruch}
\newcommand{\chrn}{Chern}

\newcommand{\funcl}{functional}
\newcommand{\invc}{invariance}
\newcommand{\invt}{invariant}
\newcommand{\ext}{extremal}
\newcommand{\hol}{holomorphic}
\newcommand{\holy}{holomorphy}

\newcommand{\cohoml}{cohomology}

\newcommand{\cohomll}{cohomological}

\newcommand{\plyhomoy}{polyhomogeneity}

\newcommand{\plyhomo}{polyhomogeneous}

\newcommand{\conrml}{conormal}
\newcommand{\har}{harmonic}

\newcommand{\psh}{plurisubharmonic}
\newcommand{\intbly}{integrability}

\newcommand{\intble}{integrable}
\newcommand{\orble}{orientable}
\newcommand{\nbd}{neighbourhood}
\makeatletter
\renewcommand\subsection
{
\@startsection{subsection}{2}
\z@{0.1875\linespacing \@plus 0.1875\linespacing}
{0.125\linespacing}{\centering\normalfont\scshape}
}
\makeatother
\makeatletter
\@namedef{subjclassname@2020}{\textup{2020} Mathematics Subject Classification}
\makeatother



\theoremstyle{plain}
\newtheorem{lemma}{Lemma}

\newtheorem{theorem}{Theorem}
\newtheorem{corollary}{Corollary}
\newtheorem{conjecture}{Conjecture}
\theoremstyle{remark}
\newtheorem{remark}{Remark}

\newtheorem{motivation}{Motivation}
\newtheorem{question}{Question}

\newtheorem*{case*}{Case}

\newtheorem*{claim*}{Claim}
\theoremstyle{definition}
\newtheorem{definition}{Definition}

\begin{document}
\allowdisplaybreaks[4]
\title[Existence of Conical Higher cscK Metrics on a Minimal Ruled Surface]{Existence of Conical Higher cscK Metrics on a Minimal Ruled Surface}
\author[Rajas Sandeep Sompurkar]{Rajas Sandeep Sompurkar}
\address{Department of Mathematics, Indian Institute of Science Education and Research (IISER) Pune, Dr. Homi Bhabha Road, Pashan, Pune - 411 008, Maharashtra, India}
\email{\href{mailto:rajassompurkar@gmail.com}{rajassompurkar@gmail.com}, \href{mailto:rajas.sompurkar@acads.iiserpune.ac.in}{rajas.sompurkar@acads.iiserpune.ac.in}}
\urladdr{\url{https://sites.google.com/view/rajassompurkar/home}}
\curraddr{\url{https://orcid.org/0009-0003-9668-5752}}
\thanks{The author was previously supported by the DST-INSPIRE Fellowship No. IF 180232 during his Ph.D. hosted at IISc Bengaluru, and then was supported by an Institute Postdoc. Fellowship with Employee ID No. 20255006 at IISER Pune. The author is currently supported by the DAE-NBHM Postdoc. Fellowship with Sanction Order No. 0204/4/2026/R\&D-II/2872 also hosted at IISER Pune.}
\date{May 7, 2026}
\begin{abstract}
A `higher extremal K\"ahler metric' is defined (motivated by analogy with the definition of an extremal K\"ahler metric) as a K\"ahler metric whose top Chern form equals a globally defined smooth function multiplied by its volume form such that the gradient of the smooth function is a holomorphic vector field. A special case of this kind of metric is a `higher constant scalar curvature K\"ahler (higher cscK) metric' which is defined (again by analogy with the definition of a constant scalar curvature K\"ahler (cscK) metric) as one whose top Chern form is a constant multiple of its volume form or equivalently whose top Chern form is harmonic. In our previous paper on higher extremal K\"ahler metrics we had looked at a certain class of minimal ruled surfaces called as `pseudo-Hirzebruch surfaces' all of which contain two special divisors (viz. the zero and infinity divisors) and serve as the primary example manifolds in the momentum construction method (the Calabi ansatz procedure) which is used for producing explicit examples of the above-mentioned kinds of canonical K\"ahler metrics. We had proven that every K\"ahler class on such a surface admits a momentum-constructed higher extremal K\"ahler metric which is not higher cscK and we had further proven by using the `top Bando-Futaki invariant' that higher cscK metrics do not exist in any K\"ahler class on the surface. In this paper we will see that if we allow our metrics to develop `conical singularities' along at least one of the two special divisors of the surface then we do get `conical higher cscK metrics' in each K\"ahler class of the surface by the momentum construction method. We will show that our constructed metrics satisfy the ``polyhomogeneous condition'' for conical K\"ahler metrics and we will interpret the conical higher cscK equation ``globally on the surface'' in terms of the currents of integration along the zero and infinity divisors. We will introduce the `top $\log$ Bando-Futaki invariant' and then try to prove some standard expected results about it, with the final aim being to employ it to arrive at a certain linear relationship, that we conjecture must exist, between the cone angles of the conical singularities along the zero and infinity divisors.
\end{abstract}
\maketitle
\numberwithin{equation}{subsection}
\numberwithin{figure}{subsection}
\numberwithin{table}{subsection}
\numberwithin{lemma}{subsection}
\numberwithin{proposition}{subsection}
\numberwithin{result}{subsection}
\numberwithin{theorem}{subsection}
\numberwithin{corollary}{subsection}
\numberwithin{conjecture}{subsection}
\numberwithin{remark}{subsection}
\numberwithin{note}{subsection}
\numberwithin{motivation}{subsection}
\numberwithin{question}{subsection}
\numberwithin{answer}{subsection}
\numberwithin{case}{subsection}
\numberwithin{claim}{subsection}
\numberwithin{definition}{subsection}
\numberwithin{example}{subsection}
\numberwithin{hypothesis}{subsection}
\numberwithin{statement}{subsection}
\numberwithin{ansatz}{subsection}
\section{Introduction}\label{sec:Intro}
\subsection{Objective of the Paper}\label{subsec:Obj}
In this paper we aim to construct higher cscK metrics (see Definition \ref{def:hcscK}) on the pseudo-Hirzebruch surface (a minimal ruled surface) of genus $2$ and degree $-1$ \cite{Fujiki:1992:eKruledmani,Tonnesen:1998:eKminruledsurf}, which are smooth away from the zero and infinity divisors and which have got conical singularities with arbitrary positive cone angles along the zero and infinity divisors. We will employ the well-known momentum construction (the Calabi ansatz) \cite{Hwang:1994:cscK,Hwang:2002:MomentConstruct}, whereby in the ODE boundary value problem for the momentum profile we will substitute the correct boundary conditions involving the cone angles, which are required for the resultant metric to develop conical singularities along the zero and infinity divisors \cite{Edwards:2019:ContractDivConicKRFlow,Hashimoto:2019:cscKConeSing,Li:2012:eKEngyFunctProjBund,Rubinstein:2022:KEedgeHirzebruch,Schlitzer:2023:dHYM}. In Theorem \ref{thm:mainconesing} and Corollary \ref{cor:mainconesing} we will see that in every \krcl{} of the minimal ruled surface we can get momentum-constructed conical higher cscK metrics for a set of pairs of positive values of the cone angles depending on the value of the parameter characterizing the \krcl{}. The momentum construction method attributed to Hwang-Singer \cite{Hwang:2002:MomentConstruct} applied to the case of \krm{s} with conical singularities (specifically conical higher cscK metrics) is outlined in Subsection \ref{subsec:MomentConstructConeSing} along with all the relevant references. The resultant ODE boundary value problem in the conical higher cscK case (which is given by equation (\ref{eq:ODEBVPConeSing})) is analyzed with its complete solution in Subsection \ref{subsec:AnalysisODEBVPConeSing} and Section \ref{sec:ProofConeSing}, where we need to do some heavy lifting analysis with respect to the various parameters involved in the boundary value problem because the ODE is not readily \intble{} and is non-autonomous and hence the existence of a solution satisfying all the boundary conditions is far from obvious. \par
We will explain in detail in Subsection \ref{subsec:KhlrConeSing} what it exactly means for a \krm{} on a compact complex manifold to develop a conical singularity along a simple closed (complex) hypersurface. We have listed some four definitions of conical \krm{s} having varying degrees of `strongness', which were (according to the needs of our situation) gathered and pieced together from the available literature on conical \krm{s}, e.g. the works of Brendle \cite{Brendle:2013:Ricflatedgesing}, Campana-Guenancia-P\u{a}un \cite{Campana:2013:ConeSingNormCrossDiv}, Datar \cite{Datar:2014:CanonConeSing}, Donaldson \cite{Donaldson:2012:ConeSingDiv}, Hashimoto \cite{Hashimoto:2019:cscKConeSing}, Jeffres-Mazzeo-Rubinstein \cite{Jeffres:2016:KEedgesing}, Song-Wang \cite{Song:2016:RicciGLBconicKE}, Zheng \cite{Zheng:2015:UniqueConeSing} among many others. We will see in Section \ref{sec:polyhomoConeSing} that the conical \krm{s} yielded by the momentum construction (with real analytic momentum profile) satisfy the strongest among these definitions, which we have dubbed as ``\plyhomo{} smooth'' (Definition \ref{def:coneKr4}) due to close analogy with Jeffres-Mazzeo-Rubinstein \cite{Jeffres:2016:KEedgesing}; Theorems 1 and 2. \par
In Section \ref{sec:logFutlogMab} we will write down the expression (\ref{eq:logBFdef}) for the `top $\log$ Bando-Futaki invariant' specially customized to the case of momentum-constructed conical higher cscK metrics on the minimal ruled surface under consideration. This will be done by imitating the expression for the $\log$ Futaki invariant obtained in the works of Hashimoto \cite{Hashimoto:2019:cscKConeSing} (for momentum-constructed conical cscK metrics on the projective completion of a certain kind of \hol{} line bundle over a \kr{}-Einstein Fano manifold), Aoi-Hashimoto-Zheng \cite{Aoi:2025:cscKConeSing} (for conical cscK metrics on smooth projective varieties) and Donaldson \cite{Donaldson:2012:ConeSingDiv}, Li \cite{Li:2015:logKStab} (for conical \kr{}-Einstein metrics singular along an anticanonical divisor of a Fano manifold). We will need to forcefully set this top $\log$ Bando-Futaki invariant to zero, to derive a certain linear relationship (given by equation (\ref{eq:boboomstrghtline})), that we conjecture should exist, between the positive cone angles of the conical higher cscK metrics along the zero and infinity divisors of the surface. \par
But in order to get to the expression (\ref{eq:logBFdef}) for the top $\log$ Bando-Futaki invariant, we will need a ``global interpretation'' of the top Chern form and the `higher scalar curvature' (to be explained in Subsection \ref{subsec:CanonKhlrSmooth} along with all the related preliminaries) of a conical higher cscK metric on the minimal ruled surface in terms of the currents of integration along its zero and infinity divisors (exactly as one would expect by analogy with the case of the $\log$ Futaki invariant acting as an obstruction to the existence of conical cscK or conical \kr{}-Einstein metrics \cite{Aoi:2025:cscKConeSing,Donaldson:2012:ConeSingDiv,Hashimoto:2019:cscKConeSing,Li:2015:logKStab}). This is precisely the expression of currents (\ref{eq:Cherncurrentomega''}) globally characterizing the higher scalar curvature (which is derived in Subsection \ref{subsec:Curvcurrent}), and (\ref{eq:Cherncurrentomega''}) clearly looks analogous to the corresponding equation for momentum-constructed conical (usual) cscK metrics obtained by Hashimoto \cite{Hashimoto:2019:cscKConeSing} as well as to the one for momentum-constructed conical \kr{}-Einstein metrics derived in some works like Rubinstein-Zhang \cite{Rubinstein:2022:KEedgeHirzebruch}, Song-Wang \cite{Song:2016:RicciGLBconicKE}. But as we shall discuss in Subsection \ref{subsec:CanonKhlrConeSing}, in the cases of conical cscK and conical higher cscK metrics, the global expression for the respective notion of curvature (viz. the scalar and higher scalar curvatures) involves taking wedge products of current terms (which are not defined in general), unlike the case of conical \kr{}-Einstein metrics in which only the current of integration over the hypersurface of conical singularity appears (see Hashimoto \cite{Hashimoto:2019:cscKConeSing}, Li \cite{Li:2018:conicMab}, Zheng \cite{Zheng:2015:UniqueConeSing} for the complications that arise in interpreting the wedge products of current terms arising in the conical cscK equation). We will justify the wedge products of current terms involved in the global expression (\ref{eq:Cherncurrentomega''}) in our conical higher cscK case in terms of Bedford-Taylor theory \cite{Bedford:1982:cpctypsh,Bedford:1976:DirichletMongeAmpere} in Subsection \ref{subsec:currenteq}. In Subsection \ref{subsec:SmoothapproxConehcscK} we will explicitly provide approximations to momentum-constructed conical higher cscK metrics by momentum-constructed smooth \krm{s} on the minimal ruled surface, whose smooth top Chern forms converge in the sense of currents to our expected expression (\ref{eq:Cherncurrentomega''}) for the top Chern current of the conical higher cscK metric \cite{Campana:2013:ConeSingNormCrossDiv,Edwards:2019:ContractDivConicKRFlow,Shen:2016:SmoothApproxConeSingRicciLB,Wang:2016:SmoothApproxConeKRFlow}, thereby providing another interpretation to the wedge products of the concerned current terms. \par
Another important thing to be mentioned here is that the top $\log$ Bando-Futaki invariant as well as the respective $\log$ Futaki invariants for conical cscK and conical \kr{}-Einstein metrics require \cohomll{} invariance of the higher scalar curvature and the scalar and Ricci curvatures respectively in order for these algebraic objects to be invariants of the \krcl{} \cite{Aoi:2025:cscKConeSing,Donaldson:2012:ConeSingDiv,Hashimoto:2019:cscKConeSing,Li:2015:logKStab}. The global equations for all three of these notions of curvature given in terms of the current of integration along the divisor of conical singularity should then land up in the correct \cohoml{} classes of the underlying complex manifold, and we will explicitly check this thing in Subsection \ref{subsec:CohomolInvcurrent} for the higher scalar curvature of momentum-constructed conical higher cscK metrics which is given by the expression (\ref{eq:Cherncurrentomega''}).
\subsection{Background of the Paper}\label{subsec:Background}
In our previous paper \cite{Sompurkar:2023:heKsmooth} that is closely related to the present one, we had applied the momentum construction method \cite{Hwang:2002:MomentConstruct} for constructing (smooth) higher \ext{} \krm{s} (see Definition \ref{def:heK}) which turned out to be not higher cscK in each \krcl{} of the same minimal ruled surface (the pseudo-\hbch{} surface \cite{Tonnesen:1998:eKminruledsurf}). We had used the top Bando-Futaki invariant \cite{Bando:2006:HarmonObstruct} (which provides an obstruction to the existence of higher cscK metrics in exactly the same way as the usual Futaki invariant \cite{Calabi:1985:eK2,Futaki:1983:ObstructKE} does in the case of usual cscK or \kr{}-Einstein metrics) to prove that smooth higher cscK metrics (even without Calabi symmetry) do not exist in any \krcl{} of the minimal ruled surface. It is for this reason that in this paper we are exploring the problem of constructing higher cscK metrics with conical singularities along (at least one of) the two special divisors of the surface. The problem of constructing higher \ext{} \krm{s} which we had dealt with in \cite{Sompurkar:2023:heKsmooth} is summarized below with its important results just to set the results of this paper in context: \par
For many \kr{} geometric PDEs (like the ones arising in the definitions of canonical \kr{} metrics) a fertile testing ground is provided by the momentum construction method of Hwang-Singer \cite{Hwang:2002:MomentConstruct} which is used to produce concrete examples of these kinds of metrics on a certain special kind of minimal ruled (complex) surface (more generally on a certain special class of ruled complex manifolds) by exploiting the nice symmetries present in the surface and by imposing the Calabi ansatz on the metric to be constructed (see \cite{Hwang:2002:MomentConstruct} for the details). Following the exposition given in Sz\'ekelyhidi \cite{Szekelyhidi:2014:eKintro}; Section 4.4 of a special case of the momentum construction method for extremal \krm{s} on the minimal ruled surface (a problem that was tackled by T{\o}nnesen-Friedman \cite{Tonnesen:1998:eKminruledsurf}), Pingali \cite{Pingali:2018:heK}; Section 2 had posed the following analogous problem of constructing higher extremal \krm{s} on this surface by this method and the author in his earlier work \cite{Sompurkar:2023:heKsmooth} had provided a complete solution to this problem along with a comparison between the results obtained in the extremal \kr{} and the higher extremal \kr{} analogues of the problem: \par
Let $\left(\Sigma, \omega_\Sigma\right)$ be a genus $2$ Riemann surface equipped with a K\"ahler metric of constant scalar curvature $-2$ (and hence surface area $2 \pi$) and $\left(L, h\right)$ be a degree $-1$ holomorphic line bundle on $\Sigma$ equipped with a Hermitian metric whose curvature form is $-\omega_\Sigma$. Consider the \textit{minimal ruled (complex) surface} $X = \mathbb{P} \left(L \oplus \mathcal{O}\right)$ where $\mathcal{O}$ is the trivial line bundle on $\Sigma$ and $\mathbb{P}$ denotes vector bundle projectivization. Let $\mathsf{C}$ be the typical fibre of $X$, $S_0 = \mathbb{P} \left(\left\lbrace 0 \right\rbrace \oplus \mathcal{O}\right)$ be the \textit{zero divisor} of $X$ and $S_\infty = \mathbb{P} \left(L \oplus \left\lbrace 0 \right\rbrace\right)$ be the \textit{infinity divisor} of $X$, so that $\mathsf{C}$ is a copy of the Riemann sphere $S^2$ (or the complex projective line $\C \PP^1$) sitting in $X$, and $S_0$ and $S_\infty$ are copies of $\Sigma$ sitting in $X$ (with $\Sigma$ being identified with $S_0$ as a (complex) curve in $X$ and $\soo$ being dubbed as the copy of $\Sigma$ in $X$ ``sitting at infinity''). By using the Leray-Hirsch theorem \cite{Demailly:2012:CmplxDifferGeom} and the Nakai-Moishezon criterion \cite{Barth:2004:CmpctCmplxSurf} (see Buchdahl \cite{Buchdahl:1999:CmpctKahler} and Lamari \cite{Lamari:1999:Kcone} for the more general result on compact \kr{} surfaces), the \textit{K\"ahler cone} (i.e. the set of all K\"ahler classes) of $X$ is (up to Poincar\'e duality) precisely the following set (see Fujiki \cite{Fujiki:1992:eKruledmani}; Proposition 1, Lemma 5 and T{\o}nnesen-Friedman \cite{Tonnesen:1998:eKminruledsurf}; Lemma 1 for the more specific result on (minimal) ruled manifolds and (minimal) ruled surfaces respectively):
\begin{equation}\label{eq:KConeX0}
H^{\left(1,1\right)} \left(X, \mathbb{R}\right)^+ = \left\lbrace \hspace{2pt} a \mathsf{C} + b S_\infty \hspace{4pt} \big\vert \hspace{4pt} a, b > 0 \hspace{2pt} \right\rbrace \subseteq H^{\left(1,1\right)} \left(X, \mathbb{R}\right) \subseteq H^2 \left(X, \mathbb{R}\right) = \mathbb{R} \mathsf{C} \oplus \mathbb{R} S_\infty
\end{equation} \par
The higher \ext{} \kr{} equation for a \krm{} $\eta$ on the surface $X$ according to Definition \ref{def:heK} is the following ($c_2 \left(\eta\right)$ being the top \chrn{} form of $\eta$ given by the expression (\ref{eq:deftopChern})):
\begin{equation}\label{eq:topChern0}
c_2 \left(\eta\right) = \frac{\lambda \left(\eta\right)}{2 \left(2 \pi\right)^2} \eta^2
\end{equation}
where $\lambda \left(\eta\right) : X \to \R$ is a smooth function called as the \textit{``higher scalar curvature''} of $\eta$ (defined by equation (\ref{eq:defhcscKheK})) such that $\nabla^{\left(1,0\right)} \lambda \left(\eta\right) = \left(\bar{\partial} \lambda \left(\eta\right)\right)^{\sharp}$ is a real \hol{} vector field on $X$. According to Hwang-Singer \cite{Hwang:2002:MomentConstruct}; Sections 1 and 2, the \textit{Calabi ansatz} for the K\"ahler metric $\eta$ on $X \smallsetminus \left(\so \cup \soo\right)$ is given as follows (as written in for example \cite{Pingali:2018:heK,Szekelyhidi:2014:eKintro}):
\begin{equation}\label{eq:ansatz0}
\eta = \mathtt{p}^* \omega_\Sigma + \sqrt{-1} \partial \bar{\partial} \rho \left(s\right)
\end{equation}
where $\mathtt{p} : X \to \Sigma$ is the fibre bundle projection, $s$ is the logarithm of the square of the fibrewise norm function on $L$ induced by $h$ and $\rho : \R \to \R$ is a strictly convex smooth function chosen suitably such that the function $\R \to \R$, $s \mapsto s + \rho \left(s\right)$ is strictly increasing. The \krm{} $\eta$ given by (\ref{eq:ansatz0}) is supposed to extend smoothly across $\so$ and $\soo$, is (considering (\ref{eq:KConeX0})) required to be in the \krcl{} $\acsoo$ for some $a, b > 0$ and is further required to be higher extremal \kr{}. For the sake of calculations we take $a = 2 \pi$, $b = 2 m \pi$ for some $m > 0$, and we can obtain the results in all the \krcl{es} as the property of being higher extremal \kr{} (or even higher cscK) is invariant under rescaling the metric by a positive constant (see \cite{Sompurkar:2023:heKsmooth}; Subsection 2.3 and Section 4). Imposing these conditions on $\eta$, doing a certain set of calculations in `bundle-adapted' local holomorphic coordinates \cite{Hwang:2002:MomentConstruct} and applying a certain change of variables (called as the \textit{momentum construction}) from $\rho \left(s\right)$, $s \in \R$ to $\psi \left(x\right) = \rho'' \left(s\right)$, $x = 1 + \rho' \left(s\right) \in \left[1, m + 1\right]$ involving the Legendre transform (as done in \cite{Hwang:2002:MomentConstruct,Pingali:2018:heK,Szekelyhidi:2014:eKintro} and followed in \cite{Sompurkar:2023:heKsmooth}), the problem of finding the metric $\eta$ with the desired properties finally boils down to solving the following ODE boundary value problem for $\psi \left(x\right)$, $x \in \left[1, m + 1\right]$ where $A, B, C \in \R$ are arbitrary constants \cite{Pingali:2018:heK,Sompurkar:2023:heKsmooth}:
\begin{equation}\label{eq:ODEBVP0}
\begin{gathered}
\left(2 x + \psi\right) \psi' = A \frac{x^4}{3} + B \frac{x^3}{2} + C x \hspace{2pt}; \hspace{4pt} x \in \left[1, m + 1\right] \\
\psi \left(1\right) = \psi \left(m + 1\right) = 0 \\
\psi' \left(1\right) = - \psi' \left(m + 1\right) = 1 \\
\psi \left(x\right) > 0 \hspace{2pt}; \hspace{4pt} x \in \left(1, m + 1\right)
\end{gathered}
\end{equation}
The ODE in (\ref{eq:ODEBVP0}) is unfortunately not directly integrable (unlike the usual \ext{} \kr{} case of T{\o}nnesen-Friedman \cite{Tonnesen:1998:eKminruledsurf}) and requires a very delicate and tricky analysis for getting the existence of a solution satisfying all the (boundary) conditions of (\ref{eq:ODEBVP0}) for an arbitrary $m > 0$ (see Pingali \cite{Pingali:2018:heK}; Section 2, \cite{Sompurkar:2023:heKsmooth}; Subsection 2.3 and Section 3). \par
Referring to \cite{Sompurkar:2023:heKsmooth}; Subsection 2.2, the higher scalar curvature of the momentum-constructed metric $\eta$ can be written as a linear polynomial in the variable $x$ with the coefficients being in terms of the constants $A, B, C$ appearing in the right hand side of the ODE in (\ref{eq:ODEBVP0}), precisely as $\lambda \left(\eta\right) = A x + B$. We had also obtained the following expressions for $A, B$ in terms of $C, m$ in \cite{Sompurkar:2023:heKsmooth}; Subsection 2.3, equation (2.3.2) by simply substituting the boundary conditions in the ODE in (\ref{eq:ODEBVP0}):
\begin{equation}\label{eq:ABCm0}
\begin{gathered}
A\left(C\right) = \frac{3 C}{m}\left(1 - \frac{1}{\left(m + 1\right)^2}\right) - \frac{6}{m}\left(1 + \frac{1}{\left(m + 1\right)^2}\right) \\
B\left(C\right) = -\frac{2 C}{m}\left(m + 1 - \frac{1}{\left(m + 1\right)^2}\right) + \frac{4}{m}\left(m + 1 + \frac{1}{\left(m + 1\right)^2}\right)
\end{gathered}
\end{equation}
The main result of our previous work \cite{Sompurkar:2023:heKsmooth}; Theorem 2.3.2, \textit{Remark} 2.3.1 states that for each $m > 0$ there exist unique $A, B, C \in \R$ (depending only on $m$) with $A \neq 0$ (this fact is from Pingali \cite{Pingali:2018:heK}; Theorem 1.1, \textit{Remark} 1.1) such that the ODE boundary value problem (\ref{eq:ODEBVP0}) has a unique smooth solution $\psi : \left[1, m + 1\right] \to \R$.
\begin{theorem}[Pingali \cite{Pingali:2018:heK}; Theorem 1.1, \cite{Sompurkar:2023:heKsmooth}; Theorem 2.3.2]\label{thm:heKsmooth}
For each $m > 0$ there exists a unique $C = C \left(m\right) \in \R$ for which the ODE boundary value problem (\ref{eq:ODEBVP0}), with $A = A \left(C \left(m\right)\right) \in \R$, $B = B \left(C \left(m\right)\right) \in \R$ being given by the expressions (\ref{eq:ABCm0}), has a unique smooth solution $\psi \left(x\right)$ on $\cllml$ satisfying all the (boundary) conditions. Further these values of $A, B, C$ which yield the required solution to (\ref{eq:ODEBVP0}) satisfy $C > 2$, $A > 0$, $B < 0$ for any $m > 0$.
\end{theorem}
{\noindent This in turn gives (\cite{Sompurkar:2023:heKsmooth}; Corollary 2.3.1) for each $m > 0$ the existence of a \krm{} $\eta$ satisfying the ansatz (\ref{eq:ansatz0}) and extending smoothly across $\so$ and $\soo$, which is higher \ext{} \kr{} but not higher cscK and which lies in the \krcl{} $\pcsoo$, and then rescaling these metrics by suitable positive constants (\cite{Sompurkar:2023:heKsmooth}; Corollary 2.3.2) gives the existence of the required metrics in all the \krcl{es} of $X$ given by the expression (\ref{eq:KConeX0}).}
\begin{corollary}[\cite{Sompurkar:2023:heKsmooth}; Corollary 4.1]\label{cor:heKsmooth}
For all $a, b > 0$ there exists a smooth K\"ahler metric $\eta$ in the \krcl{} $a \mathsf{C} + b S_\infty$ on the minimal ruled surface $X = \prj$ satisfying the ansatz (\ref{eq:ansatz0}), which is higher extremal K\"ahler but not higher cscK, i.e. $\nabla^{\left(1,0\right)} \lambda \left(\eta\right) = \frac{d}{d x} \left(\lambda \left(\eta\right)\right) = A$ is a non-zero real \hol{} vector field on $X$.
\end{corollary}
{\noindent Corollary \ref{cor:heKsmooth} can be starkly contrasted with the existence result obtained in the usual \ext{} \kr{} analogue of this problem by T{\o}nnesen-Friedman \cite{Tonnesen:1998:eKminruledsurf}, in which non-cscK \ext{} \krm{s} exist only in some \krcl{es} of the minimal ruled surface but not in all, because the \ext{} \kr{} analogue of the ODE boundary value problem (\ref{eq:ODEBVP0}) has a solution only for values of the parameter $m$ smaller than a particular positive value \cite{Szekelyhidi:2014:eKintro}.} \par
We had further proven in \cite{Sompurkar:2023:heKsmooth}; Theorem 4.2, by using the \textit{top Bando-Futaki invariant} introduced by Bando \cite{Bando:2006:HarmonObstruct} and given by the following expression on the \krcl{} $\pcsoo$ (where $Y$ is a \textit{gradient real \hol{} vector field} on the surface $X$ with \textit{real holomorphy potential} $\ttf : X \to \R$, i.e. $Y = \nabla^{\left(1,0\right)} \ttf = \left(\bar{\partial} \ttf\right)^{\sharp}$ and $\lambda_0 \left(\eta\right) = \frac{\int\limits_{X} \lambda \left(\eta\right) \eta^2}{\int\limits_{X} \eta^2}$ is the \textit{average higher scalar curvature} of a \krm{} $\eta$ belonging to $\pcsoo$):
\begin{equation}\label{eq:BFdef0}
\mathcal{F} \left(Y, \pcsoo\right) = -\frac{1}{2 \left(2 \pi\right)^2} \int\limits_X \ttf \left(\lambda \left(\eta\right) - \lambda_0 \left(\eta\right)\right) \eta^2
\end{equation}
that in the K\"ahler class of a higher cscK metric every higher extremal K\"ahler representative has to be higher cscK (compare this with the analogous statement about cscK and \ext{} \krm{s} proven by Calabi \cite{Calabi:1985:eK2}; Theorem 4 using the Futaki invariant \cite{Futaki:1983:ObstructKE}, which is mentioned in T{\o}nnesen-Friedman \cite{Tonnesen:1998:eKminruledsurf}; Proposition 3, Corollary 2 and Sz\'ekelyhidi \cite{Szekelyhidi:2014:eKintro}; Corollary 4.22). This along with Corollary \ref{cor:heKsmooth} finally gives us the following:
\begin{corollary}[\cite{Sompurkar:2023:heKsmooth}; Corollary 4.2]\label{cor:nonexisthcscKsmooth}
There do not exist any (smooth) higher cscK metrics on the minimal ruled surface $X = \prj$.
\end{corollary}
{\noindent Corollary \ref{cor:nonexisthcscKsmooth} agrees exactly by analogy with the non-existence result for (usual) cscK metrics on the minimal ruled surface proven by T{\o}nnesen-Friedman \cite{Tonnesen:1998:eKminruledsurf}.} \par
All the \krm{s} in the above discussed problem on higher \ext{} \krm{s} \cite{Pingali:2018:heK,Sompurkar:2023:heKsmooth} were required to be smooth throughout the surface $X$ as was the case with the analogous problem on \ext{} \krm{s} \cite{Szekelyhidi:2014:eKintro,Tonnesen:1998:eKminruledsurf}. In this paper we will require the \krm{s} to be smooth only on the non-compact surface $X \smallsetminus \left(\so \cup \soo\right)$ and we will allow them to develop conical singularities of some kind (made precise by the discussions in Subsections \ref{subsec:KhlrConeSing} and \ref{subsec:CanonKhlrConeSing}) along the simple closed curves (hypersurfaces) $\so$ and $\soo$ with cone angles $2 \pi \bo > 0$ and $2 \pi \boo > 0$ respectively. We will see that in this setup we can indeed construct conical higher cscK metrics on the compact surface $X$ by the momentum construction method of Hwang-Singer \cite{Hwang:2002:MomentConstruct}. \par
Similar to the ODE boundary value problem (\ref{eq:ODEBVP0}) in the smooth higher \ext{} \kr{} case, the momentum construction applied to the conical higher cscK problem on the minimal ruled surface $X$ throws up the following ODE boundary value problem for some smooth function $\phi : \cllml \to \R$ and some constants $B, C \in \R$ (as we will see in detail in Section \ref{sec:MomentConstructConehcscK}):
\begin{equation}\label{eq:ODEBVPConeSing0}
\begin{gathered}
\left(2 \gamma + \phi\right) \phi' = B \frac{\gamma^3}{2} + C \gamma \hspace{2pt}; \hspace{4pt} \gamma \in \left[1, m + 1\right] \\
\phi\left(1\right) = 0 \hspace{1pt}, \hspace{5pt} \phi\left(m + 1\right) = 0 \\
\phi'\left(1\right) = \beta_0 \hspace{1pt}, \hspace{5pt} \phi'\left(m + 1\right) = - \beta_\infty \\
\phi\left(\gamma\right) > 0 \hspace{2pt}; \hspace{4pt} \gamma \in \left(1, m + 1\right)
\end{gathered}
\end{equation}
The method of attack for solving (\ref{eq:ODEBVPConeSing0}) is very similar to the one used for (\ref{eq:ODEBVP0}) but with some subtle differences in the values and the interdependence of the parameters involved. Proceeding like in \cite{Sompurkar:2023:heKsmooth}; Subsection 2.3 we first derive the expressions (\ref{eq:BCConeSing}) for $B, C$ in terms of $\bo, \boo, m$ analogous to the expressions (\ref{eq:ABCm0}), and then analyze the behaviour of the polynomial $B \frac{\gamma^3}{2} + C \gamma$ appearing in the right hand side of the ODE in (\ref{eq:ODEBVPConeSing0}) in Lemma \ref{lem:polynomconesing} (in Subsection \ref{subsec:AnalysisODEBVPConeSing}). Analogous to Theorem \ref{thm:heKsmooth} we have the existence result for the ODE boundary value problem (\ref{eq:ODEBVPConeSing0}) given by Theorem \ref{thm:mainconesing} (proven in Section \ref{sec:ProofConeSing}) which states that for all positive values of the independent parameters $m, \bo$ there exists a unique positive value of the dependent parameter $\boo$ and there exist unique values of $B, C$ given by the expressions (\ref{eq:BCConeSing}) such that (corresponding to these values of the respective parameters) there exists a unique smooth solution $\phi\left(\gamma\right)$ on $\cllml$ to (\ref{eq:ODEBVPConeSing0}) satisfying all the (boundary) conditions. Corollary \ref{cor:mainconesing} then immediately follows from the construction described in Subsection \ref{subsec:MomentConstructConeSing}, stating that in every \krcl{} of the form $\pcsoo$ of the surface $X$ there exists a momentum-constructed conical higher cscK metric $\omega$ with cone angles $2 \pi \bo$ and $2 \pi \boo$ along the divisors $\so$ and $\soo$ respectively and with higher scalar curvature on $\xsosoo$ (see Definition \ref{def:conehcscK}) being given by $\lambda \left(\omega\right) = B$. Of course then just like in the smooth higher \ext{} \kr{} case, we can rescale the constructed metric $\omega$ by a suitable positive constant to get a conical higher cscK metric with the same cone angles as $\omega$ in the general \krcl{} of the form $\acsoo$ given by the expression (\ref{eq:KConeX0}), because even for conical metrics the property of being higher cscK is invariant under rescaling the metric by a positive constant. \par
According to Definition \ref{def:conehcscK} the higher scalar curvature of the conical higher cscK metric $\omega$ is defined a priori as the smooth (constant) function $\lambda \left(\omega\right)$ on the non-compact surface $\xsosoo$ only. Following the discussion in Subsection \ref{subsec:CanonKhlrConeSing} (specifically referring to equation (\ref{eq:highScalconehcscK})) our real aim in the study of the conical higher cscK equation is to get an expression for the higher scalar curvature of $\omega$ which is globally applicable on the whole of $X$, and this is achieved by the following global expression for the \textit{top \chrn{} current} of $\omega$ (to be derived in Section \ref{sec:highScalcurrent}):
\begin{equation}\label{eq:Cherncurrentomega0}
c_2 \left(\omega\right) = \frac{\lambda \left(\omega\right)}{2 \left(2 \pi\right)^2} \omega^2 + \frac{\beta_0 - 1}{\pi} \omega \wedge \left[\so\right] + \frac{\beta_\infty - 1}{\left(m + 1\right) \pi} \omega \wedge \left[\soo\right]
\end{equation}
where $\left[\so\right]$ and $\left[\soo\right]$ are the currents of integration on $X$ along $\so$ and $\soo$ respectively. The wedge products of the conically singular closed positive $\left(1,1\right)$-form $\omega$ with the closed positive $\left(1,1\right)$-currents $\left[\so\right]$ and $\left[\soo\right]$ are rigorously explained by using two different approaches in this paper, viz. by means of the Bedford-Taylor wedge product \cite{Bedford:1982:cpctypsh,Bedford:1976:DirichletMongeAmpere} in Subsection \ref{subsec:currenteq}, and then in Subsection \ref{subsec:SmoothapproxConehcscK} by taking some special kind of smooth approximations $\omega_\epsilon$ to the conical metric $\omega$ (akin to the method outlined in \cite{Campana:2013:ConeSingNormCrossDiv,Edwards:2019:ContractDivConicKRFlow,Shen:2016:SmoothApproxConeSingRicciLB,Wang:2016:SmoothApproxConeKRFlow}) whose smooth top \chrn{} forms $c_2 \left(\omega_\epsilon\right)$ converge weakly in the sense of currents to $c_2 \left(\omega\right)$ given by (\ref{eq:Cherncurrentomega0}). \par
Finally in Section \ref{sec:logFutlogMab} we introduce the \textit{top $\log$ Bando-Futaki invariant} on the \krcl{} $\pcsoo$ of the minimal ruled surface $X$ given by the following expression (where the vector field $Y$ must be \textit{parallel} to both the curves $\so$ and $\soo$, meaning $Y$ should restrict to vector fields on $\so$ as well as $\soo$):
\begin{multline}\label{eq:logBFdef0}
\mathcal{F}_{\log; \bo, \boo} \left(Y, \pcsoo\right) = -\frac{1}{2 \left(2 \pi\right)^2} \int\limits_X \ttf \left(\lambda \left(\eta\right) - \lambda_0 \left(\eta\right)\right) \eta^2 \\
+ \frac{\beta_0 - 1}{\pi} \left( \int\limits_{S_0} \ttf \eta - \frac{\int\limits_{S_0} \eta}{\int\limits_X \eta^2} \int\limits_X \ttf \eta^2 \right) + \frac{\beta_\infty - 1}{\left(m + 1\right) \pi} \left( \int\limits_{S_\infty} \ttf \eta - \frac{\int\limits_{S_\infty} \eta}{\int\limits_X \eta^2} \int\limits_X \ttf \eta^2 \right)
\end{multline}
We can observe that the invariant in (\ref{eq:logBFdef0}) is precisely equal to the classical top Bando-Futaki invariant \cite{Bando:2006:HarmonObstruct} given by the expression (\ref{eq:BFdef0}) plus two ``correction factors'' along the two divisors of the conical singularities, and (\ref{eq:logBFdef0}) can be reasonably expected to provide an obstruction to the existence of (momentum-constructed) conical higher cscK metrics with cone angles $2\pi\bo > 0$ and $2\pi\boo > 0$ along the divisors $\so$ and $\soo$ respectively in the \krcl{} $\pcsoo$ in exactly the same way as (\ref{eq:BFdef0}) does for smooth higher cscK metrics. So if we evaluate the invariant (\ref{eq:logBFdef0}) at the momentum-constructed smooth higher \ext{} \krm{} $\eta$ yielded by Corollary \ref{cor:heKsmooth} and forcefully set it equal to zero, then we obtain the following linear equation in $\bo, \boo$ dependent on the parameter $m > 0$ determining the \krcl{} under consideration:
\begin{equation}\label{eq:boboomstrghtline0}
\frac{2 \left(m + 3\right)}{m + 2} \beta_\infty - \frac{2 \left(2 m + 3\right)}{m + 2} \beta_0 = \frac{m^2 \left(m^2 + 6 m + 6\right)}{4 \left(m + 1\right)^2} C \left(m\right) - \frac{m \left(m + 2\right)^3}{2 \left(m + 1\right)^2}
\end{equation}
Also note that (\ref{eq:boboomstrghtline0}) gives a precise relationship in terms of $m$ between the value of the parameter $C = C \left(m\right)$ and the values of the parameters $\bo, \boo$ given by \cite{Sompurkar:2023:heKsmooth}; Theorem 2.3.2, Corollary 2.3.1 and by Theorem \ref{thm:mainconesing}, Corollary \ref{cor:mainconesing} in the smooth higher \ext{} \kr{} and the conical higher cscK cases of the momentum construction method respectively, and thus relates the work in our previous paper \cite{Sompurkar:2023:heKsmooth} with that in this paper. But the author would like to admit here that the linear relationship between the cone angles given by equation (\ref{eq:boboomstrghtline0}) is however conjectural, as one of the expected properties of the top $\log$ Bando-Futaki invariant (viz. Conjecture \ref{conj:logBFconiccohomllinvar}) itself could not be proven in its entirety (see Section \ref{sec:logFutlogMab}).
\subsection{Comparison with Other Related Research Works}\label{subsec:AnalogyConeSing}
We just briefly mention here the fact that all the results and expressions that are going to be derived in this paper, and all the calculations and analysis that are involved therein, carry over to general pseudo-Hirzebruch surfaces of genera $\ttg \geq 2$ and degrees $\ttd \neq 0$ \cite{Tonnesen:1998:eKminruledsurf} as well, rather than merely the special case of $\ttg = 2$ and $\ttd = -1$ which we are going to consider. Following the terminology of T{\o}nnesen-Friedman \cite{Tonnesen:1998:eKminruledsurf}; Definition 1 a \textit{pseudo-Hirzebruch surface} is defined to be a minimal ruled complex surface of the form $X = \mathbb{P} \left(L \oplus \mathcal{O}\right)$ where $L$ is a holomorphic line bundle of \textit{degree} $\mathtt{d} \neq 0$ (which means the intersection number $c_1 \left(L\right) \cdot \left[\Sigma\right] = \ttd$ where $c_1 \left(L\right)$ is the first \chrn{} class of $L$ and $\left[\Sigma\right]$ is the fundamental class of $\Sigma$) over a compact Riemann surface $\Sigma$ of genus $\mathtt{g} \geq 2$. We always equip $\Sigma$ with a K\"ahler metric $\omega_\Sigma$ of constant scalar curvature $-2 \left(\mathtt{g} - 1\right)$ (which is the Euler characteristic of $\Sigma$) and we then equip $L$ with a Hermitian metric $h$ of curvature form $\mathtt{d} \omega_\Sigma$. Exactly as it happened in our first paper \cite{Sompurkar:2023:heKsmooth}; Section 5 with the ODE boundary value problem (\ref{eq:ODEBVP0}) in the smooth higher \ext{} \kr{} case of the momentum construction, some factors involving $\ttg$ and $\ttd$ will pop up in the ODE boundary value problem (\ref{eq:ODEBVPConeSing0}) in the conical higher cscK case of this paper as well, and as a result also in the global expression of currents (\ref{eq:Cherncurrentomega0}) for the top Chern current, in the expression (\ref{eq:logBFdef0}) defining the top $\log$ Bando-Futaki invariant and in the equation (\ref{eq:boboomstrghtline0}) giving the linear relationship between the positive cone angles at the two special divisors. Because of lack of space and for the sake of simplicity of the calculations we are skipping all the details that are involved in this generalization to the case $\ttg \geq 2$ and $\ttd \neq 0$, just like the exposition given in Sz\'ekelyhidi \cite{Szekelyhidi:2014:eKintro}; Section 4.4 of the work of T{\o}nnesen-Friedman \cite{Tonnesen:1998:eKminruledsurf} dealing with the construction of smooth (usual) \ext{} \krm{s}, presents only the special case $\ttg = 2$ and $\ttd = -1$. \par
Again just like in the smooth higher \ext{} \kr{} equation with the Calabi ansatz studied in our first paper \cite{Sompurkar:2023:heKsmooth} the genus $\ttg$ of the base \rms{} $\Sigma$ does not seem to be playing a hurdle in getting all the results in the conical higher cscK case of this paper as well. So the author believes that results similar to the ones gotten in this paper for the case $\ttg = 2$ and $\ttd = -1$ (i.e. a pseudo-\hbch{} surface) are obtainable for the case with $\ttg = 0$, i.e. when the base \rms{} is the complex projective line $\C \PP^1$ (or the Riemann sphere $S^2$). A \textit{Hirzebruch surface} is defined to be a ruled complex surface of the form $X = \mathbb{P} \left(\mathcal{O} \left(\mathtt{d}\right) \oplus \mathcal{O}\right)$ where $\mathcal{O} \left(\mathtt{d}\right)$ is the $-\ttd$\textsuperscript{th} tensor power of the tautological line bundle $\mathcal{O} \left(-1\right)$ over $\cp^1$ and $\ttd \neq 0$ \cite{Barth:2004:CmpctCmplxSurf}. Here as well the special case with $\ttd = -1$, i.e. the \hbch{} surface $\mathbb{P} \left(\mathcal{O} \left(-1\right) \oplus \mathcal{O}\right)$ which can also be seen as the blowup of the complex projective plane $\cp^2$ at a point, can be tried for the momentum construction of conical higher cscK metrics along the lines of the work of Calabi \cite{Calabi:1982:eK} which is concerned with the construction of smooth \ext{} \krm{s} on this \hbch{} surface (which is discussed briefly in Sz\'ekelyhidi \cite{Szekelyhidi:2014:eKintro}; Section 4.4, Exercise 4.32). \par
Now to take a short peek into the enormous body of research papers on canonical \krm{s} with conical singularities along certain kinds of hypersurface divisors of compact \krmf{s}, the simplest among these are conical \kr{}-Einstein metrics which have been a topic of active research interest as seen from the works of Donaldson \cite{Donaldson:2012:ConeSingDiv}, Brendle \cite{Brendle:2013:Ricflatedgesing}, Campana-Guenancia-P\u{a}un \cite{Campana:2013:ConeSingNormCrossDiv}, Datar \cite{Datar:2014:CanonConeSing}, Li \cite{Li:2015:logKStab}, Jeffres-Mazzeo-Rubinstein \cite{Jeffres:2016:KEedgesing}, Song-Wang \cite{Song:2016:RicciGLBconicKE} just to name a few. But on the contrary, conical cscK metrics have been studied only in a few works and that too quite recently, e.g. Zheng \cite{Zheng:2015:UniqueConeSing}, Keller-Zheng \cite{Keller:2018:cscKConeSing}, Li \cite{Li:2018:conicMab}, Hashimoto \cite{Hashimoto:2019:cscKConeSing}, Aoi-Hashimoto-Zheng \cite{Aoi:2025:cscKConeSing}, because the conical cscK PDE is in general much harder than the conical \kr{}-Einstein PDE and it involves some more complications like for example giving the ``correct interpretation'' to the wedge products of closed positive currents that will appear in the generalization of the scalar curvature of the conical metric globally on the underlying compact manifold (this same issue of taking wedge products of currents will appear in our conical higher cscK PDE as well, and Section \ref{sec:highScalcurrent} is devoted precisely for this). For all three of these notions of canonical \krm{s} with conical singularities the respective kinds of curvature (as in, the higher scalar curvature for conical higher cscK metrics, and so on) are expected to have global interpretations in terms of the current of integration along the divisor of their singularities, like the model equations (\ref{eq:RicconeKE}), (\ref{eq:ScalconecscK}) and (\ref{eq:highScalconehcscK}) discussed in Subsection \ref{subsec:CanonKhlrConeSing}. For arriving at these types of equations one first needs to compute the global expression for the Ricci curvature current $\Ric \left(\omega\right)$ of the conical \kr{} metric $\omega$ by using a fairly standard method followed in for example Li \cite{Li:2015:logKStab}; Section 2, equation (2) and Hashimoto \cite{Hashimoto:2019:cscKConeSing}; Lemma 4.2, equations (3.11), (4.2) and (4.3) of expressing the (singular) volume form of $\omega$ as:
\begin{equation}\label{eq:defRiccurt1}
\omega^n = \left\lvert z \right\rvert^{2 \beta - 2} \xi
\end{equation}
where $z = 0$ is the divisor of the conical singularity of $\omega$, $2 \pi \beta > 0$ is the cone angle and $\xi$ is some positive $\left(n, n\right)$-form which is smooth away from $z = 0$ and which contains coefficient terms of the order of $\left\lvert z \right\rvert^{2 \beta}$ near $z = 0$. Since $\omega^n$ and $\xi$ will induce their corresponding (singular and non-smooth respectively) \hmnm{s} on the anticanonical line bundle of the underlying \kr{} manifold, we can then simply take the \textit{Ricci curvature current} of $\omega$ to be (using the Poincar\'e-Lelong formula \cite{Demailly:2012:CmplxDifferGeom}):
\begin{equation}\label{eq:defRiccurt2}
\Ric \left(\omega\right) = \Ric \left(\omega^n\right) = \Ric \left(\xi\right) + 2\pi \left(1-\beta\right) \left[z=0\right]
\end{equation}
where $\Ric \left(\omega^n\right)$ and $\Ric \left(\xi\right)$ are the curvature forms of the induced \hmnm{s} $\omega^n$ and $\xi$ respectively and $\left[z=0\right]$ is the current of integration along the divisor $z = 0$. We have used this same method (as described above in equations (\ref{eq:defRiccurt1}), (\ref{eq:defRiccurt2})) for obtaining the current term in the $\left(2, 2\right)$-entry of the curvature form matrix $\Theta \left(\omega\right)$ of the momentum-constructed conical \krm{} $\omega$ in equation (\ref{eq:Curvcurrent22}), but unlike the conical \kr{}-Einstein case the issue that occurs after this point in our conical higher cscK case is that we have to take the determinant of the curvature form matrix (and not its trace) which involves taking wedge products of currents (rather than a plain sum of currents) which then need to be justified rigorously (in the conical cscK case on the other hand one needs to take the wedge product of $\Ric \left(\omega\right)$ given by (\ref{eq:defRiccurt2}) with the conically singular $\left(n-1, n-1\right)$-form $\omega^{n-1}$ in order to get to the global interpretation of the scalar curvature of the conical \krm{} $\omega$ \cite{Hashimoto:2019:cscKConeSing,Li:2018:conicMab}). \par
Hashimoto \cite{Hashimoto:2019:cscKConeSing} had studied conical cscK metrics in the setting of the Calabi ansatz by applying the momentum construction method of \cite{Hwang:2002:MomentConstruct}, while the theory of conical cscK metrics has been developed in much more general settings than the Calabi ansatz (like on smooth projective varieties) by for example Zheng \cite{Zheng:2015:UniqueConeSing}, Keller-Zheng \cite{Keller:2018:cscKConeSing}, Li \cite{Li:2018:conicMab}, Aoi-Hashimoto-Zheng \cite{Aoi:2025:cscKConeSing}. \par
Conical cscK metrics were constructed by Hashimoto \cite{Hashimoto:2019:cscKConeSing} on (actual) Hirzebruch surfaces \cite{Barth:2004:CmpctCmplxSurf}, and in fact on a much more general class of ruled complex $n$-manifolds, viz. the projective completions of pluricanonical and plurianticanonical line bundles over \kr{}-Einstein Fano $\left(n-1\right)$-manifolds with second Betti number $1$. In this paper we are dealing with the conical higher cscK analogue of the problem of \cite{Hashimoto:2019:cscKConeSing}, with the only difference in the setup being that the base manifold on which the projective fibre bundle lies is a negatively curved (hyperbolic) \rms{} in our case. Hashimoto \cite{Hashimoto:2019:cscKConeSing}; Theorem 1.12 gives the global equation of currents for the scalar curvature of a momentum-constructed conical cscK metric, which is like the model equation (\ref{eq:ScalconecscK}) briefly discussed in Subsection \ref{subsec:CanonKhlrConeSing}. We can clearly observe the analogy between the conical scalar curvature equation in \cite{Hashimoto:2019:cscKConeSing}; Theorem 1.12 and the equation of currents (\ref{eq:Cherncurrentomega0}) for the higher scalar curvature of a momentum-constructed conical higher cscK metric that we have obtained in this paper. We are using Bedford-Taylor theory \cite{Bedford:1982:cpctypsh,Bedford:1976:DirichletMongeAmpere} (in Subsection \ref{subsec:currenteq}) and also the method of taking smooth approximations to conical metrics \cite{Campana:2013:ConeSingNormCrossDiv,Edwards:2019:ContractDivConicKRFlow,Shen:2016:SmoothApproxConeSingRicciLB,Wang:2016:SmoothApproxConeKRFlow} (in Subsection \ref{subsec:SmoothapproxConehcscK}) for giving meaning to the wedge products of currents that appear in our conical higher cscK equation, but we could have also used the method detailed in Hashimoto \cite{Hashimoto:2019:cscKConeSing}; Subsection 4.2 in which one first thinks of these wedge products as the corresponding integrals (as remarked in equation (\ref{eq:wedgenaivint}) in Subsection \ref{subsec:CanonKhlrConeSing}) and then proves that these integrals are indeed well-defined and finite by using various estimates and some limiting arguments. Even the fact, that these wedge products of currents can be acted on non-smooth test functions which are asymptotically of the order of $\left\lvert z \right\rvert^{\beta}$ near the divisor of conical singularity given by $z = 0$ (where again $2\pi\beta > 0$ is the cone angle), which will be needed in the proof of Theorem \ref{thm:logBFconicevalue} to justify the evaluation of the top $\log$ Bando-Futaki invariant with respect to a conically singular metric, relates well with the work in \cite{Hashimoto:2019:cscKConeSing}; Subsubsections 4.3.2 and 4.3.3. And most notably this top $\log$ Bando-Futaki invariant given by the expression (\ref{eq:logBFdef0}), which we are going to study in Section \ref{sec:logFutlogMab}, is directly seen to be the top-dimensional analogue of the $\log$ Futaki invariant of Hashimoto \cite{Hashimoto:2019:cscKConeSing}; equation (2.3), with the rigorous verifications in Section \ref{sec:logFutlogMab} being along similar lines as those in \cite{Hashimoto:2019:cscKConeSing}; Subsection 4.3, Section 5. \par
However the most important difference between our higher scalar curvature equation with the Calabi ansatz and its usual scalar curvature counterpart is that (as we will be seeing in equation (\ref{eq:lambdaconesing}) in Subsection \ref{subsec:MomentConstructConeSing}) the higher scalar curvature $\lambda \left(\omega\right)$ of the momentum-constructed metric $\omega$ has a second-order fully non-linear differential expression in terms of the momentum profile (which is the function $\phi \left(\gamma\right)$ seen in the ODE (\ref{eq:ODEBVPConeSing0})) whereas the corresponding expression for the scalar curvature is second-order linear \cite{Hashimoto:2019:cscKConeSing,Hwang:2002:MomentConstruct,Szekelyhidi:2014:eKintro,Tonnesen:1998:eKminruledsurf}. The ODE (\ref{eq:ODEBVPConeSing0}) arising in the momentum construction of conical higher cscK metrics (just like the ODE (\ref{eq:ODEBVP0}) in the smooth higher \ext{} \kr{} problem \cite{Sompurkar:2023:heKsmooth}) cannot be solved explicitly in closed form by any of the elementary ODE methods (being a version of Chini's equation) \cite{Pingali:2018:heK}, and hence it requires a very delicate analysis for getting the existence of a solution satisfying all the concerned boundary conditions and that too for each \krcl{}. But on the contrary, for the ODE yielded by the scalar curvature equation the momentum profile is obtainable as an explicit rational function which can then be solved for the boundary conditions individually, as demonstrated in \cite{Hashimoto:2019:cscKConeSing,Hwang:2002:MomentConstruct,Szekelyhidi:2014:eKintro,Tonnesen:1998:eKminruledsurf}. \par
Li \cite{Li:2018:conicMab} had also studied the conical cscK equation but in a very different setup than Hashimoto \cite{Hashimoto:2019:cscKConeSing}, and had employed the non-pluripolar wedge product of Boucksom-Eyssidieux-Guedj-Zeriahi \cite{Boucksom:2010:MAeqnonpluripolarprod} (a generalization of the Bedford-Taylor product for unbounded \psh{} functions which charges no mass on the divisor of the conical singularity as the divisor is a pluripolar set) to interpret the wedge product of the current terms that appeared in the conical cscK equation in \cite{Li:2018:conicMab}, and had further noted that if instead the Bedford-Taylor-Demailly wedge product \cite{Demailly:2012:CmplxDifferGeom} (a different generalization of the Bedford-Taylor product for unbounded \psh{} functions which charges mass on the divisor of the conical singularity) is used to determine the wedge product of the concerned current terms then the answer will be different from the case of the non-pluripolar product. In our case of the conical higher cscK equation with the Calabi symmetry condition this problem does not arise since the momentum construction ensures that the \psh{} functions which turn out to be the potentials of the closed positive currents (viz. the functions $\ln \left(1 + f' \left(s\right)\right)$, $f \left(s\right)$ and $f \left(s\right) - m s$ to be seen in Subsection \ref{subsec:currenteq}, where $f \left(s\right)$ is the convex function yielding the Calabi ansatz similar to the function $\rho \left(s\right)$ in equation (\ref{eq:ansatz0})) are bounded at least locally around the divisors. \par
Just like the top $\log$ Bando-Futaki invariant is designed to provide an algebro-geometric obstruction to the existence of conical higher cscK metrics in a given \krcl{}, the notions of a `higher $\log$ Mabuchi functional', `higher $\log$ $K$-stability' would also characterize conical higher cscK metrics and would provide deeper insights into their study, and so we would be interested in studying these things in our future works by following the already well-known analogous notions of these (which are the $\log$ Futaki invariant, the $\log$ Mabuchi functional and $\log$ $K$-stability respectively) given in the works \cite{Aoi:2025:cscKConeSing,Hashimoto:2019:cscKConeSing,Keller:2018:cscKConeSing,Li:2012:eKEngyFunctProjBund,Li:2018:conicMab,Zheng:2015:UniqueConeSing} in the case of conical cscK metrics and in the works \cite{Donaldson:2012:ConeSingDiv,Jeffres:2016:KEedgesing,Li:2015:logKStab,Song:2016:RicciGLBconicKE} in the case of conical \kr{}-Einstein metrics. \par
Conical \ext{} \krm{s} are another obvious and important notion of canonical \krm{s} with conical singularities about which the author could find only one work viz. Li \cite{Li:2012:eKEngyFunctProjBund} which studies them in the setting of the Calabi ansatz on the same minimal ruled surface (the pseudo-\hbch{} surface \cite{Tonnesen:1998:eKminruledsurf}) as ours. Li \cite{Li:2012:eKEngyFunctProjBund}; Corollary 1.2, Theorem 3.2 state that conical \ext{} \krm{s} can be constructed in all the \krcl{es} of the surface, in contrast to the fact that smooth \ext{} \krm{s} do not exist in some \krcl{es}, which is proven in \cite{Szekelyhidi:2014:eKintro,Tonnesen:1998:eKminruledsurf}. Looking at the naturally analogous concept of a conical higher \ext{} \krm{} might also lead us to some new research directions.
\numberwithin{equation}{subsection}
\numberwithin{figure}{subsection}
\numberwithin{table}{subsection}
\numberwithin{lemma}{subsection}
\numberwithin{proposition}{subsection}
\numberwithin{result}{subsection}
\numberwithin{theorem}{subsection}
\numberwithin{corollary}{subsection}
\numberwithin{conjecture}{subsection}
\numberwithin{remark}{subsection}
\numberwithin{note}{subsection}
\numberwithin{motivation}{subsection}
\numberwithin{question}{subsection}
\numberwithin{answer}{subsection}
\numberwithin{case}{subsection}
\numberwithin{claim}{subsection}
\numberwithin{definition}{subsection}
\numberwithin{example}{subsection}
\numberwithin{hypothesis}{subsection}
\numberwithin{statement}{subsection}
\numberwithin{ansatz}{subsection}
\section{Preliminaries}\label{sec:Prelim}
\subsection{The Definitions of the Requisite Notions of Canonical K\"ahler Metrics}\label{subsec:CanonKhlrSmooth}
A recurring theme in \kr{} geometry (or even in \hmn{} and \rmn{} geometry) is to find ``canonical'' metrics on (complex or real) manifolds, i.e. to find metrics which enjoy some nice properties with respect to their curvatures \cite{Szekelyhidi:2014:eKintro}. Along these lines in \kr{} geometry we have \textit{\kr{}-Einstein metrics} which are special cases of \textit{constant scalar curvature \kr{} (cscK) metrics} on compact \krmf{s} \cite{Szekelyhidi:2014:eKintro}. Calabi \cite{Calabi:1982:eK,Calabi:1985:eK2} introduced \textit{extremal K\"ahler metrics} as the critical points of a certain energy \funcl{} defined on a fixed \krcl{} (called as the Calabi \funcl{}) and these are a further generalization of cscK metrics. These three notions involve the Ricci curvature form (or the first Chern form for \krm{s}) satisfying some nice equations as we will briefly see below \cite{Szekelyhidi:2014:eKintro}: \par
Let $M$ be a compact \kr{} $n$-manifold and $\omega$ be a \krm{} on $M$. The \textit{Ricci curvature form} of $\omega$ is defined as \cite{Szekelyhidi:2014:eKintro}:
\begin{equation}\label{eq:defRic}
\Ric \left(\omega\right) = - \ideldb \ln \det \left(\omega\right)
\end{equation}
where $\det \left(\omega\right) = \det H \left(\omega\right)$, $H \left(\omega\right)$ being the Hermitian matrix of (the underlying \hmnm{} associated with) $\omega$ given in terms of local holomorphic coordinates on $M$. The metric $\omega$ is said to be \textit{\kr{}-Einstein} if there exists a constant $\lambda \in \R$ such that $\Ric \left(\omega\right) = \lambda \omega$, in which case the constant $\lambda$ is called as the \textit{Ricci curvature} of $\omega$ \cite{Szekelyhidi:2014:eKintro}. \par
The \textit{scalar curvature} of $\omega$, denoted by $S \left(\omega\right) : M \to \R$, can be read off from the following formula \cite{Szekelyhidi:2014:eKintro}:
\begin{equation}\label{eq:defScal}
n \Ric \left(\omega\right) \wedge \omega^{n-1} = S \left(\omega\right) \omega^n
\end{equation}
The metric $\omega$ is said to be \textit{constant scalar curvature \kr{} (cscK)} if (as the name suggests) $S \left(\omega\right) \in \R$ is a constant \cite{Szekelyhidi:2014:eKintro}. \par
Letting $\Omega^{\left(0,1\right)} \left(M\right)$ and $\mathfrak{X}^{\left(1,0\right)} \left(M\right)$ denote the set of all real smooth $\left(0,1\right)$-forms on $M$ and the set of all real smooth $\left(1,0\right)$-vector fields on $M$ respectively, and $\flat : \mathfrak{X}^{\left(1,0\right)} \left(M\right) \to \Omega^{\left(0,1\right)} \left(M\right)$ and $\sharp : \Omega^{\left(0,1\right)} \left(M\right) \to \mathfrak{X}^{\left(1,0\right)} \left(M\right)$ denote the `musical isomorphisms' induced by (the underlying \hmnm{} associated with) $\omega$, then the set of all \textit{real holomorphic vector fields} on $M$ is defined as \cite{Szekelyhidi:2014:eKintro}:
\begin{equation}\label{eq:defholVF}
\fkh \left(M\right) = \left\lbrace \hspace{2pt} Y \in \mathfrak{X}^{\left(1,0\right)} \left(M\right) \hspace{4pt} \Big\vert \hspace{4pt} \db Y = 0 \hspace{2pt} \right\rbrace
\end{equation}
The metric $\omega$ is said to be \textit{extremal \kr{}} if $\nabla^{\left(1,0\right)} S \left(\omega\right) = \left(\bar{\partial} S \left(\omega\right)\right)^{\sharp} \in \fkh \left(M\right)$, where $\nabla^{\left(1,0\right)} \left(\cdot\right) = \left(\bar{\partial} \left(\cdot\right) \right)^{\sharp}$ denotes the $\left(1,0\right)$-gradient with respect to $\omega$ (see Calabi \cite{Calabi:1982:eK}; Theorem 2.1 and Sz\'ekelyhidi \cite{Szekelyhidi:2014:eKintro}; Theorem 4.2). \par
These three definitions are directly related to the notions of the \textit{first Chern form} of the \krm{} $\omega$ and the \textit{first Chern class} of $M$ (which turns out to be independent of the choice of the \krm{} $\omega$) \cite{Szekelyhidi:2014:eKintro}.
\begin{equation}\label{eq:deffirstChern}
c_1 \left(\omega\right) = \frac{1}{2 \pi} \Ric \left(\omega\right) \hspace{1pt}, \hspace{5pt} c_1 \left(M\right) = \left[c_1 \left(\omega\right)\right] = \frac{1}{2 \pi} \left[\Ric \left(\omega\right)\right] \in H^{\left(1,1\right)} \left(M, \mathbb{R}\right)
\end{equation}
where $H^{\left(1,1\right)} \left(M, \mathbb{R}\right)$ is the real cohomology space of all closed real smooth $\left(1,1\right)$-forms on $M$. \par
An equivalent characterization for cscK metrics is, $\omega$ is a cscK metric on $M$ if and only if its first Chern form $c_1 \left(\omega\right)$ is \textit{harmonic}, i.e. $\Delta c_1 \left(\omega\right) = - \bar{\partial}^* \bar{\partial} c_1 \left(\omega\right) - \bar{\partial} \bar{\partial}^* c_1 \left(\omega\right) = 0$ \cite{Calabi:1985:eK2,Futaki:1983:ObstructKE}, where $\Delta = - \bar{\partial}^* \bar{\partial} - \bar{\partial} \bar{\partial}^*$ is the Hodge $\bar{\partial}$-Laplacian on $M$ and $\bar{\partial}^*$ is the Hodge formal adjoint of $\bar{\partial}$ \cite{Demailly:2012:CmplxDifferGeom}. \par
It can be verified that the \textit{average scalar curvature} of $\omega$, given by $\widehat{S} \left(\omega\right) = \frac{\int\limits_{M} S \left(\omega\right) \omega^n}{\int\limits_{M} \omega^n} \in \R$, is a cohomological invariant of the \krcl{} of $\omega$, as in it depends only on the real \cohoml{} classes $\left[ \omega \right], c_1 \left(M\right) \in H^{\left(1,1\right)} \left(M, \mathbb{R}\right)$ in the following way \cite{Szekelyhidi:2014:eKintro}:
\begin{equation}\label{eq:avgScal}
\widehat{S} \left(\omega\right) = \frac{2 n \pi c_1 \left(M\right) \smile \left[ \omega \right]^{n-1}}{\left[ \omega \right]^n}
\end{equation}
where $\smile$ denotes the cup product of real \cohoml{} classes of $M$. Clearly then $\omega$ is a cscK metric in its \krcl{} $\left[ \omega \right]$ if and only if $S \left(\omega\right) = \widehat{S} \left(\omega\right) \in \R$ \cite{Calabi:1985:eK2,Futaki:1983:ObstructKE,Szekelyhidi:2014:eKintro}. \par
Taking the analogy of these definitions to the level of the top real cohomology, Pingali \cite{Pingali:2018:heK} introduced \textit{`higher constant scalar curvature K\"ahler (higher cscK) metrics'} and \textit{`higher extremal K\"ahler metrics'}. Denoting by $\Theta \left(\omega\right)$ the \textit{curvature form matrix} of the \krm{} $\omega$ (which in terms of its Hermitian matrix $H \left(\omega\right)$ in local holomorphic coordinates is given by $\Theta \left(\omega\right) = \bar{\partial} \left(H^{-1} \partial H\right) \left(\omega\right)$), the \textit{top Chern form} of $\omega$ and the \textit{top Chern class} of $M$ (which is again independent of the choice of $\omega$) are defined as follows \cite{Demailly:2012:CmplxDifferGeom}:
\begin{equation}\label{eq:deftopChern}
c_n \left(\omega\right) = \det \left( \frac{\sqrt{-1}}{2 \pi} \Theta \left(\omega\right) \right) \hspace{1pt}, \hspace{5pt} c_n \left(M\right) = \left[c_n \left(\omega\right)\right] \in H^{\left(n,n\right)} \left(M, \mathbb{R}\right)
\end{equation}
where $H^{\left(n,n\right)} \left(M, \mathbb{R}\right)$ is the top-dimensional real cohomology space of $M$. Since $M$ is compact and orientable and $\omega^n$ is nowhere vanishing, there exists a unique smooth (bounded) function $\lambda \left(\omega\right) : M \to \R$ (which can be dubbed as the \textit{``higher scalar curvature''} of $\omega$) such that \cite{Pingali:2018:heK,Sompurkar:2023:heKsmooth}:
\begin{equation}\label{eq:defhcscKheK}
c_n \left(\omega\right) = \frac{\lambda \left(\omega\right)}{n! \left(2 \pi\right)^n} \omega^n
\end{equation}
\begin{definition}[Higher cscK Metric; Pingali \cite{Pingali:2018:heK}]\label{def:hcscK}
The \kr{} metric $\omega$ on $M$ is said to be \textit{higher constant scalar curvature K\"ahler (higher cscK)} if $\lambda \left(\omega\right) \in \R$ is a constant in equation (\ref{eq:defhcscKheK}).
\end{definition}
\begin{definition}[Higher Extremal K\"ahler Metric; Pingali \cite{Pingali:2018:heK}]\label{def:heK}
The \kr{} metric $\omega$ on $M$ is said to be \textit{higher extremal K\"ahler} if $\nabla^{\left(1,0\right)} \lambda \left(\omega\right) = \left(\bar{\partial} \lambda \left(\omega\right)\right)^{\sharp} \in \fkh \left(M\right)$ in equation (\ref{eq:defhcscKheK}).
\end{definition}
{\noindent Observing the analogy between equations (\ref{eq:defScal}) and (\ref{eq:defhcscKheK}), one sees that Definitions \ref{def:hcscK} and \ref{def:heK} are motivated by the above three ``classical'' notions of canonical metrics in \kr{} geometry with the first Chern form being replaced by the top Chern form \cite{Pingali:2018:heK}.} \par
Just like for cscK metrics, there are two other equivalent characterizations for higher cscK metrics. In terms of \textit{\har{ity}} of closed forms (given in terms of Hodge theory) \cite{Demailly:2012:CmplxDifferGeom}, $\omega$ is higher cscK if and only if its top Chern form $c_n \left(\omega\right)$ is \textit{harmonic}, i.e. $\Delta c_n \left(\omega\right) = - \bar{\partial}^* \bar{\partial} c_n \left(\omega\right) - \bar{\partial} \bar{\partial}^* c_n \left(\omega\right) = 0$ \cite{Bando:2006:HarmonObstruct,Calabi:1985:eK2,Futaki:1983:ObstructKE,Pingali:2018:heK}. The other one is analogously given in terms of the \textit{average higher scalar curvature} of $\omega$, which is defined as $\lambda_0 \left(\omega\right) = \frac{\int\limits_{M} \lambda \left(\omega\right) \omega^n}{\int\limits_{M} \omega^n} \in \R$ and which can again be checked to be a cohomological invariant of the \krcl{} of $\omega$, meaning it depends only on the real \cohoml{} classes $\left[ \omega \right] \in H^{\left(1,1\right)} \left(M, \mathbb{R}\right)$, $c_n \left(M\right) \in H^{\left(n,n\right)} \left(M, \mathbb{R}\right)$ in the following way \cite{Bando:2006:HarmonObstruct,Pingali:2018:heK,Sompurkar:2023:heKsmooth,Szekelyhidi:2014:eKintro}:
\begin{equation}\label{eq:avghighScal}
\lambda_0 \left(\omega\right) = \frac{n! \left(2 \pi\right)^n c_n \left(M\right)}{\left[ \omega \right]^n}
\end{equation}
and then the formulation is obviously, $\omega$ is higher cscK if and only if $\lambda \left(\omega\right) = \lambda_0 \left(\omega\right) \in \R$ \cite{Bando:2006:HarmonObstruct,Pingali:2018:heK,Sompurkar:2023:heKsmooth}.
\subsection{Some Definitions of K\"ahler Metrics with Conical Singularities}\label{subsec:KhlrConeSing}
We will now briefly look at some definitions of \krm{s} on compact complex manifolds having conical singularities of positive cone angles along simple closed (complex) hypersurfaces of the manifolds, attributed to Brendle \cite{Brendle:2013:Ricflatedgesing}, Donaldson \cite{Donaldson:2012:ConeSingDiv}, Jeffres-Mazzeo-Rubinstein \cite{Jeffres:2016:KEedgesing}, Song-Wang \cite{Song:2016:RicciGLBconicKE} among many others. Starting with the basic motivation behind the definitions, a \textit{K\"ahler metric with a conical singularity} is a \krc{} which represents a smooth \krm{} away from the hypersurface of its conical singularity and which in a \nbd{} of the hypersurface is ``comparable'' with the \textit{model edge metric} on $\C^n$ (given by the expression (\ref{eq:modeledge})) in some or the other way \cite{Donaldson:2012:ConeSingDiv,Jeffres:2016:KEedgesing}. The sense in which it is comparable with the model edge metric on $\C^n$ gives us various different notions of conical \krm{s} (out of which we are going to be seeing only four, viz. Definitions \ref{def:coneKr1}, \ref{def:coneKr2}, \ref{def:coneKr3} and \ref{def:coneKr4}) \cite{Donaldson:2012:ConeSingDiv,Jeffres:2016:KEedgesing}. The theory of conical \krm{s} (especially conical \kr{}-Einstein metrics) has been one of active research interest, starting with the early works of for example Troyanov \cite{Troyanov:1991:CurvCmpctSurfConeSing}, Luo-Tian \cite{Luo:1992:LiouvillePolytope} (who considered the special case of metrics on \rms{s} with conical singularities on a finite set of points) and going all the way till the later works of Donaldson \cite{Donaldson:2012:ConeSingDiv}, Brendle \cite{Brendle:2013:Ricflatedgesing}, Jeffres-Mazzeo-Rubinstein \cite{Jeffres:2016:KEedgesing}, Song-Wang \cite{Song:2016:RicciGLBconicKE} (which deal with the more general case of metrics on \krmf{s} having conical singularities along a smooth effective simple normal crossing divisor). Donaldson \cite{Donaldson:2012:ConeSingDiv} and Jeffres-Mazzeo-Rubinstein \cite{Jeffres:2016:KEedgesing} have developed elaborate linear theories of various kinds of weighted H\"older spaces that are involved in the many definitions of conical \krm{s} (satisfying various additional properties) that they have dealt with in their works, but for our purpose in this paper we will be dealing only with what are called as \textit{`smooth conical \krm{s}'} (which are basically conical \krm{s} which can be ``smoothened out'' by two certain singular coordinate transformations as we shall see in Definition \ref{def:coneKr2}) \cite{Donaldson:2012:ConeSingDiv,Jeffres:2016:KEedgesing,Song:2016:RicciGLBconicKE}, and hence the entire linear theories given in \cite{Donaldson:2012:ConeSingDiv,Jeffres:2016:KEedgesing} are not needed at least for this paper. \kr{}-Einstein metrics with conical singularities are also studied in the works of Campana-Guenancia-P\u{a}un \cite{Campana:2013:ConeSingNormCrossDiv}, Datar \cite{Datar:2014:CanonConeSing}, Li \cite{Li:2015:logKStab}, Shen \cite{Shen:2016:SmoothApproxConeSingRicciLB}, Rubinstein-Zhang \cite{Rubinstein:2022:KEedgeHirzebruch} and many others, while only a few works like Zheng \cite{Zheng:2015:UniqueConeSing}, Keller-Zheng \cite{Keller:2018:cscKConeSing}, Li \cite{Li:2018:conicMab}, Hashimoto \cite{Hashimoto:2019:cscKConeSing}, Aoi-Hashimoto-Zheng \cite{Aoi:2025:cscKConeSing}, Li \cite{Li:2012:eKEngyFunctProjBund} deal with conical cscK metrics (the last one also with conical \ext{} \krm{s}). \par
The \textit{model edge metric} on $\C^n$ in standard coordinates $\left(z_1, \ldots, z_{n-1}, z\right)$ with \textit{cone angle} $2 \pi \beta > 0$ along the \textit{hyperplane} $\left\lbrace z = 0 \right\rbrace \subseteq \C^n$ is the following \krf{} (together with its corresponding \krm{}) on $\C^n \smallsetminus \left\lbrace z = 0 \right\rbrace$ \cite{Brendle:2013:Ricflatedgesing,Datar:2014:CanonConeSing,Donaldson:2012:ConeSingDiv,Jeffres:2016:KEedgesing}:
\begin{equation}\label{eq:modeledge}
\begin{gathered}
\omega_\beta = \sqrt{-1} \sum\limits_{i=1}^{n-1} d z_i \wedge d \bar{z}_i + \sqrt{-1} \left\lvert z \right\rvert^{2 \beta - 2} d z \wedge d \bar{z} \hspace{1.5pt}; \hspace{5pt} g_\beta = 2 \sum\limits_{i=1}^{n-1} \left\lvert d z_i \right\rvert^2 + 2 \left\lvert z \right\rvert^{2 \beta - 2} \left\lvert d z \right\rvert^2
\end{gathered}
\end{equation}
Since $\omega_\beta$ is locally integrable on the whole of $\C^n$, we regard $\omega_\beta$ as a \textit{\krc{}} (i.e. a closed strictly positive $\left(1,1\right)$-current) on $\C^n$ \cite{Demailly:2012:CmplxDifferGeom}. Since $\omega_\beta = \ideldb \left( \sum\limits_{i=1}^{n-1} \left\lvert z_i \right\rvert^2 + \frac{\left\lvert z \right\rvert^{2 \beta}}{\beta^2} \right)$ on $\C^n$, we have $\sum\limits_{i=1}^{n-1} \left\lvert z_i \right\rvert^2 + \frac{\left\lvert z \right\rvert^{2 \beta}}{\beta^2}$ as a global \krp{} for $\omega_\beta$ on $\C^n$ \cite{Datar:2014:CanonConeSing,Demailly:2012:CmplxDifferGeom}.
\begin{remark}
In this paper we are going to be dealing with conical singularities of arbitrary positive cone angles, so in the expression (\ref{eq:modeledge}) and in the Definitions \ref{def:coneKr1}, \ref{def:coneKr2}, \ref{def:coneKr3} and \ref{def:coneKr4} we are going to be assuming the condition $\beta \in \left(0, \infty\right)$ only. But we note here that if $\beta \in \left(0, 1\right)$ then the model edge metric $\omega_\beta$ on $\C^n$ as well as the conical \krm{s} $\omega$ on the general \krmf{} $M$ to be seen in these four definitions will have a pole of fractional order $2 - 2 \beta$ and hence will be singular along the divisor $\left\lbrace z = 0 \right\rbrace$, while if $\beta \in \left(1, \infty\right)$ then $\omega_\beta$ and $\omega$ will have a zero of fractional multiplicity $2 \beta - 2$ and hence will be degenerate along $\left\lbrace z = 0 \right\rbrace$, and the case $\beta = 1$ obviously yields the standard smooth metric $\omega_1$ on $\C^n$ and smooth \krm{s} $\omega$ on the manifold $M$ respectively \cite{Hashimoto:2019:cscKConeSing,Jeffres:2016:KEedgesing,Li:2012:eKEngyFunctProjBund,Li:2018:conicMab,Schlitzer:2023:dHYM}.
\end{remark} \par
Let $M$ be a compact \kr{} $n$-manifold and $D \subseteq M$ be a \textit{smooth simple closed (complex) hypersurface} (i.e. a (complex) codimension-$1$ submanifold with no self-intersections). We first see the weakest notion of a \krm{} on $M$ with a \textit{conical singularity of cone angle} $2 \pi \beta > 0$ along $D$ \cite{Brendle:2013:Ricflatedgesing,Datar:2014:CanonConeSing,Donaldson:2012:ConeSingDiv,Jeffres:2016:KEedgesing}.
\begin{definition}[Conical \kr{} Metric; \cite{Brendle:2013:Ricflatedgesing,Datar:2014:CanonConeSing,Donaldson:2012:ConeSingDiv,Jeffres:2016:KEedgesing}]\label{def:coneKr1}
A \krc{} $\omega$ on $M$ is said to be a \textit{conical \krm{}} on $M$ with cone angle $2 \pi \beta > 0$ along $D$ if:
\begin{enumerate}
\item $\omega$ is a smooth \krm{} on the (non-compact) \kr{} $n$-manifold $M \smallsetminus D$.
\item {Around every point of $D$ there exist local \hol{} coordinates $\left(z_1, \ldots, z_{n-1}, z\right)$ with $D$ being given by $\left\lbrace z = 0 \right\rbrace$ and with $\left(z_1, \ldots, z_{n-1}\right)$ restricting to a complex coordinate chart on $D$, such that $\omega$ (expressed in the coordinates $\left(z_1, \ldots, z_{n-1}, z\right)$) is \textit{asymptotically quasi-isometric} to the model edge metric $\omega_\beta$ on $\C^n$ (described in the expression (\ref{eq:modeledge})), i.e. there exist constants $C_1, C_2 > 0$ such that (locally around the point of $D$) we have:
\begin{equation}\label{eq:qisometr}
C_1 \omega_\beta \leq \omega \leq C_2 \omega_\beta
\end{equation}
where for closed $\left(1,1\right)$-currents $\xi, \eta$ on $M$ we say $\xi \leq \eta$ if the current $\eta - \xi$ is semipositive.}\label{itm:qisometr}
\end{enumerate}
\end{definition} \par
Since Definition \ref{def:coneKr1} is too general, we formulate a much stronger notion of a conical \krm{}. Define a new (non-\hol{}) singular coordinate $\zeta = \left\lvert z \right\rvert^{\beta-1} z$ on $\C$ \cite{Brendle:2013:Ricflatedgesing,Datar:2014:CanonConeSing,Donaldson:2012:ConeSingDiv,Jeffres:2016:KEedgesing}. Then in the \textit{``conical coordinates''} (singular coordinates) $\left(z_1, \ldots, z_{n-1}, \zeta\right)$ on $\C^n$ the expression for the model edge metric $\omega_\beta$ turns out to be similar to that for the standard smooth metric on $\C^n$ in the coordinates $\left(z_1, \ldots, z_{n-1}, \zeta\right)$ as seen from the following \cite{Brendle:2013:Ricflatedgesing,Datar:2014:CanonConeSing,Donaldson:2012:ConeSingDiv,Jeffres:2016:KEedgesing}:
\begin{equation}\label{eq:zetaz}
\begin{gathered}
d \zeta = \frac{\beta+1}{2} \left\lvert z \right\rvert^{\beta-1} d z + \frac{\beta-1}{2} \frac{z}{\bar{z}} \left\lvert z \right\rvert^{\beta-1} d \bar{z} \\
\sqrt{-1} d \zeta \wedge d \bar{\zeta} = \sqrt{-1} \beta \left\lvert z \right\rvert^{2 \beta - 2} d z \wedge d \bar{z} \hspace{1.5pt}; \hspace{5pt} \frac{\sqrt{-1}}{\beta} \frac{d \zeta \wedge d \bar{\zeta}}{\left\lvert \zeta \right\rvert^2} = \sqrt{-1} \frac{d z \wedge d \bar{z}}{\left\lvert z \right\rvert^2} \\
\omega_\beta = \sqrt{-1} \sum\limits_{i=1}^{n-1} d z_i \wedge d \bar{z}_i + \frac{\sqrt{-1}}{\beta} d \zeta \wedge d \bar{\zeta}
\end{gathered}
\end{equation}
But as the coordinate transformation $\zeta$ is non-\hol{}, there will be a non-\hol{} part appearing in the transformed coordinate expressions where it would not have otherwise occurred, as we can see from the following \cite{Brendle:2013:Ricflatedgesing,Datar:2014:CanonConeSing,Donaldson:2012:ConeSingDiv,Jeffres:2016:KEedgesing}:
\begin{equation}\label{eq:zzeta}
\begin{gathered}
z = \left\lvert \zeta \right\rvert^{\frac{1}{\beta} - 1} \zeta \\
\left\lvert z \right\rvert^{\beta-1} d z = \frac{1}{2} \left( \frac{1}{\beta} + 1 \right) d \zeta + \frac{1}{2} \left( \frac{1}{\beta} - 1 \right) \frac{\zeta}{\bar{\zeta}} d \bar{\zeta} \hspace{1.5pt}; \hspace{5pt} \frac{d z}{z} = \frac{1}{2} \left( \frac{1}{\beta} + 1 \right) \frac{d \zeta}{\zeta} + \frac{1}{2} \left( \frac{1}{\beta} - 1 \right) \frac{d \bar{\zeta}}{\bar{\zeta}}
\end{gathered}
\end{equation}
So as a matter of convention while stating any result in terms of the coordinate $\zeta$, we ignore the non-\hol{} part and consider only the \hol{} part in the concerned coordinate expression, and then the said statement holds true up to the \hol{} part \cite{Datar:2014:CanonConeSing,Donaldson:2012:ConeSingDiv,Jeffres:2016:KEedgesing,Li:2018:conicMab}. \par
The issue of non-\hol{ity} in $\zeta$ can be mitigated by defining another new (multivalued or non-injective) singular coordinate $\tilde{\zeta} = z^\beta$ on $\C$ as it will give us $z^{\beta-1} d z = \frac{1}{\beta} d \tilde{\zeta}$ and then $\frac{\sqrt{-1}}{\beta^2} \frac{d \tilde{\zeta} \wedge d \bar{\tilde{\zeta}}}{\left\lvert \tilde{\zeta} \right\rvert^2} = \sqrt{-1} \frac{d z \wedge d \bar{z}}{\left\lvert z \right\rvert^2}$ \cite{Donaldson:2012:ConeSingDiv,Jeffres:2016:KEedgesing}. The two singular coordinates $\left(z_1, \ldots, z_{n-1}, \zeta\right)$ and $\left(z_1, \ldots, z_{n-1}, \tilde{\zeta}\right)$ on $\C^n$ can be checked to be equivalent to each other (as in every definition of conical singularities stated in terms of one of them holds true in terms of the other as well) and both of them can be dubbed as \textit{`conical coordinates'} on $\C^n$ with their sole purpose being to ``conceal'' the conical singularities present in the metric $\omega_\beta$ (or in the metric $\omega$ of Definition \ref{def:coneKr2}) \cite{Donaldson:2012:ConeSingDiv,Jeffres:2016:KEedgesing}. $\zeta$ is single-valued and injective on $\C$ but non-\hol{} for any $\beta \in \left(0, \infty\right) \smallsetminus \stone$, while $\tilde{\zeta}$ is either \hol{} on $\C \smallsetminus \stzero$ but multivalued (if $\beta \in \left(0, 1\right)$) or \hol{} on $\C$ but non-injective (if $\beta \in \left(1, \infty\right)$) \cite{Jeffres:2016:KEedgesing}. Thus we can state all of our results about conical \krm{s} in terms of $\zeta$ as well as $\tilde{\zeta}$ equivalently (see Jeffres-Mazzeo-Rubinstein \cite{Jeffres:2016:KEedgesing}; Subsection 2.1 for more on the relationship between $\zeta$ and $\tilde{\zeta}$).
\begin{remark}
Since the coordinate $\tilde{\zeta} = z^\beta$ is either multivalued or non-injective, we have to work with the logarithmic \rms{} (either as the domain or as the codomain) in order to make it single-valued, injective and \hol{} \cite{Jeffres:2016:KEedgesing}. If $\beta \in \left(0, 1\right]$ then the branched \hol{} covering of fractional degree $\frac{1}{\beta}$ viz. $\C^n \to \C^n$, $\left(z_1, \ldots, z_{n-1}, \tilde{\zeta}\right) \mapsto \left(z_1, \ldots, z_{n-1}, z = \tilde{\zeta}^{\frac{1}{\beta}}\right)$ smoothens out the model edge metric $\omega_\beta$ as well as the smooth conical \krm{} $\omega$ given by Definition \ref{def:coneKr2}, while if $\beta \in \left[1, \infty\right)$ then $\C^n \to \C^n$, $\left(z_1, \ldots, z_{n-1}, z\right) \mapsto \left(z_1, \ldots, z_{n-1}, \tilde{\zeta}\right)$ (which has fractional degree $\beta$) smoothens out $\omega_\beta$ and the concerned conical metric $\omega$ \cite{Jeffres:2016:KEedgesing}. Thus for $\beta \in \left(0, 1\right]$ the geometry of the conical \krm{} of cone angle $2 \pi \beta \leq 2 \pi$ can be clearly visualized in the usual way and in this case $\tilde{\zeta}$ covers $z$ to smoothen out the conical metric, whereas the geometric interpretation of cone angles $2 \pi \beta \geq 2 \pi$ can be thought of as having the conical singularity of cone angle $\frac{2 \pi}{\beta} \leq 2 \pi$ but instead with $z$ covering $\tilde{\zeta}$ \cite{Jeffres:2016:KEedgesing}.
\end{remark} \par
A \textit{smooth conical \krm{}} on a general \krmf{} is one for which in item (\ref{itm:qisometr}) of Definition \ref{def:coneKr1}, if the coordinate $\zeta = \left\lvert z \right\rvert^{\beta-1} z$ is substituted then the local expression for the metric in the new coordinates turns out to be that for a smooth \krm{} (up to the \hol{} part), or equivalently if $\tilde{\zeta} = z^\beta$ is substituted then the local coordinate expression becomes smooth (just as it happened above in (\ref{eq:zetaz}) in the case of the model edge metric $\omega_\beta$ on $\C^n$) \cite{Brendle:2013:Ricflatedgesing,Datar:2014:CanonConeSing,Donaldson:2012:ConeSingDiv,Jeffres:2016:KEedgesing,Song:2016:RicciGLBconicKE}.
\begin{definition}[Smooth Conical \kr{} Metric; \cite{Brendle:2013:Ricflatedgesing,Datar:2014:CanonConeSing,Donaldson:2012:ConeSingDiv,Jeffres:2016:KEedgesing,Song:2016:RicciGLBconicKE}]\label{def:coneKr2}
A \krc{} $\omega$ on $M$ is said to be a \textit{smooth conical \krm{}} on $M$ with cone angle $2 \pi \beta > 0$ along $D$ if:
\begin{enumerate}
\item $\omega$ is a smooth \krm{} on $M \smallsetminus D$.
\item {Around every point of $D$ there exist local coordinates $\left(z_1, \ldots, z_{n-1}, z\right)$ with $D$ being given by $\left\lbrace z = 0 \right\rbrace$ and with $\left(z_1, \ldots, z_{n-1}\right)$ restricting to a coordinate chart on $D$, such that after substituting the coordinate $\zeta = \left\lvert z \right\rvert^{\beta-1} z$ the expression for $\omega$ in the coordinates $\left(z_1, \ldots, z_{n-1}, \zeta\right)$ takes the form of a smooth \krm{} on $M$ (up to the \hol{} part), or equivalently after substituting $\tilde{\zeta} = z^\beta$ the expression for $\omega$ in $\left(z_1, \ldots, z_{n-1}, \tilde{\zeta}\right)$ takes the form of a smooth \krm{} (just like it happened with $\omega_\beta$ in (\ref{eq:zetaz})).}\label{itm:smoothen}
\end{enumerate}
\end{definition}
{\noindent In other words a \textit{smooth conical \krm{}} is one which can be ``smoothened out'' by doing any of the singular coordinate transformations $\zeta$ or $\tilde{\zeta}$ just like the model edge metric $\omega_\beta$, meaning the conical singularities present in the metric are mild enough to be completely transferred to the singular conical coordinates $\left(z_1, \ldots, z_{n-1}, \zeta\right)$ or $\left(z_1, \ldots, z_{n-1}, \tilde{\zeta}\right)$ \cite{Brendle:2013:Ricflatedgesing,Datar:2014:CanonConeSing,Donaldson:2012:ConeSingDiv,Jeffres:2016:KEedgesing,Song:2016:RicciGLBconicKE}.} \par
In this paper we will be dealing with smooth conical \krm{s} that satisfy the following even more restrictive condition (which we thought of dubbing as \textit{``\conrml{ity}''} considering the implications between Definitions \ref{def:coneKr3} and \ref{def:coneKr4} and following the works of Hashimoto \cite{Hashimoto:2019:cscKConeSing} and Jeffres-Mazzeo-Rubinstein \cite{Jeffres:2016:KEedgesing}):
\begin{definition}[Conormal Smooth Conical \kr{} Metric; Hashimoto \cite{Hashimoto:2019:cscKConeSing}; Definition 1.1, Jeffres-Mazzeo-Rubinstein \cite{Jeffres:2016:KEedgesing}; Equation (6), Subsubsection 2.6.4]\label{def:coneKr3}
A \krc{} $\omega$ on $M$ is said to be a \textit{\conrml{} smooth conical \krm{}} on $M$ with cone angle $2 \pi \beta > 0$ along $D$ if:
\begin{enumerate}
\item $\omega$ is a smooth \krm{} on $M \smallsetminus D$.
\item {Around every point of $D$ there exist local coordinates $\left(z_1, \ldots, z_{n-1}, z\right)$ with $D$ being given by $\left\lbrace z = 0 \right\rbrace$ and with $\left(z_1, \ldots, z_{n-1}\right)$ restricting to a coordinate chart on $D$, such that writing the following local expression for $\omega$:
\begin{equation}\label{eq:gij}
\omega = \sqrt{-1} \sum\limits_{i,j=1}^{n-1} g_{i \bar{\jmath}} d z_i \wedge d \bar{z}_j + \sqrt{-1} \sum\limits_{i=1}^{n-1} g_i d z_i \wedge d \bar{z} + \sqrt{-1} \sum\limits_{i=1}^{n-1} g_{\bar{\imath}} d z \wedge d \bar{z}_i + \sqrt{-1} g_0 dz \wedge d \bar{z}
\end{equation}
we have the following conditions on the coefficient functions in (\ref{eq:gij}):
\begin{enumerate}
\item $g_{i \bar{\jmath}} \in O \left(1\right)$ as $z \to 0$, i.e. $g_{i \bar{\jmath}}$ is bounded for all $1 \leq i, j \leq n - 1$. \label{itm:gij}
\item $g_i, g_{\bar{\imath}} \in O \left(\left\lvert z \right\rvert^{2 \beta - 1}\right)$ as $z \to 0$, i.e. $\frac{g_i}{\left\lvert z \right\rvert^{2 \beta - 1}}$, $\frac{g_{\bar{\imath}}}{\left\lvert z \right\rvert^{2 \beta - 1}}$ are bounded for all $1 \leq i \leq n - 1$. \label{itm:gi}
\item $g_0 = F_0 \left\lvert z \right\rvert^{2 \beta - 2}$, where $F_0$ is bounded and strictly positive. \label{itm:g0}
\item $g_{i \bar{\jmath}}$, $g_i$, $g_{\bar{\imath}}$, $g_0$ and $F_0$ are functions of $\left(z_1, \ldots, z_{n-1}, z\right)$ such that $g_{i \bar{\jmath}}$, $\bar{z} g_i$, $z g_{\bar{\imath}}$, $\left\lvert z \right\rvert^2 g_0$ and $F_0$ are smooth away from $\left\lbrace z = 0 \right\rbrace$ and these when considered as functions of $\left(z_1, \ldots, z_{n-1}, \zeta\right)$ with $\zeta = \left\lvert z \right\rvert^{\beta-1} z$ (or even of $\left(z_1, \ldots, z_{n-1}, \tilde{\zeta}\right)$ with $\tilde{\zeta} = z^\beta$) become smooth everywhere.
\end{enumerate}}\label{itm:gijbigOz}
\end{enumerate}
\end{definition}
\begin{remark}
We should note here that the conditions given in the items (\ref{itm:gij}), (\ref{itm:gi}) and (\ref{itm:g0}) in Definition \ref{def:coneKr3} can be related to conical \krm{s} satisfying Definitions \ref{def:coneKr1} and \ref{def:coneKr2} in the following ways:
\begin{enumerate}
\item $\omega$ is a smooth conical \krm{} on $M$ given by Definition \ref{def:coneKr2} if and only if the coefficient functions in the coordinate expression (\ref{eq:gij}) for $\omega$ satisfy all the conditions of Definition \ref{def:coneKr3} except with the condition in item (\ref{itm:gi}) weakened to $g_i, g_{\bar{\imath}} \in O \left(\left\lvert z \right\rvert^{\beta - 1}\right)$ as $z \to 0$ for all $1 \leq i \leq n - 1$.
\item $\omega$ is a conical \krm{} on $M$ given by Definition \ref{def:coneKr1} if and only if the coefficient functions in (\ref{eq:gij}) satisfy the following conditions:
\begin{enumerate}
\item $g_{i \bar{\jmath}} \in O \left(1\right)$ as $z \to 0$ for all $1 \leq i, j \leq n - 1$. \label{itm:gijqisometr}
\item $g_0 = F_0 \left\lvert z \right\rvert^{2 \beta - 2}$, where $F_0$ is bounded and strictly positive. \label{itm:g0qisometr}
\item $g_{i \bar{\jmath}}$, $g_0$ and $F_0$ are functions of $\left(z_1, \ldots, z_{n-1}, z\right)$ such that $g_{i \bar{\jmath}}$, $\left\lvert z \right\rvert^2 g_0$ and $F_0$ are continuous everywhere and smooth away from $\left\lbrace z = 0 \right\rbrace$.
\end{enumerate}
\end{enumerate}
\end{remark} \par
Along the lines of the terminology of Jeffres-Mazzeo-Rubinstein \cite{Jeffres:2016:KEedgesing} we define a \textit{\plyhomo{} smooth conical \krm{}} as one having a certain nice kind of complete asymptotic power series expansion with smooth coefficients along the hypersurface of the conical singularity.
\begin{definition}[Polyhomogeneous Smooth Conical \kr{} Metric; Hashimoto \cite{Hashimoto:2019:cscKConeSing}; Lemma 3.6, Jeffres-Mazzeo-Rubinstein \cite{Jeffres:2016:KEedgesing}; Theorems 1 and 2, Subsubsection 2.6.4]\label{def:coneKr4}
A \krc{} $\omega$ on $M$ is said to be a \textit{\plyhomo{} smooth conical \krm{}} on $M$ with cone angle $2 \pi \beta > 0$ along $D$ if:
\begin{enumerate}
\item $\omega$ is a smooth \krm{} on $M \smallsetminus D$.
\item {Around every point of $D$ there exist local coordinates $\left(z_1, \ldots, z_{n-1}, z\right)$ with $D$ being given by $\left\lbrace z = 0 \right\rbrace$ and with $\left(z_1, \ldots, z_{n-1}\right)$ restricting to a coordinate chart on $D$, such that the coefficient functions in the coordinate expression (\ref{eq:gij}) have got the following (absolutely convergent and locally uniformly convergent) complete power series expansions in a tubular \nbd{} of $\left\lbrace z = 0 \right\rbrace$:
\begin{enumerate}
\item $\left\lvert g_{i \bar{\jmath}} \right\rvert = \sum\limits_{k=0}^{\infty} c_{2 k, i j} \left\lvert z \right\rvert^{2 k \beta}$
\item $\left\lvert g_i \right\rvert = \left\lvert g_{\bar{\imath}} \right\rvert = \sum\limits_{k=1}^{\infty} c_{2 k,i} \left\lvert z \right\rvert^{2 k \beta - 1}$
\item $g_0 = \left\lvert g_0 \right\rvert = \sum\limits_{k=1}^{\infty} c_{2 k} \left\lvert z \right\rvert^{2 k \beta - 2}$
\end{enumerate}
where $c_{2 k, i j} = c_{2 k, j i}$, $c_{2 k,i}$ and $c_{2 k}$ are real-valued, bounded and smooth functions of $\left(z_1, \ldots, z_{n-1}\right)$ (i.e. are independent of $z$) such that $c_{0, i i}$ and $c_2$ are strictly positive and bounded below away from $0$.}\label{itm:polyhomo}
\end{enumerate}
\end{definition}
\begin{remark}
It should be noted with caution that `\plyhomo{} smooth' (and also `\conrml{} smooth') in this paper can be at best thought of as analogous to `\plyhomo{}' (and `\conrml{}' respectively) of Jeffres-Mazzeo-Rubinstein \cite{Jeffres:2016:KEedgesing}. But we are dealing with a much restricted class of conical \krm{s} (viz. smooth conical \krm{s}) than \cite{Jeffres:2016:KEedgesing}, so our Definitions \ref{def:coneKr3} and \ref{def:coneKr4} are certainly not identical to the respective definitions of \conrml{ity} and \plyhomoy{} made in \cite{Jeffres:2016:KEedgesing}.
\end{remark} \par
We can now clearly see the hierarchy of implications of Definitions \ref{def:coneKr1}, \ref{def:coneKr2}, \ref{def:coneKr3} and \ref{def:coneKr4} as follows with each implication being strictly one-directional \cite{Brendle:2013:Ricflatedgesing,Datar:2014:CanonConeSing,Donaldson:2012:ConeSingDiv,Hashimoto:2019:cscKConeSing,Jeffres:2016:KEedgesing,Song:2016:RicciGLBconicKE}:
\begin{equation*}
\text{Definition \ref{def:coneKr4}} \impl \text{Definition \ref{def:coneKr3}} \impl \text{Definition \ref{def:coneKr2}} \impl \text{Definition \ref{def:coneKr1}}
\end{equation*}
Our conical \kr{} metrics constructed by the momentum construction method \cite{Hwang:1994:cscK,Hwang:2002:MomentConstruct} on pseudo-Hirzebruch surfaces \cite{Barth:2004:CmpctCmplxSurf,Fujiki:1992:eKruledmani,Tonnesen:1998:eKminruledsurf} will be \plyhomo{} smooth \cite{Hashimoto:2019:cscKConeSing,Jeffres:2016:KEedgesing} (if the momentum profile is taken to be real analytic) and will be \conrml{} smooth \cite{Hashimoto:2019:cscKConeSing,Jeffres:2016:KEedgesing} (if the momentum profile is taken to be just smooth) as we shall see in Section \ref{sec:polyhomoConeSing}, and hence will be satisfying the relatively strongest of all the above-mentioned conditions for conical \krm{s}.
\subsection{Canonical K\"ahler Metrics with Conical Singularities}\label{subsec:CanonKhlrConeSing}
Now that we have made sense of the concept of a \krm{} developing a conical singularity along a smooth complex hypersurface of a compact complex manifold, we will briefly discuss the respective notions of \textit{canonical \krm{s} with conical singularities} \cite{Datar:2014:CanonConeSing} i.e. conical \krm{s} which are additionally \kr{}-Einstein, cscK or higher cscK. Since a conical \krm{} is a smooth \krm{} away from the hypersurface of its conical singularity, the notions of a \textit{conical \kr{}-Einstein metric} \cite{Brendle:2013:Ricflatedgesing,Donaldson:2012:ConeSingDiv,Jeffres:2016:KEedgesing,Song:2016:RicciGLBconicKE}, a \textit{conical cscK metric} \cite{Aoi:2025:cscKConeSing,Hashimoto:2019:cscKConeSing,Keller:2018:cscKConeSing,Li:2018:conicMab,Zheng:2015:UniqueConeSing} and a \textit{conical higher cscK metric} can be naturally defined as those conical \krm{s} which are (smooth) \kr{}-Einstein, cscK and higher cscK respectively away from the hypersurface. However since this involves dealing with the complement of the hypersurface which is a non-compact manifold, there arise some issues related to the \cohomll{} invariance of the respective curvatures i.e. the Ricci, scalar and higher scalar curvatures (see for example Donaldson \cite{Donaldson:2012:ConeSingDiv}, Jeffres-Mazzeo-Rubinstein \cite{Jeffres:2016:KEedgesing}, Song-Wang \cite{Song:2016:RicciGLBconicKE} for the Ricci curvature of conical \kr{}-Einstein metrics and Aoi-Hashimoto-Zheng \cite{Aoi:2025:cscKConeSing}, Hashimoto \cite{Hashimoto:2019:cscKConeSing}, Li \cite{Li:2018:conicMab} for the scalar curvature of conical cscK metrics). In order to sort out these issues (for which Section \ref{sec:highScalcurrent} is dedicated for the case of our momentum-constructed conical higher cscK metrics) one needs a global interpretation of the respective curvatures on the underlying compact manifold in terms of the current of integration along the hypersurface of the conical singularity \cite{Aoi:2025:cscKConeSing,Donaldson:2012:ConeSingDiv,Hashimoto:2019:cscKConeSing,Jeffres:2016:KEedgesing,Li:2018:conicMab,Song:2016:RicciGLBconicKE}. \par
Let $M$ be a compact \kr{} $n$-manifold and $D \subseteq M$ be a simple closed hypersurface. Let $\omega$ be any of the four kinds (mentioned in Subsection \ref{subsec:KhlrConeSing}) of conical \krm{s} on $M$ with cone angle $2 \pi \beta > 0$ along $D$. Since $\omega$ is a smooth \krm{} on the non-compact \krmf{} $M \smallsetminus D$, the \textit{Ricci curvature form} $\Ric \left(\omega\right) \bigr\rvert_{M \smallsetminus D}$ (and hence the \textit{first Chern form} $c_1 \left(\omega\right) \bigr\rvert_{M \smallsetminus D}$) given by (\ref{eq:defRic}) (and by (\ref{eq:deffirstChern}) respectively) are well-defined closed smooth $\left(1,1\right)$-forms on $M \smallsetminus D$. So a \textit{conical \kr{}-Einstein metric} on $M$ can be obviously defined as follows \cite{Brendle:2013:Ricflatedgesing,Donaldson:2012:ConeSingDiv,Jeffres:2016:KEedgesing,Song:2016:RicciGLBconicKE}:
\begin{definition}[Conical \kr{}-Einstein Metric; \cite{Brendle:2013:Ricflatedgesing,Donaldson:2012:ConeSingDiv,Jeffres:2016:KEedgesing,Song:2016:RicciGLBconicKE}]\label{def:coneKE}
A conical \krm{} $\omega$ on $M$ with cone angle $2 \pi \beta > 0$ along $D$ is said to be a \textit{conical \kr{}-Einstein metric} if $\operatorname{Ric} \left(\omega\right) \bigr\rvert_{M \smallsetminus D} = \lambda \omega$ on $M \smallsetminus D$ for some constant $\lambda \in \mathbb{R}$. The constant $\lambda$ is then called as the \textit{Ricci curvature} of $\omega$ on $M \smallsetminus D$.
\end{definition} \par
Note that as $\Ric \left(\omega\right) \bigr\rvert_{M \smallsetminus D}$ and $c_1 \left(\omega\right) \bigr\rvert_{M \smallsetminus D}$ (for any conical \krm{} $\omega$ on $M$) are locally integrable on the whole of $M$, they can be regarded as closed $\left(1,1\right)$-currents on $M$, and hence it makes sense to talk about their de Rham cohomology classes $\left[ \Ric \left(\omega\right) \bigr\rvert_{M \smallsetminus D} \right], \left[ c_1 \left(\omega\right) \bigr\rvert_{M \smallsetminus D} \right] \in H^{\left(1,1\right)} \left(M, \mathbb{R}\right)$ \cite{Demailly:2012:CmplxDifferGeom,Szekelyhidi:2014:eKintro}. However $c_1 \left(\omega\right) \bigr\rvert_{M \smallsetminus D}$ given in this way is in general not a \cohomll{} representative of the first Chern class $c_1 \left(M\right)$ (defined by (\ref{eq:deffirstChern})) as we are dealing with the non-compact manifold $M \smallsetminus D$ (see for example \cite{Hashimoto:2019:cscKConeSing,Li:2015:logKStab,Li:2018:conicMab,Song:2016:RicciGLBconicKE} for more on this). It is for this reason that a ``global interpretation'' of the objects $\Ric \left(\omega\right)$ and $c_1 \left(\omega\right)$ on the whole compact manifold $M$ in terms of the current of integration $\left[ D \right]$ (which is a closed positive $\left(1,1\right)$-current on $M$) is sought after \cite{Hashimoto:2019:cscKConeSing,Li:2015:logKStab,Li:2018:conicMab,Song:2016:RicciGLBconicKE}. This global expression of currents for the \textit{Ricci curvature form} (or the \textit{Ricci curvature current}) $\Ric \left(\omega\right)$ on $M$ is given as (studied in many works e.g. \cite{Brendle:2013:Ricflatedgesing,Campana:2013:ConeSingNormCrossDiv,Datar:2014:CanonConeSing,Donaldson:2012:ConeSingDiv,Jeffres:2016:KEedgesing,Li:2015:logKStab,Shen:2016:SmoothApproxConeSingRicciLB,Song:2016:RicciGLBconicKE}):
\begin{equation}\label{eq:RicconeKhlr}
\Ric \left(\omega\right) = \rho + 2 \pi \left(1 - \beta\right) \left[ D \right]
\end{equation}
where $\left[ D \right]$ is the current of integration on $M$ along the hypersurface $D$ and $\rho$ is some closed $\left(1,1\right)$-current on $M$ given by a closed $\left(1,1\right)$-form smooth on $M \smallsetminus D$ and locally integrable on $M$ (similar to $\omega$) where in fact $\rho = \Ric \left(\omega\right) \bigr\rvert_{M \smallsetminus D}$. An expression of the form (\ref{eq:RicconeKhlr}) can be obtained for $\Ric \left(\omega\right)$ by following the general method outlined in equations (\ref{eq:defRiccurt1}), (\ref{eq:defRiccurt2}) briefly discussed in Subsection \ref{subsec:AnalogyConeSing}. The \textit{first Chern form} (or the \textit{first Chern current}) $c_1 \left(\omega\right) = \frac{1}{2 \pi} \Ric \left(\omega\right)$ given by (\ref{eq:RicconeKhlr}) then indeed turns out to be a \cohomll{} representative of the first Chern class $c_1 \left(M\right)$ \cite{Hashimoto:2019:cscKConeSing,Li:2015:logKStab,Li:2018:conicMab,Song:2016:RicciGLBconicKE}. If in particular $\omega$ is a conical \kr{}-Einstein metric on $M$ with Ricci curvature $\lambda \in \R$ then this global expression for $\Ric \left(\omega\right)$ on $M$ looks like \cite{Brendle:2013:Ricflatedgesing,Campana:2013:ConeSingNormCrossDiv,Datar:2014:CanonConeSing,Donaldson:2012:ConeSingDiv,Jeffres:2016:KEedgesing,Li:2015:logKStab,Shen:2016:SmoothApproxConeSingRicciLB,Song:2016:RicciGLBconicKE}:
\begin{equation}\label{eq:RicconeKE}
\Ric \left(\omega\right) = \lambda \omega + 2 \pi \left(1 - \beta\right) \left[ D \right]
\end{equation} \par
Now we will discuss conical cscK metrics which are defined in a similar way as conical \kr{}-Einstein metrics \cite{Hashimoto:2019:cscKConeSing,Li:2018:conicMab,Zheng:2015:UniqueConeSing}. For the conical \krm{} $\omega$ on $M$, we again first consider the closed smooth $\left(1,1\right)$-forms $\Ric \left(\omega\right) \bigr\rvert_{M \smallsetminus D}$ and $c_1 \left(\omega\right) \bigr\rvert_{M \smallsetminus D}$ (given by (\ref{eq:defRic}) and (\ref{eq:deffirstChern}) respectively) only on $M \smallsetminus D$ (and not on the whole of $M$), and $\omega^n$ will be a nowhere vanishing smooth top-dimensional form on $M \smallsetminus D$. So the \textit{scalar curvature} $S \left(\omega\right) : M \smallsetminus D \to \R$ is a well-defined smooth (but possibly unbounded) function satisfying (\ref{eq:defScal}) on $M \smallsetminus D$. But to begin with we can restrict our attention to \textit{conical \krm{s} with bounded scalar curvature} only. Amongst these are the \textit{conical cscK metrics} which are obviously defined as follows \cite{Hashimoto:2019:cscKConeSing,Li:2018:conicMab,Zheng:2015:UniqueConeSing}:
\begin{definition}[Conical cscK Metric; \cite{Aoi:2025:cscKConeSing,Hashimoto:2019:cscKConeSing,Keller:2018:cscKConeSing,Li:2018:conicMab,Zheng:2015:UniqueConeSing}]\label{def:conecscK}
A conical \krm{} $\omega$ on $M$ with cone angle $2 \pi \beta > 0$ along $D$ is said to be a \textit{conical cscK metric} if the scalar curvature $S \left(\omega\right) \in \R$ is a constant on $M \smallsetminus D$.
\end{definition} \par
Now we will define conical higher cscK metrics along the same lines as conical cscK metrics (or even conical \kr{}-Einstein metrics). Again since the conical \krm{} $\omega$ on $M$ is a smooth \krm{} on $M \smallsetminus D$, the \textit{top Chern form} $c_n \left(\omega\right) \bigr\rvert_{M \smallsetminus D}$ (as defined in (\ref{eq:deftopChern})) and the volume form $\frac{\omega^n}{n!}$ are well-defined smooth $\left(n,n\right)$-forms on the non-compact \orble{} manifold $M \smallsetminus D$, so the \textit{higher scalar curvature} $\lambda \left(\omega\right) : M \smallsetminus D \to \R$ is a well-defined smooth (but possibly unbounded) function defined by (\ref{eq:defhcscKheK}) on $M \smallsetminus D$. Again considering \textit{conical \krm{s} with bounded higher scalar curvature} only, we have amongst these our \textit{conical higher cscK metrics} which can be obviously defined by combining Definition \ref{def:hcscK} with any one out of Definitions \ref{def:coneKr1}, \ref{def:coneKr2}, \ref{def:coneKr3} and \ref{def:coneKr4} as follows:
\begin{definition}[Conical Higher cscK Metric]\label{def:conehcscK}
A conical \krm{} $\omega$ on $M$ with cone angle $2 \pi \beta > 0$ along $D$ is said to be a \textit{conical higher cscK metric} if the higher scalar curvature $\lambda \left(\omega\right) \in \R$ is a constant on $M \smallsetminus D$.
\end{definition} \par
Now we will come to the global interpretations of the conical cscK and the conical higher cscK equations on the whole compact manifold which will result in the \cohomll{} invariance of the scalar and the higher scalar curvatures respectively on the manifold. For any conical \krm{} $\omega$ on $M$, just like $\Ric \left(\omega\right) \bigr\rvert_{M \smallsetminus D}$ and $c_1 \left(\omega\right) \bigr\rvert_{M \smallsetminus D}$ given by (\ref{eq:defRic}) and (\ref{eq:deffirstChern}) respectively are closed $\left(1,1\right)$-currents on $M$ given by locally integrable closed smooth $\left(1,1\right)$-forms on $M \smallsetminus D$, in exactly the same way $c_n \left(\omega\right) \bigr\rvert_{M \smallsetminus D}$ (given by (\ref{eq:deftopChern}) on $M \smallsetminus D$) is locally integrable on the whole of $M$ and so can be regarded as an $\left(n,n\right)$-current on $M$ \cite{Demailly:2012:CmplxDifferGeom}. So it makes sense to talk about the de Rham cohomology class $\left[ c_n \left(\omega\right) \bigr\rvert_{M \smallsetminus D} \right] \in H^{\left(n,n\right)} \left(M, \mathbb{R}\right)$ and similarly, also the classes $\left[ \Ric \left(\omega\right) \bigr\rvert_{M \smallsetminus D} \right], \left[ c_1 \left(\omega\right) \bigr\rvert_{M \smallsetminus D} \right] \in H^{\left(1,1\right)} \left(M, \mathbb{R}\right)$ \cite{Demailly:2012:CmplxDifferGeom}. Unfortunately as we are on the non-compact manifold $M \smallsetminus D$, just like $c_1 \left(\omega\right) \bigr\rvert_{M \smallsetminus D}$ is not a \cohomll{} representative of $c_1 \left(M\right)$, even $c_n \left(\omega\right) \bigr\rvert_{M \smallsetminus D}$ given in this way is not in general a de Rham \cohomll{} representative of the top Chern class $c_n \left(M\right)$ (which is defined in (\ref{eq:deftopChern})) \cite{Hashimoto:2019:cscKConeSing,Li:2015:logKStab,Li:2018:conicMab,Song:2016:RicciGLBconicKE}. \par
The global expression of currents for the scalar curvature $S \left(\omega\right)$ of a conical \krm{} $\omega$ on $M$ studied by Zheng \cite{Zheng:2015:UniqueConeSing}, Li \cite{Li:2018:conicMab}, Hashimoto \cite{Hashimoto:2019:cscKConeSing} looks like the following (similar to equations (\ref{eq:RicconeKhlr}) and (\ref{eq:RicconeKE})):
\begin{equation}\label{eq:ScalconecscK}
n \Ric \left(\omega\right) \wedge \omega^{n-1} = S \left(\omega\right) \omega^n + 2 n \pi \left(1 - \beta\right) \left[ D \right] \wedge \omega^{n-1}
\end{equation}
The \textit{average scalar curvature} of $\omega$ on the non-compact $M \smallsetminus D$ is given by $\bar{S} \left(\omega\right) = \frac{\int\limits_{M} S \left(\omega\right) \omega^n}{\int\limits_{M} \omega^n}$ which is not going to be an invariant of the \krcl{} $\left[\omega\right]$, while the \textit{average scalar curvature} of $\omega$ on the whole of $M$ is given by $\widehat{S} \left(\omega\right) = n \frac{\int\limits_{M} \Ric \left(\omega\right) \wedge \omega^{n-1}}{\int\limits_{M} \omega^n}$ (with the substitution of (\ref{eq:ScalconecscK})) which will turn out to be equal to the \cohomll{} value (\ref{eq:avgScal}) \cite{Hashimoto:2019:cscKConeSing,Li:2018:conicMab}. Further $\widehat{S} \left(\omega\right)$ and $\bar{S} \left(\omega\right)$ are related precisely by the following (Hashimoto \cite{Hashimoto:2019:cscKConeSing}; \textit{Remark} 4.8):
\begin{equation}\label{eq:widehatSbarS}
\widehat{S} \left(\omega\right) = \bar{S} \left(\omega\right) + 2 n \pi \left(1 - \beta\right) \frac{\int\limits_{D} \omega^{n-1}}{\int\limits_{M} \omega^n}
\end{equation} \par
Along very similar lines we try to mimic equation (\ref{eq:ScalconecscK}) to obtain a global expression of currents for the higher scalar curvature $\lambda \left(\omega\right)$ on $M$ which is supposed to look somewhat like the following (which we will study in the special case of the momentum construction method \cite{Hwang:2002:MomentConstruct} in Section \ref{sec:highScalcurrent}):
\begin{equation}\label{eq:highScalconehcscK}
c_n \left(\omega\right) = \frac{\lambda \left(\omega\right)}{n! \left(2 \pi\right)^n} \omega^n + \frac{1 - \beta}{\left(2 \pi\right)^{n-1}} \alpha \wedge \left[ D \right]
\end{equation}
where $\alpha$ is some closed $\left(n-1,n-1\right)$-form which is smooth on $M \smallsetminus D$ and locally integrable on $M$ thereby giving a closed $\left(n-1,n-1\right)$-current on $M$ (exactly like $\omega^{n-1}$). The \textit{top Chern form} (or the \textit{top Chern current}) $c_n \left(\omega\right)$ given by (\ref{eq:highScalconehcscK}) should then turn out to be a \cohomll{} representative of the top Chern class $c_n \left(M\right)$ (as we will see in our special case in Subsection \ref{subsec:CohomolInvcurrent}). Here also the \textit{average higher scalar curvature} on $M \smallsetminus D$ is given by $\lambda_1 \left(\omega\right) = \frac{\int\limits_{M} \lambda \left(\omega\right) \omega^n}{\int\limits_{M} \omega^n}$ which may not a \cohomll{} invariant, while the \textit{average higher scalar curvature} on $M$ is given by $\lambda_0 \left(\omega\right) = n! \left(2 \pi\right)^n \frac{\int\limits_{M} c_n \left(\omega\right)}{\int\limits_{M} \omega^n}$ (again with the substitution of (\ref{eq:highScalconehcscK})) which is expected to be equal to the \cohomll{} value (\ref{eq:avghighScal}). Further $\lambda_0 \left(\omega\right)$ and $\lambda_1 \left(\omega\right)$ should be related by the following equation analogous to the equation (\ref{eq:widehatSbarS}) (also to be seen in Subsection \ref{subsec:CohomolInvcurrent}):
\begin{equation}\label{eq:lambda0lambda1}
\lambda_0 \left(\omega\right) = \lambda_1 \left(\omega\right) + n! \left(2 \pi\right) \left(1 - \beta\right) \frac{\int\limits_{D} \alpha}{\int\limits_{M} \omega^n}
\end{equation} \par
The most important issue however which arises in the study of equations (\ref{eq:ScalconecscK}) and (\ref{eq:highScalconehcscK}) (and which does not arise in equations (\ref{eq:RicconeKhlr}) or (\ref{eq:RicconeKE})) is that the wedge products of the current terms $\left[ D \right]$ and $\omega^{n-1}$ (respectively $\alpha$) can be ``na\"ively'' thought of as the following integrals \cite{Hashimoto:2019:cscKConeSing}:
\begin{equation}\label{eq:wedgenaivint}
\left[ D \right] \wedge \omega^{n-1} \left(\varphi\right) = \int\limits_D \varphi \omega^{n-1} \hspace{2pt}, \hspace{7pt} \alpha \wedge \left[ D \right] \left(\varphi\right) = \int\limits_D \varphi \alpha
\end{equation}
where $\varphi : M \to \R$ is a smooth test function. But the problem is that the closed $\left(n-1,n-1\right)$-forms $\omega^{n-1}$ and $\alpha$ are singular (and not smooth) on $M$ and have got singularities precisely along $D$, so it is not at all clear why the integrals in (\ref{eq:wedgenaivint}) even make sense (meaning are well-defined and finite) \cite{Hashimoto:2019:cscKConeSing}. Resolving this problem for general conical \krm{s} on arbitrary compact complex manifolds seems to be out of hand, but it is possible to give ``correct interpretations'' to these wedge products in some special cases like the momentum construction method \cite{Hwang:2002:MomentConstruct} and some others \cite{Donaldson:2012:ConeSingDiv,Jeffres:2016:KEedgesing} (see Hashimoto \cite{Hashimoto:2019:cscKConeSing} and Li \cite{Li:2018:conicMab} for a detailed account of this issue). In the case of our momentum-constructed conical higher cscK metrics, we will interpret the wedge products of closed currents (which will be having a sign at least locally around the two special divisors) arising in the specific form of equation (\ref{eq:highScalconehcscK}) by using Bedford-Taylor theory \cite{Bedford:1982:cpctypsh,Bedford:1976:DirichletMongeAmpere} (to be seen in Subsection \ref{subsec:currenteq}). In Subsection \ref{subsec:SmoothapproxConehcscK} we will provide another way of thinking about the wedge products of currents in equation (\ref{eq:highScalconehcscK}) by following a method attributed to \cite{Campana:2013:ConeSingNormCrossDiv,Edwards:2019:ContractDivConicKRFlow,Shen:2016:SmoothApproxConeSingRicciLB,Wang:2016:SmoothApproxConeKRFlow} of taking certain explicit smooth approximations $\omega_\epsilon$ to the conical \kr{} metric $\omega$ (with all the metrics having Calabi symmetry in this case) such that the smooth top Chern forms $c_n \left(\omega_\epsilon\right)$ converge weakly in the sense of currents to the top Chern current $c_n \left(\omega\right)$ given by (\ref{eq:highScalconehcscK}). \par
With these wedge products of currents in equations (\ref{eq:ScalconecscK}) and (\ref{eq:highScalconehcscK}) rigorously justified, we can give a precise meaning to the integrals $\int\limits_{D} \omega^{n-1} = \left[ D \right] \wedge \omega^{n-1} \left(1\right)$ and $\int\limits_{D} \alpha = \alpha \wedge \left[ D \right] \left(1\right)$ appearing in equations (\ref{eq:widehatSbarS}) and (\ref{eq:lambda0lambda1}). And the integral $\int\limits_{M} \omega^n = \int\limits_{M \smallsetminus D} \omega^n$ appearing in equations (\ref{eq:widehatSbarS}) and (\ref{eq:lambda0lambda1}) and in the definitions of $\widehat{S} \left(\omega\right)$, $\bar{S} \left(\omega\right)$ and $\lambda_0 \left(\omega\right)$, $\lambda_1 \left(\omega\right)$ is clearly well-defined, as $\omega^n$ is a locally \intble{} top-dimensional form on $M$ which is smooth on $M \smallsetminus D$.
\numberwithin{equation}{subsection}
\numberwithin{figure}{subsection}
\numberwithin{table}{subsection}
\numberwithin{lemma}{subsection}
\numberwithin{proposition}{subsection}
\numberwithin{result}{subsection}
\numberwithin{theorem}{subsection}
\numberwithin{corollary}{subsection}
\numberwithin{conjecture}{subsection}
\numberwithin{remark}{subsection}
\numberwithin{note}{subsection}
\numberwithin{motivation}{subsection}
\numberwithin{question}{subsection}
\numberwithin{answer}{subsection}
\numberwithin{case}{subsection}
\numberwithin{claim}{subsection}
\numberwithin{definition}{subsection}
\numberwithin{example}{subsection}
\numberwithin{hypothesis}{subsection}
\numberwithin{statement}{subsection}
\numberwithin{ansatz}{subsection}
\section{The Momentum Construction Method for Conical Higher cscK Metrics on a Minimal Ruled Surface}\label{sec:MomentConstructConehcscK}
\subsection{A Brief Description of the Momentum Construction Method}\label{subsec:MomentConstructConeSing}
We will now apply the momentum construction method of Hwang-Singer \cite{Hwang:2002:MomentConstruct} for explicitly constructing conical higher cscK metrics on the \textit{minimal ruled surface} $X = \mathbb{P} \left(L \oplus \mathcal{O}\right)$ where $L$ is a degree $-1$ holomorphic line bundle over a genus $2$ compact Riemann surface $\Sigma$ and $\Sigma$ is equipped with a K\"ahler metric $\omega_\Sigma$ of constant scalar curvature $-2$ (or equivalently of area $2 \pi$) and $L$ is equipped with a Hermitian metric $h$ of curvature form $-\omega_\Sigma$ (as briefly described in Subsection \ref{subsec:Background}). This surface $X$ is the first example coming from a family of similar kind of compact complex surfaces called as \textit{`pseudo-Hirzebruch surfaces'} (as termed by T{\o}nnesen-Friedman \cite{Tonnesen:1998:eKminruledsurf}; Definition 1) and it has got some nice symmetries in terms of its typical fibre $\mathsf{C}$, its \textit{zero divisor} $S_0 = \mathbb{P} \left(\left\lbrace 0 \right\rbrace \oplus \mathcal{O}\right)$ and its \textit{infinity divisor} $S_\infty = \mathbb{P} \left(L \oplus \left\lbrace 0 \right\rbrace\right)$ which enable the use of the momentum construction method. In our previous paper \cite{Sompurkar:2023:heKsmooth} we had constructed smooth higher \ext{} \krm{s} in all the \krcl{es} of $X$ by the momentum construction method, however these metrics were not higher cscK (see \cite{Sompurkar:2023:heKsmooth} for the details and Subsection \ref{subsec:Background} for a quick summary). \par
Considering the Poincar\'e duals of $\mathsf{C}$, $S_\infty$ and $S_0$ which will be elements of the de Rham \cohoml{} space $H^{\left(1,1\right)} \left(X, \mathbb{R}\right) \subseteq H^2 \left(X, \mathbb{R}\right)$ we have the following intersection formulae (refer to Barth-Hulek-Peters-Van de Ven \cite{Barth:2004:CmpctCmplxSurf}; Sections I.1, II.9 and II.10, Sz\'ekelyhidi \cite{Szekelyhidi:2014:eKintro}; Section 4.4 and T{\o}nnesen-Friedman \cite{Tonnesen:1998:eKminruledsurf}; Proposition 4):
\begin{equation}\label{eq:IntersectForm}
\mathsf{C}^2 = 0 \hspace{1pt}, \hspace{5pt} S_\infty^2 = 1 \hspace{1pt}, \hspace{5pt} S_0^2 = -1 \hspace{1pt}, \hspace{5pt} \mathsf{C} \cdot S_\infty = 1 \hspace{1pt}, \hspace{5pt} \mathsf{C} \cdot S_0 = 1 \hspace{1pt}, \hspace{5pt} S_\infty \cdot S_0 = 0
\end{equation}
where $\cdot$ denotes the intersection product of real \cohoml{} classes of $X$. And considering $c_1 \left(L\right) \in H^{\left(1,1\right)} \left(\Sigma, \mathbb{R}\right) = H^2 \left(\Sigma, \mathbb{R}\right)$ as the first Chern class of $L$, $\left[\omega_\Sigma\right] \in H^{\left(1,1\right)} \left(\Sigma, \mathbb{R}\right)$ as the K\"ahler class of $\omega_\Sigma$ and $\left[\Sigma\right] \in H^2 \left(\Sigma, \mathbb{R}\right)$ as the fundamental class of $\Sigma$ (which is identified with the Poincar\'e dual of $S_0$ as an element in $H^2 \left(X, \mathbb{R}\right)$) we also have the following intersection formulae \cite{Barth:2004:CmpctCmplxSurf,Szekelyhidi:2014:eKintro,Tonnesen:1998:eKminruledsurf}:
\begin{equation}\label{eq:Sigma}
c_1 \left(L\right) \cdot \left[\Sigma\right] = -1 \hspace{1pt}, \hspace{5pt} \left[\omega_\Sigma\right] \cdot \left[\Sigma\right] = 2 \pi \hspace{1pt}, \hspace{5pt} \mathsf{C} \cdot \left[\Sigma\right] = 1 \hspace{1pt}, \hspace{5pt} S_\infty \cdot \left[\Sigma\right] = 0 \hspace{1pt}, \hspace{5pt} S_0 \cdot \left[\Sigma\right] = -1
\end{equation}
The intersection formulae (\ref{eq:IntersectForm}) and (\ref{eq:Sigma}) are needed for obtaining the numerical characterization of the \kr{} cone of $X$ as well as for doing some \cohomll{} computations in the momentum construction method. \par
We will first recall a description of the \kr{} cone of $X$ in terms of two real parameters due to Fujiki \cite{Fujiki:1992:eKruledmani} and T{\o}nnesen-Friedman \cite{Tonnesen:1998:eKminruledsurf}. By the Leray-Hirsch theorem \cite{Demailly:2012:CmplxDifferGeom} the second real cohomology space of $X$ up to Poincar\'e duality is given by \cite{Barth:2004:CmpctCmplxSurf,Szekelyhidi:2014:eKintro,Tonnesen:1998:eKminruledsurf}:
\begin{equation}\label{eq:H2XR}
H^2 \left(X, \mathbb{R}\right) = \mathbb{R} \mathsf{C} \oplus \mathbb{R} S_\infty = \left\lbrace \hspace{2pt} a \mathsf{C} + b S_\infty \hspace{4pt} \big\vert \hspace{4pt} a, b \in \R \hspace{2pt} \right\rbrace
\end{equation}
Then by the Nakai-Moishezon criterion \cite{Barth:2004:CmpctCmplxSurf} (extended from its original integral cohomology case to the real cohomology case on general compact \kr{} surfaces by Buchdahl \cite{Buchdahl:1999:CmpctKahler} and Lamari \cite{Lamari:1999:Kcone} independently) the \textit{K\"ahler cone} of $X$ (which is the set of all K\"ahler classes of $X$) is given by \cite{Szekelyhidi:2014:eKintro,Tonnesen:1998:eKminruledsurf}:
\begin{equation}\label{eq:KConeX}
H^{\left(1,1\right)} \left(X, \mathbb{R}\right)^+ = \left\lbrace \hspace{2pt} a \mathsf{C} + b S_\infty \hspace{4pt} \big\vert \hspace{4pt} a, b > 0 \hspace{2pt} \right\rbrace = \mathbb{R}_{> 0} \mathsf{C} \oplus \mathbb{R}_{> 0} S_\infty \subseteq H^2 \left(X, \mathbb{R}\right)
\end{equation}
One can refer to Fujiki \cite{Fujiki:1992:eKruledmani}; Proposition 1, Lemma 5 and T{\o}nnesen-Friedman \cite{Tonnesen:1998:eKminruledsurf}; Lemma 1 for the proof of the numerical description (\ref{eq:KConeX}) applicable in the special case of our minimal ruled surface $X$ which uses the intersection formulae (\ref{eq:IntersectForm}). An exposition of the numerical description (\ref{eq:KConeX}) of the K\"ahler cone of $X$ containing all the required details with the concerned references is given in the author's earlier work \cite{Sompurkar:2023:heKsmooth}; Subsection 2.1. \par
In the rest of Subsection \ref{subsec:MomentConstructConeSing} we will follow the exposition of the special case of the momentum construction method given in Sz\'ekelyhidi \cite{Szekelyhidi:2014:eKintro}; Section 4.4 (for smooth \ext{} \krm{s}) and in Pingali \cite{Pingali:2018:heK}; Section 2 (for smooth higher \ext{} \krm{s}) which we had followed in \cite{Sompurkar:2023:heKsmooth} for smooth higher \ext{} \krm{s} but which we will follow now for conical higher cscK metrics. The surface $X$ is obtained as the \textit{projective completion} of the line bundle $L$ by attaching to $L$ a bi\hol{} copy of the base \rms{} $\Sigma$ ``at infinity'' (which then becomes the infinity divisor of $X$), so that $\Sigma$ sits in $L$ as its zero section and hence in $X$ as its zero divisor and $X$ is a projective fibre bundle on $\Sigma$ i.e. the fibres of $X$ are bi\hol{} copies of $\cp^1$ or $\C \cup \left\lbrace \infty \right\rbrace$ \cite{Szekelyhidi:2014:eKintro,Tonnesen:1998:eKminruledsurf}. The basic idea for producing a \krm{} $\omega$ on the surface $X$ by using the given metrics $\omega_\Sigma$ and $h$ on $\Sigma$ and $L$ respectively (evolved by Hwang-Singer \cite{Hwang:2002:MomentConstruct}) is to first consider an ansatz for $\omega$ on the total space of $L$ minus its zero section (which is the same as $X$ minus its zero and infinity divisors) and to then extend $\omega$ across the zero and infinity divisors of $X$ by means of some appropriate boundary conditions applied on the fibres of $L$, and this is done by taking the pullback of $L$ to its total space and then adding the curvature of the resultant bundle to the pullback of $\omega_\Sigma$ to $X$ \cite{Szekelyhidi:2014:eKintro}. \par
Let $z$ be a local \hol{} coordinate on $\Sigma$, $w$ be a local \hol{} coordinate on the fibres of $L$ corresponding to a local \hol{} trivialization around $z$ (where $w = 0$ gives the zero section of $L$), $\left\lvert \left(z,w\right) \right\rvert_{h}^2 = \left\lvert w \right\rvert^2 h \left(z\right)$ be the \textit{fibrewise squared norm function} on $L$ induced by the \hmnm{} $h$ (where the function $h \left(z\right)$ is strictly positive and smooth) and $s = \ln \left\lvert \left(z,w\right) \right\rvert_{h}^2 = \ln \left\lvert w \right\rvert^2 + \ln h \left(z\right)$ be the coordinate on the total space of $L$ minus the zero section i.e. on the non-compact surface $X \smallsetminus \left(S_0 \cup S_\infty\right)$. The local coordinates $\left(z,w\right)$ on $X \smallsetminus \left(S_0 \cup S_\infty\right)$ are called as \textit{`bundle-adapted' local holomorphic coordinates} on the surface \cite{Hwang:2002:MomentConstruct,Szekelyhidi:2014:eKintro}. Let $\mathtt{p} : X \to \Sigma$ be the fibre bundle projection, $f$ be a strictly convex smooth function of $s$ such that $s + f \left(s\right)$ is strictly increasing, and $\omega$ be a smooth K\"ahler metric on $X \smallsetminus \left(S_0 \cup S_\infty\right)$ given by the \textit{Calabi ansatz} as follows \cite{Pingali:2018:heK,Szekelyhidi:2014:eKintro}:
\begin{equation}\label{eq:ansatzconesing}
\omega = \mathtt{p}^* \omega_\Sigma + \sqrt{-1} \partial \bar{\partial} f \left(s\right)
\end{equation}
where $\mathtt{p}^* \omega_\Sigma$ denotes the pullback to $X$ of the \krm{} $\omega_\Sigma$ with respect to $\mathtt{p}$. \par
The ansatz (\ref{eq:ansatzconesing}) produces \krm{s} $\omega$, which are said to have \textit{Calabi symmetry}, depending only on a suitable choice of the convex function $f : \R \to \R$ \cite{Hwang:2002:MomentConstruct}. Our task in Section \ref{sec:MomentConstructConehcscK} is to construct $\omega$ given by (\ref{eq:ansatzconesing}) which extends to the whole of $X$ by developing conical singularities of cone angles $2 \pi \beta_0 > 0$ and $2 \pi \beta_\infty > 0$ along the hypersurface divisors $S_0$ and $S_\infty$ respectively and which is (conical) higher cscK on $X$. The top Chern form $c_2 \left(\omega\right) \bigr\rvert_{X \smallsetminus \left(S_0 \cup S_\infty\right)}$ and the higher scalar curvature $\lambda \left(\omega\right) : X \smallsetminus \left(S_0 \cup S_\infty\right) \to \R$ should a priori satisfy the following PDE on the non-compact surface $X \smallsetminus \left(S_0 \cup S_\infty\right)$:
\begin{equation}\label{eq:topChernXminus}
c_2 \left(\omega\right) \bigr\rvert_{X \smallsetminus \left(S_0 \cup S_\infty\right)} = \frac{\lambda \left(\omega\right)}{2 \left(2 \pi\right)^2} \omega^2
\end{equation}
where $\lambda \left(\omega\right) \in \R$ should be a constant, for $\omega$ to be conical higher cscK on $X$ (refer to Definition \ref{def:conehcscK}). \par
By doing local calculations in the coordinates $\left(z, w\right)$, where $w = 0$ gives the zero divisor $S_0$, we get the following expression for $\omega$ on $X \smallsetminus \left(S_0 \cup S_\infty\right)$ \cite{Hwang:2002:MomentConstruct}:
\begin{multline}\label{eq:omega1conesing}
\omega \hspace{2pt} = \hspace{2pt} \left(1 + f' \left(s\right)\right) \mathtt{p}^* \omega_\Sigma \hspace{2pt} + \hspace{2pt} \sqrt{-1} \frac{f'' \left(s\right)}{h \left(z\right)^2} \left\lvert \frac{\partial h}{\partial z} \right\rvert^2 d z \wedge d \bar{z} \\
+ \hspace{2pt} \sqrt{-1} \frac{f'' \left(s\right)}{\bar{w} h \left(z\right)} \frac{\partial h}{\partial z} d z \wedge d \bar{w} \hspace{2pt} + \hspace{2pt} \sqrt{-1} \frac{f'' \left(s\right)}{w h \left(z\right)} \frac{\partial h}{\partial \bar{z}} d w \wedge d \bar{z} \hspace{2pt} + \hspace{2pt} f'' \left(s\right) \sqrt{-1} \frac{d w \wedge d \bar{w}}{\left\lvert w \right\rvert^2}
\end{multline}
where we used the fact that $\ideldb \ln h \left(z\right) = \mathtt{p}^* \omega_\Sigma$, as the curvature form of $h$ is given to be $-\omega_\Sigma$. It can be checked that (\ref{eq:omega1conesing}) remains unchanged under a bi\hol{} change of coordinates $z' = \psi \left(z\right)$ and hence does not depend on the choice of the local coordinate $z$ on the base $\Sigma$ \cite{Hwang:2002:MomentConstruct}. \par
Now in order to simplify the coordinate expression (\ref{eq:omega1conesing}), we use a clever trick given in Sz\'ekelyhidi \cite{Szekelyhidi:2014:eKintro}; Section 4.4 and followed in Pingali \cite{Pingali:2018:heK}; Section 2 and in \cite{Sompurkar:2023:heKsmooth}; Subsection 2.2: For any point $Q \in X \smallsetminus \left(S_0 \cup S_\infty\right)$ we can choose a local trivialization $\left(z, w\right)$ around $Q$ such that $Q = \left(z_0, w_0\right)$ and $d \ln h \left(z_0\right) = 0$, then the local expression for $\omega$ at the point $Q$ in terms of these chosen coordinates $\left(z,w\right)$ simplifies to the following \cite{Pingali:2018:heK,Szekelyhidi:2014:eKintro}:
\begin{equation}\label{eq:omegaconesing}
\omega = \left(1 + f' \left(s\right)\right) \mathtt{p}^* \omega_\Sigma + f'' \left(s\right) \sqrt{-1} \frac{d w \wedge d \bar{w}}{\left\lvert w \right\rvert^2}
\end{equation}
It can again be checked that if any other local trivialization around the coordinate $z$ is chosen which will be of the form $w' = \varphi \left(z\right) w$ for some non-zero \hol{} function $\varphi \left(z\right)$ then the same expression (\ref{eq:omegaconesing}) for $\omega$ is obtained \cite{Pingali:2018:heK,Szekelyhidi:2014:eKintro}. Thus the local expression (\ref{eq:omegaconesing}) for $\omega$ does not depend on the choice of the local trivialization $w$ for the line bundle $L$ and hence holds true at all points of the surface $X \smallsetminus \left(S_0 \cup S_\infty\right)$ \cite{Pingali:2018:heK,Szekelyhidi:2014:eKintro}. \par
We then want to determine the \krcl{} of the ansatz (\ref{eq:ansatzconesing}) while considering the numerical characterization of the \kr{} cone of $X$ given by (\ref{eq:KConeX}). Though a general \krcl{} on the surface $X$ is given by $a \mathsf{C} + b S_\infty$ where $a, b > 0$, we will consider only the \krcl{es} of the form $\pcsoo$ where $m > 0$ following the convention of \cite{Pingali:2018:heK,Szekelyhidi:2014:eKintro}. The property of a \krm{} (smooth or conical) on any compact complex manifold being any one kind of the five notions of canonical \krm{s} discussed in Subsection \ref{subsec:CanonKhlrSmooth} (\kr{}-Einstein, cscK, \ext{} \kr{}, higher cscK or higher \ext{} \kr{}) does not change under scaling the metric by a positive constant (see the concerned definitions and expressions given in Subsection \ref{subsec:CanonKhlrSmooth} for this). So we can very well rescale our ansatz (\ref{eq:ansatzconesing}) obtained in the \krcl{es} $\pcsoo$ by a factor of $\frac{k}{2 \pi}$ where $k > 0$ to get the desired \krm{s} in all the \krcl{es} of $X$ (which are going to be of the form $\kcsoo$). So it suffices to obtain all our results for the specially chosen subset $\left\lbrace \hspace{2pt} \pcsoo \hspace{1.5pt} \big\vert \hspace{2.5pt} m > 0 \hspace{2pt} \right\rbrace$ of the \kr{} cone of $X$. \par
So we now want $\omega$ given by (\ref{eq:ansatzconesing}) to be in the K\"ahler class $2 \pi \left(\mathsf{C} + m S_\infty\right)$ where $m > 0$, and for that to happen we must have $0 \leq f' \left(s\right) \leq m$ along with $\lim\limits_{s \to -\infty} f' \left(s\right) = 0$ and $\lim\limits_{s \to \infty} f' \left(s\right) = m$ (which we get by integrating the equation (\ref{eq:omegaconesing}) over $\mathsf{C}$ and $S_\infty$ individually and computing the lengths of the curves $\mathsf{C}$ and $S_\infty$ with respect to the metric $\omega$ by using the intersection formulae (\ref{eq:IntersectForm}) and (\ref{eq:Sigma}), as done in \cite{Pingali:2018:heK}; Section 2 and \cite{Szekelyhidi:2014:eKintro}; Section 4.4). \par
We compute the volume form associated with $\omega$ which is $\frac{\omega^2}{2}$ and the \textit{curvature form matrix} of $\omega$ given by $\Theta \left(\omega\right) = \bar{\partial} \left( H^{-1} \partial H \right) \left(\omega\right)$, where $H \left(\omega\right)$ is the Hermitian matrix of $\omega$ in the coordinates $\left(z,w\right)$, as follows \cite{Pingali:2018:heK,Szekelyhidi:2014:eKintro}:
\begin{equation}\label{eq:omega2conesing}
\omega^2 = 2 \left(1 + f' \left(s\right)\right) f'' \left(s\right) \mathtt{p}^* \omega_\Sigma \wedge \sqrt{-1} \frac{d w \wedge d \bar{w}}{\left\lvert w \right\rvert^2}
\end{equation}
\begin{equation}\label{eq:Curv1conesing}
\sqrt{-1} \Theta \left(\omega\right) =
\begin{bmatrix}
- \sqrt{-1} \partial \bar{\partial} \ln \left(1 + f' \left(s\right)\right) - 2 \mathtt{p}^* \omega_\Sigma & 0 \\
0 & - \sqrt{-1} \partial \bar{\partial} \ln \left(f'' \left(s\right)\right)
\end{bmatrix}
\end{equation}
This is where the conditions $1 + f' \left(s\right) > 0$ and $f'' \left(s\right) > 0$ assumed a priori are needed (and also in (\ref{eq:omegaconesing})). \par
We then use the Legendre transform $F \left(\tau\right)$ in the variable $\tau = f' \left(s\right) \in \clom$ (with $\tau$ being called as the \textit{momentum variable} of $\omega$ and the interval $\clom$ being called as the \textit{momentum interval} of $\omega$) given as $f \left(s\right) + F \left(\tau\right) = s \tau$ \cite{Hwang:2002:MomentConstruct,Pingali:2018:heK,Szekelyhidi:2014:eKintro}. We then define the \textit{momentum profile} of $\omega$ as $\phi \left(\tau\right) = \frac{1}{F'' \left(\tau\right)} = f'' \left(s\right)$ (again as in \cite{Hwang:2002:MomentConstruct,Pingali:2018:heK,Szekelyhidi:2014:eKintro}). This change of variables from $f \left(s\right)$ to $\phi \left(\tau\right)$ developed in the works \cite{Hwang:1994:cscK,Hwang:2002:MomentConstruct} is then called as the \textit{momentum construction}. \par
We can also write down the expression for the curvature form matrix $\Theta \left(\omega\right)$ in terms of $\phi \left(\gamma\right) = f'' \left(s\right)$, where $\gamma = \tau + 1 \in \left[1, m + 1\right]$, as follows \cite{Pingali:2018:heK}:
\begin{equation}\label{eq:Curv2conesing}
\sqrt{-1} \Theta \left(\omega\right) =
\begin{bmatrix}
\frac{\phi}{\gamma} \left( \frac{\phi}{\gamma} - \phi' \right) \sqrt{-1} \frac{d w \wedge d \bar{w}}{\left\lvert w \right\rvert^2} - \left(\frac{\phi}{\gamma} + 2\right) \mathtt{p}^* \omega_\Sigma & 0 \\
0 & - \phi'' \phi \sqrt{-1} \frac{d w \wedge d \bar{w}}{\left\lvert w \right\rvert^2} - \phi' \mathtt{p}^* \omega_\Sigma
\end{bmatrix}
\end{equation} \par
Before proceeding further we first note down the following relations between the variables $w$, $s$, $\tau$ and $\gamma$ and the functions $f$, $F$ and $\phi$ that have been introduced so far, which need to be used for getting the expression (\ref{eq:Curv2conesing}) \cite{Hwang:2002:MomentConstruct,Pingali:2018:heK,Szekelyhidi:2014:eKintro}:
\begin{equation}\label{eq:variablechange}
s = F' \left(\tau\right) \hspace{1pt}, \hspace{5pt} f''' \left(s\right) = - \frac{F''' \left(\tau\right)}{\left(F'' \left(\tau\right)\right)^3} = \phi' \left(\tau\right) \phi \left(\tau\right) \hspace{1pt}, \hspace{5pt} d s = F'' \left(\tau\right) d \tau = \frac{d \tau}{\phi \left(\tau\right)}
\end{equation}
\begin{equation}\label{eq:wgammasosoo}
\begin{gathered}
w \to 0 \iff s \to -\infty \iff \tau \to 0 \iff \gamma \to 1 \hspace{7pt} \text{(Corresponding to $S_0$)} \\
w \to \infty \iff s \to \infty \iff \tau \to m \iff \gamma \to m + 1 \hspace{7pt} \text{(Corresponding to $S_\infty$)}
\end{gathered}
\end{equation}
Thus the momentum construction taking $f \left(s\right)$, $s \in \R$ to $\phi \left(\gamma\right)$, $\gamma \in \left[1, m + 1\right]$ outlined in Hwang-Singer \cite{Hwang:2002:MomentConstruct}; Section 1 makes $\omega$ completely and uniquely determined by $\phi$ (as $\omega$ was determined by $f$ alone in the ansatz (\ref{eq:ansatzconesing})), makes the behaviour of $\omega$ at $\so$ and $\soo$ (in terms of smoothness or conical singularities) easily readable from the boundary conditions on $\phi$, $\phi'$ at the endpoints $1$ and $m + 1$ respectively as will be seen in (\ref{eq:BVPConeSing}) below, and most importantly gives the higher scalar curvature $\lambda \left(\omega\right)$ as a second-order fully non-linear differential expression in $\phi \left(\gamma\right)$ (rather than a fourth-order fully non-linear differential expression in $f \left(s\right)$ which it would have been) as will be seen in (\ref{eq:lambdaconesing}) below. \par
The \textit{top Chern form} of $\omega$ is given by $c_2 \left(\omega\right) \bigr\rvert_{X \smallsetminus \left(S_0 \cup S_\infty\right)} = \frac{1}{\left(2 \pi\right)^2} \det \left( \sqrt{-1} \Theta \left(\omega\right) \right)$ and in terms of $\phi \left(\gamma\right)$ is given by \cite{Pingali:2018:heK}:
\begin{equation}\label{eq:Chernconesing}
c_2 \left(\omega\right) \bigr\rvert_{X \smallsetminus \left(S_0 \cup S_\infty\right)} = \frac{1}{\left(2 \pi\right)^2} \mathtt{p}^* \omega_\Sigma \wedge \sqrt{-1} \frac{d w \wedge d \bar{w}}{\left\lvert w \right\rvert^2} \frac{\phi}{\gamma^2} \left( \gamma \left(\phi + 2 \gamma\right) \phi'' + \phi' \left(\phi' \gamma - \phi\right) \right)
\end{equation} \par
Comparing the expressions (\ref{eq:Chernconesing}) and (\ref{eq:omega2conesing}) for $c_2 \left(\omega\right) \bigr\rvert_{X \smallsetminus \left(S_0 \cup S_\infty\right)}$ and $\omega^2$ respectively with the equation (\ref{eq:topChernXminus}) that is to be studied for the ansatz (\ref{eq:ansatzconesing}), we obtain the expression for the \textit{higher scalar curvature} $\lambda \left(\omega\right)$ in terms of $\phi \left(\gamma\right)$ as follows \cite{Pingali:2018:heK}:
\begin{equation}\label{eq:lambdaconesing}
\lambda \left(\omega\right) = \frac{1}{\gamma^3} \left( \gamma \left(\phi + 2 \gamma\right) \phi'' + \phi' \left(\phi' \gamma - \phi\right) \right)
\end{equation} \par
Now substituting $\lambda \left(\omega\right) = B \in \R$ to be some constant (since $\omega$ is required to be higher cscK on $X \smallsetminus \left(S_0 \cup S_\infty\right)$), we obtain the following ODE for the momentum profile $\phi : \left[1, m + 1\right] \to \mathbb{R}$ for some constant $C \in \mathbb{R}$ (which is obtained by integrating the equation (\ref{eq:lambdaconesing}) once with respect to $\gamma$) \cite{Pingali:2018:heK}:
\begin{equation}\label{eq:ODEConeSing}
\left(2 \gamma + \phi\right) \phi' = B \frac{\gamma^3}{2} + C \gamma \hspace{1pt}, \hspace{5pt} \gamma \in \left[1, m + 1\right]
\end{equation} \par
Now for $\omega$ to develop conical singularities with cone angles $2 \pi \beta_0 > 0$ and $2 \pi \beta_\infty > 0$ along $S_0$ and $S_\infty$ respectively, it can be seen from the works of Edwards \cite{Edwards:2019:ContractDivConicKRFlow}, Hashimoto \cite{Hashimoto:2019:cscKConeSing}, Li \cite{Li:2012:eKEngyFunctProjBund}, Rubinstein-Zhang \cite{Rubinstein:2022:KEedgeHirzebruch}, Schlitzer-Stoppa \cite{Schlitzer:2023:dHYM} that the correct boundary conditions on $\phi : \left[1, m + 1\right] \to \mathbb{R}$ are the following:
\begin{equation}\label{eq:BVPConeSing}
\begin{gathered}
\phi\left(1\right) = 0 \hspace{1pt}, \hspace{5pt} \phi\left(m + 1\right) = 0 \\
\phi'\left(1\right) = \beta_0 \hspace{1pt}, \hspace{5pt} \phi'\left(m + 1\right) = - \beta_\infty
\end{gathered}
\end{equation}
The boundary conditions on $\phi'$ in (\ref{eq:BVPConeSing}) involve the cone angles $2 \pi \beta_0$ and $2 \pi \beta_\infty$, and we note that the momentum-constructed \krm{} $\omega$ will be smooth along $\so$ if $\bo = 1$ and similarly will be smooth along $\soo$ if $\boo = 1$ \cite{Edwards:2019:ContractDivConicKRFlow,Hashimoto:2019:cscKConeSing,Hwang:2002:MomentConstruct,Li:2012:eKEngyFunctProjBund,Rubinstein:2022:KEedgeHirzebruch,Schlitzer:2023:dHYM,Szekelyhidi:2006:eKKStab}. \par
And we also have the following condition:
\begin{equation}\label{eq:positiveConeSing}
\phi > 0 \hspace{5.5pt} \text{on} \hspace{3pt} \left(1, m + 1\right) \hspace{0.02pt}, \hspace{3.7pt} \text{as} \hspace{5pt} \phi \left(\gamma\right) = f'' \left(s\right) > 0 \hspace{1.2pt}, \hspace{3.3pt} \text{for all} \hspace{4.2pt} s \in \R
\end{equation}
which is required for $\omega$ to be a \krm{} i.e. a closed strictly positive $\left(1,1\right)$-form \cite{Hwang:2002:MomentConstruct,Pingali:2018:heK,Szekelyhidi:2014:eKintro}. \par
Thus the problem of constructing the \krm{} $\omega$ with the required properties on the surface $X$ by the momentum construction method finally boils down to solving the ODE (\ref{eq:ODEConeSing}) for the momentum profile $\phi \left(\gamma\right)$ on $\left[1, m + 1\right]$ with the boundary conditions (\ref{eq:BVPConeSing}) and the additional condition (\ref{eq:positiveConeSing}), i.e. the momentum construction method has converted the PDE (\ref{eq:topChernXminus}) defining the higher scalar curvature of a conical higher cscK metric into an explicit ODE boundary value problem on an interval of the real line whose solution uniquely determines the metric satisfying the PDE. \par
Since the ODE boundary value problem (\ref{eq:ODEConeSing}), (\ref{eq:BVPConeSing}), (\ref{eq:positiveConeSing}) depends on the parameters $m$ (determining the \krcl{} of the ansatz), $\bo$, $\boo$ (the values of the cone angles at $\so$, $\soo$ respectively up to a factor of $2 \pi$), $B$ (the constant value of the higher scalar curvature) and $C$ (a constant of integration), its analysis involves studying the relationships between these parameters and the boundary conditions, and also analyzing the behaviour of the polynomial $B \frac{\gamma^3}{2} + C \gamma$, as we shall see in Subsection \ref{subsec:AnalysisODEBVPConeSing} and Section \ref{sec:ProofConeSing}.
\subsection{Analysis of the ODE Boundary Value Problem for the Momentum Profile}\label{subsec:AnalysisODEBVPConeSing}
Following Pingali \cite{Pingali:2018:heK}, we start by defining the polynomial $p \left(\gamma\right) = B \frac{\gamma^2}{2} + C$ and applying the transformation $v = \frac{\left(2 \gamma + \phi\right)^2}{2}$, $\gamma \in \left[1, m + 1\right]$, after which the ODE (\ref{eq:ODEConeSing}) along with the boundary conditions (\ref{eq:BVPConeSing}) and the condition (\ref{eq:positiveConeSing}), which we had obtained at the end of Subsection \ref{subsec:MomentConstructConeSing}, reduces to the following ODE boundary value problem:
\begin{equation}\label{eq:ODEBVPConeSing}
\begin{gathered}
v' = 2 \sqrt{2} \sqrt{v} + p \left(\gamma\right) \gamma \hspace{1pt}, \hspace{5pt} \gamma \in \left[1, m + 1\right] \\
v\left(1\right) = 2 \hspace{1pt}, \hspace{5pt} v\left(m + 1\right) = 2 \left(m + 1\right)^2 \\
v'\left(1\right) = 2 \left(\beta_0 + 2\right) \hspace{1pt}, \hspace{5pt} v'\left(m + 1\right) = - 2 \left(m + 1\right) \left(\beta_\infty - 2\right) \\
v\left(\gamma\right) > 2 \gamma^2 \hspace{1pt}, \hspace{5pt} \gamma \in \left(1, m + 1\right)
\end{gathered}
\end{equation} \par
Again as done in \cite{Pingali:2018:heK}, imposing all the boundary conditions on the ODE in (\ref{eq:ODEBVPConeSing}) and evaluating individually at $\gamma = 1$ and $\gamma = m + 1$ gives us $B$, $C$ as linear functions of $\beta_0$, $\beta_\infty$ as follows:
\begin{equation}\label{eq:BCConeSing}
B \left(\bo, \boo\right) = - \frac{4 \left(\beta_0 + \beta_\infty\right)}{m \left(m + 2\right)} \hspace{1pt}, \hspace{5pt} C \left(\bo, \boo\right) = \frac{2 \left(\beta_0 \left(m + 1\right)^2 + \beta_\infty\right)}{m \left(m + 2\right)}
\end{equation}
Note that $B < 0$ and $C > 0$ as $\bo, \boo > 0$ and $m > 0$. \par
Considering the values of $B$, $C$ given in (\ref{eq:BCConeSing}), we can check that the boundary condition $v\left(1\right) = 2$ with the ODE in (\ref{eq:ODEBVPConeSing}) will automatically imply the boundary condition $v'\left(1\right) = 2 \left(\beta_0 + 2\right)$, and similarly the boundary condition $v\left(m + 1\right) = 2 \left(m + 1\right)^2$ will imply the boundary condition $v'\left(m + 1\right) = - 2 \left(m + 1\right) \left(\beta_\infty - 2\right)$. So just like in \cite{Pingali:2018:heK}, the boundary conditions $v'\left(1\right) = 2 \left(\beta_0 + 2\right)$ and $v'\left(m + 1\right) = - 2 \left(m + 1\right) \left(\beta_\infty - 2\right)$ in the ODE boundary value problem (\ref{eq:ODEBVPConeSing}) are redundant, as in they simply follow from the boundary conditions $v\left(1\right) = 2$ and $v\left(m + 1\right) = 2 \left(m + 1\right)^2$ respectively after substituting the expressions for $B$, $C$ given by (\ref{eq:BCConeSing}). \par
The following result analyzing the polynomials $p \left(\gamma\right)$ and $p \left(\gamma\right) \gamma$ with their coefficients containing $B = B \left(\bo, \boo\right)$ and $C = C \left(\bo, \boo\right)$ will give us the redundancy of the additional condition $v\left(\gamma\right) > 2 \gamma^2$ for all $\gamma \in \left(1, m + 1\right)$ in (\ref{eq:ODEBVPConeSing}) as well, and will also play a very crucial role in the entire analysis of the ODE boundary value problem (\ref{eq:ODEBVPConeSing}) to be developed in Subsection \ref{subsec:AnalysisODEBVPConeSing} and Section \ref{sec:ProofConeSing} (very similar to the analysis of the ODE obtained in the smooth higher \ext{} \kr{} case studied in \cite{Pingali:2018:heK,Sompurkar:2023:heKsmooth}).
\begin{lemma}\label{lem:polynomconesing}
The polynomials $p \left(\gamma\right)$ and $p \left(\gamma\right) \gamma$ both have exactly one root $\gamma_0 = \sqrt{-\frac{2 C}{B}}$ in the interval $\cllml{}$ with $1 < \gamma_0 < m + 1$. The quadratic $p \left(\gamma\right)$ is strictly decreasing on $\cllml{}$ with $p \left(1\right) = 2 \bo > 0$ and $p \left(m + 1\right) = -2 \boo < 0$, while the cubic $p \left(\gamma\right) \gamma$ can have at most one critical point (which would be a point of local maximum) $\gamma_{0 0} = \sqrt{-\frac{2 C}{3 B}}$ in the interval $\cllml{}$ with $\gamma_{0 0} < \gamma_0$. Both $p \left(\gamma\right)$ and $p \left(\gamma\right) \gamma$ are strictly positive on $\left[1, \gamma_0\right)$ and are strictly negative on $\left(\gamma_0, m + 1\right]$.
\end{lemma}
\begin{proof}
The proof of Lemma \ref{lem:polynomconesing} is entirely straightforward and we just outline it here. Solve the quadratic and cubic equations $p \left(\gamma\right) = 0$ and $p \left(\gamma\right) \gamma = 0$ respectively to get the value of $\gamma_0$, and also the quadratic equation $\frac{d}{d \gamma} \left( p \left(\gamma\right) \gamma \right) = 0$ to get the value of $\gamma_{0 0}$. Verify that $\frac{d}{d \gamma} \left( p \left(\gamma\right) \right) < 0$ and $\frac{d^2}{d \gamma^2} \left( p \left(\gamma\right) \gamma \right) < 0$ on $\cllml{}$ by using the fact that $B < 0$. Write down the expressions for $\gamma_0$ and $\gamma_{0 0}$ in terms of $\bo$, $\boo$ and $m$ by using the values of $B$ and $C$ given by (\ref{eq:BCConeSing}). Conclude the inequalities $1 < \gamma_0 < m + 1$ and also $\gamma_{0 0} < \gamma_0$ by observing the signs of $B = B \left(\bo, \boo\right)$ and $C = C \left(\bo, \boo\right)$ as well as $\bo$, $\boo$ and $m$. Evaluate $p \left(\gamma\right)$ and $p \left(\gamma\right) \gamma$ at $\gamma = 1$ and at $\gamma = m + 1$ individually, and again use the relations (\ref{eq:BCConeSing}) to conclude the signs of $p \left(\gamma\right)$ and $p \left(\gamma\right) \gamma$ on the subintervals $\left[1, \gamma_0\right)$ and $\left(\gamma_0, m + 1\right]$.
\end{proof} \par
With the help of the unique root $\gamma_0$ of the polynomial $p \left(\gamma\right) \gamma$ in $\cllml$, it can be proven (again as in \cite{Pingali:2018:heK}) that the boundary conditions $v\left(1\right) = 2$ and $v\left(m + 1\right) = 2 \left(m + 1\right)^2$ in (\ref{eq:ODEBVPConeSing}) taken together will automatically imply the condition $v\left(\gamma\right) > 2 \gamma^2$ for all $\gamma \in \left(1, m + 1\right)$. By using Lemma \ref{lem:polynomconesing}, we observe that if $v$ is a solution of the ODE in (\ref{eq:ODEBVPConeSing}) on the interval $\left[1, m + 1\right]$ satisfying both the boundary conditions $v \left(1\right) = 2$ and $v \left(m + 1\right) = 2 \left(m + 1\right)^2$, then integrating the expression for $\left(\sqrt{v}\right)'$ derived from (\ref{eq:ODEBVPConeSing}) on $\left[1, \gamma_0\right]$ and $\left[\gamma_0, m + 1\right]$ separately and noting the sign of $p \left(\gamma\right) \gamma$ on both the subintervals will help us conclude that $v \left(\gamma\right) \geq 2 \gamma^2$ on $\left[1, m + 1\right]$, and then rewriting the ODE in (\ref{eq:ODEBVPConeSing}) as $2 \sqrt{v} \left(\sqrt{v} - \sqrt{2} \gamma\right)' = p \left(\gamma\right) \gamma$ and using the mean value theorem for derivatives and the uniqueness of the root $\gamma_0$ in the interval $\cllml$ will help us conclude that $v \left(\gamma\right) > 2 \gamma^2$ on $\left(1, m + 1\right)$. Thus the ODE boundary value problem (\ref{eq:ODEBVPConeSing}) can be further simplified to the following:
\begin{equation}\label{eq:ODEBVP2ConeSing}
\begin{gathered}
v' = 2 \sqrt{2} \sqrt{v} + p \left(\gamma\right) \gamma \hspace{1pt}, \hspace{5pt} \gamma \in \left[1, m + 1\right] \\
v\left(1\right) = 2 \hspace{1pt}, \hspace{5pt} v\left(m + 1\right) = 2 \left(m + 1\right)^2
\end{gathered}
\end{equation} \par
Over here also (exactly like the smooth case of \cite{Pingali:2018:heK,Sompurkar:2023:heKsmooth}) the strategy of solving the ODE boundary value problem (\ref{eq:ODEBVP2ConeSing}) with $m > 0$ being fixed will be to first neglect the final boundary condition $v\left(m + 1\right) = 2 \left(m + 1\right)^2$ and consider the resulting ODE initial value problem:
\begin{equation}\label{eq:ODEIVPConeSing}
\begin{gathered}
v' = 2 \sqrt{2} \sqrt{v} + p \left(\gamma\right) \gamma \hspace{1pt}, \hspace{5pt} \gamma \in \left[1, m + 1\right] \\
v\left(1\right) = 2
\end{gathered}
\end{equation}
and then try to get the existence of solutions $v \left(\cdot; \bo, \boo\right)$ to (\ref{eq:ODEIVPConeSing}) on the whole of $\cllml$ for different values of the parameters $\bo, \boo$ and finally try to get a pair of values of $\bo, \boo$ for which the corresponding solution $v = v \left(\cdot; \bo, \boo\right)$ satisfies $v\left(m + 1\right) > 2 \left(m + 1\right)^2$ and another pair of values of $\bo, \boo$ for which the solution $v$ satisfies $v\left(m + 1\right) < 2 \left(m + 1\right)^2$. Since the ODE initial value problem (\ref{eq:ODEIVPConeSing}) varies continuously with respect to the parameters $\bo, \boo$, this should imply the existence of a pair of values of these parameters for which the required final boundary condition viz. $v\left(m + 1\right) = 2 \left(m + 1\right)^2$ is satisfied. Again as in the case of \cite{Pingali:2018:heK,Sompurkar:2023:heKsmooth}, proving the $<$ part is much more difficult than proving the $>$ part. We have the following existence result for the ODE boundary value problem (\ref{eq:ODEBVP2ConeSing}) (whose proof uses nothing more than elementary real analysis and standard ODE theory, but requires a careful handling of various estimates given by variations in the parameters $\bo, \boo$, as we shall see in detail in Section \ref{sec:ProofConeSing}):
\begin{theorem}[Existence Result for the ODE Boundary Value Problem (\ref{eq:ODEBVP2ConeSing})]\label{thm:mainconesing}
For every $m > 0$ and for every $\beta_0 > 0$ there exists a unique $\beta_\infty > 0$ (depending on both $m$ and $\bo$) with $\boo > \bo$ and there exist unique $B, C \in \R$ (given by the expressions (\ref{eq:BCConeSing}) with the substitution of the respective values of $m$, $\bo$ and $\boo$) with $B < 0$ and $C > 0$, such that there exists a unique smooth solution $v : \cllml{} \to \R$ (depending on all these parameters) to the ODE boundary value problem (\ref{eq:ODEBVP2ConeSing}) satisfying both the boundary conditions, and as a consequence also to the ODE boundary value problem (\ref{eq:ODEBVPConeSing}) satisfying all the conditions therein.
\end{theorem} \par
Remember that $m$ characterizes the \krcl{} under consideration, $\bo$, $\boo$ give the values of the cone angles at the divisors $\so$, $\soo$ respectively and $B$ is the value of the constant higher scalar curvature. As can be seen from the arguments in Subsections \ref{subsec:MomentConstructConeSing} and \ref{subsec:AnalysisODEBVPConeSing}, the following existence result for the required kind of \krm{} on the surface $X$ follows directly from Theorem \ref{thm:mainconesing}:
\begin{corollary}[Existence Result for Conical Higher cscK Metrics on the Minimal Ruled Surface]\label{cor:mainconesing}
For every $m > 0$ and for every $\beta_0 > 0$ there exists a unique $\beta_\infty > 0$ with $\boo > \bo$, such that there exists a unique higher cscK metric $\omega$ having Calabi symmetry on the minimal ruled surface $X = \prj$, which belongs to the K\"ahler class $2 \pi \left(\sfc + m S_\infty\right)$, which is smooth on $X \smallsetminus \left(S_0 \cup S_\infty\right)$ and which has got conical singularities with cone angles $2 \pi \beta_0$ and $2 \pi \beta_\infty$ along the divisors $S_0$ and $S_\infty$ respectively.
\end{corollary} \par
\begin{remark}
As we shall see in Section \ref{sec:ProofConeSing}, the ODE initial value problem (\ref{eq:ODEIVPConeSing}) depends on the parameters $m > 0$ and $\alpha = \bo - \boo$ (and not independently on each one of the two parameters $\bo, \boo > 0$). Out of all these parameters, $m$ and $\bo$ are allowed to take any arbitrary positive values but are kept fixed throughout the proof of Theorem \ref{thm:mainconesing}, while $\alpha$ is taken as the independent variable parameter with its values ranging over the interval $\left(-\infty, \bo\right)$ and with the parameter $\boo = - \alpha + \bo$ (and of course the parameters $B, C$ as well, being given by (\ref{eq:BCConeSing})) depending on and varying with $\alpha$. Then the analysis required for the proof of Theorem \ref{thm:mainconesing} will yield for every $m > 0$ and for every $\bo > 0$ a unique $\alpha = \alpha \left(m, \bo\right) \in \left(-\infty, \bo\right)$ with $\alpha < 0$ such that there will exist a unique smooth solution $v = v \left(\cdot; \alpha\right)$ to the ODE initial value problem (\ref{eq:ODEIVPConeSing}) satisfying the correct final boundary condition viz. $v\left(m + 1\right) = 2 \left(m + 1\right)^2$ and hence to the ODE boundary value problems (\ref{eq:ODEBVP2ConeSing}) as well as (\ref{eq:ODEBVPConeSing}) satisfying all the required (boundary) conditions. That is the reason why in Theorem \ref{thm:mainconesing} and Corollary \ref{cor:mainconesing} we are having $m$ and $\bo$ arbitrary and $\boo$ depending on both of them with the condition $\boo > \bo$ coming from the fact $\alpha \left(m, \bo\right) < 0$.
\end{remark}
\numberwithin{equation}{subsection}
\numberwithin{figure}{subsection}
\numberwithin{table}{subsection}
\numberwithin{lemma}{subsection}
\numberwithin{proposition}{subsection}
\numberwithin{result}{subsection}
\numberwithin{theorem}{subsection}
\numberwithin{corollary}{subsection}
\numberwithin{conjecture}{subsection}
\numberwithin{remark}{subsection}
\numberwithin{note}{subsection}
\numberwithin{motivation}{subsection}
\numberwithin{question}{subsection}
\numberwithin{answer}{subsection}
\numberwithin{case}{subsection}
\numberwithin{claim}{subsection}
\numberwithin{definition}{subsection}
\numberwithin{example}{subsection}
\numberwithin{hypothesis}{subsection}
\numberwithin{statement}{subsection}
\numberwithin{ansatz}{subsection}
\section{Proof of Theorem \ref{thm:mainconesing}}\label{sec:ProofConeSing}
\subsection{First Part of the Proof}\label{subsec:Proof1}
The final goal over here in Section \ref{sec:ProofConeSing} is to prove Theorem \ref{thm:mainconesing}. The method of attack of the ODE boundary value problem (\ref{eq:ODEBVP2ConeSing}) and the central ideas in the proof of Theorem \ref{thm:mainconesing} are all nearly the same as those for the very similar ODE boundary value problem obtained in the momentum construction of smooth higher \ext{} \krm{s}, whose analysis was initiated in Pingali \cite{Pingali:2018:heK}; Section 2 and was completed in the author's first work \cite{Sompurkar:2023:heKsmooth}; Section 3. As mentioned in Subsection \ref{subsec:AnalysisODEBVPConeSing} our main aim in Section \ref{sec:ProofConeSing} is to find a smooth solution $v : \left[1, m + 1\right] \to \R$ to the ODE initial value problem (\ref{eq:ODEIVPConeSing}) satisfying the correct final boundary condition viz. $v\left(m + 1\right) = 2 \left(m + 1\right)^2$, and for that we need to study the variation of the solution of the ODE in (\ref{eq:ODEIVPConeSing}) with respect to the parameters $\bo$, $\boo$ (or actually $\alpha = \bo - \boo$) and determine for what values of these parameters does a solution exist satisfying the required boundary conditions. As was the case in the smooth higher \ext{} \kr{} setting of this problem studied in \cite{Pingali:2018:heK}; Section 2 and \cite{Sompurkar:2023:heKsmooth}; Section 3, the (common) unique root $\gamma_0$ of the polynomials $p \left(\gamma\right)$ and $p \left(\gamma\right) \gamma$ in $\cllml$, which in this case is explicitly determined by Lemma \ref{lem:polynomconesing}, plays a major role in the analysis of the ODE in (\ref{eq:ODEIVPConeSing}) as we shall see here in Subsection \ref{subsec:Proof1}. \par
Throughout Section \ref{sec:ProofConeSing}, $m > 0$ will be kept fixed. An important thing that happens over here is that $\bo > 0$ can also be kept fixed throughout (along with $m$), and then the independent variable parameter can be taken to be $\alpha \in \left(-\infty, \bo\right)$ (which represents the difference between the cone angles at the two divisors) with the other three parameters $\boo = \boo \left(\alpha\right) = - \alpha + \bo \in \left(0, \infty\right)$, $B = B \left(\alpha\right)$, $C = C \left(\alpha\right)$ (both being given by the following expressions derived from the expressions (\ref{eq:BCConeSing})) all depending on $\alpha$, and hence even the ODE initial value problem (\ref{eq:ODEIVPConeSing}) will depend on $\alpha$.
\begin{equation}\label{eq:BCalphabeta0}
\begin{gathered}
B \left(\alpha\right) = \frac{4 \alpha}{m \left(m + 2\right)} - \frac{8 \beta_0}{m \left(m + 2\right)} \in \left(-\infty, -\frac{4 \bo}{m \left(m + 2\right)}\right) \\
C \left(\alpha\right) = - \frac{2 \alpha}{m \left(m + 2\right)} + \frac{2 \beta_0 \left(\left(m + 1\right)^2 + 1\right)}{m \left(m + 2\right)} \in \left(\frac{2 \beta_0 \left(m + 1\right)^2}{m \left(m + 2\right)}, \infty\right)
\end{gathered}
\end{equation} \par
\begin{motivation}
We will first prove that for every $\alpha \in \left(-\infty, \bo\right)$ there exists a unique $\mathcal{C}^1$ solution $v$ to the ODE initial value problem (\ref{eq:ODEIVPConeSing}) on a non-degenerate interval containing $1$, and in fact this $v$ exists and is strictly increasing on $\left[1, \gamma_0\right]$ where $\gamma_0 = \sqrt{-\frac{2 C}{B}}$ is the unique root of the polynomial $p \left(\gamma\right) \gamma = B \frac{\gamma^3}{2} + C \gamma$ in $\left[1, m + 1\right]$. We will then prove that the $\mathcal{C}^1$ solution $v$ defined on any interval is always strictly positive on the interval, and as a consequence is smooth (i.e. $\mathcal{C}^\infty$) on the interval. We will finally prove a necessary and sufficient condition for the continuation of the solution $v$ defined a priori on $\left[1, \tilde{r}\right)$ for a given $\tilde{r} \in \left(1, m + 1\right]$.
\end{motivation} \par
As was noted by Pingali \cite{Pingali:2018:heK}; Section 2 (for the ODE in the smooth analogue of this problem), if $v$ is a $\mathcal{C}^1$ solution to (\ref{eq:ODEIVPConeSing}) on any interval then substituting $\sqrt{v} < v + 1$ and $\left\lvert p \left(\gamma\right) \gamma \right\rvert \leq l$ (for some $l > 0$) in the expression for $v' = \left(v + 1\right)'$ in (\ref{eq:ODEIVPConeSing}) and applying Gr\"onwall's inequality will give us a $K > 0$ such that $v \left(\gamma\right) \leq K$ on the interval. \par
Thus solutions to (\ref{eq:ODEIVPConeSing}) are always bounded above (and always bounded below by $0$) on any interval on which they exist. So by standard ODE theory the existence of a strictly positive lower bound on a solution of (\ref{eq:ODEIVPConeSing}) is a sufficient condition for the continuation of the solution beyond its prior interval of definition (\cite{Pingali:2018:heK}; Section 2).
\begin{lemma}[Continuation of Solutions]\label{lem:continue}
For a given $\alpha \in \left(-\infty, \bo\right)$ let $v$ be a $\mathcal{C}^1$ solution to the ODE initial value problem (\ref{eq:ODEIVPConeSing}) existing on $\left[1, \tilde{r}\right) \subseteq \left[1, m + 1\right]$. If there exists an $\epsilon > 0$ such that $v \left(\gamma\right) \geq \epsilon$ on $\left[1, \tilde{r}\right)$ then $v$ can be continued beyond $\tilde{r}$.
\end{lemma}
\begin{proof}
We are given a lower bound $\epsilon > 0$ for the solution $v$ and we always have an upper bound $K > 0$ for $v$ as discussed above. So $0 < \sqrt{\epsilon} \leq \sqrt{v} \leq \sqrt{K}$ on $\left[1, \tilde{r}\right)$ and so, the right hand side of the ODE in (\ref{eq:ODEIVPConeSing}) is $\mathcal{C}^1$ in $v$. So by standard ODE theory the solution $v$ defined a priori on $\left[1, \tilde{r}\right)$ can be continued beyond $\tilde{r}$.
\end{proof} \par
We will prove the converse of Lemma \ref{lem:continue} (viz. Theorem \ref{thm:continue}), but before that we prove some more basic results:
\begin{lemma}[Existence of Solutions]\label{lem:firstexist}
For every $\alpha \in \left(-\infty, \bo\right)$ there exists a unique $\mathcal{C}^1$ solution $v$ to the ODE initial value problem (\ref{eq:ODEIVPConeSing}) on $\left[1, r\right)$ for some $r \in \left(1, m + 1\right]$ such that $v' > 0$ on $\left[1, r\right)$. If $\left[1, r'\right) \subseteq \left[1, m + 1\right]$ is the maximal interval of existence of $v$ then $\left[1, \gamma_0\right] \subseteq \left[1, r'\right)$ and $v' > 0$ on $\left[1, \gamma_0\right]$. Similarly if $\left[1, m + 1\right]$ is the maximal interval of existence of $v$ then (obviously) $\left[1, \gamma_0\right] \subseteq \left[1, m + 1\right]$ and $v' > 0$ on $\left[1, \gamma_0\right]$.
\end{lemma}
\begin{proof}
As noted in Subsection \ref{subsec:AnalysisODEBVPConeSing}, $v \left(1\right) = 2$ will automatically imply $v' \left(1\right) = 2 \left(\beta_0 + 2\right) > 4$ in (\ref{eq:ODEIVPConeSing}). And then since $\sqrt{v} \geq \sqrt{2} > 0$ near $\gamma = 1$, so the right hand side of (\ref{eq:ODEIVPConeSing}) is continuous in $\left(\gamma, v\right)$ and Lipschitz in $v$ locally in a neighbourhood of $\gamma = 1$ and by standard ODE theory there exists a unique $\mathcal{C}^1$ solution $v$ to (\ref{eq:ODEIVPConeSing}) on $\left[1, r\right)$ for some $r \in \left(1, m + 1\right]$. Since $v' \left(1\right) = 2 \left(\beta_0 + 2\right) > 0$ so this $r \in \left(1, m + 1\right]$ can be chosen so that $v' > 0$ on $\left[1, r\right)$. \\
Let $\left[1, r'\right) \subseteq \left[1, m + 1\right]$ be the maximal interval of existence of $v$. If $\gamma_0 \geq r'$ then by Lemma \ref{lem:polynomconesing}, $p \left(\gamma\right) \gamma \geq 0$ on $\left[1, r'\right)$ and hence $v' \geq 0$ on $\left[1, r'\right)$. So $v \left(\gamma\right) \geq v \left(1\right) = 2 > 0$ on $\left[1, r'\right)$ and by Lemma \ref{lem:continue}, $v$ can be continued beyond $r'$ contradicting the maximality of $r'$. So $\left[1, \gamma_0\right] \subseteq \left[1, r'\right)$ and $v' \geq 0$ on $\left[1, \gamma_0\right]$, but as $\sqrt{v \left(\gamma\right)} > 0$ on $\left[1, \gamma_0\right]$ so $v' > 0$ on $\left[1, \gamma_0\right]$. \\
If $\left[1, m + 1\right]$ is the maximal interval of existence of $v$ then by Lemma \ref{lem:polynomconesing}, $\left[1, \gamma_0\right] \subseteq \left[1, m + 1\right]$ and by the same arguments as in the above case, $v' > 0$ on $\left[1, \gamma_0\right]$.
\end{proof} \par
\begin{remark}
Note that $v' > 0$ on $\left[1, \gamma_0\right]$ actually implies $v' > 0$ on $\left[1, \gamma_0'\right) \subseteq \left[1, r'\right)$ (or $\left[1, \gamma_0'\right) \subseteq \left[1, m + 1\right]$) for some $\gamma_0' > \gamma_0$.
\end{remark} \par
Observe that if there exists a $\mathcal{C}^1$ solution $v$ to (\ref{eq:ODEIVPConeSing}) on any interval then $v \geq 0$ on the interval, but Lemma \ref{lem:positivity} below is saying that $v > 0$ on the interval.
\begin{lemma}[Positivity of Solutions]\label{lem:positivity}
For a given $\alpha \in \left(-\infty, \bo\right)$ let $v$ be the unique $\mathcal{C}^1$ solution to the ODE initial value problem (\ref{eq:ODEIVPConeSing}) existing on some interval containing $1$.
\begin{enumerate}
\item If $\left[1, r'\right) \subseteq \left[1, m + 1\right]$ is the maximal interval of existence of $v$ then $v \left(\gamma\right) > 0$ for all $\gamma \in \left[1, r'\right)$ and $\lim\limits_{\gamma \to r'} v \left(\gamma\right) = 0$ and $\lim\limits_{\gamma \to r'} v' \left(\gamma\right) < 0$. \label{itm:posit1}
\item If $\left[1, m + 1\right]$ is the maximal interval of existence of $v$ then $v \left(\gamma\right) > 0$ for all $\gamma \in \left[1, m + 1\right]$. \label{itm:posit2}
\end{enumerate}
\end{lemma}
\begin{proof}
\begin{case*}(\ref{itm:posit1}) \hspace{2pt}
If $\left[1, r'\right)$ is the maximal interval of existence of $v$ then by Lemma \ref{lem:firstexist}, $\gamma_0 \in \left[1, r'\right)$ and $v' \left(\gamma_0\right) > 0$ and also $v \left(\gamma_0\right) \geq v \left(1\right) = 2 > 0$. Since $v$ cannot be continued beyond $r'$ so by Lemma \ref{lem:continue}, $\inf\limits_{\gamma \in \left[1, r'\right)} v \left(\gamma\right) = 0$. If $t_0 \in \left(1, r'\right)$ is such that $v \left(t_0\right) = 0$ then $t_0$ must be a point of local minimum of $v$ and so $v' \left(t_0\right) = 0$ which will imply $p \left(t_0\right) t_0 = 0$ (from the ODE in (\ref{eq:ODEIVPConeSing})) and then $t_0 = \gamma_0$ (by the uniqueness of $\gamma_0$) which contradicts the first assertion above. So $v \left(\gamma\right) > 0$ for all $\gamma \in \left[1, r'\right)$. \\
Now since $v$ and $p \left(\gamma\right) \gamma$ are bounded on $\left[1, r'\right)$, from the expression of $v'$ in (\ref{eq:ODEIVPConeSing}) we get $v'$ is bounded, thereby implying $v$ is Lipschitz on $\left[1, r'\right)$. So $\lim\limits_{\gamma \to r'} v \left(\gamma\right)$ exists and as $v > 0$ on $\left[1, r'\right)$ so $\lim\limits_{\gamma \to r'} v \left(\gamma\right) = \inf\limits_{\gamma \in \left[1, r'\right)} v \left(\gamma\right) = 0$. \\
Now as $\sqrt{\cdot}$ on $\left[0, \infty\right)$, $v$ on $\left[1, r'\right)$ and $p \left(\gamma\right) \gamma$ on $\left[1, m + 1\right]$ are uniformly continuous so from (\ref{eq:ODEIVPConeSing}), $v'$ is uniformly continuous on $\left[1, r'\right)$ and so $\lim\limits_{\gamma \to r'} v' \left(\gamma\right)$ exists. Since $\lim\limits_{\gamma \to r'} v \left(\gamma\right) = 0$ and $v > 0$ on $\left[1, r'\right)$ so $\lim\limits_{\gamma \to r'} v' \left(\gamma\right) \leq 0$. If $\lim\limits_{\gamma \to r'} v' \left(\gamma\right) = 0$ then from (\ref{eq:ODEIVPConeSing}) we will get $r'$ is a root of $p \left(\gamma\right) \gamma$ which is not possible by Lemma \ref{lem:firstexist}. So $\lim\limits_{\gamma \to r'} v' \left(\gamma\right) < 0$.
\end{case*}
\begin{case*}(\ref{itm:posit2}) \hspace{2pt}
If $\left[1, m + 1\right]$ is the maximal interval of existence of $v$ then $v \geq 0$ on $\left[1, m + 1\right]$ and $v \left(1\right) = 2 > 0$ and from Lemma \ref{lem:firstexist}, $v' \left(\gamma_0\right) > 0$ and $v \left(\gamma_0\right) > 0$. So by the same argument as in Case (\ref{itm:posit1}), there cannot exist a $t_0 \in \left(1, m + 1\right)$ such that $v \left(t_0\right) = 0$. So $v \left(\gamma\right) > 0$ for all $\gamma \in \left[1, m + 1\right)$. \\
Let if possible $v \left(m + 1\right) = 0$. Since $v \geq 0$ and $v$ is $\mathcal{C}^1$ on $\left[1, m + 1\right]$ so $v' \left(m + 1\right) \leq 0$. If $v' \left(m + 1\right) = 0$ then from (\ref{eq:ODEIVPConeSing}) we will get $m + 1$ is a root of $p \left(\gamma\right) \gamma$ which is not possible by Lemma \ref{lem:polynomconesing}. So $v' \left(m + 1\right) < 0$ i.e. $v$ is strictly decreasing in a neighbourhood of $m + 1$. Now $v$ is the $\mathcal{C}^1$ solution of (\ref{eq:ODEIVPConeSing}) and $v > 0$ on $\left[1, m + 1\right)$, and for $v$ to be extendable as the $\mathcal{C}^1$ solution to (\ref{eq:ODEIVPConeSing}) on an interval strictly containing $\left[1, m + 1\right)$ we must have $v \geq 0$ on the larger interval which will not be possible with $v' \left(m + 1\right) < 0$. So in that case, $v$ will exist as the $\mathcal{C}^1$ solution of (\ref{eq:ODEIVPConeSing}) maximally on $\left[1, m + 1\right)$, a contradiction to the hypothesis. So $v \left(m + 1\right) > 0$ and hence $v \left(\gamma\right) > 0$ for all $\gamma \in \left[1, m + 1\right]$.
\end{case*}
\end{proof} \par
Before proceeding further let us observe the following about the regularity of the $\mathcal{C}^1$ solutions of the ODE initial value problem (\ref{eq:ODEIVPConeSing}):
\begin{corollary}[Smoothness and Real Analyticity of Solutions]\label{cor:InfinitDiff}
Let $v$ be the $\mathcal{C}^1$ solution of the ODE initial value problem (\ref{eq:ODEIVPConeSing}) on a non-degenerate interval $\mathsf{J} \subseteq \cllml$. Then $v'$ is bounded (thus $v$ is Lipschitz) on $\mathsf{J}$ and $v'$ is uniformly continuous on $\mathsf{J}$, and $v^{\left(i\right)}$ exists on $\mathsf{J}$ for all $i \in \mathbb{N}_{\geq 2}$, thereby $v$ is smooth (i.e. $\mathcal{C}^\infty$) on $\mathsf{J}$. For all $i \in \mathbb{N}_{\geq 2}$ if $\inf\limits_{\mathsf{J}} v > 0$ then $v^{\left(i\right)}$ is bounded on $\mathsf{J}$ and if $\inf\limits_{\mathsf{J}} v = 0$ then $v^{\left(i\right)}$ is unbounded on $\mathsf{J}$. Further $v$ is in fact $\mathcal{C}^\omega$ (i.e. real analytic) on $\mathsf{J}$.
\end{corollary}
\begin{proof}
By Lemma \ref{lem:positivity}, $v > 0$ on $\mathsf{J}$ and so $\frac{1}{\sqrt{v}}$ makes sense. Considering the ODE in (\ref{eq:ODEIVPConeSing}) and its $\left(i - 1\right)$\textsuperscript{th} derivative, substituting the implied property of $\sqrt{\cdot}$, $v$ and $p \left(\gamma\right) \gamma$ in both the cases into these ODEs and using induction will give us the required results for all $i \in \mathbb{N}_{\geq 2}$. So $v$ will be smooth on $\mathsf{J}$. \\
Regarding the real analyticity of $v$, again as $v > 0$ on $\mathsf{J}$ so the right hand side of the ODE in (\ref{eq:ODEIVPConeSing}) is real analytic in $\left(\gamma, v\right)$ with the initial condition being $v \left(1\right) = 2$. So again by standard ODE theory $v$ is real analytic on $\mathsf{J}$.
\end{proof} \par
\begin{remark}
Corollary \ref{cor:InfinitDiff} will imply that the $\mathcal{C}^1$ solution $v$ of the ODE initial value problem (\ref{eq:ODEIVPConeSing}) is automatically $\mathcal{C}^\infty$ and furthermore even $\mathcal{C}^\omega$ on its maximal interval of existence. This is because by Lemma \ref{lem:positivity}, $v$ remains strictly positive as long as it exists as a solution to (\ref{eq:ODEIVPConeSing}), i.e. $v > 0$ on the maximal interval of existence. In particular if $\cllml$ is the maximal interval of existence of $v$ then $v \in \mathcal{C}^\omega \left[1, m + 1\right] \subseteq \mathcal{C}^\infty \left[1, m + 1\right]$.
\end{remark} \par
Observe one more thing that the solution $v$ cannot be constant on any non-degenerate interval $\mathsf{J}$, as that would imply (again from (\ref{eq:ODEIVPConeSing})) that the polynomial $p \left(\gamma\right) \gamma$ is a constant polynomial which is not possible by Lemma \ref{lem:polynomconesing}. We have the following result giving the precise behaviour of the solution $v$ with regards to its monotonicity and its points of local extrema:
\begin{corollary}\label{cor:incredecreextrm}
Let $v$ be the solution of the ODE initial value problem (\ref{eq:ODEIVPConeSing}) existing on some non-degenerate interval $\left[1, \tilde{r}\right) \subseteq \left[1, m + 1\right]$. Then $v$ can have at most one point of local extremum (which if it exists will always be a point of local maximum) in $\left[1, \tilde{r}\right)$. Furthermore,
\begin{enumerate}
\item If $\left[1, r'\right) \subseteq \left[1, m + 1\right]$ is the maximal interval of existence of $v$ then $v$ will have a unique point of local maximum $t \in \left[1, r'\right)$ with $t > \gamma_0$ (and no local minima in $\left[1, r'\right)$). \label{itm:extrm1}
\item If $\left[1, m + 1\right]$ is the maximal interval of existence of $v$ then either $v$ will have a unique point of local maximum $t \in \left[1, m + 1\right]$ with $t > \gamma_0$ (just like Case (\ref{itm:extrm1})) or $v$ will be strictly increasing on $\left[1, m + 1\right]$. \label{itm:extrm2}
\end{enumerate}
\end{corollary}
\begin{proof}
By Corollary \ref{cor:InfinitDiff} $v$ is real analytic on $\left[1, \tilde{r}\right)$, and since $v$ is non-constant, the set of critical points of $v$ in $\left[1, \tilde{r}\right)$ is finite, and so $v$ can have only finitely many local extrema in $\left[1, \tilde{r}\right)$. Let if possible $v$ have two or more local extrema in $\left[1, \tilde{r}\right)$. Since $v \left(1\right) = 2$ and $v' \left(1\right) = 2 \left(\beta_0 + 2\right) > 4$, $v$ is strictly increasing in a \nbd{} of $1$. So given the existence of at least two local extrema, $v$ first has to attain a local maximum say $t$ and then must have a local minimum say $t_0$ in $\left[1, \tilde{r}\right)$ (with there being no other local extrema in between $1$ and $t$ as well as in between $t$ and $t_0$). In such a situation considering Lemma \ref{lem:firstexist}, we must have $1 < \gamma_0 < t < t_0 < \tilde{r}$ (as $v' > 0$ on $\left[1, \gamma_0\right]$). As $t$ and $t_0$ are critical points of $v$ we get $v\left(t\right) = \frac{p\left(t\right)^2 t^2}{8}$ and $v\left(t_0\right) = \frac{p\left(t_0\right)^2 t_0^2}{8}$ from the equation of $v'$ in (\ref{eq:ODEIVPConeSing}). \\
From Lemma \ref{lem:polynomconesing}, $\frac{p\left(\gamma\right)^2 \gamma^2}{8} \geq 0$ and $\gamma_0$ (which is the unique root of $p\left(\gamma\right) \gamma$) is also the unique root of $\frac{p\left(\gamma\right)^2 \gamma^2}{8}$ in $\cllml$, with $\frac{p\left(\gamma\right)^2 \gamma^2}{8}$ having a double root at $\gamma_0$. Apart from $\gamma_0$ (which will be a point of local minimum), $\frac{p\left(\gamma\right)^2 \gamma^2}{8}$ can have at most one other critical point (which would be a point of local maximum) in $\cllml{}$, and this point if it exists in $\cllml{}$ will be the same $\gamma_{0 0}$ (which was the only possible critical point of $p\left(\gamma\right) \gamma$ in $\cllml{}$) with $\gamma_{0 0} < \gamma_0$. Thus $\frac{p\left(\gamma\right)^2 \gamma^2}{8}$ is strictly increasing on $\left[\gamma_0, m + 1\right]$ and so we must have $\frac{p\left(t\right)^2 t^2}{8} < \frac{p\left(t_0\right)^2 t_0^2}{8}$. But on the contrary $v$ is strictly decreasing on $\left[t, t_0\right]$ and that gives us $v\left(t\right) > v\left(t_0\right)$. This contradicts the statement derived above that the values of the solution $v$ and the polynomial $\frac{p\left(\gamma\right)^2 \gamma^2}{8}$ are equal to each other at both the points $t$ and $t_0$. So $v$ can have at most one point of local extremum in $\left[1, \tilde{r}\right)$. Since $v$ is strictly increasing near $1$ and remains strictly increasing at least till $\gamma_0$, if $v$ has a point of local extremum in $\left[1, \tilde{r}\right)$ then clearly it will always be a point of local maximum and it will be strictly greater than $\gamma_0$.
\begin{case*}(\ref{itm:extrm1}) \hspace{2pt}
If $\left[1, r'\right) \subseteq \left[1, m + 1\right]$ is the maximal interval of existence of $v$ then from Lemmas \ref{lem:firstexist} and \ref{lem:positivity}, $v' > 0$ on $\left[1, \gamma_0\right] \subseteq \left[1, r'\right)$ and $v > 0$ on $\left[1, r'\right)$ and $\lim\limits_{\gamma \to r'} v \left(\gamma\right) = 0$ and $\lim\limits_{\gamma \to r'} v' \left(\gamma\right) < 0$. Here since $v$ is strictly increasing near $1$ and is strictly decreasing near $r'$, it must have at least one point of local maximum say $t$ in between $1$ and $r'$, and then by the very first assertion of Corollary \ref{cor:incredecreextrm} which we have already proven above, this point $t \in \left(1, r'\right)$ will be the only point of local extremum of $v$ and $t > \gamma_0$.
\end{case*}
\begin{case*}(\ref{itm:extrm2}) \hspace{2pt}
Let $\left[1, m + 1\right]$ be the maximal interval of existence of $v$ then again by Lemmas \ref{lem:firstexist} and \ref{lem:positivity}, $v' > 0$ on $\left[1, \gamma_0\right]$ and $v > 0$ on $\left[1, m + 1\right]$. Let us assume that $v$ is not strictly increasing on $\left[\gamma_0, m + 1\right]$ then $v$ will have at least one point of local maximum $t$ between $\gamma_0$ and $m + 1$ (as $v'\left(\gamma_0\right) > 0$). Then again as in Case (\ref{itm:extrm1}), $t \in \left(1, m + 1\right)$ will be the only point of local extremum of $v$ and $t > \gamma_0$.
\end{case*}
\end{proof} \par
We now have the following necessary and sufficient condition for the continuation of the solution to (\ref{eq:ODEIVPConeSing}) defined a priori on some interval:
\begin{theorem}[Criterion for Continuation of Solutions]\label{thm:continue}
For any $\alpha \in \left(-\infty, \bo\right)$ if $v$ is the smooth solution to (\ref{eq:ODEIVPConeSing}) defined on an interval $\left[1, \tilde{r}\right) \subseteq \left[1, m + 1\right]$ then:
\begin{enumerate}
\item $v$ can be continued beyond $\tilde{r}$ if and only if $\inf\limits_{\gamma \in \left[1, \tilde{r}\right)} v \left(\gamma\right) = \epsilon_1 > 0$ if and only if $\lim\limits_{\gamma \to \tilde{r}} v \left(\gamma\right) = \epsilon_2 > 0$. \label{itm:cont1}
\item $\left[1, \tilde{r}\right)$ is the maximal interval of existence of $v$ if and only if $\inf\limits_{\gamma \in \left[1, \tilde{r}\right)} v \left(\gamma\right) = 0$ if and only if $\lim\limits_{\gamma \to \tilde{r}} v \left(\gamma\right) = 0$. \label{itm:cont2}
\end{enumerate}
\end{theorem}
\begin{proof}
From Lemma \ref{lem:positivity} and Corollary \ref{cor:InfinitDiff}, $v > 0$ and $v$ is Lipschitz on $\left[1, \tilde{r}\right)$ so $\inf\limits_{\gamma \in \left[1, \tilde{r}\right)} v \left(\gamma\right)$ and $\lim\limits_{\gamma \to \tilde{r}} v \left(\gamma\right)$ both exist and are non-negative. So again by using Lemma \ref{lem:positivity} and Corollary \ref{cor:InfinitDiff} it can be easily checked that either both $\inf\limits_{\gamma \in \left[1, \tilde{r}\right)} v \left(\gamma\right)$ and $\lim\limits_{\gamma \to \tilde{r}} v \left(\gamma\right)$ are simultaneously strictly positive or both are simultaneously zero. From Lemma \ref{lem:continue} we already have that if both are simultaneously positive then $v$ can be continued beyond $\tilde{r}$. \\
For proving the converse let $\inf\limits_{\gamma \in \left[1, \tilde{r}\right)} v \left(\gamma\right) = \lim\limits_{\gamma \to \tilde{r}} v \left(\gamma\right) = 0$. Let if possible $v$ be extendable as the smooth solution of (\ref{eq:ODEIVPConeSing}) to an interval $\left[1, r'\right)$ with $\left[1, \tilde{r}\right] \subseteq \left[1, r'\right) \subseteq \left[1, m + 1\right]$. Then by Lemma \ref{lem:positivity}, $v > 0$ on $\left[1, r'\right)$ and as $\tilde{r} \in \left[1, r'\right)$ and $v$ is continuous on $\left[1, r'\right)$ so $0 = \lim\limits_{\gamma \to \tilde{r}} v \left(\gamma\right) = v \left(\tilde{r}\right) > 0$, a contradiction. So $\left[1, \tilde{r}\right)$ is the maximal interval of existence of $v$.
\end{proof} \par
So if the smooth solution $v$ to (\ref{eq:ODEIVPConeSing}) cannot be defined on $\left[1, m + 1\right]$ then there exists a unique $\gamma_\star \in \left(1, m + 1\right]$ such that $\left[1, \gamma_\star\right)$ is the maximal interval of existence of $v$. \par
\begin{remark}
Lemma \ref{lem:positivity} and Theorem \ref{thm:continue} are together saying that the solution of (\ref{eq:ODEIVPConeSing}) continues to exist as long as it is strictly positive, but the moment it attains zero, it `breaks down' i.e. it cannot be continued further.
\end{remark} \par
So finally for every $\alpha \in \left(-\infty, \bo\right)$ considering the ODE initial value problem (\ref{eq:ODEIVPConeSing}) depending on $\alpha$ we have exactly one of the following two scenarios (as a consequence of Lemma \ref{lem:firstexist} and Theorem \ref{thm:continue}):
\begin{enumerate}
\item There exists a unique smooth solution $v_\alpha = v \left(\cdot; \alpha\right)$ on $\left[1, m + 1\right]$. \label{itm:case1}
\item There exists a unique smooth solution $v_\alpha = v \left(\cdot; \alpha\right)$ with maximal interval of existence $\left[1, \gamma_{\star, \alpha}\right)$ for a unique $\gamma_{\star, \alpha} = \gamma_\star \left(\alpha\right) \in \left(1, m + 1\right]$. \label{itm:case2}
\end{enumerate} \par
\begin{motivation}
In order to prove Theorem \ref{thm:mainconesing}, we will first show that the set of all $\alpha \in \left(-\infty, \bo\right)$, for which the condition (\ref{itm:case1}) above holds true, is precisely the interval $\left(M, \bo\right)$ for a unique $M < 0$, and then we will check the limits of $v_\alpha \left(m + 1\right)$ as $\alpha \to M^+$ (which will turn out to be $0$) and as $\alpha \to \bo^-$ (which will turn out to be some $b > 2 \left(m + 1\right)^2$) respectively to conclude that the range of the function $\left(M, \bo\right) \to \mathbb{R}$, $\alpha \mapsto v_\alpha \left(m + 1\right)$ is precisely the interval $\left(0, b\right)$, whence the fact that $2 \left(m + 1\right)^2 \in \left(0, b\right)$ should yield the existence of an $\alpha \in \left(M, \bo\right)$ for which $v_\alpha \left(m + 1\right) = 2 \left(m + 1\right)^2$. For doing this we will prove some preparatory results in the remainder of Subsection \ref{subsec:Proof1} and in Subsection \ref{subsec:Proof2}.
\end{motivation} \par
Let $\left( \mathcal{C} \left[1, m + 1\right], \left\lVert \cdot \right\rVert_\infty \right)$ be the Banach space of all continuous functions on $\left[1, m + 1\right]$ equipped with the supremum norm. For each $\alpha \in \left(-\infty, \bo\right)$ define $u \left(\cdot; \alpha\right) : \left[1, m + 1\right] \to \mathbb{R}$ as follows:
\begin{enumerate}
\item If the smooth solution $v_\alpha$ to (\ref{eq:ODEIVPConeSing}) exists on $\left[1, m + 1\right]$ then $u \left(\gamma; \alpha\right) = v_\alpha \left(\gamma\right)$ for all $\gamma \in \left[1, m + 1\right]$. \label{itm:defu1}
\item If the smooth solution $v_\alpha$ to (\ref{eq:ODEIVPConeSing}) has maximal interval of existence $\left[1, \gamma_{\star, \alpha}\right)$ then $u \left(\gamma; \alpha\right) = v_\alpha \left(\gamma\right)$ for all $\gamma \in \left[1, \gamma_{\star, \alpha}\right)$ and $u \left(\gamma; \alpha\right) = 0$ for all $\gamma \in \left[\gamma_{\star, \alpha}, m + 1\right]$. \label{itm:defu2}
\end{enumerate}
By Lemma \ref{lem:positivity}, $u \left(\cdot; \alpha\right)$ is continuous on $\left[1, m + 1\right]$ in the Case (\ref{itm:defu2}) above as well, and hence $u \left(\cdot; \alpha\right) \in \mathcal{C} \left[1, m + 1\right]$ in both the Cases (\ref{itm:defu1}) and (\ref{itm:defu2}) above. Thus we get a function $\Phi : \left(-\infty, \bo\right) \to \mathcal{C} \left[1, m + 1\right]$ defined as $\Phi \left(\alpha\right) = u \left(\cdot; \alpha\right)$ for all $\alpha \in \left(-\infty, \bo\right)$. It can be readily checked from (\ref{eq:ODEIVPConeSing}) that $\Phi$ is well-defined and injective. \par
\begin{motivation}
We want to prove that $\Phi$ is continuous and considering the pointwise partial order $\leq$ on $\mathcal{C} \left[1, m + 1\right]$, $\Phi$ is monotone increasing.
\end{motivation} \par
For a given $\alpha \in \left(-\infty, \bo\right)$ let $\gamma_{0, \alpha} = \gamma_0 \left(\alpha\right)$ be the unique root of the polynomial $p_\alpha \left(\gamma\right) \gamma = p \left(\gamma; \alpha\right) \gamma = B \left(\alpha\right) \frac{\gamma^3}{2} + C \left(\alpha\right) \gamma$ in $\left[1, m + 1\right]$ and similarly let $u_\alpha = u \left(\cdot; \alpha\right)$ be the function defined on $\left[1, m + 1\right]$ in the Cases (\ref{itm:defu1}) and (\ref{itm:defu2}) above and if it exists then let $t_\alpha = t \left(\alpha\right)$ be the unique point of local maximum of the solution $v_\alpha$ in its maximal interval of existence as given by Corollary \ref{cor:incredecreextrm}. \par
\begin{theorem}\label{thm:UnifConv}
Let $\left(\alpha_n\right) \to \alpha_0$ in $\left(-\infty, \bo\right)$ and $u_n = u_{\alpha_n}$ and $u_0 = u_{\alpha_0}$ on $\left[1, m + 1\right]$. Then there exists a subsequence $\left(u_{n_k}\right)$ of $\left(u_n\right)$ such that $\left(u_{n_k}\right) \to u_0$ uniformly on $\left[1, m + 1\right]$.
\end{theorem}
\begin{proof}
Let $v_n = v_{\alpha_n}$ and $v_0 = v_{\alpha_0}$, and $p_n \left(\gamma\right) \gamma = p_{\alpha_n} \left(\gamma\right) \gamma$ and $p_0 \left(\gamma\right) \gamma = p_{\alpha_0} \left(\gamma\right) \gamma$. Then from the expressions (\ref{eq:BCalphabeta0}), $\left(B \left(\alpha_n\right)\right) \to B \left(\alpha_0\right)$ and $\left(C \left(\alpha_n\right)\right) \to C \left(\alpha_0\right)$, and hence $\left( p_n \left(\gamma\right) \gamma \right) \to p_0 \left(\gamma\right) \gamma$ uniformly on $\left[1, m + 1\right]$. Let $\gamma_{0, n} = \gamma_{0, \alpha_n} = \sqrt{-\frac{2 C \left(\alpha_n\right)}{B \left(\alpha_n\right)}}$ and $\gamma_{0, 0} = \gamma_{0, \alpha_0} = \sqrt{-\frac{2 C \left(\alpha_0\right)}{B \left(\alpha_0\right)}}$ be the roots of $p_n \left(\gamma\right) \gamma$ and $p_0 \left(\gamma\right) \gamma$ respectively in $\left[1, m + 1\right]$ given by Lemma \ref{lem:polynomconesing}, then clearly $\left( \gamma_{0, n} \right) \to \gamma_{0, 0}$. Also note that $u_n \left(1\right) = v_n \left(1\right) = 2 = v_0 \left(1\right) = u_0 \left(1\right)$ and $u_n' \left(1\right) = v_n' \left(1\right) = 2 \left(\beta_0 + 2\right) = v_0' \left(1\right) = u_0' \left(1\right)$. After this the proof of Theorem \ref{thm:UnifConv} will be divided into the following three Cases:
\begin{case*}($1$) \hspace{2pt}
The solution $v_n$ exists on $\left[1, m + 1\right]$ for all $n \in \mathbb{N}$ and $\inf\limits_{\gamma \in \left[1, m + 1\right], n \in \mathbb{N}} v_n \left(\gamma\right) = \epsilon > 0$.
\end{case*}
{\noindent Here $u_n = v_n$ for all $n \in \mathbb{N}$. Since $\left( p_n \left(\gamma\right) \gamma \right)$ is uniformly norm bounded, substituting $\left\lvert p_n \left(\gamma\right) \gamma \right\rvert \leq l$ (for some $l > 0$) and $\sqrt{v_n} < v_n + 1$ for all $n \in \mathbb{N}$ in the expression for $v_n' = \left(v_n + 1\right)'$ in (\ref{eq:ODEIVPConeSing}) and using Gr\"onwall's inequality will yield a $K > 0$ such that $\epsilon \leq v_n \left(\gamma\right) \leq K$ for all $\gamma \in \left[1, m + 1\right]$ and for all $n \in \mathbb{N}$, thereby proving that $\left(v_n\right)$ is uniformly norm bounded on $\left[1, m + 1\right]$. Now substituting $\sqrt{v_n} \leq \sqrt{K}$ and $\left\lvert p_n \left(\gamma\right) \gamma \right\rvert \leq l$ in the expression for $v_n'$ in (\ref{eq:ODEIVPConeSing}) will give us an $R > 0$ such that $\left\lvert v_n' \left(\gamma\right) \right\rvert \leq R$ for all $\gamma \in \left[1, m + 1\right]$ and for all $n \in \mathbb{N}$, proving that $\left(v_n'\right)$ is also uniformly norm bounded and thus implying that $\left(v_n\right)$ is uniformly equicontinuous on $\left[1, m + 1\right]$. So by Arzel\`a-Ascoli theorem there exists a subsequence $\left(v_{n_k}\right)$ of $\left(v_n\right)$ converging uniformly on $\left[1, m + 1\right]$ to some $w \in \mathcal{C} \left[1, m + 1\right]$. As $\left( p_n \left(\gamma\right) \gamma \right)$ and $\left(\sqrt{v_{n_k}}\right)$ are uniformly convergent so $\left(v_{n_k}'\right)$ is uniformly convergent on $\left[1, m + 1\right]$ (again from (\ref{eq:ODEIVPConeSing})). Then by standard uniform convergence theory $w$ is differentiable and satisfies the ODE initial value problem (\ref{eq:ODEIVPConeSing}) for $\alpha = \alpha_0$ on $\left[1, m + 1\right)$. As $\inf\limits_{\gamma \in \left[1, m + 1\right], n \in \mathbb{N}} v_n \left(\gamma\right) = \epsilon > 0$ so $\inf\limits_{\gamma \in \left[1, m + 1\right)} w \left(\gamma\right) = \tilde{\epsilon} > 0$ and so by Theorem \ref{thm:continue}, $w$ is differentiable and satisfies (\ref{eq:ODEIVPConeSing}) on $\left[1, m + 1\right]$, and hence the solution $v_0$ exists on $\left[1, m + 1\right]$ and by uniqueness, $w = v_0$ and by definition, $u_0 = v_0$ as well. Thus $\left(u_{n_k}\right) \to u_0$ uniformly on $\left[1, m + 1\right]$.}
\begin{case*}($2$) \hspace{2pt}
The solution $v_n$ exists on $\left[1, m + 1\right]$ for all $n \in \mathbb{N}$ and $\inf\limits_{\gamma \in \left[1, m + 1\right], n \in \mathbb{N}} v_n \left(\gamma\right) = 0$.
\end{case*}
{\noindent Here also $u_n = v_n$ for all $n \in \mathbb{N}$. The Gr\"onwall's inequality argument as in Case ($1$) above using the uniform norm boundedness of $\left( p_n \left(\gamma\right) \gamma \right)$ will prove here that $0 \leq v_n \left(\gamma\right) \leq K$ i.e. $\left(v_n\right)$ is uniformly norm bounded on $\left[1, m + 1\right]$. Once again by substituting this in the ODE in (\ref{eq:ODEIVPConeSing}) with $\alpha = \alpha_n$, we can obtain $\left\lvert v_n' \left(\gamma\right) \right\rvert \leq R$ i.e. $\left(v_n'\right)$ will also be uniformly norm bounded on $\left[1, m + 1\right]$, and so there will exist a subsequence $\left(v_{n_k}\right) \to w \in \mathcal{C} \left[1, m + 1\right]$ uniformly on $\left[1, m + 1\right]$. By the same arguments as in Case ($1$), $\left(v_{n_k}'\right)$ is uniformly convergent on $\left[1, m + 1\right]$ and $w$ is differentiable and satisfies the ODE initial value problem (\ref{eq:ODEIVPConeSing}) for $\alpha = \alpha_0$ on $\left[1, m + 1\right)$. \\
If $v_n$ is strictly increasing on $\left[1, m + 1\right]$ for each $n \in \mathbb{N}$ then $v_n \left(\gamma\right) \geq v_n \left(1\right) = 2$ for all $\gamma \in \left[1, m + 1\right]$ which contradicts the hypothesis $\inf\limits_{\gamma \in \left[1, m + 1\right], n \in \mathbb{N}} v_n \left(\gamma\right) = 0$ given here in Case ($2$). Then by using Corollary \ref{cor:incredecreextrm} Case (\ref{itm:extrm2}), $v_n$ must have a unique point of local maximum $t_n = t_{\alpha_n} \in \left(1, m + 1\right)$ (and no local minima), and we can assume this to be true for each $n \in \mathbb{N}$ since our aim here was to find some uniformly convergent subsequence of $\left(v_n\right)$. So $\inf\limits_{\gamma \in \left[1, m + 1\right]} v_n \left(\gamma\right) = \min\limits_{\gamma \in \left[1, m + 1\right]} v_n \left(\gamma\right) = \min \left\lbrace v_n \left(1\right), v_n \left(m + 1\right) \right\rbrace = \min \left\lbrace 2, v_n \left(m + 1\right) \right\rbrace$. As $\inf\limits_{\gamma \in \left[1, m + 1\right], n \in \mathbb{N}} v_n \left(\gamma\right) = 0$, we must have $\inf\limits_{n \in \mathbb{N}} v_n \left(m + 1\right) = 0$. As $\left(v_{n_k}\right) \to w$ uniformly on $\left[1, m + 1\right]$ (which we have gotten above for Case ($2$)) so $w \left(m + 1\right) = 0$ and hence by Theorem \ref{thm:continue}, $\left[1, m + 1\right)$ is the maximal interval of existence of $w$ as the smooth solution to (\ref{eq:ODEIVPConeSing}) with $\alpha = \alpha_0$, and so $w = v_0 = u_0$ on $\left[1, m + 1\right)$ and $w = u_0$ on $\left[1, m + 1\right]$ by continuity. With this, $\left(u_{n_k}\right) \to u_0$ uniformly on $\left[1, m + 1\right]$.}
\begin{case*}($3$) \hspace{2pt}
The solution $v_n$ has maximal interval of existence $\left[1, \gamma_{\star, n}\right)$ with $\gamma_{\star, n} = \gamma_{\star, \alpha_n} \in \left(1, m + 1\right]$ for all $n \in \mathbb{N}$. Without loss of generality $\left(\gamma_{\star, n}\right)$ is a monotone sequence converging to some $\sigma \in \left[1, m + 1\right]$.
\end{case*}
{\noindent Here $u_n = v_n$ on $\left[1, \gamma_{\star, n}\right)$ and $u_n = 0$ on $\left[\gamma_{\star, n}, m + 1\right]$ for all $n \in \mathbb{N}$. As was noted in the beginning, the polynomials $\left( p_n \left(\gamma\right) \gamma \right) \to p_0 \left(\gamma\right) \gamma$ uniformly on $\left[1, m + 1\right]$ and their respective roots $\left( \gamma_{0, n} \right) \to \gamma_{0, 0}$. By Lemmas \ref{lem:polynomconesing} and \ref{lem:firstexist}, $1 < \gamma_{0,n} < \gamma_{\star,n} \leq m + 1$ for all $n \in \mathbb{N}$ and taking limits as $n \to \infty$ we get $1 < \gamma_{0,0} \leq \sigma \leq m + 1$ ($\gamma_{0,0} > 1$ by Lemma \ref{lem:polynomconesing}) and specifically $\sigma \in \left(1, m + 1\right]$. By Corollary \ref{cor:incredecreextrm} Case (\ref{itm:extrm1}), $v_n$ has a unique point of local maximum $t_n = t_{\alpha_n} \in \left(1, \gamma_{\star, n}\right)$.
\begin{claim*}
There exists a $\tilde{K} > 0$ such that for each $n \in \mathbb{N}$ we have $v_n \left(t_n\right) \leq \tilde{K}$.
\end{claim*}
{\noindent For any $n \in \mathbb{N}$, as $t_n$ is a critical point of $v_n$ so $v_n'\left(t_n\right) = 0$ which will imply $v_n\left(t_n\right) = \frac{p_n\left(t_n\right)^2 t_n^2}{8}$ (from the equation in (\ref{eq:ODEIVPConeSing}) with $\alpha = \alpha_n$), and as $\left(p_n\left(\gamma\right)\gamma\right)$ is uniformly norm bounded on $\left[1, m + 1\right]$ in all the three Cases, so there exists a $\tilde{K} > 0$ such that $v_n\left(t_n\right) = \frac{p_n\left(t_n\right)^2 t_n^2}{8} \leq \tilde{K}$ for all $n \in \mathbb{N}$.} \\
Looking at the definition of $u_n$ in Case ($3$) and using Lemma \ref{lem:positivity} and Corollary \ref{cor:incredecreextrm}, we will observe for each $n \in \mathbb{N}$ that $\max\limits_{\gamma \in \left[1, m + 1\right]} u_n \left(\gamma\right) = \sup\limits_{\gamma \in \left[1, \gamma_{\star, n}\right)} v_n \left(\gamma\right) = v_n \left(t_n\right) \leq \tilde{K}$, as $t_n$ is the unique global maximum of $v_n$ and hence also of $u_n$. Thus we see $0 \leq u_n \left(\gamma\right) \leq \tilde{K}$ for all $\gamma \in \left[1, m + 1\right]$ and for all $n \in \mathbb{N}$ i.e. $\left(u_n\right)$ is uniformly norm bounded on $\left[1, m + 1\right]$. By substituting the uniform norm bounds on $\left(\sqrt{v_n}\right)$ and $\left( p_n \left(\gamma\right) \gamma \right)$ in the expression for $v_n'$ in (\ref{eq:ODEIVPConeSing}) on $\left[1, \gamma_{\star, n}\right)$, we get an $\tilde{R} > 0$ such that for each $n \in \mathbb{N}$, $\left\lvert v_n' \left(\gamma\right) \right\rvert \leq \tilde{R}$ for all $\gamma \in \left[1, \gamma_{\star, n}\right)$, and so considering Lemma \ref{lem:positivity}, we also get $\left\lvert \lim\limits_{\gamma \to \gamma_{\star, n}} v_n' \left(\gamma\right) \right\rvert \leq \tilde{R}$. So by its definition here, $\left(u_n\right)$ is uniformly Lipschitz and hence uniformly equicontinuous on $\left[1, m + 1\right]$. So we can extract a subsequence $\left(u_{n_k}\right) \to w \in \mathcal{C} \left[1, m + 1\right]$ uniformly on $\left[1, m + 1\right]$. As $u_{n_k} \geq 0$ so $w \geq 0$.
\begin{claim*}
$w \left(\sigma\right) = 0$.
\end{claim*}
{\noindent Note that by Lemma \ref{lem:positivity}, $u_n \left(\gamma_{\star, n}\right) = 0$ for all $n \in \mathbb{N}$ and by the above arguments in Case ($3$), $\tilde{R}$ is the uniform Lipschitz constant for $\left(u_n\right)$ on $\left[1, m + 1\right]$. Then by considering the following estimates for any appropriately chosen $k \in \mathbb{N}$:
\begin{align}\label{eq:Case3}
\left\lvert w\left(\sigma\right) \right\rvert &\leq \left\lvert u_{n_k}\left(\sigma\right) - w\left(\sigma\right) \right\rvert + \left\lvert u_{n_k}\left(\gamma_{\star, n_k}\right) - u_{n_k}\left(\sigma\right) \right\rvert + \left\lvert u_{n_k}\left(\gamma_{\star, n_k}\right) \right\rvert \\
&\leq \left\lvert u_{n_k}\left(\sigma\right) - w\left(\sigma\right) \right\rvert + \tilde{R} \left\lvert \gamma_{\star, n_k} - \sigma \right\rvert \nonumber
\end{align}
it can be easily seen that $w \left(\sigma\right) = 0$.} \\
After this the proof of $w = u_0$ in Case ($3$) will depend upon whether $\left(\gamma_{\star, n}\right)$ is monotonically increasing or decreasing.
\begin{case*}($3.1$) \hspace{2pt}
$\left(\gamma_{\star, n}\right)$ decreases to $\sigma$.
\end{case*}
{\noindent So $\left[1, \sigma\right] = \bigcap\limits_{k \in \mathbb{N}} \left[1, \gamma_{\star,n_k}\right)$ and so each $v_{n_k}$ will satisfy the ODE initial value problem (\ref{eq:ODEIVPConeSing}) with $\alpha = \alpha_{n_k}$ on $\left[1, \sigma\right]$. But $v_{n_k} = u_{n_k}$ on $\left[1, \gamma_{\star, n_k}\right)$ and $\left(u_{n_k}\right) \to w$ uniformly on $\left[1, m + 1\right]$. So $\left( \sqrt{v_{n_k}} \right) \to \sqrt{w}$ uniformly on $\left[1, \sigma\right]$ and this will imply (again from (\ref{eq:ODEIVPConeSing}) with $\alpha = \alpha_{n_k}$) that $\left(v_{n_k}'\right)$ is uniformly convergent on $\left[1, \sigma\right]$. Hence $w$ is differentiable and satisfies (\ref{eq:ODEIVPConeSing}) for $\alpha = \alpha_0$ on $\left[1, \sigma\right)$. As $w \left(\sigma\right) = 0$, from Lemma \ref{lem:positivity} and Theorem \ref{thm:continue}, $\left[1, \sigma\right)$ is the maximal interval of existence of the solution $v_0 \left(= v_{\alpha_0}\right)$ of (\ref{eq:ODEIVPConeSing}), and $w = v_0$ on $\left[1, \sigma\right)$ and $\sigma = \gamma_{\star,0} = \gamma_{\star,\alpha_0}$. Now note that $u_{n_k} \equiv 0$ on $\left[\gamma_{\star, n_k}, m + 1\right]$ for each $k \in \mathbb{N}$ and so (the uniform limit) $w \equiv 0$ on $\left(\sigma, m + 1\right] = \bigcup\limits_{k \in \mathbb{N}} \left[\gamma_{\star,n_k}, m + 1\right]$, as $\left(\gamma_{\star, n}\right)$ is decreasing to $\sigma$. So by its definition, $w = u_0$ on $\left[1, m + 1\right]$ and we already have $\sigma = \gamma_{\star,0}$.}
\begin{case*}($3.2$) \hspace{2pt}
$\left(\gamma_{\star, n}\right)$ increases to $\sigma$.
\end{case*}
{\noindent So $\left[1, \sigma\right) = \bigcup\limits_{k \in \mathbb{N}} \left[1, \gamma_{\star,n_k}\right)$ and $\left[1, \gamma_{\star,n_k}\right) = \bigcap\limits_{j \geq k} \left[1, \gamma_{\star,n_j}\right)$ for each $k \in \mathbb{N}$. So for a fixed $k \in \mathbb{N}$, $v_{n_j}$ satisfies the ODE initial value problem (\ref{eq:ODEIVPConeSing}) with $\alpha = \alpha_{n_j}$ on $\left[1, \gamma_{\star,n_k}\right)$ for all $j \geq k$. By using the same set of arguments as in Case ($3.1$) for the subsequence $\left(v_{n_j}\right)_{j \geq k}$ converging uniformly to $w$ on $\left[1, \gamma_{\star,n_k}\right)$, we will see that $\left(v_{n_j}'\right)_{j \geq k}$ is uniformly convergent on $\left[1, \gamma_{\star,n_k}\right)$. So $w$ is differentiable and satisfies (\ref{eq:ODEIVPConeSing}) for $\alpha = \alpha_0$ on $\left[1, \gamma_{\star,n_k}\right)$, and as this holds true for each $k \in \mathbb{N}$ so $w$ satisfies (\ref{eq:ODEIVPConeSing}) for $\alpha = \alpha_0$ on $\left[1, \sigma\right) = \bigcup\limits_{k \in \mathbb{N}} \left[1, \gamma_{\star,n_k}\right)$. After this, the same arguments as in Case ($3.1$) will give us that $\sigma = \gamma_{\star,0} \left(= \gamma_{\star,\alpha_0}\right)$ and $w = v_0$ on $\left[1, \sigma\right)$ and $w \equiv 0$ on $\left[\sigma, m + 1\right] = \bigcap\limits_{k \in \mathbb{N}} \left[\gamma_{\star,n_k}, m + 1\right]$ (as $\left(\gamma_{\star, n}\right)$ is increasing to $\sigma$), thereby giving $w = u_0$ on $\left[1, m + 1\right]$.} \\
In both the Cases ($3.1$) and ($3.2$), $\left(u_{n_k}\right) \to u_0$ uniformly on $\left[1, m + 1\right]$ and $\left(\gamma_{\star,n_k}\right) \to \gamma_{\star,0}$.} \\
Since our aim was to find only a subsequence of $\left(u_n\right)$ which is uniformly convergent, the above Cases ($1$), ($2$) and ($3$) suffice in proving Theorem \ref{thm:UnifConv}.
\end{proof} \par
\begin{remark}
The Cases ($1$), ($2$) and ($3$) in Theorem \ref{thm:UnifConv} have given us a hint that the interval $\left(-\infty, \bo\right)$ (which is the range of the values of the parameter $\alpha$) may be expressed as the disjoint set union of the set of all $\alpha \in \left(-\infty, \bo\right)$ for which the solution $v_\alpha$ to (\ref{eq:ODEIVPConeSing}) exists on the whole of $\left[1, m + 1\right]$ with the set of all $\alpha \in \left(-\infty, \bo\right)$ for which $v_\alpha$ breaks down at $\gamma_{\star, \alpha} \in \left(1, m + 1\right]$, and that the first set is an open interval and the second one is a closed interval.
\end{remark}
\subsection{Second Part of the Proof}\label{subsec:Proof2}
\begin{motivation}
In Subsection \ref{subsec:Proof2} we will try to analyze the variation of the ODE initial value problem (\ref{eq:ODEIVPConeSing}) with respect to the parameter $\alpha \in \left(-\infty, \bo\right)$ by differentiating the ODE and its solution with respect to $\alpha$. This is going to be along the exact same lines as the ODE appearing in the analogous smooth higher \ext{} \kr{} problem seen in Pingali \cite{Pingali:2018:heK}; Section 2 and \cite{Sompurkar:2023:heKsmooth}; Subsection 3.2 with only some subtle details in the calculations and estimates with respect to the parameters changing over here.
\end{motivation} \par
We first do some calculations with the polynomial $p_\alpha \left(\gamma\right) \gamma = B \left(\alpha\right) \frac{\gamma^3}{2} + C \left(\alpha\right) \gamma$ on $\left[1, m + 1\right]$ (in the following three results which are in continuation of Lemma \ref{lem:polynomconesing}) which will be needed further in Subsection \ref{subsec:Proof2}. \par
\begin{lemma}\label{lem:polynom1conesing}
For each $\alpha \in \left(-\infty, \bo\right)$ define $P_\alpha \left(\gamma\right) = P \left(\gamma; \alpha\right) = \int\limits_1^\gamma p_\alpha \left(y\right) y d y$ for all $\gamma \in \left[1, m + 1\right]$. Then $P_\alpha \left(m + 1\right) = \frac{m \left(m + 2\right)}{2} \alpha$ and further $P_\alpha \left(\gamma\right) = B \left(\alpha\right) \frac{\gamma^4 - 1}{8} + C \left(\alpha\right) \frac{\gamma^2 - 1}{2} \geq \min \left\lbrace 0, \frac{m \left(m + 2\right)}{2} \alpha \right\rbrace$ for all $\gamma \in \left[1, m + 1\right]$.
\end{lemma}
\begin{proof}
\begin{align}\label{eq:intpolynomvalue}
P_\alpha \left(m + 1\right) \hspace{1pt} &= \int\limits_1^{m + 1} p_\alpha \left(\gamma\right) \gamma d \gamma \hspace{1pt} = \int\limits_1^{m + 1} \left( B \left(\alpha\right) \frac{\gamma^3}{2} + C \left(\alpha\right) \gamma \right) d \gamma \nonumber \\
&= B \left(\alpha\right) \frac{\left(m + 1\right)^4 - 1}{8} + C \left(\alpha\right) \frac{\left(m + 1\right)^2 - 1}{2} \nonumber \\
&= \left( \frac{B \left(\alpha\right)}{4} \left(\left(m + 1\right)^2 + 1\right) + C \left(\alpha\right) \right) \frac{m \left(m + 2\right)}{2} \nonumber \\
&= \frac{m \left(m + 2\right)}{2} \alpha && \left(\text{by using the expressions (\ref{eq:BCalphabeta0})}\right)
\end{align}
As $\frac{d}{d \gamma} \left( P_\alpha \left(\gamma\right) \right) = p_\alpha \left(\gamma\right) \gamma$, by Lemma \ref{lem:polynomconesing}, $\gamma_{0, \alpha} = \gamma_0 \left(\alpha\right) = \sqrt{-\frac{2 C \left(\alpha\right)}{B \left(\alpha\right)}}$ is the only critical point of the polynomial $P_\alpha \left(\gamma\right) = B \left(\alpha\right) \frac{\gamma^4 - 1}{8} + C \left(\alpha\right) \frac{\gamma^2 - 1}{2}$ in the interval $\cllml$, and further as $\frac{d^2}{d \gamma^2} \left( P_\alpha \left(\gamma\right) \right) \Bigr\rvert_{\gamma = \gamma_{0, \alpha}} = \frac{d}{d \gamma} \left( p_\alpha \left(\gamma\right) \gamma \right) \Bigr\rvert_{\gamma = \gamma_{0, \alpha}} = \frac{3 B \left(\alpha\right)}{2} \gamma_{0, \alpha}^2 + C \left(\alpha\right) = - 2 C \left(\alpha\right) < 0$ (from (\ref{eq:BCalphabeta0})), $\gamma_{0, \alpha}$ is a point of local maximum of $P_\alpha \left(\gamma\right)$. So we have $\inf\limits_{\gamma \in \left[1, m + 1\right]} P_\alpha \left(\gamma\right) = \min\limits_{\gamma \in \left[1, m + 1\right]} P_\alpha \left(\gamma\right) = \min \left\lbrace P_\alpha \left(1\right), P_\alpha \left(m + 1\right) \right\rbrace = \min \left\lbrace 0, \frac{m \left(m + 2\right)}{2} \alpha \right\rbrace$.
\end{proof} \par
\begin{lemma}\label{lem:derpolynomconesing}
The polynomial $q \left(\gamma\right) = \frac{d}{d\alpha}\left(p_\alpha\left(\gamma\right)\gamma\right)$, $\gamma \in \left[1, m + 1\right]$ is independent of $\alpha$, and further $q \left(\gamma\right) > 0$ for all $\gamma \in \left(1, m + 1\right]$ and $q \left(1\right) = 0$.
\end{lemma}
\begin{proof}
\begin{equation}\label{eq:derBC}
\frac{d}{d\alpha}\left(B\left(\alpha\right)\right) = \frac{4}{m \left(m + 2\right)} \hspace{2pt}, \hspace{5pt} \frac{d}{d\alpha}\left(C\left(\alpha\right)\right) = - \frac{2}{m \left(m + 2\right)}
\end{equation}
where we again use the expressions (\ref{eq:BCalphabeta0}).
\begin{equation}\label{eq:qfactor}
\frac{d}{d\alpha}\left(p_\alpha\left(\gamma\right)\gamma\right) \hspace{2pt} = \hspace{2pt} \frac{2}{m \left(m + 2\right)} \gamma^3 - \frac{2}{m \left(m + 2\right)} \gamma \hspace{2pt} = \hspace{2pt} \frac{2}{m \left(m + 2\right)} \gamma \left(\gamma + 1\right) \left(\gamma - 1\right)
\end{equation}
So $q \left(\gamma\right) = \frac{d}{d\alpha}\left(p_\alpha\left(\gamma\right)\gamma\right)$ is independent of $\alpha$, has its root at $\gamma = 1$ in $\cllml$, and does not change its sign in $\left(1, m + 1\right]$. Evaluating $q \left(\gamma\right)$ at $\gamma = m + 1 \in \left(1, m + 1\right]$ to check its sign, we see $q \left(m + 1\right) = 2 \left(m + 1\right) > 0$. So for any $\gamma \in \left(1, m + 1\right]$, we must have $q \left(\gamma\right) > 0$, and $q \left(1\right) = 0$.
\end{proof} \par
\begin{corollary}\label{cor:derpolynomconesing}
Define $Q \left(\gamma\right) = \int\limits_1^\gamma q \left(y\right) d y$ for all $\gamma \in \left[1, m + 1\right]$. Then $Q \left(\gamma\right) > 0$ for all $\gamma \in \left(1, m + 1\right]$ and $Q \left(\gamma\right)$ is strictly increasing on $\left[1, m + 1\right]$. Further $\frac{d}{d \alpha}\left(P_\alpha \left(\gamma\right)\right) = \frac{\left(\gamma^2 - 1\right)^2}{2 m \left(m + 2\right)}$ is also independent of $\alpha$ and $Q \left(\gamma\right) = \frac{d}{d \alpha}\left(P_\alpha \left(\gamma\right)\right)$ on $\cllml$.
\end{corollary}
\begin{proof}
\begin{align}\label{eq:intderpolynom}
\int\limits_1^\gamma \frac{d}{d\alpha}\left(p_\alpha\left(y\right)y\right) dy &= \frac{d}{d \alpha} \left( \int\limits_1^\gamma p_\alpha \left(y\right) y d y \right) \nonumber \\
&= \frac{\gamma^4 - 1}{2 m \left(m + 2\right)} - \frac{\gamma^2 - 1}{m \left(m + 2\right)} && \left(\text{from Lemma \ref{lem:polynom1conesing} and the expressions (\ref{eq:derBC})}\right) \nonumber \\
&= \frac{\left(\gamma^2 - 1\right)^2}{2 m \left(m + 2\right)}
\end{align}
So $Q \left(\gamma\right) = \frac{d}{d \alpha}\left(P_\alpha \left(\gamma\right)\right)$ is also independent of $\alpha$. Note that in particular $Q\left(m + 1\right) = \frac{m \left(m + 2\right)}{2}$.
\end{proof} \par
\begin{motivation}
The following calculations and estimates are going to give us that for each $\gamma > 1$, $\frac{d}{d \alpha} \left( v_\alpha \left(\gamma\right) \right) > 0$ on appropriate intervals of $\gamma$ and $\alpha$ i.e. $v \left(\gamma; \alpha\right)$ is strictly increasing in $\alpha$. From the analysis developed in Subsection \ref{subsec:Proof1} the smooth solution $v_\alpha \left(\cdot\right) = v \left(\cdot; \alpha\right)$ of the ODE initial value problem (\ref{eq:ODEIVPConeSing}) depending on $\alpha$ will always exist on some non-degenerate subinterval of $\cllml$ containing $1$. Also note that we are allowed to differentiate $v_\alpha$ with respect to $\alpha$ because the ODE and the initial condition in (\ref{eq:ODEIVPConeSing}) have smooth dependence on the parameter $\alpha$.
\end{motivation} \par
\begin{theorem}\label{thm:StrctIncrConeSing}
Let $\mathsf{V}$ be a non-degenerate subinterval of $\left(-\infty, \bo\right)$ and $\mathsf{J}$ be a non-degenerate subinterval of $\left[1, m + 1\right]$ containing $1$, such that the smooth solution $v_\alpha$ to the ODE initial value problem (\ref{eq:ODEIVPConeSing}) exists on $\mathsf{J}$ for every $\alpha \in \mathsf{V}$. Then for any $\alpha_1, \alpha_2 \in \mathsf{V}$ with $\alpha_1 < \alpha_2$, we have $v\left(\gamma; \alpha_1\right) \leq v\left(\gamma; \alpha_2\right) - Q\left(\gamma\right) \left(\alpha_2 - \alpha_1\right) \leq v\left(\gamma; \alpha_2\right)$ for all $\gamma \in \mathsf{J}$, with the second inequality being strict if $\gamma > 1$.
\end{theorem}
\begin{proof}
Consider the following operations performed on the equation (\ref{eq:ODEIVPConeSing}) with $'$ and $\frac{d}{d\alpha}$ denoting derivatives with respect to $\gamma$ and $\alpha$ respectively for $\alpha \in \mathsf{V}$ and $\gamma \in \mathsf{J}$, and use the polynomials $q \left(\gamma\right)$ and $Q \left(\gamma\right)$ from Lemma \ref{lem:derpolynomconesing} and Corollary \ref{cor:derpolynomconesing} respectively:
\begin{equation}\label{eq:ODEIVPalpha}
v_\alpha'\left(\gamma\right) = 2 \sqrt{2} \sqrt{v_\alpha\left(\gamma\right)} + p_\alpha\left(\gamma\right)\gamma \hspace{3pt}, \hspace{5pt} v_\alpha\left(1\right) = 2
\end{equation}
\begin{equation}\label{eq:derODEIVPalpha}
\frac{d}{d\alpha} \left( v_\alpha'\left(\gamma\right) \right) = \frac{\sqrt{2}}{\sqrt{v_\alpha\left(\gamma\right)}} \frac{d}{d\alpha} \left( v_\alpha\left(\gamma\right) \right) + q\left(\gamma\right)
\end{equation}
By Lemma \ref{lem:positivity}, $\sqrt{v_\alpha\left(\gamma\right)} > 0$ for all $\gamma \in \mathsf{J}$. Multiplying by $e^{- \int\limits_1^\gamma \frac{\sqrt{2}}{\sqrt{v\left(y; \alpha\right)}} dy}$:
\begin{equation}\label{eq:intfact}
\left( \frac{d}{d\alpha} \left( v_\alpha\left(\gamma\right) \right) \right)' e^{- \int\limits_1^\gamma \frac{\sqrt{2}}{\sqrt{v\left(y; \alpha\right)}} dy} + \frac{d}{d\alpha} \left( v_\alpha\left(\gamma\right) \right) \left( e^{- \int\limits_1^\gamma \frac{\sqrt{2}}{\sqrt{v\left(y; \alpha\right)}} dy} \right)' = q\left(\gamma\right) e^{- \int\limits_1^\gamma \frac{\sqrt{2}}{\sqrt{v\left(y; \alpha\right)}} dy}
\end{equation}
For $\gamma \in \mathsf{J}$, integrating on $\left[1, \gamma\right]$:
\begin{equation}\label{eq:intODEIVPalpha}
\frac{d}{d\alpha} \left( v_\alpha\left(\gamma\right) \right) e^{- \int\limits_1^\gamma \frac{\sqrt{2}}{\sqrt{v\left(y; \alpha\right)}} dy} = \int\limits_1^\gamma q\left(y\right) e^{- \int\limits_1^y \frac{\sqrt{2}}{\sqrt{v\left(x; \alpha\right)}} dx} dy
\end{equation}
\begin{equation}\label{eq:dersol}
\frac{d}{d\alpha} \left( v_\alpha\left(\gamma\right) \right) \hspace{3pt} = \hspace{3pt} e^{\int\limits_1^\gamma \frac{\sqrt{2}}{\sqrt{v\left(y; \alpha\right)}} dy} \int\limits_1^\gamma q\left(y\right) e^{- \int\limits_1^y \frac{\sqrt{2}}{\sqrt{v\left(x; \alpha\right)}} dx} dy \hspace{3pt} \geq \hspace{3pt} e^{\int\limits_1^\gamma \frac{\sqrt{2}}{\sqrt{v\left(y; \alpha\right)}} dy} \int\limits_1^\gamma q\left(y\right) e^{- \int\limits_1^\gamma \frac{\sqrt{2}}{\sqrt{v\left(x; \alpha\right)}} dx} dy \hspace{3pt} = \hspace{3pt} Q\left(\gamma\right)
\end{equation}
Now for any $\alpha_1, \alpha_2 \in \mathsf{V}$ with $\alpha_1 < \alpha_2$ and for any $\gamma \in \mathsf{J}$ we have (for some $\alpha_3 \in \left(\alpha_1, \alpha_2\right)$):
\begin{equation}\label{eq:LMVT}
v\left(\gamma; \alpha_2\right) - v\left(\gamma; \alpha_1\right) \hspace{2pt} = \hspace{2pt} \frac{d}{d\alpha} \left( v_\alpha\left(\gamma\right) \right) \biggr\rvert_{\alpha = \alpha_3} \left(\alpha_2 - \alpha_1\right) \hspace{2pt} \geq \hspace{2pt} Q\left(\gamma\right) \left(\alpha_2 - \alpha_1\right) \hspace{2pt} \geq \hspace{2pt} 0
\end{equation}
So $v\left(\gamma; \alpha_1\right) \leq v\left(\gamma; \alpha_2\right) - Q\left(\gamma\right) \left(\alpha_2 - \alpha_1\right) \leq v\left(\gamma; \alpha_2\right)$ for all $\gamma \in \mathsf{J}$ and from Corollary \ref{cor:derpolynomconesing}, the second inequality here is clearly strict if $\gamma > 1$.
\end{proof} \par
Define $\mathscr{A} = \left\lbrace \hspace{2pt} \alpha \in \left(-\infty, \bo\right) \hspace{1.5pt} \big\vert \hspace{2.5pt} \text{$v_\alpha$ exists on the whole of $\left[1, m + 1\right]$} \hspace{2pt} \right\rbrace \subseteq \left(-\infty, \bo\right)$. \par
We then first get the following result proving the existence of a value of the parameter $\alpha$ satisfying $v_\alpha \left(m + 1\right) > 2 \left(m + 1\right)^2$ thus completing the proof of the $>$ inequality for the desired final boundary condition on the ODE in (\ref{eq:ODEIVPConeSing}) (as we had discussed about the strategy of establishing the final boundary condition in Subsection \ref{subsec:AnalysisODEBVPConeSing}):
\begin{lemma}\label{lem:VPPConeSing}
$\left[0, \bo\right) \subseteq \mathscr{A}$ and $\lim\limits_{\alpha \to \bo^-} v\left(m + 1; \alpha\right) = b$ where $2 \left(m + 1\right)^2 < b = b \left(m, \bo\right) < \infty$. There exists an $\alpha \in \left[0, \bo\right)$ such that $v_\alpha \left(m + 1\right) > 2 \left(m + 1\right)^2$.
\end{lemma}
\begin{proof}
For any $\alpha \in \left[0, \bo\right) \subseteq \left(-\infty, \bo\right)$, by Lemma \ref{lem:firstexist} the solution $v_\alpha$ to (\ref{eq:ODEIVPConeSing}) a priori exists on some interval $\left[1, \tilde{\gamma}\right)$ with $\tilde{\gamma} > 1$. Integrating the ODE in (\ref{eq:ODEIVPConeSing}) on $\left[1, \gamma\right]$ for $\gamma \in \left[1, \tilde{\gamma}\right)$ and using Lemma \ref{lem:polynom1conesing} will give $v_\alpha \left(\gamma\right) \geq 2 + P_\alpha \left(\gamma\right) \geq 2$ (as $\alpha \geq 0$), and then by Lemma \ref{lem:continue} $v_\alpha$ can be continued beyond $\tilde{\gamma}$, and this will be true for each such $\tilde{\gamma} > 1$. So $v_\alpha$ exists on $\left[1, m + 1\right]$ if $\alpha \in \left[0, \bo\right)$, i.e. $\left[0, \bo\right) \subseteq \mathscr{A}$. In particular we have also proven over here that for $\alpha \in \left[0, \bo\right)$, $v_\alpha \left(\gamma\right) \geq 2$ for all $\gamma \in \cllml$. \\
If $\alpha \in \left[0, \bo\right)$ then from (\ref{eq:BCalphabeta0}) it can be observed that $B\left(\alpha\right), C\left(\alpha\right)$ are bounded with bounds depending only on $m, \bo$ (which have been kept fixed throughout this analysis in Section \ref{sec:ProofConeSing}). So there will exist an $l > 0$ such that $\left\lvert p_\alpha \left(\gamma\right) \gamma \right\rvert \leq l$ for all $\alpha \in \left[0, \bo\right)$. As done earlier in Subsection \ref{subsec:Proof1}, substituting this uniform norm bound $l$ and $\sqrt{v_\alpha} < v_\alpha + 1$ in the expression for $v_\alpha' = \left(v_\alpha + 1\right)'$ in (\ref{eq:ODEIVPConeSing}) and using Gr\"onwall's inequality will give us a $K > 0$ such that $2 \leq v_\alpha \left(\gamma\right) \leq K$ for all $\gamma \in \left[1, m + 1\right]$ and for all $\alpha \in \left[0, \bo\right)$. So in particular we have $2 \leq v\left(m + 1; \alpha\right) \leq K$ for all $\alpha \in \left[0, \bo\right)$. From Theorem \ref{thm:StrctIncrConeSing}, $\alpha \mapsto v \left(m + 1; \alpha\right)$ is strictly increasing in $\alpha$ on $\left[0, \bo\right)$. So $\lim\limits_{\alpha \to \bo^-} v\left(m + 1; \alpha\right)$ exists and equals some $b = b \left(m, \bo\right)$ and further $2 \leq b \leq K < \infty$.
\begin{claim*}
$b > 2 \left(m + 1\right)^2$.
\end{claim*}
{\noindent Take $\alpha \in \left[0, \bo\right)$ and rewrite the ODE in (\ref{eq:ODEIVPConeSing}) as follows:
\begin{equation}\label{eq:ODEIVPalpha'}
v_\alpha'\left(\gamma\right) = 2 \sqrt{2} \sqrt{v_\alpha\left(\gamma\right)} + p_\alpha\left(\gamma\right)\gamma \hspace{3pt}, \hspace{5pt} v_\alpha\left(1\right) = 2
\end{equation}
\begin{equation}\label{eq:ODEsqrtv'}
\left(\sqrt{v_\alpha\left(\gamma\right)}\right)' = \sqrt{2} + \frac{p_\alpha\left(\gamma\right)\gamma}{2 \sqrt{v_\alpha\left(\gamma\right)}}
\end{equation}
Integrating over $\cllml$, using the unique root $\gamma_{0, \alpha}$ of the polynomial $p_\alpha\left(\gamma\right)\gamma$ given by Lemma \ref{lem:polynomconesing}, noting the sign of $p_\alpha\left(\gamma\right)\gamma$ on $\left[1, \gamma_{0, \alpha}\right]$ and $\left[\gamma_{0, \alpha}, m + 1\right]$ separately and substituting the value of the integral of $p_\alpha\left(\gamma\right)\gamma$ over $\cllml$ given by Lemma \ref{lem:polynom1conesing}, we obtain the following estimate:
\begin{align}\label{eq:bg2m12}
\sqrt{v\left(m + 1; \alpha\right)} &\geq \sqrt{2} \left(m + 1\right) + \frac{1}{2 \sqrt{K}} \int\limits_1^{\gamma_{0, \alpha}} p_\alpha\left(\gamma\right)\gamma d\gamma + \frac{1}{2 \sqrt{2}} \int\limits_{\gamma_{0, \alpha}}^{m + 1} p_\alpha\left(\gamma\right)\gamma d\gamma \nonumber \\
&= \sqrt{2} \left(m + 1\right) + \frac{m \left(m + 2\right)}{4 \sqrt{K}} \alpha + \frac{\sqrt{K} - \sqrt{2}}{2 \sqrt{2} \sqrt{K}} \int\limits_{\gamma_{0, \alpha}}^{m + 1} p_\alpha\left(\gamma\right)\gamma d\gamma
\end{align}
Since $\gamma_{0, \alpha} = \sqrt{-\frac{2 C \left(\alpha\right)}{B \left(\alpha\right)}}$, it can be checked by using the expressions (\ref{eq:BCalphabeta0}) for $B \left(\alpha\right), C \left(\alpha\right)$ that $\lim\limits_{\alpha \to \bo^-} \gamma_{0, \alpha} = m + 1$. We also have $\left\lvert p_\alpha \left(\gamma\right) \gamma \right\rvert \leq l$ where the bound $l$ does not vary with $\alpha \in \left[0, \bo\right)$. So we must have $\lim\limits_{\alpha \to \bo^-} \int\limits_{\gamma_{0, \alpha}}^{m + 1} p_\alpha\left(\gamma\right)\gamma d\gamma = 0$. Now passing the limits as $\alpha \to \bo^-$ in the estimate (\ref{eq:bg2m12}) we get $\sqrt{b} \geq \sqrt{2} \left(m + 1\right) + \frac{m \left(m + 2\right)}{4 \sqrt{K}} \bo > \sqrt{2} \left(m + 1\right)$, thus proving the Claim.}
\end{proof} \par
\begin{motivation}
In the remainder of Subsection \ref{subsec:Proof2} we will prove that $\mathscr{A} = \left(M, \bo\right)$ for some $M = M \left(m, \bo\right) \in \left(-\infty, 0\right)$ and $\lim\limits_{\alpha \to M^+} v\left(m + 1; \alpha\right) = 0$, which will give us an $\alpha \in \mathscr{A}$ satisfying $v \left(m + 1; \alpha\right) < 2 \left(m + 1\right)^2$ which will prove the $<$ inequality needed for concluding the required final boundary condition viz. $v \left(m + 1; \alpha\right) = 2 \left(m + 1\right)^2$.
\end{motivation} \par
\begin{theorem}\label{thm:scrA}
We have the following properties of the set $\mathscr{A}$:
\begin{enumerate}
\item $\mathscr{A}$ is an interval. \label{itm:scrAinterval}
\item $\mathscr{A}$ is open. \label{itm:scrAopen}
\item $\mathscr{A} \subsetneq \left(-\infty, \bo\right)$. \label{itm:scrAprprsbst}
\item There exists an $M = M \left(m, \bo\right) < 0$ such that $\mathscr{A} = \left(M, \bo\right)$. \label{itm:scrAfinal}
\end{enumerate}
\end{theorem}
\begin{proof}
$ $ \vspace{5pt} \newline
(\ref{itm:scrAinterval}) Let $A_1, A_2 \in \mathscr{A}$ with $A_1 < A_2$.
\begin{claim*}
There exists a common $\tilde{\gamma} \in \left(1, m + 1\right]$ such that $v_\alpha$ exists at least on $\left[1, \tilde{\gamma}\right)$ for all $\alpha \in \left[A_1, A_2\right]$.
\end{claim*}
{\noindent If not true then there exists a sequence $\left(\alpha_n\right)$ in $\left[A_1, A_2\right]$ such that $v_n = v_{\alpha_n}$ exists maximally on $\left[1, \gamma_{\star,n}\right)$ where $\gamma_{\star,n} = \gamma_{\star,\alpha_n} \in \left(1, m + 1\right]$ and the sequence $\left(\gamma_{\star,n}\right) \to 1$. Passing to a subsequence if necessary, assume $\left(\alpha_n\right) \to \alpha_0 \in \left[A_1, A_2\right]$. By Theorem \ref{thm:UnifConv} Case ($3$), the sequence $\left(u_n = u_{\alpha_n}\right)$ in $\mathcal{C} \left[1, m + 1\right]$ has a subsequence $\left(u_{n_k}\right) \to u_0 = u_{\alpha_0}$ uniformly on $\left[1, m + 1\right]$ with the subsequence $\left(\gamma_{\star,n_k}\right) \to \gamma_{\star,0} = \gamma_{\star,\alpha_0}$ where $\left[1, \gamma_{\star,0}\right)$ is the maximal interval of existence of $v_0 = v_{\alpha_0}$. By Lemma \ref{lem:firstexist}, $\gamma_{\star,0} > 1$. So $\left(\gamma_{\star,n_k}\right) \to \gamma_{\star,0} > 1$ and $\left(\gamma_{\star,n}\right) \to 1$, a contradiction. Hence the Claim.} \\
Take any $\tilde{\gamma} > 1$ with the property mentioned in the Claim above. Applying Theorem \ref{thm:StrctIncrConeSing} with $\mathsf{V} = \left[A_1, A_2\right]$ and $\mathsf{J} = \left[1, \tilde{\gamma}\right)$ we will get for any $\alpha \in \left[A_1, A_2\right]$ and for all $\gamma \in \left[1, \tilde{\gamma}\right)$, $v\left(\gamma; \alpha\right) \geq v\left(\gamma; A_1\right)$. Since $A_1 \in \mathscr{A}$ so $v_{A_1}$ exists on the whole of $\left[1, m + 1\right]$ and by Theorem \ref{thm:continue}, there exists an $\epsilon > 0$ (depending only on $A_1$) such that $v\left(\gamma; A_1\right) \geq \epsilon$ for all $\gamma \in \left[1, m + 1\right]$ and hence in particular for all $\gamma \in \left[1, \tilde{\gamma}\right)$. So $v\left(\gamma; \alpha\right) \geq \epsilon$ for all $\gamma \in \left[1, \tilde{\gamma}\right)$ and so by Theorem \ref{thm:continue}, $v_\alpha$ can be continued beyond $\tilde{\gamma}$ for all $\alpha \in \left[A_1, A_2\right]$. Since this holds true for any $\tilde{\gamma} > 1$ with the property mentioned in the Claim above and the lower bound $\epsilon > 0$ on $v_\alpha$ does not depend on $\alpha \in \left[A_1, A_2\right]$ as well as on $\tilde{\gamma} > 1$ so $v_\alpha$ has to exist on $\left[1, m + 1\right]$ for all $\alpha \in \left[A_1, A_2\right]$ i.e. $\left[A_1, A_2\right] \subseteq \mathscr{A}$. So $\mathscr{A}$ is an interval. \vspace{5pt} \\
(\ref{itm:scrAopen}) Take a sequence $\left(\alpha_n\right) \to \alpha_0 \in \left(-\infty, \bo\right)$ of points in $\mathscr{A}^c = \left(-\infty, \bo\right) \smallsetminus \mathscr{A}$. Then by definition of $\mathscr{A}$, the solution $v_n$ to (\ref{eq:ODEIVPConeSing}) with $\alpha = \alpha_n$ has maximal interval of existence $\left[1, \gamma_{\star,n}\right)$ and we are in Theorem \ref{thm:UnifConv} Case ($3$). So there exists a subsequence $\left(u_{n_k}\right) \to u_0$ uniformly on $\left[1, m + 1\right]$ with $\left(\gamma_{\star,n_k}\right) \to \gamma_{\star,0}$. So $\left[1, \gamma_{\star,0}\right)$ is the maximal interval of existence of $v_0$ as the solution to (\ref{eq:ODEIVPConeSing}) with $\alpha = \alpha_0$ and so by definition, $\alpha_0 \in \mathscr{A}^c$. So $\mathscr{A}$ is open. \vspace{5pt} \\
(\ref{itm:scrAprprsbst}) Let if possible $v_\alpha$ exist on $\left[1, m + 1\right]$ for all $\alpha \in \left(-\infty, \bo\right)$. Taking $\mathsf{V} = \left(-\infty, \bo\right)$ and $\mathsf{J} = \left[1, m + 1\right]$ in Theorem \ref{thm:StrctIncrConeSing} we get for a fixed $\alpha_0 \in \left(-\infty, \bo\right)$ and for any $\alpha < \alpha_0$ and with $\gamma = m + 1$, $v\left(m + 1; \alpha\right) \leq v\left(m + 1; \alpha_0\right) - Q\left(m + 1\right) \left(\alpha_0 - \alpha\right)$. Taking $\alpha_n = \alpha_0 - n$ for $n \in \mathbb{N}$ and as $Q\left(m + 1\right) = \frac{m \left(m + 2\right)}{2}$ (from Corollary \ref{cor:derpolynomconesing}), we get $v\left(m + 1; \alpha_n\right) \leq v\left(m + 1; \alpha_0\right) - n \frac{m \left(m + 2\right)}{2}$. Since $v\left(m + 1; \alpha_n\right), v\left(m + 1; \alpha_0\right) > 0$ (from Lemma \ref{lem:positivity}), we have $n < \frac{2 v\left(m + 1; \alpha_0\right)}{m \left(m + 2\right)}$ for all $n \in \mathbb{N}$, a contradiction. So there exists an $\alpha \in \left(-\infty, \bo\right)$ such that $v_\alpha$ has maximal interval of existence $\left[1, \gamma_{\star,\alpha}\right)$ for some $\gamma_{\star,\alpha} \in \left(1, m + 1\right]$ i.e. $v_\alpha$ does not exist till $m + 1$ and so $\alpha \in \mathscr{A}^c = \left(-\infty, \bo\right) \smallsetminus \mathscr{A}$. So $\mathscr{A} \subsetneq \left(-\infty, \bo\right)$. \vspace{5pt} \\
(\ref{itm:scrAfinal}) From Lemma \ref{lem:VPPConeSing} and (\ref{itm:scrAinterval}), (\ref{itm:scrAopen}) and (\ref{itm:scrAprprsbst}), there exists an $M = M \left(m, \bo\right) < 0$ such that $\mathscr{A} = \left(M, \bo\right)$.
\end{proof} \par
For each $\alpha \in \left(-\infty, \bo\right)$ let $I_\alpha \subseteq \left[1, m + 1\right]$ denote the maximal interval of existence of the solution $v_\alpha$ to the ODE initial value problem (\ref{eq:ODEIVPConeSing}). Let $\mathscr{P} \left[1, m + 1\right]$ denote the power set of $\left[1, m + 1\right]$. By Lemma \ref{lem:firstexist} and Theorem \ref{thm:continue} we get the set map $\left(-\infty, \bo\right) \to \mathscr{P} \left[1, m + 1\right]$, $\alpha \mapsto I_\alpha$. Then using Theorems \ref{thm:UnifConv}, \ref{thm:StrctIncrConeSing} and \ref{thm:scrA} and the definitions of $u_\alpha$, $\Phi$ (seen earlier in Subsection \ref{subsec:Proof1}) and $\mathscr{A}$, we get the following two results:
\begin{corollary}\label{cor:StrctIncrConeSing}
The set map $\left(-\infty, \bo\right) \to \mathscr{P} \left[1, m + 1\right]$, $\alpha \mapsto I_\alpha$ and the function $\Phi : \left(-\infty, \bo\right) \to \mathcal{C} \left[1, m + 1\right]$, $\Phi \left(\alpha\right) = u_\alpha \left(\cdot\right) = u \left(\cdot; \alpha\right)$ are monotone increasing in $\alpha$:
\begin{enumerate}
\item If $\alpha_1, \alpha_2 \in \mathscr{A}^c = \left(-\infty, \bo\right) \smallsetminus \mathscr{A}$, $\alpha_1 < \alpha_2$ then $\left[\alpha_1, \alpha_2\right] \subseteq \mathscr{A}^c$ and $\gamma_{\star,\alpha_1} < \gamma_{\star,\alpha_2}$ i.e. $I_{\alpha_1} \subsetneq I_{\alpha_2}$. In general if $\alpha_1, \alpha_2 \in \left(-\infty, \bo\right)$, $\alpha_1 < \alpha_2$ then $I_{\alpha_1} \subseteq I_{\alpha_2}$, with the set containment being strict if $\alpha_1 \in \mathscr{A}^c$ and it being set equality otherwise. \label{itm:Ialpha}
\item If $\alpha_1, \alpha_2 \in \mathscr{A}$, $\alpha_1 < \alpha_2$ then $\left[\alpha_1, \alpha_2\right] \subseteq \mathscr{A}$ and $v\left(\gamma; \alpha_1\right) < v\left(\gamma; \alpha_2\right)$ i.e. $u\left(\gamma; \alpha_1\right) < u\left(\gamma; \alpha_2\right)$ for all $\gamma \in \left(1, m + 1\right]$. In general if $\alpha_1, \alpha_2 \in \left(-\infty, \bo\right)$, $\alpha_1 < \alpha_2$ then $u\left(\gamma; \alpha_1\right) \leq u\left(\gamma; \alpha_2\right)$ for all $\gamma \in \left(1, m + 1\right]$, with the inequality being strict if $\gamma \in I_{\alpha_2} \smallsetminus \stone$ and it being equality otherwise. \label{itm:ualpha}
\end{enumerate}
\end{corollary}
\begin{proof}
$ $ \vspace{5pt} \newline
(\ref{itm:Ialpha}) Given $\alpha_1, \alpha_2 \in \mathscr{A}^c$, $\alpha_1 < \alpha_2$ then as $\mathscr{A}^c = \left(-\infty, M\right]$ (by Theorem \ref{thm:scrA} (\ref{itm:scrAfinal})) so clearly $\left[\alpha_1, \alpha_2\right] \subseteq \mathscr{A}^c$. So for every $\alpha \in \left[\alpha_1, \alpha_2\right]$, $I_\alpha = \left[1, \gamma_{\star, \alpha}\right)$.
\begin{claim*}
There exists a $\tilde{\gamma} \in \left(1, m + 1\right]$ such that $v_\alpha$ exists at least on $\left[1, \tilde{\gamma}\right)$ i.e. $\left[1, \tilde{\gamma}\right) \subseteq I_\alpha$ for all $\alpha \in \left[\alpha_1, \alpha_2\right]$.
\end{claim*}
{\noindent The proof of the above Claim is exactly the same as that of the Claim in Theorem \ref{thm:scrA} (\ref{itm:scrAinterval}) with $A_1, A_2$ being replaced by $\alpha_1, \alpha_2$.} \\
Take any $\tilde{\gamma} > 1$ with the property mentioned in the above Claim. Then clearly $\tilde{\gamma} \leq \gamma_{\star, \alpha_1}$. If $\tilde{\gamma} = \gamma_{\star, \alpha_1}$ then $\left[1, \gamma_{\star, \alpha_1}\right) = I_{\alpha_1} \subseteq I_\alpha$ for all $\alpha \in \left[\alpha_1, \alpha_2\right]$ which is precisely what we are trying to prove now over here. Otherwise taking any such $\tilde{\gamma} < \gamma_{\star, \alpha_1}$ and applying Theorem \ref{thm:StrctIncrConeSing} with $\mathsf{V} = \left[\alpha_1, \alpha_2\right]$ and $\mathsf{J} = \left[1, \tilde{\gamma}\right)$ we get $v\left(\gamma; \alpha\right) \geq v\left(\gamma; \alpha_1\right)$ for all $\gamma \in \left[1, \tilde{\gamma}\right)$ and for any $\alpha \in \left[\alpha_1, \alpha_2\right]$. Since $\left[1, \tilde{\gamma}\right) \subsetneq I_{\alpha_1}$ so by Lemma \ref{lem:positivity} and Theorem \ref{thm:continue}, there exists an $\epsilon_{\tilde{\gamma}} > 0$ (depending only on $\tilde{\gamma}$ and $\alpha_1$) such that $v\left(\gamma; \alpha_1\right) \geq \epsilon_{\tilde{\gamma}}$ for all $\gamma \in \left[1, \tilde{\gamma}\right)$. So for each $\alpha \in \left[\alpha_1, \alpha_2\right]$, $v\left(\gamma; \alpha\right) \geq \epsilon_{\tilde{\gamma}}$ for all $\gamma \in \left[1, \tilde{\gamma}\right)$ and so by Theorem \ref{thm:continue}, $v_\alpha$ can be continued beyond $\tilde{\gamma}$ for all $\alpha \in \left[\alpha_1, \alpha_2\right]$. Since this holds true for any $1 < \tilde{\gamma} < \gamma_{\star, \alpha_1}$ with the property mentioned in the above Claim and the lower bound $\epsilon_{\tilde{\gamma}} > 0$ on $v_\alpha$ does not depend on $\alpha \in \left[\alpha_1, \alpha_2\right]$ so $v_\alpha$ has to exist on $\left[1, \gamma_{\star, \alpha_1}\right)$ i.e. $I_{\alpha_1} \subseteq I_\alpha$ for all $\alpha \in \left[\alpha_1, \alpha_2\right]$. \\
So in particular $I_{\alpha_1} \subseteq I_{\alpha_2}$. Let if possible $I_{\alpha_1} = I_{\alpha_2}$ i.e. $\gamma_{\star, \alpha_1} = \gamma_{\star, \alpha_2}$. So by Lemma \ref{lem:positivity} we have $\lim\limits_{\gamma \to \gamma_{\star, \alpha_1}} v \left(\gamma; \alpha_1\right) = 0 = \lim\limits_{\gamma \to \gamma_{\star, \alpha_2}} v \left(\gamma; \alpha_2\right)$. Taking $\mathsf{V} = \left[\alpha_1, \alpha_2\right]$ and $\mathsf{J} = \left[1,
\gamma_{\star, \alpha_1}\right)$ in Theorem \ref{thm:StrctIncrConeSing} we get $v\left(\gamma; \alpha_2\right) \geq v\left(\gamma; \alpha_1\right) + Q\left(\gamma\right) \left(\alpha_2 - \alpha_1\right)$ for all $\gamma \in \left[1, \gamma_{\star, \alpha_1}\right)$. Applying limits as $\gamma \to \gamma_{\star, \alpha_1}$ we get $Q\left(\gamma_{\star, \alpha_1}\right) \leq 0$ where $\gamma_{\star, \alpha_1} > 1$ which is a contradiction to Corollary \ref{cor:derpolynomconesing}. So $I_{\alpha_1} \subsetneq I_{\alpha_2}$ i.e. $\gamma_{\star,\alpha_1} < \gamma_{\star,\alpha_2}$. \\
If $\alpha_1 \in \mathscr{A}^c$ and $\alpha_2 \in \mathscr{A}$ then $\alpha_1 < \alpha_2$ (by Theorem \ref{thm:scrA} (\ref{itm:scrAfinal})) and $I_{\alpha_1} = \left[1, \gamma_{\star, \alpha_1}\right) \subsetneq \left[1, m + 1\right] = I_{\alpha_2}$. If $\alpha_1, \alpha_2 \in \mathscr{A}$, $\alpha_1 < \alpha_2$ then $I_{\alpha_1} = \left[1, m + 1\right] = I_{\alpha_2}$. Thus the general statement for $\alpha_1, \alpha_2 \in \left(-\infty, \bo\right)$, $\alpha_1 < \alpha_2$ holds true. \vspace{5pt} \\
(\ref{itm:ualpha}) Given $\alpha_1, \alpha_2 \in \mathscr{A}$, $\alpha_1 < \alpha_2$ then from Theorem \ref{thm:scrA} (\ref{itm:scrAinterval}), $\left[\alpha_1, \alpha_2\right] \subseteq \mathscr{A}$. So for every $\alpha \in \left[\alpha_1, \alpha_2\right]$, $I_\alpha = \left[1, m + 1\right]$ and $u_\alpha = v_\alpha$ on $\left[1, m + 1\right]$. Applying Theorem \ref{thm:StrctIncrConeSing} with $\mathsf{V} = \left[\alpha_1, \alpha_2\right]$ and $\mathsf{J} = \left[1, m + 1\right]$ we get $v\left(\gamma; \alpha_1\right) < v\left(\gamma; \alpha_2\right)$ i.e. $u\left(\gamma; \alpha_1\right) < u\left(\gamma; \alpha_2\right)$ for all $\gamma \in \left(1, m + 1\right]$. \\
If $\alpha_1 \in \mathscr{A}^c$ and $\alpha_2 \in \mathscr{A}$ then $\alpha_1 < \alpha_2$ and $u_{\alpha_1} = v_{\alpha_1}$ on $I_{\alpha_1} = \left[1, \gamma_{\star, \alpha_1}\right)$ and $u_{\alpha_1} \equiv 0$ on $\left[\gamma_{\star, \alpha_1}, m + 1\right]$ and $u_{\alpha_2} = v_{\alpha_2}$ on $I_{\alpha_2} = \left[1, m + 1\right]$. From (\ref{itm:Ialpha}) above, $I_{\alpha_1} \subseteq I_\alpha$ for all $\alpha \in \left[\alpha_1, \alpha_2\right]$ and so by Theorem \ref{thm:StrctIncrConeSing} with $\mathsf{V} = \left[\alpha_1, \alpha_2\right]$ and $\mathsf{J} = \left[1, \gamma_{\star, \alpha_1}\right)$ we get $v\left(\gamma; \alpha_1\right) < v\left(\gamma; \alpha_2\right)$ for all $\gamma \in \left(1, \gamma_{\star, \alpha_1}\right)$. On $\left[\gamma_{\star, \alpha_1}, m + 1\right]$ by Lemma \ref{lem:positivity}, $u_{\alpha_1} = 0 < v_{\alpha_2} = u_{\alpha_2}$. \\
If $\alpha_1, \alpha_2 \in \mathscr{A}^c$, $\alpha_1 < \alpha_2$ then by (\ref{itm:Ialpha}) above, $u_\alpha = v_\alpha$ on $I_\alpha = \left[1, \gamma_{\star, \alpha}\right)$ and $u_\alpha \equiv 0$ on $\left[\gamma_{\star, \alpha}, m + 1\right]$ for all $\alpha \in \left[\alpha_1, \alpha_2\right]$. Also by (\ref{itm:Ialpha}), $I_{\alpha_1} \subseteq I_\alpha$ for all $\alpha \in \left[\alpha_1, \alpha_2\right]$ and so again using Theorem \ref{thm:StrctIncrConeSing} with $\mathsf{V} = \left[\alpha_1, \alpha_2\right]$ and $\mathsf{J} = \left[1, \gamma_{\star, \alpha_1}\right)$ we get $v\left(\gamma; \alpha_1\right) < v\left(\gamma; \alpha_2\right)$ for all $\gamma \in \left(1, \gamma_{\star, \alpha_1}\right)$. As $\gamma_{\star, \alpha_1} < \gamma_{\star, \alpha_2}$ (again by (\ref{itm:Ialpha})) so on $\left[\gamma_{\star, \alpha_1}, \gamma_{\star, \alpha_2}\right)$ by Lemma \ref{lem:positivity} and Theorem \ref{thm:continue}, $u_{\alpha_1} = 0 < v_{\alpha_2} = u_{\alpha_2}$. On $\left[\gamma_{\star, \alpha_2}, m + 1\right]$, $u_{\alpha_2} = 0 = u_{\alpha_1}$ as $I_{\alpha_1} \subsetneq I_{\alpha_2}$. \\
Thus the general statement for $\alpha_1, \alpha_2 \in \left(-\infty, \bo\right)$, $\alpha_1 < \alpha_2$ holds true.
\end{proof} \par
Because of monotonicity we can get the uniform convergence of the whole sequence $\left(u_n\right)$ instead of just a subsequence $\left(u_{n_k}\right)$ in Theorem \ref{thm:UnifConv}:
\begin{corollary}\label{cor:Phicont}
If $\left(\alpha_n\right) \uparrow \alpha_0$ in $\left(-\infty, \bo\right)$ then $\left(u_n = u_{\alpha_n}\right) \uparrow u_0 = u_{\alpha_0}$ in $\mathcal{C} \left[1, m + 1\right]$, and more generally if $\left(\alpha_n\right) \to \alpha_0$ then $\left(u_n\right) \to u_0$ uniformly on $\left[1, m + 1\right]$. Thus $\Phi$ is continuous.
\end{corollary}
\begin{proof}
Use Theorem \ref{thm:UnifConv} and Corollary \ref{cor:StrctIncrConeSing} and apply Dini's theorem.
\end{proof} \par
We now have the final result of Section \ref{sec:ProofConeSing} which will prove Theorem \ref{thm:mainconesing} and as a consequence Corollary \ref{cor:mainconesing}:
\begin{corollary}\label{cor:final}
$\lim\limits_{\alpha \to M^+} v\left(m + 1; \alpha\right) = 0$. There exists a unique $\alpha = \alpha \left(m, \bo\right) \in \left(M, \bo\right)$ such that $v \left(m + 1; \alpha\right) = 2 \left(m + 1\right)^2$ and the $\alpha$ with this property has to be strictly negative.
\end{corollary}
\begin{proof}
By Theorem \ref{thm:scrA}, $\mathscr{A} = \left(M, \bo\right) \subseteq \left(-\infty, \bo\right)$ and so $\mathscr{A}^c = \left(-\infty, M\right]$. Let $\left(A_n\right) \uparrow M$ be a sequence of points in $\mathscr{A}^c$ then from Corollary \ref{cor:Phicont}, $\left(u_{A_n}\right) \uparrow u_M$ uniformly on $\left[1, m + 1\right]$ and by definition, $I_{A_n} = \left[1, \gamma_{\star, A_n}\right)$ and so by Theorem \ref{thm:UnifConv} Case ($3.2$) and Corollaries \ref{cor:StrctIncrConeSing} and \ref{cor:Phicont}, $I_M = \left[1, \gamma_{\star, M}\right)$ where $\left(\gamma_{\star, A_n}\right) \uparrow \gamma_{\star, M}$ i.e. $\left(I_{A_n}\right)$ is monotone increasing in $\mathscr{P} \left[1, m + 1\right]$ and $I_M = \bigcup\limits_{n \in \mathbb{N}} I_{A_n}$. \\
Let $\left(\alpha_n\right) \downarrow M \in \mathscr{A}^c$ be a sequence of points in $\mathscr{A}$ then $\left(u_{\alpha_n}\right) \downarrow u_M$ uniformly on $\left[1, m + 1\right]$ and by definition, $I_{\alpha_n} = \left[1, m + 1\right]$ and $u_{\alpha_n} = v_{\alpha_n}$ on $\left[1, m + 1\right]$ for all $n \in \mathbb{N}$. So we will land up in either of the Cases ($1$) or ($2$) in Theorem \ref{thm:UnifConv}. In Theorem \ref{thm:UnifConv} Case ($1$), $\inf\limits_{\gamma \in \left[1, m + 1\right], n \in \mathbb{N}} u_{\alpha_n} \left(\gamma\right) = \epsilon > 0$ and hence the uniform limit $u_M \geq \epsilon$ on $\left[1, m + 1\right]$ and so by Lemma \ref{lem:continue}, $I_M = \left[1, m + 1\right]$ with $u_M = v_M$ on $\left[1, m + 1\right]$, thereby implying that $M \in \mathscr{A}$, a contradiction. So we are in Theorem \ref{thm:UnifConv} Case ($2$) which has $\inf\limits_{\gamma \in \left[1, m + 1\right], n \in \mathbb{N}} u_{\alpha_n} \left(\gamma\right) = 0$ and hence $I_M = \left[1, m + 1\right)$ (i.e. $\gamma_{\star, M} = m + 1$) and $u_M = v_M > 0$ (by Lemma \ref{lem:positivity}) on $\left[1, m + 1\right)$ and $u_M \left(m + 1\right) = 0$. So by pointwise convergence, $\left( u_{\alpha_n} \left(m + 1\right) \right) \downarrow u_M \left(m + 1\right)$ i.e. $\left(v_{\alpha_n} \left(m + 1\right)\right) \downarrow 0$ and so we get $\lim\limits_{\alpha \to M^+} v\left(m + 1; \alpha\right) = 0$. \\
From this and from Lemma \ref{lem:VPPConeSing} and as the ODE and the initial condition in (\ref{eq:ODEIVPConeSing}) depend continuously on the parameter $\alpha$, there exists an $\alpha = \alpha \left(m, \bo\right) \in \mathscr{A} = \left(M, \bo\right)$ such that $v \left(m + 1; \alpha\right) = 2 \left(m + 1\right)^2 > 0$ and by the strictness of the inequalities in Theorem \ref{thm:StrctIncrConeSing} and Corollary \ref{cor:StrctIncrConeSing}, this $\alpha$ has to be unique. This unique value of the parameter $\alpha$ depends on both $m > 0$ as well as $\bo > 0$, the two parameters which we have kept fixed since the beginning of Section \ref{sec:ProofConeSing}. \\
Now to show this $\alpha$ yielding the correct final boundary condition on the solution $v_\alpha$ is strictly negative we recollect from Subsection \ref{subsec:AnalysisODEBVPConeSing} that if there exists a smooth solution $v_\alpha$ to the ODE initial value problem (\ref{eq:ODEIVPConeSing}) satisfying both the boundary conditions viz. $v_\alpha \left(1\right) = 2$ and $v_\alpha \left(m + 1\right) = 2 \left(m + 1\right)^2$ then $v_\alpha \left(\gamma\right) > 2 \gamma^2$ for all $\gamma \in \left(1, m + 1\right)$. Substituting all this and $P_\alpha \left(m + 1\right) = \frac{m \left(m + 2\right)}{2} \alpha$ (from Lemma \ref{lem:polynom1conesing}) in the ODE in (\ref{eq:ODEIVPConeSing}) and integrating it over $\left[1, m + 1\right]$ we get $2 \left(m + 1\right)^2 - 2 > 2 \left( \left(m + 1\right)^2 - 1 \right) + \frac{m \left(m + 2\right)}{2} \alpha$ which implies $\alpha < 0$. So if $\alpha \in \left(M, \bo\right)$ satisfies $v \left(m + 1; \alpha\right) = 2 \left(m + 1\right)^2$ then $\alpha < 0$.
\end{proof}
\numberwithin{equation}{section}
\numberwithin{figure}{section}
\numberwithin{table}{section}
\numberwithin{lemma}{section}
\numberwithin{proposition}{section}
\numberwithin{result}{section}
\numberwithin{theorem}{section}
\numberwithin{corollary}{section}
\numberwithin{conjecture}{section}
\numberwithin{remark}{section}
\numberwithin{note}{section}
\numberwithin{motivation}{section}
\numberwithin{question}{section}
\numberwithin{answer}{section}
\numberwithin{case}{section}
\numberwithin{claim}{section}
\numberwithin{definition}{section}
\numberwithin{example}{section}
\numberwithin{hypothesis}{section}
\numberwithin{statement}{section}
\numberwithin{ansatz}{section}
\section{Polyhomogeneity of Momentum-Constructed Conical K\"ahler Metrics}\label{sec:polyhomoConeSing}
In Section \ref{sec:polyhomoConeSing} we will see that our momentum-constructed conical \krm{s} on our minimal ruled surface are \plyhomo{} smooth conical \krm{s} (given by Definition \ref{def:coneKr4}) if the momentum profile is assumed to be real analytic on the whole momentum interval (including both the endpoints) and are only \conrml{} smooth (Definition \ref{def:coneKr3}) if the momentum profile is assumed to be just smooth on the whole momentum interval. Results similar to these in some or the other form have been given in the works of Hashimoto \cite{Hashimoto:2019:cscKConeSing}, Li \cite{Li:2012:eKEngyFunctProjBund}, Rubinstein-Zhang \cite{Rubinstein:2022:KEedgeHirzebruch}, but the explicit formulations of these results (considering the various definitions of conical \krm{s} seen in Subsection \ref{subsec:KhlrConeSing}) and their proofs (using the ingredients of the momentum construction described in Subsection \ref{subsec:MomentConstructConeSing}) are somewhat simpler and clearer in our special case of the pseudo-Hirzebruch surface $X = \prj$. \par
It was shown by Hwang \cite{Hwang:1994:cscK}, Hwang-Singer \cite{Hwang:2002:MomentConstruct} in the momentum construction of smooth \krm{s} that the boundary conditions given in (\ref{eq:ODEBVP0}) on the real analytic momentum profile $\psi$ are equivalent to the momentum variable $x$ having certain nice asymptotic power series expansions in $\left\lvert w \right\rvert^2$ in a tubular \nbd{} of $w = 0$ (corresponding to the zero divisor) and also in $\left\lvert w^{-1} \right\rvert^2$ in a \nbd{} of $w^{-1} = 0$ (corresponding to the infinity divisor), $w$ being the fibre coordinate on the surface $X$. The conical singularities version of this result is exactly the same as the smooth version but instead uses the boundary conditions (\ref{eq:BVPConeSing}) on the momentum profile $\phi\left(\gamma\right)$ and it can be found in the works of Hashimoto \cite{Hashimoto:2019:cscKConeSing}, Li \cite{Li:2012:eKEngyFunctProjBund}, Rubinstein-Zhang \cite{Rubinstein:2022:KEedgeHirzebruch}. It is this result which gives the required kind of asymptotic power series expansions at the zero and infinity divisors for the coefficient functions in the local coordinate expression of the form (\ref{eq:gij}) for a momentum-constructed conical \krm{}, thereby making it satisfy the conditions of Definition \ref{def:coneKr4}. \par
Consider the bundle-adapted local \hol{} coordinates $\left(z,w\right)$ (introduced in Subsection \ref{subsec:MomentConstructConeSing}) near the divisor $\so$ of the surface $X$, and let $\tilde{w} = w^{-1}$ then $\left(z,\tilde{w}\right)$ are also bundle-adapted local coordinates but near the divisor $\soo$. Let $r = \left\lvert w \right\rvert$ and $\tilde{r} = \left\lvert \tilde{w} \right\rvert$ then $\tilde{r} = r^{-1}$. We can get the relations between the boundary values of $r, w, \tau$ from (\ref{eq:wgammasosoo}) as: $r = 0 \iff \tau = 0$ geometrically giving the boundary behaviour of the momentum-constructed \krm{} $\omega$ at the zero divisor $\so$ and $\tilde{r} = 0 \iff \tau = m > 0$ giving the boundary behaviour of $\omega$ at the infinity divisor $\soo$. We will now be needing all the relations between the functions $f : \R \to \R$, $F : \opom \to \R$, $\phi : \clom \to \R$ as well as the variables $w, s, \tau$ which we had seen in Subsection \ref{subsec:MomentConstructConeSing} like those given by (\ref{eq:variablechange}), (\ref{eq:wgammasosoo}). Also we have by its definition $s = \ln \left\lvert w \right\rvert^2 = 2 \ln r$ and $s = \ln \left\lvert \tilde{w} \right\rvert^{-2} = -2 \ln \tilde{r}$ (following the convention $d \ln h \left(z\right) = 0$ from Sz\'ekelyhidi \cite{Szekelyhidi:2014:eKintro}; Section 4.4) and so we get $d s = \frac{2}{r} d r = -\frac{2}{\tilde{r}} d \tilde{r}$ and similarly $\frac{d}{d s} = \frac{r}{2} \frac{d}{d r} = -\frac{\tilde{r}}{2} \frac{d}{d \tilde{r}}$. \par
\begin{remark}
We would like to quickly note that because of the equation $\phi \left(\tau\right) = \frac{1}{F'' \left(\tau\right)} = f'' \left(s\right)$ with $\tau \in \opom, s \in \R$ being related by the Legendre transform $s = F' \left(\tau\right)$, we can easily see that $\phi$ is real analytic in $\tau$ on $\opom$ if and only if $f'' \left(s\right)$ is real analytic in $r$ for $r \in \left(0, \infty\right)$ if and only if $f'' \left(s\right)$ is real analytic in $\tilde{r}$ for $\tilde{r} \in \left(0, \infty\right)$. Whereas the real analyticity of $\phi\left(\tau\right)$ at the endpoints $\tau = 0$ and $\tau = m$ along with the boundary conditions (\ref{eq:BVPConeSing}) turns out to be equivalent to $f'' \left(s\right)$ having a certain specific kind of power series expansion in $r^{2\beta_0}$ in a \nbd{} of $r = 0$ and also in $\tilde{r}^{2\beta_\infty}$ in a \nbd{} of $\tilde{r} = 0$ \cite{Hashimoto:2019:cscKConeSing,Hwang:1994:cscK,Hwang:2002:MomentConstruct,Li:2012:eKEngyFunctProjBund,Rubinstein:2022:KEedgeHirzebruch}. But proving this requires some work as $\tau$ and $r$ (as well as $\tilde{r}$) are related only by means of $s$, and even though the boundary values of $\tau$ and $r$ (and also $\tilde{r}$) are finite, the boundary values of $s$ are however infinite as can be seen from (\ref{eq:wgammasosoo}).
\end{remark} \par
\begin{theorem}[Asymptotic Power Series Expansions for the Momentum Profile at the Two Divisors; Hashimoto \cite{Hashimoto:2019:cscKConeSing}; Lemma 3.6, Hwang \cite{Hwang:1994:cscK}; Lemmas 2.2 and 2.5, Proposition 2.1, Hwang-Singer \cite{Hwang:2002:MomentConstruct}; Section 2.2, Proposition 2.3, Li \cite{Li:2012:eKEngyFunctProjBund}; Lemma 2.3, Rubinstein-Zhang \cite{Rubinstein:2022:KEedgeHirzebruch}; Proposition 3.3]\label{thm:polyhomoconesing}
Let $\omega$ be the \krm{} given by the Calabi ansatz (\ref{eq:ansatzconesing}) on the minimal ruled surface $X = \prj$, which is smooth on the non-compact surface $X \smallsetminus \left(S_0 \cup S_\infty\right)$ and whose momentum profile is given by $\phi \left(\tau\right) = f'' \left(s\right)$ and momentum variable by $\tau = f' \left(s\right) \in \clom$ with $s = 2 \ln r = -2 \ln \tilde{r} \in \left(-\infty, \infty\right)$, all as in Subsection \ref{subsec:MomentConstructConeSing}.
\begin{enumerate}
\item If $\phi$ is real analytic in $\tau$ in a \nbd{} of $\tau = 0$ with the boundary conditions $\phi \left(0\right) = 0$ and $\phi' \left(0\right) = \bo$, then $f'' \left(s\right)$ is real analytic in $r^{2 \beta_0}$ and has the following asymptotic power series expansion which is absolutely convergent and locally uniformly convergent in a \nbd{} of the divisor given by $r = 0$ (i.e. the divisor $\so$):
\begin{equation}\label{eq:f2spolyhomo}
f'' \left(s\right) = \sum\limits_{k=1}^{\infty} c_{2k} \left(z\right) r^{2k\beta_0}
\end{equation}
where $c_{2k} \left(z\right)$ is a smooth and bounded real-valued function for all $k \in \N$ and $c_2 \left(z\right)$ is in addition strictly positive and bounded below away from $0$. The same statement and expression hold true for $\tau = m$ and the boundary conditions $\phi \left(m\right) = 0$ and $\phi' \left(m\right) = -\boo$ with $r$ being replaced by $\tilde{r}$, $\beta_0$ by $\beta_\infty$, $\so$ by the divisor $\soo$ (which is given by $\tilde{r} = 0$) and $c_{2k} \left(z\right)$ by some function $\tilde{c}_{2k} \left(z\right)$ having the same properties. \label{itm:polyhomodirect}
\item Conversely if $f'' \left(s\right)$ is real analytic in $r^{2 \beta_0}$ and has a power series expansion of the form (\ref{eq:f2spolyhomo}) in a \nbd{} of the divisor $r = 0$, then $\phi$ is real analytic in $\tau$ in a \nbd{} of $\tau = 0$ with the boundary conditions $\phi \left(0\right) = 0$ and $\phi' \left(0\right) = \bo$. Again the same statement holds true for $\tilde{r}^{2 \beta_\infty}$ and the divisor $\tilde{r} = 0$ with $\tau = 0$ being replaced by $\tau = m$ and $\phi \left(0\right) = 0$ and $\phi' \left(0\right) = \bo$ by the boundary conditions $\phi \left(m\right) = 0$ and $\phi' \left(m\right) = -\boo$ respectively. \label{itm:polyhomoconverse}
\end{enumerate}
\end{theorem}
\begin{remark}
Theorem \ref{thm:polyhomoconesing} is proved in more general settings by Hashimoto \cite{Hashimoto:2019:cscKConeSing}, Rubinstein-Zhang \cite{Rubinstein:2022:KEedgeHirzebruch} though their proofs differ from each other in some ways. We will give a proof of Theorem \ref{thm:polyhomoconesing} which becomes a bit simpler by following the conventions of Sz\'ekelyhidi \cite{Szekelyhidi:2014:eKintro}; Section 4.4 in the momentum construction method (which are precisely what we have followed in Subsection \ref{subsec:MomentConstructConeSing}). The proof also uses some tricks found in Sz\'ekelyhidi \cite{Szekelyhidi:2006:eKKStab}; Section 5.1.
\end{remark}
\begin{proof}
$ $ \vspace{5pt} \newline
(\ref{itm:polyhomodirect}) Let $\phi$ be real analytic in $\tau$ in a \nbd{} of $\tau = 0$ with the boundary conditions $\phi \left(0\right) = 0$ and $\phi' \left(0\right) = \bo > 0$. So $\phi$ has the following power series expansion in $\tau$ which is absolutely convergent and locally uniformly convergent in a \nbd{} of $\tau = 0$:
\begin{equation}\label{eq:phitauseries}
\phi \left(\tau\right) = \bo \tau + \sum\limits_{k=2}^{\infty} a_k \tau^k
\end{equation}
So $\frac{1}{\phi}$ has the following expression where $\sum\limits_{k=0}^{\infty} b_k \tau^k$ is also convergent in a \nbd{} of $\tau = 0$:
\begin{equation}\label{eq:reciprocphitauseries}
\frac{1}{\phi \left(\tau\right)} = \frac{1}{\bo \tau} + \sum\limits_{k=0}^{\infty} b_k \tau^k
\end{equation}
From the relations given in (\ref{eq:variablechange}) we get:
\begin{equation}\label{eq:rsphitauseries}
2 \ln r = s = \int d s = \int \frac{d \tau}{\phi \left(\tau\right)} = \frac{1}{\bo} \ln \tau + \sum\limits_{k=1}^{\infty} \frac{b_{k-1}}{k} \tau^k + \operatorname{const}
\end{equation}
So $r^{2\bo}$ is real analytic in $\tau$ near $\tau = 0$ and has the following convergent power series expansion where $C > 0$:
\begin{equation}\label{eq:r2botau}
r^{2\bo} = C \tau e^{\left( \sum\limits_{k=1}^{\infty} \frac{\bo b_{k-1}}{k} \tau^k \right)} = C \tau + \sum\limits_{k=2}^{\infty} B_k \tau^k
\end{equation}
From (\ref{eq:r2botau}) we see that $\frac{d}{d \tau} \left(r^{2\bo}\right) \Bigr\rvert_{\tau = 0} = C \neq 0$, and so we can invert the power series in (\ref{eq:r2botau}) to obtain $\tau$ as being real analytic in $r^{2\bo}$ near $r = 0$ and having the following convergent power series expansion:
\begin{equation}\label{eq:taur2bo}
\tau = \frac{r^{2\bo}}{C} + \sum\limits_{k=2}^{\infty} A_k r^{2k\bo}
\end{equation}
As $\phi \left(\tau\right) = f'' \left(s\right)$, so from (\ref{eq:phitauseries}) and (\ref{eq:taur2bo}) we get $f'' \left(s\right)$ is real analytic in $r^{2 \beta_0}$ and has the convergent power series expansion (\ref{eq:f2spolyhomo}) in a \nbd{} of $r = 0$ where $c_2 = \frac{\bo}{C} > 0$. Also as we had kept the coordinate $z$ on the base \rms{} $\Sigma$ fixed throughout this calculation and had carried out the analysis in the variable $r = \left\lvert w \right\rvert$, the coefficients $c_{2k}$ in the power series expansion (\ref{eq:f2spolyhomo}) must be functions of $z$ alone. And since all the coordinate expressions in the momentum construction given in Subsection \ref{subsec:MomentConstructConeSing} are known to transform ``correctly'' under a \hol{} change of coordinates on the compact \rms{} $\Sigma$, the coefficient functions $c_{2k} \left(z\right)$ will be smooth and bounded and $c_2 \left(z\right)$ will be strictly positive and bounded below away from $0$. \vspace{5pt} \\
(\ref{itm:polyhomoconverse}) Let $f'' \left(s\right)$ be real analytic in $r^{2 \beta_0}$ and have a power series expansion of the form (\ref{eq:f2spolyhomo}) in a \nbd{} of the divisor $r = 0$. Integrating (\ref{eq:f2spolyhomo}) and using the fact $\lim\limits_{s \to -\infty} f' \left(s\right) = 0$ from (\ref{eq:wgammasosoo}) we obtain the following power series expansion for $f' \left(s\right)$ in a \nbd{} of $r = 0$:
\begin{equation}\label{eq:f1spolyhomo}
f' \left(s\right) = \int f'' \left(s\right) d s = \int \left( \sum\limits_{k=1}^{\infty} c_{2k} \left(z\right) r^{2k\beta_0} \right) \frac{2}{r} d r = \sum\limits_{k=1}^{\infty} \frac{c_{2k} \left(z\right)}{k\beta_0} r^{2k\beta_0}
\end{equation}
As $f' \left(s\right) = \tau$, the expression (\ref{eq:f1spolyhomo}) gives $\tau$ to be real analytic in $r^{2\bo}$ near $r = 0$. Since $\frac{d}{d \left(r^{2\bo}\right)} \left(\tau\right) \Bigr\rvert_{r = 0} = \frac{c_2}{\beta_0} \neq 0$, we can invert the power series in (\ref{eq:f1spolyhomo}) to obtain the following power series expansion for $r^{2\bo}$ in terms of $\tau$ in a \nbd{} of $\tau = 0$:
\begin{equation}\label{eq:r2botauconvs}
r^{2\bo} = \frac{\beta_0}{c_2} \tau + \sum\limits_{k=2}^{\infty} B_k \tau^k
\end{equation}
As $f'' \left(s\right) = \phi \left(\tau\right)$, and as from the expression (\ref{eq:r2botauconvs}), $r^{2\bo}$ is real analytic in $\tau$ near $\tau = 0$, we can conclude using (\ref{eq:f2spolyhomo}) and (\ref{eq:r2botauconvs}) that $\phi$ is real analytic in $\tau$ in a \nbd{} of $\tau = 0$ and has the following power series expansion in terms of $\tau$:
\begin{equation}\label{eq:phitauseriesconvs}
\phi \left(\tau\right) = c_2 r^{2\beta_0} + \sum\limits_{k=2}^{\infty} c_{2k} r^{2k\beta_0} = \bo \tau + \sum\limits_{k=2}^{\infty} a_k \tau^k
\end{equation}
Clearly from the power series expression (\ref{eq:phitauseriesconvs}), we have $\phi \left(0\right) = 0$ and $\phi' \left(0\right) = \bo$.
\end{proof} \par
\begin{corollary}\label{cor:polyhomoconesing}
Let $\omega$ be a momentum-constructed \krm{} on the minimal ruled surface $X$ having momentum profile $\phi : \clom \to \R$ satisfying the boundary conditions $\phi \left(0\right) = 0$, $\phi' \left(0\right) = \bo > 0$ and $\phi \left(m\right) = 0$, $\phi' \left(m\right) = -\boo < 0$ and the condition $\phi > 0$ on $\opom$. Then $\omega$ is a \plyhomo{} smooth conical \krm{} (i.e. satisfies Definition \ref{def:coneKr4}) with cone angles $2\pi\bo$ and $2\pi\boo$ along the divisors $\so$ and $\soo$ respectively if and only if $\phi \in \mathcal{C}^\omega \clom$.
\end{corollary}
\begin{proof}
As shown in Theorem \ref{thm:polyhomoconesing} the real analyticity of the momentum profile $\phi \left(\tau\right)$ along with the boundary conditions $\phi \left(0\right) = 0$, $\phi' \left(0\right) = \bo$ and $\phi \left(m\right) = 0$, $\phi' \left(m\right) = -\boo$ is equivalent to $f'' \left(s\right)$ having convergent asymptotic power series expansions of the form (\ref{eq:f2spolyhomo}) in terms of $r^{2\beta_0}$ and $\tilde{r}^{2\beta_\infty}$ respectively near both the divisors $\so$ and $\soo$. As done in the computation in (\ref{eq:f1spolyhomo}), by integrating the power series expression of $f'' \left(s\right)$ given by (\ref{eq:f2spolyhomo}) once and then twice, we obtain the following asymptotic power series expansions for $f' \left(s\right)$ and $f \left(s\right)$ in terms of $r^{2\beta_0}$ which are convergent locally around $r = 0$:
\begin{equation}\label{eq:f1sfspolyhomo}
f' \left(s\right) = \sum\limits_{k=1}^{\infty} \frac{c_{2k} \left(z\right)}{k\beta_0} r^{2k\beta_0} \hspace{1pt}, \hspace{5pt} f \left(s\right) = \sum\limits_{k=1}^{\infty} \frac{c_{2k} \left(z\right)}{k^2\beta_0^2} r^{2k\beta_0}
\end{equation}
Similar expressions can be derived for $f' \left(s\right)$ and $f \left(s\right)$ in terms of $\tilde{r}^{2\beta_\infty}$ locally around $\tilde{r} = 0$, by considering the expression for $f'' \left(s\right)$ near $\tilde{r} = 0$ analogous to (\ref{eq:f2spolyhomo}) but having some other coefficients $\tilde{c}_{2k} \left(z\right)$ in place of $c_{2k} \left(z\right)$, and by noting $\lim\limits_{s \to \infty} f' \left(s\right) = m > 0$ from (\ref{eq:wgammasosoo}).
\begin{equation}\label{eq:f2sf1sfspolyhomotilde}
\begin{gathered}
f'' \left(s\right) = \sum\limits_{k=1}^{\infty} \tilde{c}_{2k} \left(z\right) \tilde{r}^{2k\beta_\infty} \\
f' \left(s\right) = m - \sum\limits_{k=1}^{\infty} \frac{\tilde{c}_{2k} \left(z\right)}{k\beta_\infty} \tilde{r}^{2k\beta_\infty} \hspace{1pt}, \hspace{5pt} f \left(s\right) = - 2 m \ln \tilde{r} + \sum\limits_{k=1}^{\infty} \frac{\tilde{c}_{2k} \left(z\right)}{k^2\beta_\infty^2} \tilde{r}^{2k\beta_\infty}
\end{gathered}
\end{equation}
So the coefficient functions in the coordinate expression (\ref{eq:omega1conesing}) for the \krm{} $\omega$ in terms of the bundle-adapted local \hol{} coordinates $\left(z,w\right)$ have the following asymptotic power series expansions which are absolutely convergent and locally uniformly convergent in a tubular \nbd{} of $\so$ (as can be seen from (\ref{eq:f2spolyhomo}) and (\ref{eq:f1sfspolyhomo})):
\begin{equation}\label{eq:omegapolyhomo}
\begin{gathered}
1 + f' \left(s\right) = 1 + \sum\limits_{k=1}^{\infty} \frac{c_{2k} \left(z\right)}{k\beta_0} \left\lvert w \right\rvert^{2k\beta_0} \\
\frac{f'' \left(s\right)}{\left\lvert w \right\rvert} = \sum\limits_{k=1}^{\infty} c_{2k} \left(z\right) \left\lvert w \right\rvert^{2k\beta_0 - 1} \hspace{1pt}, \hspace{5pt} \frac{f'' \left(s\right)}{\left\lvert w \right\rvert^2} = \sum\limits_{k=1}^{\infty} c_{2k} \left(z\right) \left\lvert w \right\rvert^{2k\beta_0 - 2}
\end{gathered}
\end{equation}
Again, expressions similar to (\ref{eq:omegapolyhomo}) can be derived from (\ref{eq:f2sf1sfspolyhomotilde}) for the coefficient functions in the coordinate expression (\ref{eq:omega1conesing}) for $\omega$ in a tubular \nbd{} of $\soo$. Comparing the expressions in (\ref{eq:omegapolyhomo}) and the conditions mentioned in Theorem \ref{thm:polyhomoconesing} (\ref{itm:polyhomodirect}) with the conditions given in Definition \ref{def:coneKr4}, we see that the momentum-constructed \krm{} $\omega$ on $X$ is a \plyhomo{} smooth conical \krm{} being smooth on $X \smallsetminus \left(S_0 \cup S_\infty\right)$ and having conical singularities with cone angles $2 \pi \bo$ and $2 \pi \boo$ along $\so$ and $\soo$ respectively, in the event that the momentum profile $\phi \left(\tau\right)$ is real analytic.
\end{proof} \par
Combining Corollaries \ref{cor:mainconesing}, \ref{cor:InfinitDiff} and \ref{cor:polyhomoconesing}, we obtain the following result:
\begin{corollary}\label{cor:polyhomoconehcscK}
Let $\omega$ be the momentum-constructed \krm{} on the minimal ruled surface $X$ whose momentum profile $\phi \left(\gamma\right)$ satisfies the ODE (\ref{eq:ODEConeSing}) with the boundary conditions (\ref{eq:BVPConeSing}) and the additional condition (\ref{eq:positiveConeSing}). Then $\omega$ is a \plyhomo{} smooth conical higher cscK metric (i.e. satisfies Definitions \ref{def:coneKr4} and \ref{def:conehcscK} together) with cone angles $2\pi\bo$ and $2\pi\boo$ along the divisors $\so$ and $\soo$ respectively.
\end{corollary} \par
Even though the momentum profile is real analytic in most practical applications of the momentum construction method (see \cite{Hashimoto:2019:cscKConeSing,Hwang:2002:MomentConstruct,Szekelyhidi:2014:eKintro,Szekelyhidi:2006:eKKStab,Tonnesen:1998:eKminruledsurf}), we can nevertheless get the following results with the hypothesis being only smoothness of the momentum profile on the whole momentum interval (including both the endpoints) which gives conclusions slightly weaker than those provided by real analyticity in Theorem \ref{thm:polyhomoconesing} and Corollary \ref{cor:polyhomoconesing} but still sufficient for all our purposes in Sections \ref{sec:highScalcurrent} and \ref{sec:logFutlogMab}. \par
\begin{theorem}[Asymptotic Taylor Approximations for the Momentum Profile at the Two Divisors; Hashimoto \cite{Hashimoto:2019:cscKConeSing}; Lemma 3.6, Hwang \cite{Hwang:1994:cscK}; Lemmas 2.2 and 2.5, Proposition 2.1, Hwang-Singer \cite{Hwang:2002:MomentConstruct}; Section 2.2, Proposition 2.3, Li \cite{Li:2012:eKEngyFunctProjBund}; Lemma 2.3, Rubinstein-Zhang \cite{Rubinstein:2022:KEedgeHirzebruch}; Proposition 3.3]\label{thm:semihomoconesing}
Let $\omega$ be the \krm{} given by the Calabi ansatz (\ref{eq:ansatzconesing}) on the minimal ruled surface $X = \prj$, which is smooth on the non-compact surface $X \smallsetminus \left(S_0 \cup S_\infty\right)$ and whose momentum profile is given by $\phi \left(\tau\right) = f'' \left(s\right)$ and momentum variable by $\tau = f' \left(s\right) \in \clom$ with $s = 2 \ln r = -2 \ln \tilde{r} \in \left(-\infty, \infty\right)$, all as in Subsection \ref{subsec:MomentConstructConeSing}.
\begin{enumerate}
\item If $\phi$ is smooth in $\tau$ in a \nbd{} of $\tau = 0$ with the boundary conditions $\phi \left(0\right) = 0$ and $\phi' \left(0\right) = \bo$, then $f'' \left(s\right)$ is smooth in $r^{2 \beta_0}$ and has the following asymptotic second order Taylor approximation in a \nbd{} of the divisor given by $r = 0$ (i.e. the divisor $\so$):
\begin{equation}\label{eq:f2ssemihomo}
f'' \left(s\right) = c_2 \left(z\right) r^{2\beta_0} + c_4 \left(z\right) r^{4\beta_0} + c \left(z, r^{2\beta_0}\right) r^{4\beta_0}
\end{equation}
where $c_2 \left(z\right)$ and $c_4 \left(z\right)$ are smooth and bounded real-valued functions and $c_2 \left(z\right)$ is in addition strictly positive and bounded below away from $0$, and $c \left(z, r^{2\beta_0}\right)$ is a real-valued function which is smooth and bounded in $z$ and smooth in $r^{2\beta_0}$ with $\lim\limits_{r \to 0} c \left(z, r^{2\beta_0}\right) = 0$. The same statement and expression hold true for $\tau = m$ and the boundary conditions $\phi \left(m\right) = 0$ and $\phi' \left(m\right) = -\boo$ with $r$ being replaced by $\tilde{r}$, $\beta_0$ by $\beta_\infty$, $\so$ by the divisor $\soo$ (which is given by $\tilde{r} = 0$), $c_2 \left(z\right)$ and $c_4 \left(z\right)$ by some functions $\tilde{c}_2 \left(z\right)$ and $\tilde{c}_4 \left(z\right)$ and $c \left(z, r^{2\beta_0}\right)$ by some function $\tilde{c} \left(z, \tilde{r}^{2\beta_\infty}\right)$ all three having the same respective properties. \label{itm:semihomodirect}
\item Conversely if $f'' \left(s\right)$ is smooth in $r^{2 \beta_0}$ and has a second order Taylor approximation of the form (\ref{eq:f2ssemihomo}) in a \nbd{} of the divisor $r = 0$, then $\phi$ is smooth in $\tau$ in a \nbd{} of $\tau = 0$ with the boundary conditions $\phi \left(0\right) = 0$ and $\phi' \left(0\right) = \bo$. Again the same statement holds true for $\tilde{r}^{2 \beta_\infty}$ and the divisor $\tilde{r} = 0$ with $\tau = 0$ being replaced by $\tau = m$ and $\phi \left(0\right) = 0$ and $\phi' \left(0\right) = \bo$ by the boundary conditions $\phi \left(m\right) = 0$ and $\phi' \left(m\right) = -\boo$ respectively. \label{itm:semihomoconverse}
\end{enumerate}
\end{theorem}
\begin{proof}
The proof of Theorem \ref{thm:semihomoconesing} goes along the exact lines as that of Theorem \ref{thm:polyhomoconesing} with the hypothesis of real analyticity over there being replaced by the hypothesis of smoothness over here and so all the convergent power series expansions over there will be replaced by Taylor approximations of suitable orders over here. \vspace{5pt} \\
(\ref{itm:semihomodirect}) Let $\phi$ be smooth in $\tau$ in a \nbd{} of $\tau = 0$ with the boundary conditions $\phi \left(0\right) = 0$ and $\phi' \left(0\right) = \bo > 0$. So $\phi$ has the following second order Taylor approximation in $\tau$ in a \nbd{} of $\tau = 0$ (where $\lim\limits_{\tau \to 0} a \left(\tau\right) = 0$):
\begin{equation}\label{eq:phitauTaylor}
\phi \left(\tau\right) = \bo \tau + a_2 \tau^2 + a \left(\tau\right) \tau^2
\end{equation}
So $\frac{1}{\phi}$ has the following expression in a deleted \nbd{} of $\tau = 0$ (where also $\lim\limits_{\tau \to 0} b \left(\tau\right) = 0$):
\begin{equation}\label{eq:reciprocphitauTaylor}
\frac{1}{\phi \left(\tau\right)} = \frac{1}{\bo \tau} + b_0 + b \left(\tau\right)
\end{equation}
By doing the same calculations here as those done in (\ref{eq:rsphitauseries}) we get:
\begin{equation}\label{eq:rsphitauTaylor}
2 \ln r = s = \int d s = \int \frac{d \tau}{\phi \left(\tau\right)} = \frac{1}{\bo} \ln \tau + b_0 \tau + \int b \left(\tau\right) d \tau = \frac{1}{\bo} \ln \tau + b_0 \tau + b_1 \left(\tau\right) \tau + \operatorname{const}
\end{equation}
where $\lim\limits_{\tau \to 0} b_1 \left(\tau\right) = \lim\limits_{\tau \to 0} \frac{\int b \left(\tau\right) d \tau}{\tau} = 0$. So $r^{2\bo}$ is smooth in $\tau$ near $\tau = 0$ and has the following Taylor approximation where $C > 0$:
\begin{equation}\label{eq:r2botau'}
r^{2\bo} = C \tau e^{\left( \bo b_0 \tau + \bo b_1 \left(\tau\right) \tau \right)} = C \tau + B_2 \tau^2 + B \left(\tau\right) \tau^2
\end{equation}
From (\ref{eq:r2botau'}) we see that $\frac{d}{d \tau} \left(r^{2\bo}\right) \Bigr\rvert_{\tau = 0} = C \neq 0$, and so we can invert the expression in (\ref{eq:r2botau'}) to obtain $\tau$ as being smooth in $r^{2\bo}$ near $r = 0$ and having the following Taylor approximation:
\begin{equation}\label{eq:taur2bo'}
\tau = \frac{r^{2\bo}}{C} + A_2 r^{4\bo} + A \left(r^{2\bo}\right) r^{4\bo}
\end{equation}
As $\phi \left(\tau\right) = f'' \left(s\right)$, so from (\ref{eq:phitauTaylor}) and (\ref{eq:taur2bo'}) we get $f'' \left(s\right)$ is smooth in $r^{2 \beta_0}$ and has the second order Taylor approximation (\ref{eq:f2ssemihomo}) in a \nbd{} of $r = 0$ where $c_2 = \frac{\bo}{C} > 0$ and $\lim\limits_{r \to 0} c \left(r^{2\beta_0}\right) = 0$. The properties to be proven in Theorem \ref{thm:semihomoconesing} (\ref{itm:semihomodirect}) for the functions $c_2 \left(z\right)$, $c_4 \left(z\right)$ and $c \left(z, r^{2 \beta_0}\right)$ appearing in the expression (\ref{eq:f2ssemihomo}) can be easily seen by the very same arguments as those made in the proof of Theorem \ref{thm:polyhomoconesing} (\ref{itm:polyhomodirect}) for the coefficient functions $c_{2k} \left(z\right)$ in the power series expansion (\ref{eq:f2spolyhomo}). \vspace{5pt} \\
(\ref{itm:semihomoconverse}) Let $f'' \left(s\right)$ be smooth in $r^{2 \beta_0}$ and have a second order Taylor approximation of the form (\ref{eq:f2ssemihomo}) in a \nbd{} of the divisor $r = 0$. Integrating (\ref{eq:f2ssemihomo}) over here just as done with (\ref{eq:f2spolyhomo}) in (\ref{eq:f1spolyhomo}), we obtain the following Taylor approximation for $f' \left(s\right)$ in a \nbd{} of $r = 0$:
\begingroup
\addtolength{\jot}{0.5em}
\begin{align}\label{eq:f1ssemihomo}
f' \left(s\right) &= \int f'' \left(s\right) d s = \int \left( c_2 \left(z\right) r^{2\beta_0} + c_4 \left(z\right) r^{4\beta_0} + c \left(z, r^{2\beta_0}\right) r^{4\beta_0} \right) \frac{2}{r} d r \nonumber \\
&= \frac{c_2 \left(z\right)}{\beta_0} r^{2\beta_0} + \frac{c_4 \left(z\right)}{2\beta_0} r^{4\beta_0} + \int 2 c \left(z, r^{2\beta_0}\right) r^{4\beta_0 - 1} d r = \frac{c_2 \left(z\right)}{\beta_0} r^{2\beta_0} + \frac{c_4 \left(z\right)}{2\beta_0} r^{4\beta_0} + c_0 \left(z, r^{2\beta_0}\right) r^{4\beta_0}
\end{align}
\endgroup
where $\lim\limits_{r \to 0} c_0 \left(z, r^{2\beta_0}\right) = \lim\limits_{r \to 0} \frac{\int 2 c \left(z, r^{2\beta_0}\right) r^{4\beta_0 - 1} d r}{r^{4\beta_0}} = \lim\limits_{r \to 0} \frac{\frac{r}{2\beta_0 r^{2\beta_0}} \frac{2 c \left(z, r^{2\beta_0}\right) r^{4\beta_0}}{r}}{2 r^{2\beta_0}} = 0$. As $f' \left(s\right) = \tau$, the expression (\ref{eq:f1ssemihomo}) gives $\tau$ to be smooth in $r^{2\bo}$ near $r = 0$. Since $\frac{d}{d \left(r^{2\bo}\right)} \left(\tau\right) \Bigr\rvert_{r = 0} = \frac{c_2}{\beta_0} \neq 0$, we can invert the expression in (\ref{eq:f1ssemihomo}) to obtain the following Taylor approximation for $r^{2\bo}$ in terms of $\tau$ in a \nbd{} of $\tau = 0$:
\begin{equation}\label{eq:r2botau'convs}
r^{2\bo} = \frac{\beta_0}{c_2} \tau + B_2 \tau^2 + B \left(\tau\right) \tau^2
\end{equation}
As $f'' \left(s\right) = \phi \left(\tau\right)$, and as from the expression (\ref{eq:r2botau'convs}), $r^{2\bo}$ is smooth in $\tau$ near $\tau = 0$, we can conclude using (\ref{eq:f2ssemihomo}) and (\ref{eq:r2botau'convs}) that $\phi$ is smooth in $\tau$ in a \nbd{} of $\tau = 0$ and has the following second order Taylor approximation in terms of $\tau$ (where $\lim\limits_{\tau \to 0} a \left(\tau\right) = 0$):
\begin{equation}\label{eq:phitauTaylorconvs}
\phi \left(\tau\right) = c_2 r^{2\beta_0} + c_4 r^{4\beta_0} + c \left(r^{2\beta_0}\right) r^{4\beta_0} = \bo \tau + a_2 \tau^2 + a \left(\tau\right) \tau^2
\end{equation}
Clearly from the expression (\ref{eq:phitauTaylorconvs}), we have $\phi \left(0\right) = 0$ and $\phi' \left(0\right) = \bo$.
\end{proof} \par
\begin{corollary}\label{cor:semihomoconesing}
Let $\omega$ be a momentum-constructed \krm{} on the minimal ruled surface $X$ having momentum profile $\phi : \clom \to \R$ satisfying the boundary conditions $\phi \left(0\right) = 0$, $\phi' \left(0\right) = \bo > 0$ and $\phi \left(m\right) = 0$, $\phi' \left(m\right) = -\boo < 0$ and the condition $\phi > 0$ on $\opom$. Then $\omega$ is a \conrml{} smooth conical \krm{} (i.e. satisfies Definition \ref{def:coneKr3}) with cone angles $2\pi\bo$ and $2\pi\boo$ along the divisors $\so$ and $\soo$ respectively if and only if $\phi \in \mathcal{C}^\infty \clom$.
\end{corollary}
\begin{proof}
Corollary \ref{cor:semihomoconesing} follows from Theorem \ref{thm:semihomoconesing} in the same way as Corollary \ref{cor:polyhomoconesing} from Theorem \ref{thm:polyhomoconesing}. As shown in Theorem \ref{thm:semihomoconesing} the smoothness of the momentum profile $\phi \left(\tau\right)$ along with the boundary conditions $\phi \left(0\right) = 0$, $\phi' \left(0\right) = \bo$ and $\phi \left(m\right) = 0$, $\phi' \left(m\right) = -\boo$ is equivalent to $f'' \left(s\right)$ having asymptotic second order Taylor approximations of the form (\ref{eq:f2ssemihomo}) in terms of $r^{2\beta_0}$ and $\tilde{r}^{2\beta_\infty}$ respectively near both the divisors $\so$ and $\soo$. As done in the computation in (\ref{eq:f1ssemihomo}) and analogous to the expressions obtained in (\ref{eq:f1sfspolyhomo}), we obtain the following asymptotic second order Taylor approximations for $f' \left(s\right)$ and $f \left(s\right)$ in terms of $r^{2\beta_0}$ locally around $r = 0$:
\begin{equation}\label{eq:f1sfssemihomo}
\begin{gathered}
f' \left(s\right) = \frac{c_2 \left(z\right)}{\beta_0} r^{2\beta_0} + \frac{c_4 \left(z\right)}{2\beta_0} r^{4\beta_0} + c_0 \left(z, r^{2\beta_0}\right) r^{4\beta_0} \\
f \left(s\right) = \frac{c_2 \left(z\right)}{\beta_0^2} r^{2\beta_0} + \frac{c_4 \left(z\right)}{4\beta_0^2} r^{4\beta_0} + c_1 \left(z, r^{2\beta_0}\right) r^{4\beta_0}
\end{gathered}
\end{equation}
where $\lim\limits_{r \to 0} c_0 \left(z, r^{2\beta_0}\right) = 0 = \lim\limits_{r \to 0} c_1 \left(z, r^{2\beta_0}\right)$. Similar expressions can be derived for $f' \left(s\right)$ and $f \left(s\right)$ in terms of $\tilde{r}^{2\beta_\infty}$ locally around $\tilde{r} = 0$ again analogous to the expressions obtained in (\ref{eq:f2sf1sfspolyhomotilde}).
\begin{equation}\label{eq:f2sf1sfssemihomotilde}
\begin{gathered}
f'' \left(s\right) = \tilde{c}_2 \left(z\right) \tilde{r}^{2\beta_\infty} + \tilde{c}_4 \left(z\right) \tilde{r}^{4\beta_\infty} + \tilde{c} \left(z, \tilde{r}^{2\beta_\infty}\right) \tilde{r}^{4\beta_\infty} \\
f' \left(s\right) = m - \frac{\tilde{c}_2 \left(z\right)}{\beta_\infty} \tilde{r}^{2\beta_\infty} - \frac{\tilde{c}_4 \left(z\right)}{2\beta_\infty} \tilde{r}^{4\beta_\infty} - \tilde{c}_0 \left(z, \tilde{r}^{2\beta_\infty}\right) \tilde{r}^{4\beta_\infty} \\
f \left(s\right) = - 2 m \ln \tilde{r} + \frac{\tilde{c}_2 \left(z\right)}{\beta_\infty^2} \tilde{r}^{2\beta_\infty} + \frac{\tilde{c}_4 \left(z\right)}{4\beta_\infty^2} \tilde{r}^{4\beta_\infty} + \tilde{c}_1 \left(z, \tilde{r}^{2\beta_\infty}\right) \tilde{r}^{4\beta_\infty}
\end{gathered}
\end{equation}
So the coefficient functions in the coordinate expression (\ref{eq:omega1conesing}) for the \krm{} $\omega$ in terms of the bundle-adapted local \hol{} coordinates $\left(z,w\right)$ have the following asymptotic second order Taylor approximations in a tubular \nbd{} of $\so$ (as can be seen from (\ref{eq:f2ssemihomo}) and (\ref{eq:f1sfssemihomo})):
\begin{equation}\label{eq:omegasemihomo}
\begin{gathered}
1 + f' \left(s\right) = 1 + \frac{c_2 \left(z\right)}{\beta_0} \left\lvert w \right\rvert^{2\beta_0} + \frac{c_4 \left(z\right)}{2\beta_0} \left\lvert w \right\rvert^{4\beta_0} + c_0 \left(z, \left\lvert w \right\rvert^{2\beta_0}\right) \left\lvert w \right\rvert^{4\beta_0} \in O \left(1\right) \hspace{1.5pt} \text{as} \hspace{3.5pt} w \to 0 \\
\frac{f'' \left(s\right)}{\left\lvert w \right\rvert} = c_2 \left(z\right) \left\lvert w \right\rvert^{2\beta_0 - 1} + c_4 \left(z\right) \left\lvert w \right\rvert^{4\beta_0 - 1} + c \left(z, \left\lvert w \right\rvert^{2\beta_0}\right) \left\lvert w \right\rvert^{4\beta_0 - 1} \in O \left(\left\lvert w \right\rvert^{2 \beta_0 - 1}\right) \hspace{1.5pt} \text{as} \hspace{3.5pt} w \to 0 \\
\frac{f'' \left(s\right)}{\left\lvert w \right\rvert^2} = c_2 \left(z\right) \left\lvert w \right\rvert^{2\beta_0 - 2} + c_4 \left(z\right) \left\lvert w \right\rvert^{4\beta_0 - 2} + c \left(z, \left\lvert w \right\rvert^{2\beta_0}\right) \left\lvert w \right\rvert^{4\beta_0 - 2} \in O \left(\left\lvert w \right\rvert^{2 \beta_0 - 2}\right) \hspace{1.5pt} \text{as} \hspace{3.5pt} w \to 0
\end{gathered}
\end{equation}
Again, expressions similar to (\ref{eq:omegasemihomo}) can be derived from (\ref{eq:f2sf1sfssemihomotilde}) for the coefficient functions in the coordinate expression (\ref{eq:omega1conesing}) for $\omega$ in a tubular \nbd{} of $\soo$. Comparing the expressions in (\ref{eq:omegasemihomo}) and the conditions mentioned in Theorem \ref{thm:semihomoconesing} (\ref{itm:semihomodirect}) with the conditions given in Definition \ref{def:coneKr3}, we see that the momentum-constructed \krm{} $\omega$ on $X$ is a \conrml{} smooth conical \krm{} being smooth on $X \smallsetminus \left(S_0 \cup S_\infty\right)$ and having conical singularities with cone angles $2 \pi \bo$ and $2 \pi \boo$ along $\so$ and $\soo$ respectively, in the event that the momentum profile $\phi \left(\tau\right)$ is smooth.
\end{proof}
\numberwithin{equation}{subsection}
\numberwithin{figure}{subsection}
\numberwithin{table}{subsection}
\numberwithin{lemma}{subsection}
\numberwithin{proposition}{subsection}
\numberwithin{result}{subsection}
\numberwithin{theorem}{subsection}
\numberwithin{corollary}{subsection}
\numberwithin{conjecture}{subsection}
\numberwithin{remark}{subsection}
\numberwithin{note}{subsection}
\numberwithin{motivation}{subsection}
\numberwithin{question}{subsection}
\numberwithin{answer}{subsection}
\numberwithin{case}{subsection}
\numberwithin{claim}{subsection}
\numberwithin{definition}{subsection}
\numberwithin{example}{subsection}
\numberwithin{hypothesis}{subsection}
\numberwithin{statement}{subsection}
\numberwithin{ansatz}{subsection}
\section{Interpreting the Higher Scalar Curvature Globally as a Current}\label{sec:highScalcurrent}
\subsection{Computing the Curvature Form Matrix for Deriving the Expression for the Top Chern Current}\label{subsec:Curvcurrent}
The Calabi ansatz (\ref{eq:ansatzconesing}) (or even (\ref{eq:ansatz0})) defines the desired kind of metric $\omega$ (e.g. higher cscK or higher \ext{} \kr{}) only on $X \smallsetminus \left(S_0 \cup S_\infty\right)$, and the boundary conditions (\ref{eq:BVPConeSing}) (or those given in (\ref{eq:ODEBVP0}) respectively) on the momentum profile $\phi$ determine the behaviour of $\omega$ near the divisors $\so$ and $\soo$ (as in whether it has got conical singularities along $\so$ and $\soo$ or whether it extends smoothly across $\so$ and $\soo$). So the coordinate expressions for the curvature form matrix $\Theta \left(\omega\right)$, the top Chern form $c_2 \left(\omega\right)$ and the higher scalar curvature $\lambda \left(\omega\right)$, which we had derived in Subsection \ref{subsec:MomentConstructConeSing}, hold a priori only on $X \smallsetminus \left(S_0 \cup S_\infty\right)$ (as $\omega$ is smooth on $X \smallsetminus \left(S_0 \cup S_\infty\right)$ in both the conical as well as the smooth cases of the momentum construction method). In Subsection \ref{subsec:Curvcurrent} we compute the expressions for $\Theta \left(\omega\right)$ and $c_2 \left(\omega\right)$ near the zero and infinity divisors of the minimal ruled surface $X$ in the bundle-adapted local \hol{} coordinates $\left(z, w\right)$ and $\left(z, \tilde{w}\right)$ respectively where $\tilde{w} = w^{-1}$. \par
In a \nbd{} of the zero divisor $\so$ which is given by $w = 0$, the curvature form matrix $\Theta \left(\omega\right) = \bar{\partial} \left( H^{-1} \partial H \right) \left(\omega\right)$, where $H \left(\omega\right)$ is the Hermitian matrix of $\omega$ in the coordinates $\left(z, w\right)$, is computed as follows (see equation (\ref{eq:Curv1conesing}) derived in Subsection \ref{subsec:MomentConstructConeSing} from the coordinate expressions (\ref{eq:omegaconesing}) and (\ref{eq:omega2conesing})):
\begin{equation}\label{eq:Curvcurrent}
\sqrt{-1} \Theta \left(\omega\right) =
\begin{bmatrix}
-\sqrt{-1} \partial \bar{\partial} \ln \left(1 + f' \left(s\right)\right) - 2 \ttp^* \omega_\Sigma & 0 \\
0 & -\sqrt{-1} \partial \bar{\partial} \ln \left(\frac{f'' \left(s\right)}{\left\lvert w \right\rvert^2}\right)
\end{bmatrix}
\end{equation}
Assuming the momentum-constructed $\omega$ to be \plyhomo{} smooth, the $\left(2,2\right)$-entry in (\ref{eq:Curvcurrent}) is computed as follows, by looking at the expressions (\ref{eq:omegapolyhomo}) obtained in Section \ref{sec:polyhomoConeSing} and by applying the Poincar\'e-Lelong formula (Demailly \cite{Demailly:2012:CmplxDifferGeom}):
\begin{align}\label{eq:Curvcurrent22}
-\sqrt{-1} \partial \bar{\partial} \ln \left(\frac{f'' \left(s\right)}{\left\lvert w \right\rvert^2}\right) &= -\sqrt{-1} \partial \bar{\partial} \ln \left(\sum\limits_{k=1}^{\infty} c_{2k} \left(z\right) \left\lvert w \right\rvert^{2\left(k-1\right)\beta_0}\right) - \sqrt{-1} \partial \bar{\partial} \ln \left\lvert w \right\rvert^{2 \left(\beta_0 - 1\right)} \nonumber \\
&= -\sqrt{-1} \partial \bar{\partial} \ln \left(g \left(s\right)\right) + 2 \pi \left(1 - \beta_0\right) \left[w = 0\right]
\end{align}
where $g \left(s\right) = \sum\limits_{k=1}^{\infty} c_{2k} \left(z\right) \left\lvert w \right\rvert^{2\left(k-1\right)\beta_0}$ is a smooth function of $s \in \R$ and $\left[w = 0\right]$ denotes the current of integration on the surface $X$ along the hyperplane divisor given by $w = 0$. And if $\omega$ is given to be only \conrml{} smooth, then the only thing that changes in the computation in (\ref{eq:Curvcurrent22}) is $g \left(s\right) = c_2 \left(z\right) + c_4 \left(z\right) \left\lvert w \right\rvert^{2\beta_0} + c \left(z, \left\lvert w \right\rvert^{2\beta_0}\right) \left\lvert w \right\rvert^{2\beta_0}$ obtained by considering the expressions (\ref{eq:omegasemihomo}) instead of (\ref{eq:omegapolyhomo}). And the $\left(1,1\right)$-entry in (\ref{eq:Curvcurrent}) is given in terms of the momentum profile $\phi \left(\gamma\right)$ as follows (see equation (\ref{eq:Curv2conesing})):
\begin{equation}\label{eq:Curvcurrent11}
-\sqrt{-1} \partial \bar{\partial} \ln \left(1 + f' \left(s\right)\right) - 2 \ttp^* \omega_\Sigma \hspace{1pt} = \hspace{1pt} \frac{1}{\gamma} \left( \frac{\phi}{\gamma} - \phi' \right) \sqrt{-1} \frac{\phi}{\left\lvert w \right\rvert^2} d w \wedge d \bar{w} - \left(\frac{\phi}{\gamma} + 2\right) \ttp^* \omega_\Sigma
\end{equation}
Note that here $\sqrt{-1} \frac{\phi}{\left\lvert w \right\rvert^2} d w \wedge d \bar{w}$ is a conical \krm{} and hence a \krc{} as $\phi \left(\gamma\right) = f'' \left(s\right)$ (see the expressions (\ref{eq:omegapolyhomo}) and (\ref{eq:omegasemihomo})). In the coordinates $\left(z, \tilde{w}\right)$ where $\tilde{w} = w^{-1}$, taken in a \nbd{} of the infinity divisor $\soo$ which is given by $\tilde{w} = 0$, similar expressions as (\ref{eq:Curvcurrent}), (\ref{eq:Curvcurrent22}) and (\ref{eq:Curvcurrent11}) hold with $w$ replaced by $\tilde{w}$, $\bo$ by $\boo$ and $g \left(s\right)$ by some $\tilde{g} \left(s\right)$ with the same property and by taking the corresponding boundary conditions on $\phi \left(\gamma\right)$. \par
For computing the top Chern form $c_2 \left(\omega\right) = \frac{1}{\left(2 \pi\right)^2} \det \left(\sqrt{-1} \Theta \left(\omega\right)\right)$, we will have to take the wedge product of the $\left(1,1\right)$-entry and the $\left(2,2\right)$-entry in (\ref{eq:Curvcurrent}). But as can be seen in (\ref{eq:Curvcurrent22}) and (\ref{eq:Curvcurrent11}), both the entries involve current terms and the wedge product of currents cannot be na\"ively taken \cite{Demailly:2012:CmplxDifferGeom}, so some Bedford-Taylor theory \cite{Bedford:1982:cpctypsh,Bedford:1976:DirichletMongeAmpere} needs to be invoked for this purpose. In Subsection \ref{subsec:Curvcurrent} we will try to first just na\"ively compute the determinant of the matrix in (\ref{eq:Curvcurrent}) by ``intuitively'' assuming the wedge product of a current of integration over a divisor with a closed conically singular form to be the ``integration'' of the form over the divisor (as we had remarked in the expressions (\ref{eq:wedgenaivint}) discussed in Subsection \ref{subsec:CanonKhlrConeSing}). The rigorous explanation of this wedge product by using Bedford-Taylor theory will be given in Subsection \ref{subsec:currenteq}. \par
We will try to compute the wedge products of the whole term in (\ref{eq:Curvcurrent11}) with the two terms in (\ref{eq:Curvcurrent22}) separately. First note the following for the first term in (\ref{eq:Curvcurrent22}) (look at the expression (\ref{eq:Curv2conesing})):
\begin{equation}\label{eq:gsphi}
-\sqrt{-1} \partial \bar{\partial} \ln \left(g \left(s\right)\right) = - \phi'' \sqrt{-1} \frac{\phi}{\left\lvert w \right\rvert^2} d w \wedge d \bar{w} - \phi' \mathtt{p}^* \omega_\Sigma
\end{equation}
So the wedge product of the term in (\ref{eq:Curvcurrent11}) with the first term in (\ref{eq:Curvcurrent22}) is easily seen to be the following (from the expressions (\ref{eq:Chernconesing}) and (\ref{eq:lambdaconesing})):
\begin{multline}\label{eq:firstterm}
\left( -\sqrt{-1} \partial \bar{\partial} \ln \left(1 + f' \left(s\right)\right) - 2 \ttp^* \omega_\Sigma \right) \wedge \left( -\sqrt{-1} \partial \bar{\partial} \ln \left(g \left(s\right)\right) \right) \\
= \mathtt{p}^* \omega_\Sigma \wedge \sqrt{-1} \frac{d w \wedge d \bar{w}}{\left\lvert w \right\rvert^2} \frac{\phi}{\gamma^2} \left( \gamma \left(\phi + 2 \gamma\right) \phi'' + \phi' \left(\phi' \gamma - \phi\right) \right) = \frac{\lambda \left(\omega\right)}{2} \omega^2
\end{multline}
where the assumption, that the higher scalar curvature $\lambda \left(\omega\right) : X \smallsetminus \left(S_0 \cup S_\infty\right) \to \R$ is bounded, is sufficient (with the boundedness of $\lambda \left(\omega\right)$ being clear from the expression (\ref{eq:lambdaconesing}) at least when $\phi \in \mathcal{C}^\infty \cllml$). Now coming to the wedge product of the term in (\ref{eq:Curvcurrent11}) with the second term in (\ref{eq:Curvcurrent22}), we may intuitively think of the integration of a closed form of the type $\sqrt{-1} d w \wedge d \bar{w}$ over the divisor $\left\lbrace w = 0 \right\rbrace$ as zero, even though we have to be mindful that the conical singularity given by $\frac{\phi}{\left\lvert w \right\rvert^2}$ is present in the form. So we can expect to have the following expression for the wedge product of the term in (\ref{eq:Curvcurrent11}) with the second term in (\ref{eq:Curvcurrent22}) (which will be explained rigorously in Subsection \ref{subsec:currenteq}):
\begin{equation}\label{eq:secondterm}
\left( -\sqrt{-1} \partial \bar{\partial} \ln \left(1 + f' \left(s\right)\right) - 2 \ttp^* \omega_\Sigma \right) \wedge \left( 2 \pi \left(1 - \beta_0\right) \left[w = 0\right] \right) \hspace{2pt} = \hspace{2pt} 4 \pi \left(\beta_0 - 1\right) \ttp^* \omega_\Sigma \wedge \left[w = 0\right]
\end{equation}
Note that $\ttp^* \omega_\Sigma$ is a closed smooth form, and the wedge product of a current with a smooth form is defined in the usual way \cite{Demailly:2012:CmplxDifferGeom}. \par
So finally the expression for the top Chern form $c_2 \left(\omega\right) \bigr\rvert_{X \smallsetminus S_\infty}$ in the coordinates $\left(z, w\right)$ in a \nbd{} of $w = 0$ is given as follows (from (\ref{eq:firstterm}) and (\ref{eq:secondterm})):
\begin{equation}\label{eq:Cherncurrentzero}
c_2 \left(\omega\right) \bigr\rvert_{X \smallsetminus S_\infty} = \frac{\lambda \left(\omega\right)}{2 \left(2 \pi\right)^2} \omega^2 + \frac{\beta_0 - 1}{\pi} \ttp^* \omega_\Sigma \wedge \left[w = 0\right]
\end{equation}
Again in the coordinates $\left(z, \tilde{w}\right)$ in a \nbd{} of $\tilde{w} = 0$, we get the same expression as (\ref{eq:Cherncurrentzero}) for the top Chern form $c_2 \left(\omega\right) \bigr\rvert_{X \smallsetminus S_0}$ with $\tilde{w}$ in place of $w$ and $\boo$ in place of $\bo$. \par
So the global expression for the top Chern form (or the top Chern current) $c_2 \left(\omega\right)$ on the minimal ruled surface $X = \prj$ is given as follows:
\begin{equation}\label{eq:CherncurrentomegaSigma}
c_2 \left(\omega\right) = \frac{\lambda \left(\omega\right)}{2 \left(2 \pi\right)^2} \omega^2 + \frac{\beta_0 - 1}{\pi} \ttp^* \omega_\Sigma \wedge \left[\so\right] + \frac{\beta_\infty - 1}{\pi} \ttp^* \omega_\Sigma \wedge \left[\soo\right]
\end{equation}
where $\left[\so\right]$ and $\left[\soo\right]$ denote the currents of integration on $X$ along the divisors $\so$ and $\soo$ respectively. $c_2 \left(\omega\right)$ given by (\ref{eq:CherncurrentomegaSigma}) is a (closed) $\left(2,2\right)$-current on $X$ which when restricted to $X \smallsetminus \left(S_0 \cup S_\infty\right)$ agrees with the usual notion of the top Chern form $c_2 \left(\omega\right) \bigr\rvert_{X \smallsetminus \left(S_0 \cup S_\infty\right)}$ given by (\ref{eq:topChernXminus}), and we will see in Subsection \ref{subsec:CohomolInvcurrent} that the top Chern current $c_2 \left(\omega\right)$ given by (\ref{eq:CherncurrentomegaSigma}) globally on $X$ is a \cohomll{} representative of the top Chern class $c_2 \left(X\right)$ whereas $c_2 \left(\omega\right) \bigr\rvert_{X \smallsetminus \left(S_0 \cup S_\infty\right)}$ is not a \cohomll{} \invt{}. Equation (\ref{eq:CherncurrentomegaSigma}) gives additional weightage to the momentum-constructed conical higher cscK metric $\omega$ by providing in a \cohomll{ly} invariant manner a global interpretation for the higher scalar curvature $\lambda \left(\omega\right)$ on the whole of $X$ (instead of just on $X \smallsetminus \left(S_0 \cup S_\infty\right)$ which is a priori already there) in terms of the currents of integration along $S_0$ and $S_\infty$ respectively. \par
The right hand side of the equation of currents (\ref{eq:CherncurrentomegaSigma}) can be written in terms of $\omega$ (instead of $\omega_\Sigma$) by doing the following computations going along similar lines as (\ref{eq:secondterm}) and using the expressions (\ref{eq:ansatzconesing}) and (\ref{eq:omegaconesing}) and the boundary values of the variables involved given in (\ref{eq:wgammasosoo}):
\begin{equation}\label{eq:Cherncurrentomega'}
\begin{gathered}
\omega \wedge \left[w = 0\right] \hspace{2pt} = \hspace{2pt} \left( \left(1 + f' \left(s\right)\right) \mathtt{p}^* \omega_\Sigma + \sqrt{-1} \frac{f'' \left(s\right)}{\left\lvert w \right\rvert^2} d w \wedge d \bar{w} \right) \wedge \left[w = 0\right] \hspace{2pt} = \hspace{2pt} \mathtt{p}^* \omega_\Sigma \wedge \left[w = 0\right] \\
\omega \wedge \left[\tilde{w} = 0\right] \hspace{2pt} = \hspace{2pt} \left( \left(1 + f' \left(s\right)\right) \mathtt{p}^* \omega_\Sigma + \sqrt{-1} \frac{f'' \left(s\right)}{\left\lvert \tilde{w} \right\rvert^2} d \tilde{w} \wedge d \bar{\tilde{w}} \right) \wedge \left[\tilde{w} = 0\right] \hspace{2pt} = \hspace{2pt} \left(m + 1\right) \mathtt{p}^* \omega_\Sigma \wedge \left[\tilde{w} = 0\right]
\end{gathered}
\end{equation}
wherein just like (\ref{eq:secondterm}) the wedge product of a closed form of the type $\sqrt{-1} d w \wedge d \bar{w}$ with the current of integration over the divisor $\left\lbrace w = 0 \right\rbrace$ is taken to be zero, even though the form has got the conical singularity of the type $\frac{f'' \left(s\right)}{\left\lvert w \right\rvert^2}$ precisely at $w = 0$. Substituting (\ref{eq:Cherncurrentomega'}) in (\ref{eq:CherncurrentomegaSigma}) we obtain the following global expression of currents for $c_2 \left(\omega\right)$:
\begin{equation}\label{eq:Cherncurrentomega''}
c_2 \left(\omega\right) = \frac{\lambda \left(\omega\right)}{2 \left(2 \pi\right)^2} \omega^2 + \frac{\beta_0 - 1}{\pi} \omega \wedge \left[\so\right] + \frac{\beta_\infty - 1}{\left(m + 1\right) \pi} \omega \wedge \left[\soo\right]
\end{equation}
\subsection{The Conical Higher cscK Equation in Terms of the Currents of Integration along the Divisors}\label{subsec:currenteq}
We will first give a brief review of some basics of Bedford-Taylor theory \cite{Bedford:1982:cpctypsh,Bedford:1976:DirichletMongeAmpere} which will be needed here in Subsection \ref{subsec:currenteq} to give rigorous justifications to the wedge products of current terms seen in Subsection \ref{subsec:Curvcurrent}. We are following the exposition given in Demailly \cite{Demailly:2012:CmplxDifferGeom}; Section III.3 about the \textit{Bedford-Taylor wedge product} \cite{Bedford:1982:cpctypsh,Bedford:1976:DirichletMongeAmpere} of a closed (semi)positive $\left(k,k\right)$-current with a closed (semi)positive $\left(1,1\right)$-current having a (locally) bounded (global) potential. \par
Let $M$ be a complex $n$-manifold, $T$ be a closed semipositive $\left(k,k\right)$-current on $M$ and $u$ be a locally bounded \psh{} function on $M$ so that $\ideldb u$ becomes a closed semipositive $\left(1,1\right)$-current on $M$. Since $T$ is a $\left(k,k\right)$-current, it is a $\left(k,k\right)$-form with distributional coefficients. Since $T$ is closed, its distributional coefficients are distributions of order $0$ and so, are Borel measures on $M$. Since $T$ is in addition semipositive, its coefficients are positive Borel measures on $M$. Now as $u$ is a locally bounded Borel measurable function on $M$ and positive measures can be multiplied by measurable functions to obtain new measures, so $u T$ defines another $\left(k,k\right)$-form with Borel measure coefficients i.e. a $\left(k,k\right)$-current on $M$. And since the exterior derivatives of currents are well-defined, we define the \textit{Bedford-Taylor product} of the currents $\ideldb u$ and $T$ on $M$ as follows \cite{Bedford:1982:cpctypsh,Bedford:1976:DirichletMongeAmpere}:
\begin{equation}\label{eq:defBTprod}
\ideldb u \wedge T = \ideldb \left(u T\right)
\end{equation} \par
\begin{theorem}[Bedford-Taylor \cite{Bedford:1982:cpctypsh,Bedford:1976:DirichletMongeAmpere}, Demailly \cite{Demailly:2012:CmplxDifferGeom}; Section III.3]\label{thm:BTprod}
The Bedford-Taylor product of the closed semipositive $\left(1,1\right)$-current $\ideldb u$ (with $u$ being a locally bounded \psh{} function) and the closed semipositive $\left(k,k\right)$-current $T$ defined by (\ref{eq:defBTprod}) is a closed semipositive $\left(k+1,k+1\right)$-current on $M$.
\end{theorem}
{\noindent The Bedford-Taylor product also satisfies some standard continuous and monotone approximation properties which are proved in Demailly \cite{Demailly:2012:CmplxDifferGeom}; Section III.3 in addition to Theorem \ref{thm:BTprod}. So the Bedford-Taylor product provides an effective way of interpreting wedge products of closed (semi)positive currents like the ones in equations (\ref{eq:ScalconecscK}) and (\ref{eq:highScalconehcscK}) discussed in Subsection \ref{subsec:CanonKhlrConeSing} and also those in equations (\ref{eq:secondterm}) and (\ref{eq:Cherncurrentomega'}) seen in Subsection \ref{subsec:Curvcurrent}.} \par
Subsection \ref{subsec:currenteq} is dedicated to making the wedge products in (\ref{eq:secondterm}) and (\ref{eq:Cherncurrentomega'}) rigorous, because the wedge product in (\ref{eq:firstterm}) was already clear in Subsection \ref{subsec:Curvcurrent}. The bounded (or locally bounded) \psh{} (or locally \psh{}) functions $u$ to be considered over here are $\ln \left(1 + f' \left(s\right)\right)$ (for the equation (\ref{eq:secondterm})) and $f \left(s\right)$, $f \left(s\right) - m s$ (for the equations (\ref{eq:Cherncurrentomega'}) at $\so$, $\soo$ respectively) and the closed positive $\left(1,1\right)$-currents $T$ are $\left[w = 0\right]$ as well as $\left[\tilde{w} = 0\right]$ for both (\ref{eq:secondterm}) and (\ref{eq:Cherncurrentomega'}). \par
As $0 \leq f' \left(s\right) \leq m$ (from (\ref{eq:wgammasosoo})), so $0 \leq \ln \left(1 + f' \left(s\right)\right) \leq \ln \left(m + 1\right)$, i.e. $\ln \left(1 + f' \left(s\right)\right)$ is bounded on our surface $X$. Recall that on a compact \kr{} $n$-manifold $M$ a continuous function $u : M \to \R$ is \psh{} (respectively strictly \psh{}) if and only if the closed $\left(1,1\right)$-current $\ideldb u$ is semipositive (respectively strictly positive i.e. \kr{}) on $M$ \cite{Demailly:2012:CmplxDifferGeom}. Since everything is smooth in the variable $z$ on the \rms{} $\Sigma$ and the conical singularities occur only in the variable $w$ (or the variable $\tilde{w}$) on the fibres of the line bundle $L$ (and that too at the ends of the fibres), and we are following the convention $d \ln h \left(z\right) = 0$ in the variable $s$ as done by \cite{Pingali:2018:heK,Szekelyhidi:2014:eKintro}, so it suffices to treat $\ln \left(1 + f' \left(s\right)\right)$ as a function of $w$ (respectively $\tilde{w}$) and check its \psh{ity}. Looking at the expression for the current $\ideldb \ln \left(1 + f' \left(s\right)\right)$ in terms of the momentum profile $\phi \left(\gamma\right)$ given in (\ref{eq:Curvcurrent11}) and noting that $w = 0 \iff \gamma = 1$ (and $\tilde{w} = 0 \iff \gamma = m + 1$) from (\ref{eq:wgammasosoo}), we see the following:
\begin{align}\label{eq:pshKahler}
\sqrt{-1} \partial \bar{\partial} \ln \left(1 + f' \left(s\right)\right) \hspace{1pt} &= \hspace{1pt} -\frac{1}{\gamma} \left( \frac{\phi}{\gamma} - \phi' \right) \sqrt{-1} \frac{\phi}{\left\lvert w \right\rvert^2} d w \wedge d \bar{w} + \frac{\phi}{\gamma} \ttp^* \omega_\Sigma \\
&= \hspace{1pt} -\frac{1}{\gamma} \left( \frac{\phi}{\gamma} - \phi' \right) \sqrt{-1} \frac{\phi}{\left\lvert \tilde{w} \right\rvert^2} d \tilde{w} \wedge d \bar{\tilde{w}} + \frac{\phi}{\gamma} \ttp^* \omega_\Sigma \nonumber
\end{align}
where the boundary conditions (\ref{eq:BVPConeSing}) on $\phi$ will give $-\frac{1}{\gamma} \left( \frac{\phi}{\gamma} - \phi' \right) = \bo > 0$ at $\gamma = 1$, along with the facts that $\gamma \geq 1$, $\phi \left(\gamma\right) \geq 0$, $\phi \left(1\right) = 0$ and $\sqrt{-1} \frac{\phi}{\left\lvert w \right\rvert^2} d w \wedge d \bar{w}$ is a \krc{} being a conical \krm{}. So there exists a non-degenerate open tubular \nbd{} of the zero divisor $\so$ of the surface $X$, on which the current $\ideldb \ln \left(1 + f' \left(s\right)\right)$ is positive i.e. the continuous function $\ln \left(1 + f' \left(s\right)\right)$ is \psh{}. On the other hand at $\gamma = m + 1$ the boundary conditions (\ref{eq:BVPConeSing}) will give $-\frac{1}{\gamma} \left( \frac{\phi}{\gamma} - \phi' \right) = -\frac{\boo}{m + 1} < 0$, so $-\ideldb \ln \left(1 + f' \left(s\right)\right)$ is positive i.e. $-\ln \left(1 + f' \left(s\right)\right)$ is \psh{} in a tubular \nbd{} of the infinity divisor $\soo$. Since the currents of integration $\left[w = 0\right]$ and $\left[\tilde{w} = 0\right]$ are zero away from their respective divisors, it does not matter what happens to the sign of the current $\ideldb \ln \left(1 + f' \left(s\right)\right)$ away from the divisors. \par
So $\ln \left(1 + f' \left(s\right)\right)$ and $\left[w = 0\right]$ (respectively $\left[\tilde{w} = 0\right]$) do satisfy the conditions required in the hypothesis of the Bedford-Taylor product locally in some \nbd{s} of the divisors $\so$ and $\soo$ separately, and the definition of the Bedford-Taylor product is also local in nature \cite{Demailly:2012:CmplxDifferGeom}, so we are allowed to apply Bedford-Taylor theory to interpret the wedge products of currents in the equations (\ref{eq:CherncurrentomegaSigma}) and (\ref{eq:Cherncurrentomega''}) seen in Subsection \ref{subsec:Curvcurrent}. \par
So we apply the definition of the Bedford-Taylor product to the current terms in (\ref{eq:secondterm}) and (\ref{eq:Cherncurrentzero}) as follows:
\begin{equation}\label{eq:currentzeroBTprod}
\sqrt{-1} \partial \bar{\partial} \ln \left(1 + f' \left(s\right)\right) \wedge \left[w = 0\right] \hspace{2pt} = \hspace{2pt} \sqrt{-1} \partial \bar{\partial} \left( \ln \left(1 + f' \left(s\right)\right) \left[w = 0\right] \right) \hspace{2pt} = \hspace{2pt} 0
\end{equation}
since $\ln \left(1 + f' \left(s\right)\right) = 0$ at $w = 0$ and $\left[w = 0\right]$ is a closed current which vanishes away from $w = 0$. Similarly $\ln \left(1 + f' \left(s\right)\right) = \ln \left(m + 1\right) > 0$ at $\tilde{w} = 0$ and so we have:
\begin{align}\label{eq:currentinfinityBTprod}
\sqrt{-1} \partial \bar{\partial} \ln \left(1 + f' \left(s\right)\right) \wedge \left[\tilde{w} = 0\right] \hspace{2pt} &= \hspace{2pt} \sqrt{-1} \partial \bar{\partial} \left( \ln \left(1 + f' \left(s\right)\right) \left[\tilde{w} = 0\right] \right) \\
&= \hspace{2pt} \ln \left(m + 1\right) \sqrt{-1} \partial \bar{\partial} \left( \left[\tilde{w} = 0\right] \right) \hspace{2pt} = \hspace{2pt} 0 \nonumber
\end{align}
This rigorously justifies the equation of currents (\ref{eq:CherncurrentomegaSigma}) for the top Chern form $c_2 \left(\omega\right)$ globally on the surface $X$. \par
Now for justifying the wedge products of currents appearing in the equation (\ref{eq:Cherncurrentomega''}), we will have to consider the expressions for $\omega$ given by the Calabi ansatz (\ref{eq:ansatzconesing}) with local \krp{s} $f \left(s\right)$ and $f \left(s\right) - m s$ applicable on $X \smallsetminus S_\infty$ and $X \smallsetminus S_0$ respectively as follows:
\begin{equation}\label{eq:ansatzsosoo}
\begin{gathered}
\omega = \mathtt{p}^* \omega_\Sigma + \sqrt{-1} \partial \bar{\partial} f \left(s\right) \hspace{4pt} \text{on} \hspace{5.5pt} X \smallsetminus S_\infty \\
\omega = \left(m + 1\right) \mathtt{p}^* \omega_\Sigma + \sqrt{-1} \partial \bar{\partial} \left(f \left(s\right) - m s\right) \hspace{4pt} \text{on} \hspace{5.5pt} X \smallsetminus S_0
\end{gathered}
\end{equation}
Then we can check the \psh{ity} of $f \left(s\right)$ and $f \left(s\right) - m s$ again only in terms of $w$ and $\tilde{w}$ respectively (ignoring the parts in the coordinate $z$) from the following coordinate expressions easily seen from (\ref{eq:omegaconesing}):
\begin{equation}\label{eq:fsmssosoo}
\begin{gathered}
\ideldb f \left(s\right) = f' \left(s\right) \mathtt{p}^* \omega_\Sigma + \sqrt{-1} \frac{f'' \left(s\right)}{\left\lvert w \right\rvert^2} d w \wedge d \bar{w} \\
\ideldb \left(f \left(s\right) - m s\right) = \left(f' \left(s\right) - m\right) \mathtt{p}^* \omega_\Sigma + \sqrt{-1} \frac{f'' \left(s\right)}{\left\lvert \tilde{w} \right\rvert^2} d \tilde{w} \wedge d \bar{\tilde{w}}
\end{gathered}
\end{equation}
where $\sqrt{-1} \frac{f'' \left(s\right)}{\left\lvert w \right\rvert^2} d w \wedge d \bar{w}$ and $\sqrt{-1} \frac{f'' \left(s\right)}{\left\lvert \tilde{w} \right\rvert^2} d \tilde{w} \wedge d \bar{\tilde{w}}$ are conical \krm{s} and hence \krc{s} (see the concerned expressions in (\ref{eq:omegapolyhomo}), (\ref{eq:omegasemihomo})) and $f' \left(s\right) \bigr\rvert_{w = 0} = 0$, $f' \left(s\right) \bigr\rvert_{\tilde{w} = 0} = m$ as seen from (\ref{eq:wgammasosoo}). So from (\ref{eq:fsmssosoo}), $f \left(s\right)$ and $f \left(s\right) - m s$ are continuous \psh{} functions on $X \smallsetminus S_\infty$ and $X \smallsetminus S_0$ respectively. Looking at the asymptotic power series expansions for $f \left(s\right)$, $f \left(s\right) - m s$ near $\so$, $\soo$ given in (\ref{eq:f1sfspolyhomo}), (\ref{eq:f2sf1sfspolyhomotilde}) respectively for the case of real analytic momentum profiles, and their respective asymptotic second order Taylor approximations near $\so$, $\soo$ given in (\ref{eq:f1sfssemihomo}), (\ref{eq:f2sf1sfssemihomotilde}) if the momentum profile is only smooth, we clearly see that $f \left(s\right) \bigr\rvert_{w = 0} = 0$ and $\left(f \left(s\right) - m s\right) \bigr\rvert_{\tilde{w} = 0} = 0$, and hence $f \left(s\right)$ and $f \left(s\right) - m s$ are bounded locally in some tubular \nbd{s} of $S_0$ and $S_\infty$ respectively. So just as done for (\ref{eq:currentzeroBTprod}) and (\ref{eq:currentinfinityBTprod}), we can apply Bedford-Taylor theory for the \psh{} functions $f \left(s\right)$  and $f \left(s\right) - m s$ and the currents of integration $\left[w = 0\right]$ and $\left[\tilde{w} = 0\right]$ respectively to make the wedge product computations in (\ref{eq:Cherncurrentomega'}) rigorous as follows:
\begin{align}\label{eq:currentomegazeroBTprod}
\omega \wedge \left[w = 0\right] \hspace{2pt} &= \hspace{2pt} \left( \mathtt{p}^* \omega_\Sigma + \sqrt{-1} \partial \bar{\partial} f \left(s\right) \right) \wedge \left[w = 0\right] \\
&= \hspace{2pt} \mathtt{p}^* \omega_\Sigma \wedge \left[w = 0\right] + \sqrt{-1} \partial \bar{\partial} \left( f \left(s\right) \left[w = 0\right] \right) \hspace{2pt} = \hspace{2pt} \mathtt{p}^* \omega_\Sigma \wedge \left[w = 0\right] \nonumber
\end{align}
\begin{align}\label{eq:currentomegainfinityBTprod}
\omega \wedge \left[\tilde{w} = 0\right] \hspace{2pt} &= \hspace{2pt} \left( \left(m + 1\right) \mathtt{p}^* \omega_\Sigma + \sqrt{-1} \partial \bar{\partial} \left(f \left(s\right) - m s\right) \right) \wedge \left[\tilde{w} = 0\right] \nonumber \\
&= \hspace{2pt} \left(m + 1\right) \mathtt{p}^* \omega_\Sigma \wedge \left[\tilde{w} = 0\right] + \sqrt{-1} \partial \bar{\partial} \left( \left(f \left(s\right) - m s\right) \left[\tilde{w} = 0\right] \right) \\
&= \hspace{2pt} \left(m + 1\right) \mathtt{p}^* \omega_\Sigma \wedge \left[\tilde{w} = 0\right] \nonumber
\end{align} \par
So finally we can have the global expressions of currents for $c_2 \left(\omega\right)$ in terms of $\omega_\Sigma$ as well as $\omega$ as follows:
\begin{align}\label{eq:Cherncurrentomega}
c_2 \left(\omega\right) &= \frac{\lambda \left(\omega\right)}{2 \left(2 \pi\right)^2} \omega^2 + \frac{\beta_0 - 1}{\pi} \ttp^* \omega_\Sigma \wedge \left[\so\right] + \frac{\beta_\infty - 1}{\pi} \ttp^* \omega_\Sigma \wedge \left[\soo\right] \\
&= \frac{\lambda \left(\omega\right)}{2 \left(2 \pi\right)^2} \omega^2 + \frac{\beta_0 - 1}{\pi} \omega \wedge \left[\so\right] + \frac{\beta_\infty - 1}{\left(m + 1\right) \pi} \omega \wedge \left[\soo\right] \nonumber
\end{align} \par
Note that the results of Subsections \ref{subsec:Curvcurrent} and \ref{subsec:currenteq} (most importantly the equations (\ref{eq:CherncurrentomegaSigma}), (\ref{eq:Cherncurrentomega''}) and (\ref{eq:Cherncurrentomega})) apply for momentum-constructed conical \krm{s} on the minimal ruled surface whose momentum profile is at least smooth if not real analytic at both the endpoints of the momentum interval, i.e. for \conrml{} smooth conical \krm{s} in general and hence for \plyhomo{} smooth conical \krm{s} in particular.
\subsection{A Smooth Approximation Result for Momentum-Constructed Conical Higher cscK Metrics}\label{subsec:SmoothapproxConehcscK}
In Subsection \ref{subsec:SmoothapproxConehcscK} we will provide an alternate interpretation to the equation of currents (\ref{eq:Cherncurrentomega}) globally characterizing the higher scalar curvature, by following a well-known method outlined in the works of Campana-Guenancia-P\u{a}un \cite{Campana:2013:ConeSingNormCrossDiv}, Edwards \cite{Edwards:2019:ContractDivConicKRFlow}, Shen \cite{Shen:2016:SmoothApproxConeSingRicciLB}, Wang \cite{Wang:2016:SmoothApproxConeKRFlow} of approximating the momentum-constructed conical \kr{} metric by some smooth \kr{} metrics and then trying to obtain the higher scalar curvature of the conical metric as the limiting value (in the sense of currents) of the higher scalar curvatures of the approximating smooth metrics (which are going to be defined in the usual way). We will be constructing some explicit smooth approximations $\omega_\epsilon$ (also having Calabi symmetry) to the momentum-constructed conical \kr{} metric $\omega$ such that their respective smooth top Chern forms $c_2 \left(\omega_\epsilon\right)$ converge weakly in the sense of currents to the top Chern current $c_2 \left(\omega\right)$ given by the global expression (\ref{eq:Cherncurrentomega}) which we have already obtained by using Bedford-Taylor theory \cite{Bedford:1982:cpctypsh,Bedford:1976:DirichletMongeAmpere} in Subsection \ref{subsec:currenteq}. \par
Let $\omega$ be a conical \krm{} on the minimal ruled surface $X$ satisfying the Calabi ansatz (\ref{eq:ansatzconesing}) and belonging to the \krcl{} $\pcsoo$ where $m > 0$, whose momentum profile $\phi \left(\gamma\right) = f'' \left(s\right)$ is assumed to be smooth on $\cllml$ and satisfies the boundary conditions (\ref{eq:BVPConeSing}) which make it develop conical singularities of cone angles $2 \pi \bo > 0$ and $2 \pi \boo > 0$ along the divisors $\so$ and $\soo$ respectively. \par
Let $\epsilon > 0$ be fixed. Define $s_\epsilon = \ln \left(\left\lvert w \right\rvert^2 + \epsilon^2\right)$. Then since $s = 2 \ln \left\lvert w \right\rvert$ (ignoring the terms in the coordinate $z$ again due to our convention $d \ln h \left(z\right) = 0$ \cite{Pingali:2018:heK,Szekelyhidi:2014:eKintro}), we get $e^{s_\epsilon} = e^s + \epsilon^2$. From (\ref{eq:wgammasosoo}) we easily get the following:
\begin{equation}\label{eq:wsepsilonsosoo}
\begin{gathered}
w \to 0 \iff s \to -\infty \iff s_\epsilon \to 2 \ln \epsilon \\
w \to \infty \iff s \to \infty \iff s_\epsilon \to \infty
\end{gathered}
\end{equation} \par
Let $\omega_\epsilon$ be the \krm{} on $X$ defined by the following ansatz:
\begin{equation}\label{eq:ansatzepsilon}
\omega_\epsilon = \mathtt{p}^* \omega_\Sigma + \sqrt{-1} \partial \bar{\partial} f_\epsilon \left(s\right)
\end{equation}
where $f_\epsilon : \R \to \R$, $f_\epsilon \left(s\right) = f \left(s_\epsilon\right)$, $f : \R \to \R$ being the function yielding the Calabi ansatz (\ref{eq:ansatzconesing}) for the given conical metric $\omega$.
\begin{claim*}
$\omega_\epsilon$ is smooth at least on $X \smallsetminus S_\infty$ (if not on the whole of $X$), $\omega_\epsilon$ either has a conical singularity of cone angle $2\pi\boo$ along $\soo$ (same as $\omega$) or is smooth along $\soo$, and $\omega_\epsilon$ belongs to the same \krcl{} $\pcsoo$ as $\omega$.
\end{claim*} \par
We first verify that the function $f_\epsilon \left(s\right)$ satisfies all the required properties of the momentum construction mentioned in Subsection \ref{subsec:MomentConstructConeSing} from the following calculations:
\begin{equation}\label{eq:derfsepsilon}
\begin{gathered}
f_\epsilon' \left(s\right) = \frac{d}{d s} \left( f \left(s_\epsilon\right) \right) = \frac{\left\lvert w \right\rvert^2}{\left\lvert w \right\rvert^2 + \epsilon^2} f' \left(s_\epsilon\right) \\
f_\epsilon'' \left(s\right) = \frac{d}{d s} \left( \frac{e^s}{e^s + \epsilon^2} f' \left(s_\epsilon\right) \right) = \frac{\left\lvert w \right\rvert^4}{\left(\left\lvert w \right\rvert^2 + \epsilon^2\right)^2} f'' \left(s_\epsilon\right) + \frac{\epsilon^2 \left\lvert w \right\rvert^2}{\left(\left\lvert w \right\rvert^2 + \epsilon^2\right)^2} f' \left(s_\epsilon\right)
\end{gathered}
\end{equation}
It then clearly follows from (\ref{eq:derfsepsilon}) that $-1 < 0 \leq f_\epsilon' \left(s\right) \leq m$ with $\lim\limits_{w \to 0} f_\epsilon' \left(s\right) = 0$ and $\lim\limits_{w \to \infty} f_\epsilon' \left(s\right) = m$, and also $f_\epsilon'' \left(s\right) > 0$ for $w \neq 0$ with $\lim\limits_{w \to 0} f_\epsilon'' \left(s\right) = 0$ and $\lim\limits_{w \to \infty} f_\epsilon'' \left(s\right) = 0$. \par
Then we define the momentum profile of $\omega_\epsilon$ as $\phi_\epsilon : \cllml \to \R$, $\phi_\epsilon \left(\gamma\right) = f_\epsilon'' \left(s\right)$ where $\gamma = 1 + f_\epsilon' \left(s\right)$, and check the boundary conditions on $\phi_\epsilon'$ (to determine the behaviour of $\omega_\epsilon$ near $\so$ and $\soo$) as follows (using (\ref{eq:derfsepsilon}) and (\ref{eq:variablechange})):
\begin{align}\label{eq:der3fsepsilon}
f_\epsilon''' \left(s\right) &= \frac{d}{d s} \left( \frac{e^{2 s}}{\left(e^s + \epsilon^2\right)^2} f'' \left(s_\epsilon\right) + \frac{\epsilon^2 e^s}{\left(e^s + \epsilon^2\right)^2} f' \left(s_\epsilon\right) \right) \nonumber \\
&= \frac{\left\lvert w \right\rvert^6}{\left(\left\lvert w \right\rvert^2 + \epsilon^2\right)^3} f''' \left(s_\epsilon\right) + \frac{3 \epsilon^2 \left\lvert w \right\rvert^4}{\left(\left\lvert w \right\rvert^2 + \epsilon^2\right)^3} f'' \left(s_\epsilon\right) - \frac{\epsilon^2 \left\lvert w \right\rvert^2 \left(\left\lvert w \right\rvert^2 - \epsilon^2\right)}{\left(\left\lvert w \right\rvert^2 + \epsilon^2\right)^3} f' \left(s_\epsilon\right)
\end{align}
\begin{equation}\label{eq:derphiepsilon}
\phi_\epsilon' \left(\gamma\right) = \frac{f_\epsilon''' \left(s\right)}{f_\epsilon'' \left(s\right)} = \frac{\left\lvert w \right\rvert^4 f''' \left(s_\epsilon\right) + 3 \epsilon^2 \left\lvert w \right\rvert^2 f'' \left(s_\epsilon\right) - \epsilon^2 \left(\left\lvert w \right\rvert^2 - \epsilon^2\right) f' \left(s_\epsilon\right)}{\left(\left\lvert w \right\rvert^2 + \epsilon^2\right) \left(\left\lvert w \right\rvert^2 f'' \left(s_\epsilon\right) + \epsilon^2 f' \left(s_\epsilon\right)\right)}
\end{equation}
From (\ref{eq:der3fsepsilon}) and (\ref{eq:derphiepsilon}) we compute the boundary values of $f_\epsilon'''$ and $\phi_\epsilon'$ as follows:
\begin{equation}\label{eq:derphiepsilon1}
\begin{gathered}
\lim\limits_{w \to 0} f_\epsilon''' \left(s\right) = 0 = \lim\limits_{w \to \infty} f_\epsilon''' \left(s\right) \\
\phi_\epsilon' \left(1\right) = \lim\limits_{w \to 0} \frac{f_\epsilon''' \left(s\right)}{f_\epsilon'' \left(s\right)} = \frac{\epsilon^4 f' \left(2 \ln \epsilon\right)}{\epsilon^4 f' \left(2 \ln \epsilon\right)} = 1 \hspace{2.5pt} \left(\text{as $f' \left(2 \ln \epsilon\right) > 0$}\right)
\end{gathered}
\end{equation}
\begin{align}\label{eq:derphiepsilonm1}
\phi_\epsilon' \left(m + 1\right) = \lim\limits_{w \to \infty} \frac{f_\epsilon''' \left(s\right)}{f_\epsilon'' \left(s\right)} &= \lim\limits_{w \to \infty} \frac{\frac{f''' \left(s_\epsilon\right)}{f'' \left(s_\epsilon\right)} + 3 \frac{\epsilon^2}{\left\lvert w \right\rvert^2} - \left(1 - \frac{\epsilon^2}{\left\lvert w \right\rvert^2}\right) \frac{\epsilon^2 f' \left(s_\epsilon\right)}{\left\lvert w \right\rvert^2 f'' \left(s_\epsilon\right)}}{\left(1 + \frac{\epsilon^2}{\left\lvert w \right\rvert^2}\right) \left(1 + \frac{\epsilon^2 f' \left(s_\epsilon\right)}{\left\lvert w \right\rvert^2 f'' \left(s_\epsilon\right)}\right)} = -\boo && \left(\text{if $\boo \leq 1$}\right) \nonumber \\
&= \lim\limits_{w \to \infty} \frac{\left\lvert w \right\rvert^2 f'' \left(s_\epsilon\right) \frac{f''' \left(s_\epsilon\right)}{f'' \left(s_\epsilon\right)} + 3 \epsilon^2 f'' \left(s_\epsilon\right) - \left(1 - \frac{\epsilon^2}{\left\lvert w \right\rvert^2}\right) \epsilon^2 f' \left(s_\epsilon\right)}{\left(1 + \frac{\epsilon^2}{\left\lvert w \right\rvert^2}\right) \left(\left\lvert w \right\rvert^2 f'' \left(s_\epsilon\right) + \epsilon^2 f' \left(s_\epsilon\right)\right)} = -1 && \left(\text{if $\boo \geq 1$}\right)
\end{align}
because $\lim\limits_{w \to \infty} \left\lvert w \right\rvert^2 f'' \left(s_\epsilon\right) = \lim\limits_{\tilde{w} \to 0} \frac{f'' \left(s_\epsilon\right)}{\left\lvert \tilde{w} \right\rvert^2} = \lim\limits_{\tilde{w} \to 0} \tilde{c}_2 \left(z\right) \left\lvert \tilde{w} \right\rvert^{2\beta_\infty - 2}$ which is $\infty$, $\tilde{c}_2 \left(z\right)$ and $0$ if $\boo < 1$, $\boo = 1$ and $\boo > 1$ respectively, as can be seen from the expressions (\ref{eq:f2sf1sfssemihomotilde}), (\ref{eq:omegasemihomo}) (or even from (\ref{eq:f2sf1sfspolyhomotilde}), (\ref{eq:omegapolyhomo})) derived in Section \ref{sec:polyhomoConeSing}. \par
Comparing the boundary conditions on $\phi_\epsilon$ derived in (\ref{eq:derphiepsilon1}), (\ref{eq:derphiepsilonm1}) with the boundary conditions (\ref{eq:BVPConeSing}) and those given in (\ref{eq:ODEBVP0}) \cite{Edwards:2019:ContractDivConicKRFlow,Hashimoto:2019:cscKConeSing,Hwang:2002:MomentConstruct,Schlitzer:2023:dHYM,Szekelyhidi:2014:eKintro}, we conclude that $\omega_\epsilon$ is a \krm{} on $X$ having Calabi symmetry, which is smooth on $X \smallsetminus S_\infty$, has got a conical singularity with cone angle $2 \pi \boo$ along $\soo$ if $\boo \leq 1$ and extends smoothly across $\soo$ if $\boo \geq 1$, and which belongs to the \krcl{} $\pcsoo$. From (\ref{eq:ansatzepsilon}) we can derive the local expression for $\omega_\epsilon$ in the coordinates $\left(z, w\right)$ on $X \smallsetminus S_\infty$ similar to (\ref{eq:omegaconesing}) as follows:
\begin{equation}\label{eq:omegaepsilonconesing}
\omega_\epsilon = \left(1 + f_\epsilon' \left(s\right)\right) \mathtt{p}^* \omega_\Sigma + f_\epsilon'' \left(s\right) \sqrt{-1} \frac{d w \wedge d \bar{w}}{\left\lvert w \right\rvert^2}
\end{equation}
All of this holds true for every $\epsilon > 0$. We can clearly see that for $w \neq 0$, $s_\epsilon \to s$ and hence $f_\epsilon \left(s\right) \to f \left(s\right)$ as $\epsilon \to 0$, and further from (\ref{eq:derfsepsilon}) and (\ref{eq:der3fsepsilon}), $f_\epsilon' \left(s\right) \to f' \left(s\right)$, $f_\epsilon'' \left(s\right) \to f'' \left(s\right)$ and $f_\epsilon''' \left(s\right) \to f''' \left(s\right)$ as $\epsilon \to 0$, and so from (\ref{eq:derphiepsilon}), $\phi_\epsilon \to \phi$ and $\phi_\epsilon' \to \phi'$ pointwise on $\oplml$ as $\epsilon \to 0$. So from the expression (\ref{eq:omegaepsilonconesing}) it follows that $\omega_\epsilon \to \omega$ as $\epsilon \to 0$ in the $\cC^\infty$ sense on $\xsosoo$ and in the $\cL^1_{\operatorname{loc}}$ sense on $\xsoo$. \par
We now compute the curvature form matrix and the top Chern form of $\omega_\epsilon$ in the coordinates $\left(z,w\right)$ in a \nbd{} of $\so$ by applying the expressions (\ref{eq:omega2conesing}), (\ref{eq:Curv1conesing}), (\ref{eq:Curv2conesing}) and (\ref{eq:Chernconesing}) from Subsection \ref{subsec:MomentConstructConeSing} to the present case:
\begin{equation}\label{eq:omega2epsilonconesing}
\omega_\epsilon^2 = 2 \left(1 + f_\epsilon' \left(s\right)\right) f_\epsilon'' \left(s\right) \mathtt{p}^* \omega_\Sigma \wedge \sqrt{-1} \frac{d w \wedge d \bar{w}}{\left\lvert w \right\rvert^2}
\end{equation}
\begin{equation}\label{eq:Curv1epsilonconesing}
\sqrt{-1} \Theta \left(\omega_\epsilon\right) =
\begin{bmatrix}
- \sqrt{-1} \partial \bar{\partial} \ln \left(1 + f_\epsilon' \left(s\right)\right) - 2 \mathtt{p}^* \omega_\Sigma & 0 \\
0 & - \sqrt{-1} \partial \bar{\partial} \ln \left(\frac{f_\epsilon'' \left(s\right)}{\left\lvert w \right\rvert^2}\right)
\end{bmatrix}
\end{equation}
where using (\ref{eq:derfsepsilon}) and (\ref{eq:omegasemihomo}) (or (\ref{eq:omegapolyhomo})) we can write:
\begin{align}\label{eq:der2fsepsilonw2}
\frac{f_\epsilon'' \left(s\right)}{\left\lvert w \right\rvert^2} &= \left(\frac{\left\lvert w \right\rvert^2}{\left\lvert w \right\rvert^2 + \epsilon^2}\right) \frac{f'' \left(s_\epsilon\right)}{\left\lvert w \right\rvert^2 + \epsilon^2} + \left(\frac{\epsilon^2}{\left\lvert w \right\rvert^2 + \epsilon^2}\right) \frac{f' \left(s_\epsilon\right)}{\left\lvert w \right\rvert^2 + \epsilon^2} \nonumber \\
&= \left( \frac{\left\lvert w \right\rvert^2}{\left\lvert w \right\rvert^2 + \epsilon^2} g \left(s_\epsilon\right) + \frac{\epsilon^2}{\left\lvert w \right\rvert^2 + \epsilon^2} g_1 \left(s_\epsilon\right) \right) \left(\left\lvert w \right\rvert^2 + \epsilon^2\right)^{\bo - 1}
\end{align}
for $g \left(s\right) = c_2 \left(z\right) + c_4 \left(z\right) \left\lvert w \right\rvert^{2\beta_0} + c \left(z, \left\lvert w \right\rvert^{2\beta_0}\right) \left\lvert w \right\rvert^{2\beta_0}$ and $g_1 \left(s\right) = \frac{c_2 \left(z\right)}{\beta_0} + \frac{c_4 \left(z\right)}{2\beta_0} \left\lvert w \right\rvert^{2\beta_0} + c_0 \left(z, \left\lvert w \right\rvert^{2\beta_0}\right) \left\lvert w \right\rvert^{2\beta_0}$ (or $g \left(s\right) = \sum\limits_{k=1}^{\infty} c_{2k} \left(z\right) \left\lvert w \right\rvert^{2\left(k-1\right)\beta_0}$ and $g_1 \left(s\right) = \sum\limits_{k=1}^{\infty} \frac{c_{2k} \left(z\right)}{k\beta_0} \left\lvert w \right\rvert^{2\left(k-1\right)\beta_0}$) depending on whether the given conical metric $\omega$ is taken to be \conrml{} or \plyhomo{} for the asymptotic expressions of Section \ref{sec:polyhomoConeSing}. Then writing $q \left(s\right) = \frac{e^s}{e^s + \epsilon^2} g \left(\ln \left(e^s + \epsilon^2\right)\right) + \frac{\epsilon^2}{e^s + \epsilon^2} g_1 \left(\ln \left(e^s + \epsilon^2\right)\right)$ in (\ref{eq:der2fsepsilonw2}) we further compute the $\left(2,2\right)$-entry of (\ref{eq:Curv1epsilonconesing}) as follows:
\begin{equation}\label{eq:Curv1epsilonconesing22}
- \sqrt{-1} \partial \bar{\partial} \ln \left(\frac{f_\epsilon'' \left(s\right)}{\left\lvert w \right\rvert^2}\right) \hspace{1pt} = \hspace{1pt} \left( \left(\frac{q' \left(s\right)}{q \left(s\right)}\right)^2 - \frac{q'' \left(s\right)}{q \left(s\right)} \right) \sqrt{-1} \frac{d w \wedge d \bar{w}}{\left\lvert w \right\rvert^2} \hspace{4pt} - \hspace{4pt} \frac{q' \left(s\right)}{q \left(s\right)} \mathtt{p}^* \omega_\Sigma \hspace{4pt} + \hspace{4pt} \frac{\left(1 - \bo\right) \epsilon^2}{\left(\left\lvert w \right\rvert^2 + \epsilon^2\right)^2} \sqrt{-1} d w \wedge d \bar{w}
\end{equation}
The $\left(1,1\right)$-entry of (\ref{eq:Curv1epsilonconesing}) can be readily written in terms of the momentum profile $\phi_\epsilon$ (and the momentum variable $\gamma_\epsilon = 1 + f_\epsilon' \left(s\right)$) just like that in (\ref{eq:Curv2conesing}):
\begin{equation}\label{eq:Curv1epsilonconesing11}
-\sqrt{-1} \partial \bar{\partial} \ln \left(1 + f_\epsilon' \left(s\right)\right) - 2 \ttp^* \omega_\Sigma = \frac{\phi_\epsilon}{\gamma_\epsilon} \left( \frac{\phi_\epsilon}{\gamma_\epsilon} - \phi_\epsilon' \right) \sqrt{-1} \frac{d w \wedge d \bar{w}}{\left\lvert w \right\rvert^2} - \left(\frac{\phi_\epsilon}{\gamma_\epsilon} + 2\right) \ttp^* \omega_\Sigma
\end{equation}
Since all the terms in (\ref{eq:Curv1epsilonconesing22}) and (\ref{eq:Curv1epsilonconesing11}) are smooth $\left(1,1\right)$-forms even near $\so$, we can take the wedge product of (\ref{eq:Curv1epsilonconesing11}) with (\ref{eq:Curv1epsilonconesing22}) in the usual sense to obtain the expression for the top Chern form $c_2 \left(\omega_\epsilon\right)$ on $\xsoo$ as follows:
\begin{equation}\label{eq:Chernepsilonconesing}
c_2 \left(\omega_\epsilon\right) \hspace{1pt} = \hspace{1pt} \frac{1}{\left(2 \pi\right)^2} \mathtt{p}^* \omega_\Sigma \wedge \sqrt{-1} \frac{d w \wedge d \bar{w}}{\left\lvert w \right\rvert^2} P \left(\left\lvert w \right\rvert^2, \epsilon^2\right) \hspace{3.5pt} + \hspace{3.5pt} \frac{\bo - 1}{\left(2 \pi\right)^2} \left(\frac{\phi_\epsilon}{\gamma_\epsilon} + 2\right) \frac{\epsilon^2}{\left(\left\lvert w \right\rvert^2 + \epsilon^2\right)^2} \ttp^* \omega_\Sigma \wedge \sqrt{-1} d w \wedge d \bar{w}
\end{equation}
where $P \left(\left\lvert w \right\rvert^2, \epsilon^2\right)$ is going to be some expression in terms of $\left\lvert w \right\rvert^2, \epsilon^2$ obtained by taking the wedge products of the two terms in the right hand side of (\ref{eq:Curv1epsilonconesing11}) with the first two terms in the right hand side of (\ref{eq:Curv1epsilonconesing22}). \par
It can be verified by using (\ref{eq:derfsepsilon}), (\ref{eq:der3fsepsilon}) and (\ref{eq:derphiepsilon}) that $P \left(\left\lvert w \right\rvert^2, \epsilon^2\right) \to \frac{\phi}{\gamma^2} \left( \gamma \left(\phi + 2 \gamma\right) \phi'' + \phi' \left(\phi' \gamma - \phi\right) \right)$ as $\epsilon \to 0$ for $w \neq 0$, and hence from (\ref{eq:Chernconesing}) and (\ref{eq:lambdaconesing}) we get $\frac{1}{\left(2 \pi\right)^2} \mathtt{p}^* \omega_\Sigma \wedge \sqrt{-1} \frac{d w \wedge d \bar{w}}{\left\lvert w \right\rvert^2} P \left(\left\lvert w \right\rvert^2, \epsilon^2\right) \to \frac{\lambda \left(\omega\right)}{2 \left(2 \pi\right)^2} \omega^2$ as $\epsilon \to 0$ in the $\cC^\infty$ sense away from $\so$ and in the $\cL^1_{\operatorname{loc}}$ sense near $\so$ (i.e. in the exact same sense in which we had $\omega_\epsilon \to \omega$). The calculations required over here involve writing down the expression for $P \left(\left\lvert w \right\rvert^2, \epsilon^2\right)$ explicitly and then checking the limits of the individual terms appearing in that expression as $\epsilon \to 0$ for $w \neq 0$, but these calculations are extremely lengthy though elementary and hence the author would skip the tedious details involved in these calculations. \par
Now we come to the limit of the second term in the right hand side of (\ref{eq:Chernepsilonconesing}). We verify that $\frac{\bo - 1}{\left(2 \pi\right)^2} \left(\frac{\phi_\epsilon}{\gamma_\epsilon} + 2\right) \frac{\epsilon^2}{\left(\left\lvert w \right\rvert^2 + \epsilon^2\right)^2} \ttp^* \omega_\Sigma \wedge \sqrt{-1} d w \wedge d \bar{w} \to \frac{\beta_0 - 1}{\pi} \ttp^* \omega_\Sigma \wedge \left[\so\right]$ as $\epsilon \to 0$ weakly in the sense of currents on $\xsoo$ as follows ((\ref{eq:PLepsilon}) is essentially verifying the Poincar\'e-Lelong formula \cite{Demailly:2012:CmplxDifferGeom} specifically applicable in the present situation):
\begin{align}\label{eq:PLepsilon}
&\lim\limits_{\epsilon \to 0} \frac{\bo - 1}{\left(2 \pi\right)^2} \int\limits_{\xsoo} \varphi \left(\frac{\phi_\epsilon}{\gamma_\epsilon} + 2\right) \frac{\epsilon^2}{\left(\left\lvert w \right\rvert^2 + \epsilon^2\right)^2} \ttp^* \omega_\Sigma \wedge \sqrt{-1} d w \wedge d \bar{w} \\
&= \lim\limits_{\epsilon \to 0} \frac{\bo - 1}{\left(2 \pi\right)^2} \iint\limits_{\Sigma \times \C} \varphi \left(\frac{\phi_\epsilon}{\gamma_\epsilon} + 2\right) \frac{\epsilon^2}{\left(\left\lvert w \right\rvert^2 + \epsilon^2\right)^2} \omega_\Sigma \wedge \sqrt{-1} d w \wedge d \bar{w} \nonumber \\
&= \lim\limits_{\epsilon \to 0} \frac{\bo - 1}{\left(2 \pi\right)^2} \int\limits_0^{2\pi} \int\limits_0^\infty \frac{2r \epsilon^2}{\left(r^2 + \epsilon^2\right)^2} \left( \int\limits_\Sigma \varphi \left(\frac{\phi_\epsilon}{\gamma_\epsilon} + 2\right) \omega_\Sigma \right) dr d\theta \nonumber \\
&= \lim\limits_{\epsilon \to 0} \frac{\bo - 1}{\left(2 \pi\right)^2} \int\limits_0^{2\pi} \left[ \left( -\frac{\epsilon^2}{r^2 + \epsilon^2} \int\limits_\Sigma \varphi \left(\frac{\phi_\epsilon}{\gamma_\epsilon} + 2\right) \omega_\Sigma \right) \Biggr\rvert_0^\infty + \int\limits_0^\infty \frac{\epsilon^2}{r^2 + \epsilon^2} \frac{d}{dr} \left( \int\limits_\Sigma \varphi \left(\frac{\phi_\epsilon}{\gamma_\epsilon} + 2\right) \omega_\Sigma \right) dr \right] d\theta \nonumber \\
&= \lim\limits_{\epsilon \to 0} \frac{\bo - 1}{2 \pi} \int\limits_{\so} \varphi \left(z, 0\right) \left(\phi_\epsilon \left(1\right) + 2\right) \ttp^* \omega_\Sigma \hspace{5pt} \left(\text{applying dominated convergence theorem above}\right) \nonumber \\
&= \frac{\bo - 1}{\pi} \int\limits_{\so} \varphi \left(z, 0\right) \ttp^* \omega_\Sigma = \frac{\beta_0 - 1}{\pi} \ttp^* \omega_\Sigma \wedge \left[\so\right] \left(\varphi\right) \nonumber
\end{align}
where $\varphi : \xsoo \to \R$ is a test function compactly supported in some open tubular \nbd{} of $\so$, and so $\int\limits_\Sigma \varphi \left(\frac{\phi_\epsilon}{\gamma_\epsilon} + 2\right) \omega_\Sigma$ is also compactly supported near $w = 0$. \par
Thus finally the expression (\ref{eq:Cherncurrentzero}) giving the top Chern current of $\omega$ is obtained as the limiting value of the expression (\ref{eq:Chernepsilonconesing}) on $\xsoo$:
\begin{equation}\label{eq:Cherncurrentzeroepsilon}
\lim\limits_{\epsilon \to 0} c_2 \left(\omega_\epsilon\right) = \frac{\lambda \left(\omega\right)}{2 \left(2 \pi\right)^2} \omega^2 + \frac{\beta_0 - 1}{\pi} \ttp^* \omega_\Sigma \wedge \left[\so\right] = c_2 \left(\omega\right) \bigr\rvert_{X \smallsetminus S_\infty}
\end{equation}
where the convergence is in the $\cC^\infty$ sense on $\xsosoo$ and is in the sense of currents in a \nbd{} of $\so$. \par
To obtain the expression of the form (\ref{eq:Cherncurrentomega''}) for the top Chern current of $\omega$ on $\xsoo$ we need to justify the wedge product of $\omega$ with $\left[\so\right]$ computed in (\ref{eq:Cherncurrentomega'}) by means of this method of taking smooth approximations. We know that as $\epsilon \to 0$, in the appropriate manner discussed above, $\omega_\epsilon \to \omega$ and $\frac{\sqrt{-1}}{2 \pi} \deldb \ln \left(\left\lvert w \right\rvert^2 + \epsilon^2\right) = \frac{\sqrt{-1}}{2 \pi} \frac{\epsilon^2}{\left(\left\lvert w \right\rvert^2 + \epsilon^2\right)^2} d w \wedge d \bar{w} \to \left[\so\right]$ (by the Poincar\'e-Lelong formula \cite{Demailly:2012:CmplxDifferGeom}). So as expected we first take the wedge product of these two approximating smooth $\left(1,1\right)$-forms for $\epsilon > 0$ (using the coordinate expression (\ref{eq:omegaepsilonconesing}) for $\omega_\epsilon$), then pass the limits of this wedge product of smooth forms as $\epsilon \to 0$ and obtain the ``correct'' value of the wedge product of $\omega$ with $\left[\so\right]$ given by (\ref{eq:Cherncurrentomega'}) as the limiting value:
\begin{align}\label{eq:Cherncurrentomega'epsilon}
\omega \wedge \left[\so\right] &= \lim\limits_{\epsilon \to 0} \omega_\epsilon \wedge \frac{\sqrt{-1}}{2 \pi} \deldb \ln \left(\left\lvert w \right\rvert^2 + \epsilon^2\right) \nonumber \\
&= \lim\limits_{\epsilon \to 0} \frac{1 + f_\epsilon' \left(s\right)}{2 \pi} \frac{\epsilon^2}{\left(\left\lvert w \right\rvert^2 + \epsilon^2\right)^2} \mathtt{p}^* \omega_\Sigma \wedge \sqrt{-1} d w \wedge d \bar{w} = \mathtt{p}^* \omega_\Sigma \wedge \left[\so\right]
\end{align}
where the limit in (\ref{eq:Cherncurrentomega'epsilon}) is weakly in the sense of currents on $\xsoo$ and its verification is exactly the same as (\ref{eq:PLepsilon}). \par
Now note that the smooth approximations $\omega_\epsilon$ for the conical \kr{} metric $\omega$ were applicable only on $\xsoo$. We can take analogous smooth approximations $\tilde{\omega}_\epsilon$ on $\xso$ in the following manner: Taking the coordinates $\left(z, \tilde{w}\right)$ on $\xso$ where $\tilde{w} = w^{-1}$, we already have $s = -2 \ln \left\lvert \tilde{w} \right\rvert$ and then we define $\tilde{s}_\epsilon = - \ln \left(\left\lvert \tilde{w} \right\rvert^2 + \epsilon^2\right)$ for $\epsilon > 0$. Defining $\tilde{f}_\epsilon : \R \to \R$, $\tilde{f}_\epsilon \left(s\right) = f \left(\tilde{s}_\epsilon\right)$, we define the \krm{} $\tilde{\omega}_\epsilon$ by the Calabi ansatz of the form (\ref{eq:ansatzepsilon}) with the smooth strictly convex function yielding the ansatz for $\tilde{\omega}_\epsilon$ being $\tilde{f}_\epsilon \left(s\right)$. With this setup we can carry out all the calculations of Subsection \ref{subsec:SmoothapproxConehcscK} which will give us that $\tilde{\omega}_\epsilon$ is a smooth \krm{} on $\xso$ which possibly has a conical singularity at $\so$ (again depending on the value of $\bo$) and which once again belongs to the same \krcl{} $\pcsoo$. Then with $\tilde{\omega}_\epsilon$ approximating $\omega$ on $\xso$, we can compute the smooth top Chern form $c_2 \left(\tilde{\omega}_\epsilon\right)$ in the local coordinates $\left(z, \tilde{w}\right)$ and check that the weak limit of $c_2 \left(\tilde{\omega}_\epsilon\right)$ is precisely the restriction of the expression (\ref{eq:CherncurrentomegaSigma}) to $\xso$, in the same way as that done above for $\omega_\epsilon$ on $\xsoo$. The expression of the form (\ref{eq:Cherncurrentomega''}) considered locally on $\xso$ then comes by interpreting the wedge product of $\omega$ with $\left[\soo\right]$ seen in (\ref{eq:Cherncurrentomega'}) as the limiting value (in the weak sense) of the wedge product of $\tilde{\omega}_\epsilon$ with the smooth $\left(1,1\right)$-form $\frac{\sqrt{-1}}{2 \pi} \deldb \ln \left(\left\lvert \tilde{w} \right\rvert^2 + \epsilon^2\right)$ approximating the current $\left[\soo\right]$, in the same way as (\ref{eq:Cherncurrentomega'epsilon}). \par
In this way Subsection \ref{subsec:SmoothapproxConehcscK} gives a method different from that of Subsection \ref{subsec:currenteq} of interpreting the wedge products of the current terms appearing in the expression (\ref{eq:Cherncurrentomega}) for the top Chern current of momentum-constructed conical \krm{s} on the surface $X$.
\subsection{Cohomological Invariance of the Top Chern Current}\label{subsec:CohomolInvcurrent}
In Subsection \ref{subsec:CohomolInvcurrent} we will show that (the average of) the higher scalar curvature defined by the equation of currents (\ref{eq:Cherncurrentomega}) is indeed a \cohomll{} invariant, i.e. it depends only on the top Chern class of the surface $X$ and the \krcl{} of the conical metric $\omega$ (same as the way it is seen in (\ref{eq:avghighScal}) in the case of smooth \krm{s} on general compact complex manifolds). \par
Let $\omega$ be the momentum-constructed conical \krm{} on $X = \prj$, given by the Calabi ansatz (\ref{eq:ansatzconesing}), belonging to the \krcl{} $\pcsoo$ where $m > 0$ and having cone angles $2 \pi \beta_0 > 0$ and $2 \pi \beta_\infty > 0$ along the divisors $S_0$ and $S_\infty$ respectively. Let the momentum profile $\phi \left(\gamma\right) = f'' \left(s\right)$ be at least smooth on the whole of $\cllml$, which will imply that the higher scalar curvature $\lambda \left(\omega\right)$ given by (\ref{eq:lambdaconesing}) is bounded on $X \smallsetminus \left(S_0 \cup S_\infty\right)$. The top Chern current $c_2 \left(\omega\right)$ is then given by the expression (\ref{eq:Cherncurrentomega}) on $X$ with its restriction to $X \smallsetminus \left(S_0 \cup S_\infty\right)$ being given by the expression (\ref{eq:Chernconesing}) in terms of $\phi \left(\gamma\right)$. \par
Let $\eta$ be a momentum-constructed smooth \krm{} on $X$, given by the following ansatz and belonging to some \krcl{} $2 \pi \left(\sfc + k S_\infty\right)$ where $k > 0$:
\begin{equation}\label{eq:ansatzeta}
\eta = \ttp^* \omega_\Sigma + \sqrt{-1} \partial \bar{\partial} \rho \left(s\right)
\end{equation}
where $\rho \left(s\right)$ satisfies all the properties of the Calabi ansatz (\ref{eq:ansatzconesing}) mentioned in Subsection \ref{subsec:MomentConstructConeSing} (or those of the ansatz (\ref{eq:ansatz0}) in the smooth case discussed briefly in Subsection \ref{subsec:Background}), $x = 1 + \rho' \left(s\right) \in \left[1, k + 1\right]$ is the momentum variable and $\psi \left(x\right) = \rho'' \left(s\right)$ is the momentum profile which is assumed to be smooth on $\left[1, k + 1\right]$. The top Chern form $c_2 \left(\eta\right)$ and the higher scalar curvature $\lambda \left(\eta\right)$ will be given by the expressions (\ref{eq:Chernconesing}) and (\ref{eq:lambdaconesing}) respectively (with the substitution of the respective variables $x, \psi$ in place of $\gamma, \phi$). \par
We will compute the integrals of the respective top Chern form and top Chern current $c_2 \left(\eta\right)$ and $c_2 \left(\omega\right)$ over $X$ and verify that their values are the same, thereby proving that both the top-dimensional forms (currents) belong to the same top-dimensional de Rham \cohoml{} class which has to be the top Chern class $c_2 \left(X\right)$ as the metric $\eta$ is smooth everywhere on $X$. We will be using the fact that the fibre bundle $\prj$ locally looks like $\Sigma \times \left( \C \cup \left\lbrace \infty \right\rbrace \right)$, along with the boundary conditions given in (\ref{eq:ODEBVP0}) for the metric $\eta$ (required for it to extend smoothly across $\so$ and $\soo$) and the boundary conditions (\ref{eq:BVPConeSing}) for the metric $\omega$ (required for it to develop conical singularities along $\so$ and $\soo$), the intersection formulae (\ref{eq:IntersectForm}) and (\ref{eq:Sigma}) and all the relations between (the boundary values of) the variables and functions $w$, $s$, $\rho$, $f$, $\psi \left(x\right)$, $\phi \left(\gamma\right)$ seen in (\ref{eq:variablechange}) and (\ref{eq:wgammasosoo}), like $\sqrt{-1} d w \wedge d \bar{w} = 2 r d r d \theta$ where $r = \left\lvert w \right\rvert$ and $\frac{2}{r} d r = d s = \frac{1}{\psi} d x \left(= \frac{1}{\phi} d \gamma\right)$. Then it can be easily checked that:
\begin{equation}\label{eq:intChern}
\int\limits_X c_2 \left(\eta\right) = -4 = \int\limits_X c_2 \left(\omega\right)
\end{equation}
in the following way:
\begin{align}\label{eq:intCherncalceta}
\int\limits_X c_2 \left(\eta\right) &= \frac{1}{\left(2 \pi\right)^2} \iint\limits_{\prj} \ttp^* \omega_\Sigma \wedge \sqrt{-1} \frac{d w \wedge d \bar{w}}{\left\lvert w \right\rvert^2} \frac{\psi}{x^2} \left( x \left(\psi + 2 x\right) \psi'' + \psi' \left(\psi' x - \psi\right) \right) \nonumber \\
&= \frac{1}{\left(2 \pi\right)^2} \int\limits_\Sigma \omega_\Sigma \int\limits_{\C \smallsetminus \left\lbrace 0 \right\rbrace} \sqrt{-1} \frac{d w \wedge d \bar{w}}{\left\lvert w \right\rvert^2} \frac{\psi}{x^2} \left( x \left(\psi + 2 x\right) \psi'' + \psi' \left(\psi' x - \psi\right) \right) \nonumber \\
&= \frac{1}{2 \pi} \int\limits_0^{2 \pi} \int\limits_0^{\infty} \frac{\psi}{x^2} \left( x \left(\psi + 2 x\right) \psi'' + \psi' \left(\psi' x - \psi\right) \right) \frac{2}{r} d r d \theta \nonumber \\
&= \int\limits_1^{k + 1} \frac{1}{x^2} \left( x \left(\psi + 2 x\right) \psi'' + \psi' \left(\psi' x - \psi\right) \right) d x = \int\limits_1^{k + 1} \frac{d}{d x} \left( \left(\frac{\psi}{x} + 2\right) \psi' \right) d x = -4
\end{align}
\begin{align}\label{eq:intCherncalcomega}
\int\limits_X c_2 \left(\omega\right) &= \frac{1}{2 \left(2 \pi\right)^2} \int\limits_X \lambda \left(\omega\right) \omega^2 + \frac{\beta_0 - 1}{\pi} \ttp^* \omega_\Sigma \wedge \left[\so\right] \left(1\right) + \frac{\beta_\infty - 1}{\pi} \ttp^* \omega_\Sigma \wedge \left[\soo\right] \left(1\right) \nonumber \\
&= \int\limits_1^{m + 1} \frac{d}{d \gamma} \left( \left(\frac{\phi}{\gamma} + 2\right) \phi' \right) d \gamma + \frac{\beta_0 - 1}{\pi} \int\limits_{\so} \ttp^* \omega_\Sigma + \frac{\beta_\infty - 1}{\pi} \int\limits_{\soo} \ttp^* \omega_\Sigma \nonumber \\
&= -2\left(\beta_0+\beta_\infty\right) + 2\left(\beta_0-1\right) + 2\left(\beta_\infty-1\right) = -4
\end{align}
where the top Chern current $c_2 \left(\omega\right)$ is given by (\ref{eq:CherncurrentomegaSigma}), $1$ denotes the constant function $1$ on $X$, and after substituting the expressions for $\lambda \left(\omega\right)$ and $\omega^2$ in terms of $\phi \left(\gamma\right)$ given in (\ref{eq:lambdaconesing}) and (\ref{eq:omega2conesing}) respectively, the integral in (\ref{eq:intCherncalcomega}) has the same computation as the one in (\ref{eq:intCherncalceta}) but with different boundary conditions on the momentum variable and momentum profile. \par
Thus $\left[ c_2 \left(\omega\right) \right] = \left[ c_2 \left(\eta\right) \right] = c_2 \left(X\right) \in H^{\left(2,2\right)} \left(X, \mathbb{R}\right) = H^4 \left(X, \mathbb{R}\right)$, i.e. the top Chern current $c_2 \left(\omega\right)$ of the conical \krm{} $\omega$ indeed lies in the top Chern class $c_2 \left(X\right)$ of the surface $X$. This makes the equations of currents (\ref{eq:CherncurrentomegaSigma}) and (\ref{eq:Cherncurrentomega''}) for the top Chern current of the conical \krm{} $\omega$ legitimate from the point of view of the de Rham \cohoml{} of the surface $X$. \par
One more thing can be observed from the calculation in (\ref{eq:intCherncalcomega}). The top Chern form $c_2 \left(\omega\right) \bigr\rvert_{X \smallsetminus \left(S_0 \cup S_\infty\right)}$ given by (\ref{eq:topChernXminus}) (which is smooth on $\xsosoo$ and locally \intble{} on $X$) will not be in general a \cohomll{} representative of $c_2 \left(X\right)$, as $\int\limits_X c_2 \left(\omega\right) \bigr\rvert_{X \smallsetminus \left(S_0 \cup S_\infty\right)} = \frac{1}{2 \left(2 \pi\right)^2} \int\limits_X \lambda \left(\omega\right) \omega^2 = -2\left(\beta_0+\beta_\infty\right) \neq -4$. This is the reason why the average of the higher scalar curvature of the conical metric taken only on $\xsosoo$ is not a \cohomll{} \invt{}, but if the average is taken on the whole of $X$ by considering the global expression of currents (\ref{eq:CherncurrentomegaSigma}) then it indeed turns out to be equal to the \cohomll{} value given by (\ref{eq:avghighScal}). And further the two average higher scalar curvatures are going to be related by an equation of the form (\ref{eq:lambda0lambda1}) as discussed in Subsection \ref{subsec:CanonKhlrConeSing} (and this equation relating the two in our case can also be clearly discerned from the calculation (\ref{eq:intCherncalcomega})). We will be discussing more about this topic in Lemma \ref{lem:volxsosooavghighScalconic} and \textit{Remark} \ref{rem:lambda0lambda1omega}.
\numberwithin{equation}{section}
\numberwithin{figure}{section}
\numberwithin{table}{section}
\numberwithin{lemma}{section}
\numberwithin{proposition}{section}
\numberwithin{result}{section}
\numberwithin{theorem}{section}
\numberwithin{corollary}{section}
\numberwithin{conjecture}{section}
\numberwithin{remark}{section}
\numberwithin{note}{section}
\numberwithin{motivation}{section}
\numberwithin{question}{section}
\numberwithin{answer}{section}
\numberwithin{case}{section}
\numberwithin{claim}{section}
\numberwithin{definition}{section}
\numberwithin{example}{section}
\numberwithin{hypothesis}{section}
\numberwithin{statement}{section}
\numberwithin{ansatz}{section}
\section{The Top \texorpdfstring{$\log$}{log} Bando-Futaki Invariant as an Obstruction to the Existence of Momentum-Constructed Conical Higher cscK Metrics}\label{sec:logFutlogMab}
Section \ref{sec:logFutlogMab} introduces the \textit{top $\log$ Bando-Futaki invariant} which is supposed to be the appropriate Futaki-type invariant \cite{Bando:2006:HarmonObstruct,Futaki:1983:ObstructKE} giving the algebro-geometric obstruction to the existence of (momentum-constructed) conical higher cscK metrics in a given \krcl{}. It is the top-dimensional analogue of the $\log$ Futaki invariant which is meant for conical cscK and conical \kr{}-Einstein metrics, and the $\log$ Futaki invariant naturally leads to the concepts of $\log$ Mabuchi functional and $\log$ $K$-stability which are needed in the study of conical cscK and conical \kr{}-Einstein metrics. These notions can be found in the works of Li \cite{Li:2012:eKEngyFunctProjBund}, Zheng \cite{Zheng:2015:UniqueConeSing}, Keller-Zheng \cite{Keller:2018:cscKConeSing}, Li \cite{Li:2018:conicMab}, Hashimoto \cite{Hashimoto:2019:cscKConeSing}, Aoi-Hashimoto-Zheng \cite{Aoi:2025:cscKConeSing} for conical cscK metrics (and in the first one also for conical \ext{} \krm{s}) and in the works of Donaldson \cite{Donaldson:2012:ConeSingDiv}, Li \cite{Li:2015:logKStab} and some others for conical \kr{}-Einstein metrics. In this paper by studying the top $\log$ Bando-Futaki invariant only in the special case of the Calabi ansatz on the pseudo-Hirzebruch surface, we are just taking the first step towards exploring the higher dimensional analogues of these concepts i.e. their correct analogues applicable for conical higher cscK metrics. \par
Hashimoto \cite{Hashimoto:2019:cscKConeSing} and Donaldson \cite{Donaldson:2012:ConeSingDiv}, Li \cite{Li:2015:logKStab} defined the $\log$ Futaki invariant for momentum-constructed conical cscK metrics and for some special classes of conical \kr{}-Einstein metrics respectively on (certain) Fano manifolds, by considering the expression for the classical Futaki invariant for smooth cscK and \kr{}-Einstein metrics respectively \cite{Calabi:1985:eK2,Futaki:1983:ObstructKE} and adding an appropriate correction factor which will take care of the conical singularities present in the metrics under consideration. This correction factor is needed to cancel out the distributional term that will come out after the ``evaluation'' of the classical Futaki invariant with respect to a conically singular (cscK or \kr{}-Einstein) metric whose curvature (scalar or Ricci respectively) will be given by the current equations (\ref{eq:ScalconecscK}) and (\ref{eq:RicconeKE}) respectively seen in Subsection \ref{subsec:CanonKhlrConeSing}. \par
By mimicking these same ideas and arguments in our conical higher cscK case, we consider the equation of currents (\ref{eq:Cherncurrentomega}) globally determining the higher scalar curvature of a momentum-constructed conical higher cscK metric, and since the top $\log$ Bando-Futaki invariant is going to be the conical analogue of the classical top Bando-Futaki invariant \cite{Bando:2006:HarmonObstruct,Futaki:1983:ObstructKE}, we also have to use the expression obtained in Bando \cite{Bando:2006:HarmonObstruct} for the classical top Bando-Futaki invariant (which the author had studied in his first work \cite{Sompurkar:2023:heKsmooth}; Section 4). Then we define the \textit{top $\log$ Bando-Futaki invariant} for momentum-constructed conical higher cscK metrics on the pseudo-Hirzebruch surface $X = \prj$ as follows: \par
Let $\eta$ be a smooth \krm{} on $X$ (not necessarily having Calabi symmetry) belonging to the \krcl{} $\pcsoo$ with $m > 0$, $Y$ be a \textit{gradient real holomorphic vector field} on $X$ which is \textit{parallel} to both the divisors $S_0$ and $S_\infty$, i.e. $Y \rvert_{\so}$ and $Y \rvert_{\soo}$ are real \hol{} vector fields on $\so$ and $\soo$ respectively, and $\ttf : X \to \R$ be the \textit{real holomorphy potential} of $Y$ with respect to $\eta$, i.e. $Y = \nabla_\eta^{\left(1,0\right)} \ttf = \left(\bar{\partial} \ttf\right)^{\sharp_\eta}$, with $\nabla_\eta^{\left(1,0\right)}$ denoting the $\left(1,0\right)$-gradient computed with respect to $\eta$ and $\sharp_\eta$ the musical isomorphism induced by $\eta$. Then the \textit{top $\log$ Bando-Futaki invariant} on $X$ for the holomorphic vector field $Y$ and the \krcl{} $\pcsoo$ with cone angles $2\pi\bo > 0$ and $2\pi\boo > 0$ along $\so$ and $\soo$ respectively is defined as follows:
\begin{multline}\label{eq:logBFdef}
\mathcal{F}_{\log; \bo, \boo} \left(Y, \pcsoo\right) = -\frac{1}{2 \left(2 \pi\right)^2} \int\limits_X \ttf \left(\lambda \left(\eta\right) - \lambda_0 \left(\eta\right)\right) \eta^2 \\
+ \frac{\beta_0 - 1}{\pi} \left( \int\limits_{S_0} \ttf \eta - \frac{\int\limits_{S_0} \eta}{\int\limits_X \eta^2} \int\limits_X \ttf \eta^2 \right) + \frac{\beta_\infty - 1}{\left(m + 1\right) \pi} \left( \int\limits_{S_\infty} \ttf \eta - \frac{\int\limits_{S_\infty} \eta}{\int\limits_X \eta^2} \int\limits_X \ttf \eta^2 \right)
\end{multline}
where $\lambda \left(\eta\right)$ is the higher scalar curvature of $\eta$ on $X$ given by the equation (\ref{eq:defhcscKheK}) and $\lambda_0 \left(\eta\right) = \frac{\int\limits_{X} \lambda \left(\eta\right) \eta^2}{\int\limits_{X} \eta^2}$ is the average higher scalar curvature of $\eta$ on $X$ satisfying the equation (\ref{eq:avghighScal}). Then letting $\mathcal{F} \left(Y, \pcsoo\right)$ denote the classical \textit{top Bando-Futaki invariant} on $X$, with $Y$ and $\pcsoo$ same as above, we note the following about the expression (\ref{eq:logBFdef}):
\begin{multline}\label{eq:logBFcorrectfact}
\mathcal{F}_{\log; \bo, \boo} \left(Y, \pcsoo\right) = \mathcal{F} \left(Y, \pcsoo\right) \\
+ \text{Appropriate Correction Factors Corresponding to $S_0$ and $S_\infty$}
\end{multline}
where the expression $-\frac{1}{2 \left(2 \pi\right)^2} \int\limits_X \ttf \left(\lambda \left(\eta\right) - \lambda_0 \left(\eta\right)\right) \eta^2$ for the invariant $\mathcal{F} \left(Y, \pcsoo\right)$ is proven in \cite{Bando:2006:HarmonObstruct} and was used by the author in the smooth analogue of the problem of this paper in \cite{Sompurkar:2023:heKsmooth}. \par
If we ``na\"ively'' try to evaluate (like the integrals in (\ref{eq:wedgenaivint}) discussed in Subsection \ref{subsec:CanonKhlrConeSing}) the top $\log$ Bando-Futaki invariant given by (\ref{eq:logBFdef}) with respect to our momentum-constructed conical higher cscK metric $\omega$ (instead of the smooth \krm{} $\eta$) whose higher scalar curvature as a current is given by (\ref{eq:Cherncurrentomega}) on $X$, then the distributional terms thrown out by the currents of integration along $\so$ and $\soo$ will ``cancel out'' with the correction factors corresponding to $\so$ and $\soo$ which are present in the expression (\ref{eq:logBFdef}), thereby leaving behind only the expression of the classical top Bando-Futaki invariant on $\xsosoo$ which then has to be zero as the higher scalar curvature of the conical higher cscK metric $\omega$ is constant on $\xsosoo$. Over here the fact, that the top Chern current $c_2 \left(\omega\right)$ belongs to the correct \cohoml{} class which is top Chern class $c_2 \left(X\right)$ making the average higher scalar curvature of $\omega$ a \cohomll{} invariant (as we have shown in Subsection \ref{subsec:CohomolInvcurrent}), is playing a major role. \par
The rigorous proofs and the detailed computations regarding the top $\log$ Bando-Futaki invariant specifically applicable for the momentum construction on the minimal ruled surface $X$ will be shown here in Section \ref{sec:logFutlogMab}, but we note with caution that doing these same things in more general settings (other than the Calabi ansatz) seems to be very difficult as the rigorous justifications required in this special case simply boil down to computing some one variable integrals (similar to those in Subsection \ref{subsec:CohomolInvcurrent}) because of the many symmetries imposed by the Calabi ansatz procedure. After verifying (up to Conjecture \ref{conj:logBFconiccohomllinvar}) that the expression (\ref{eq:logBFdef}) is ``well-defined'' and provides a ``genuine'' Futaki-type invariant for our purpose, we will be using the smooth higher \ext{} \krm{} (which is not higher cscK) constructed in our previous work \cite{Sompurkar:2023:heKsmooth} for the `$\eta$' in (\ref{eq:logBFdef}) and the fact that momentum-constructed conical higher cscK metrics with some values of the cone angles $2 \pi \bo$ and $2 \pi \boo$ exist in each \krcl{} of the form $\pcsoo$ (given to us by Corollary \ref{cor:mainconesing}), to set the invariant in (\ref{eq:logBFdef}) to zero and obtain the (conjectured) linear relationship between the values of $\bo$ and $\boo$ that will come out as a result. \par
\begin{remark}\label{rem:EulerVF}
It is well known that the Lie algebra of all real holomorphic vector fields on the minimal ruled surface $X = \prj$ is precisely given by $\mathfrak{h} \left(X\right) = \R \left\lbrace w \frac{\partial}{\partial w} \right\rbrace$ where the generator $w \frac{\partial}{\partial w}$ of the one-dimensional real Lie algebra $\mathfrak{h} \left(X\right)$ is called as the \textit{Euler vector field} on $X$ (see for example Maruyama \cite{Maruyama:1971:autgrpruledsurf} and T{\o}nnesen-Friedman \cite{Tonnesen:1998:eKminruledsurf} for the proof of this fact and \cite{Fujiki:1992:eKruledmani,Hwang:2002:MomentConstruct,Szekelyhidi:2014:eKintro,Szekelyhidi:2006:eKKStab} for its applications). The Euler vector field $w \frac{\partial}{\partial w} = - \tilde{w} \frac{\partial}{\partial \tilde{w}}$ (where $\tilde{w} = w^{-1}$) is trivially parallel to both the divisors $S_0$ and $S_\infty$ of $X$ and on top of it, is parallel to its fibres $\sfc$ (which are bi\hol{} to $\cp^1$) as well. Also note that $S_0$ and $S_\infty$ are both bi\hol{} to the base \rms{} $\Sigma$ and $\fkh \left(\Sigma\right) = \stzero$ as $\Sigma$ is a compact \rms{} of genus $2$ \cite{Barth:2004:CmpctCmplxSurf,Szekelyhidi:2014:eKintro}, and so any globally defined \hol{} vector field on $X$ which is parallel to both $\so$ and $\soo$ anyways has to be trivially parallel to them. \\
$w \frac{\partial}{\partial w}$ is easily seen to be a \textit{gradient \hol{} vector field} on $X$, i.e. it has got a \textit{holomorphy potential} $\ttf$ defined on $X$ for any given smooth \krm{} $\eta$ on $X$. If in particular this $\eta$ is taken to be a momentum-constructed smooth \krm{} on $X$ with $\rho \left(s\right)$ being the strictly convex smooth function defining its Calabi ansatz of the form (\ref{eq:ansatzeta}), then $\ttf : X \to \R$ computed with respect to $\eta$ has a very nice and simple expression in terms of $\rho \left(s\right)$ as can be seen below (using the coordinate expression of the form (\ref{eq:omegaconesing}) for $\eta$):
\begin{gather}
\left( w \frac{\partial}{\partial w} \right)^{\flat_\eta} \left(Z\right) = q \left(w \frac{\partial}{\partial w}, Z\right) = \left(1 + \rho' \left(s\right)\right) \mathtt{p}^* g_\Sigma \left(w \frac{\partial}{\partial w}, Z\right) + 2 \rho'' \left(s\right) \frac{\left\lvert d w \right\rvert^2}{\left\lvert w \right\rvert^2} \left(w \frac{\partial}{\partial w}, Z\right) = \rho'' \left(s\right) \frac{d \bar{w}}{\bar{w}} \left(Z\right) \nonumber \\
\impl \nabla_\eta^{\left(1,0\right)} \left(1 + \rho' \left(s\right)\right) = \left( \db \left(1 + \rho' \left(s\right)\right) \right)^{\sharp_\eta} = \left( \rho'' \left(s\right) \frac{d \bar{w}}{\bar{w}} \right)^{\sharp_\eta} = w \frac{\partial}{\partial w} \impl \ttf = 1 + \rho' \left(s\right) \label{eq:holpotEulerVF}
\end{gather}
where $\flat_\eta$ and $\sharp_\eta$ are the musical isomorphisms induced by $\eta$ acting on the $\left(1,0\right)$-vector fields and the $\left(0,1\right)$-forms on $X$ respectively, $q$ and $g_\Sigma$ are the \hmnm{s} on $X$ and $\Sigma$ associated with the \krf{s} $\eta$ and $\omega_\Sigma$ respectively and $Z$ is any real smooth tangent vector field on $X$. \\
Expression (\ref{eq:logBFdef}) is clearly seen to be linear in $Y \in \mathfrak{h} \left(X\right)$ as the holomorphy potential $\ttf$ itself changes linearly with $Y$ for a fixed smooth \krm{} $\eta$ on $X$. So we can simply take the `$Y$' in (\ref{eq:logBFdef}) to be the Euler vector field $w \frac{\partial}{\partial w}$ which will greatly simplify the efforts needed to prove the results of Section \ref{sec:logFutlogMab} for the top $\log$ Bando-Futaki invariant.
\end{remark} \par
\begin{theorem}[The Top $\log$ Bando-Futaki Invariant as a Function of the K\"ahler Class]\label{thm:logBFwelldefined}
The object $\mathcal{F}_{\log; \bo, \boo} \left(w \frac{\partial}{\partial w}, \cdot\right)$ defined by the expression (\ref{eq:logBFdef}) is a function of the \krcl{} $\pcsoo$ alone and does not depend on the choice of the smooth \krm{} $\eta$ belonging to $\pcsoo$ and also does not depend on the choice of the real holomorphy potential $\ttf$ of $w \frac{\partial}{\partial w}$ with respect to $\eta$.
\end{theorem}
\begin{proof}
Since any two real holomorphy potentials of $w \frac{\partial}{\partial w}$ with respect to the same smooth \krm{} $\eta$ differ by a constant \cite{Szekelyhidi:2014:eKintro}, the expression (\ref{eq:logBFdef}) is very clearly seen to be independent of the choice of the real holomorphy potential $\ttf$. \\
As was noted in the expressions (\ref{eq:logBFdef}) and (\ref{eq:logBFcorrectfact}), $\mathcal{F} \left(w \frac{\partial}{\partial w}, \eta\right) = -\frac{1}{2 \left(2 \pi\right)^2} \int\limits_X \ttf \left(\lambda \left(\eta\right) - \lambda_0 \left(\eta\right)\right) \eta^2$ is the classical top Bando-Futaki invariant on $X$, and the proof of the fact that it is an invariant of the \krcl{} (i.e. it does not depend on the choice of the smooth \krm{} in the \krcl{}) is given in Bando \cite{Bando:2006:HarmonObstruct}; Theorem 1 in the most general setting of any compact \kr{} $n$-manifold. So it remains for us to check that the correction factors in (\ref{eq:logBFdef}) accounting for the conical singularities at the two divisors of $X$ do not depend on the choice of $\eta \in \pcsoo$. The proof of this thing uses the same arguments as those used in the proof of the invariance of the classical Futaki invariant by Futaki \cite{Futaki:1983:ObstructKE}; Section 1 and Calabi \cite{Calabi:1985:eK2}; Theorem 4, Section 4, and an exposition of this can be found in Sz\'ekelyhidi \cite{Szekelyhidi:2014:eKintro}; Theorem 4.21 (again applicable for general compact \kr{} $n$-manifolds). \\
Let $\eta, \eta' \in \pcsoo$ be any two smooth \krm{s} on $X$, then there exists a smooth function $\varphi : X \to \R$ such that $\eta' = \eta + \ideldb \varphi$ \cite{Szekelyhidi:2014:eKintro}. Define $\eta_t = \eta + t \ideldb \varphi$ for $t \in \left[0, 1\right]$, then $\eta_t \in \pcsoo$ \cite{Szekelyhidi:2014:eKintro}. Just like in Sz\'ekelyhidi \cite{Szekelyhidi:2014:eKintro}; Theorem 4.21, our task here is to check that $\frac{d}{d t} \bigr\rvert_{t = 0}$, acted separately on both the correction factor terms of the expression (\ref{eq:logBFdef}) evaluated with respect to $\eta_t$, turns out to be zero. Following \cite{Szekelyhidi:2014:eKintro} we have these expressions for the concerned terms in (\ref{eq:logBFdef}):
\begin{equation}\label{eq:dert0etat2}
\frac{d}{d t} \biggr\rvert_{t = 0} \left(\eta_t\right) = \ideldb \varphi \hspace{1pt}, \hspace{4pt} \frac{d}{d t} \biggr\rvert_{t = 0} \left(\eta_t^2\right) = 2 \ideldb \varphi \wedge \eta = \left(\Delta_\eta \varphi\right) \eta^2
\end{equation}
where $\Delta_\eta = - \bar{\partial}^{*_\eta} \bar{\partial}$ is the $\bar{\partial}$-Laplacian operator on $X$ induced by $\eta$ and $\bar{\partial}^{*_\eta}$ is the formal adjoint of $\bar{\partial}$. Since $\nabla_\eta^{\left(1,0\right)} \ttf = \left(\bar{\partial} \ttf\right)^{\sharp_\eta} = w \frac{\partial}{\partial w}$, it can be checked that $\nabla_{\eta_t}^{\left(1,0\right)} \left(\ttf + t w \frac{\partial \varphi}{\partial w}\right) = \left(\bar{\partial} \left(\ttf + t w \frac{\partial \varphi}{\partial w}\right)\right)^{\sharp_{\eta_t}} = w \frac{\partial}{\partial w}$, i.e. $\ttf_t = \ttf + t w \frac{\partial \varphi}{\partial w}$ is the real holomorphy potential of $w \frac{\partial}{\partial w}$ with respect to $\eta_t$ \cite{Szekelyhidi:2014:eKintro}, from which follows:
\begin{equation}\label{eq:dert0ft}
\frac{d}{d t} \biggr\rvert_{t = 0} \left(\ttf_t\right) = w \frac{\partial \varphi}{\partial w} = \nabla_\eta^{\left(1,0\right)} \ttf \left(\varphi\right)
\end{equation}
We can now compute the variations of the integrals appearing in (\ref{eq:logBFdef}) using the equations (\ref{eq:dert0etat2}), (\ref{eq:dert0ft}):
\begin{equation}\label{eq:dert0intftetat2}
\frac{d}{d t} \biggr\rvert_{t = 0} \left( \int\limits_X \ttf_t \eta_t^2 \right) = \int\limits_X \left( \nabla_\eta^{\left(1,0\right)} \ttf \left(\varphi\right) \eta^2 + \ttf \left(\Delta_\eta \varphi\right) \eta^2 \right) = \int\limits_X \left(\bar{\partial} \ttf\right)^{\sharp_\eta} \left(\varphi\right) \eta^2 - \int\limits_X \ttf \left(\bar{\partial}^{*_\eta} \bar{\partial} \varphi\right) \eta^2 = 0
\end{equation}
where in (\ref{eq:dert0intftetat2}), we can write down the explicit expressions for the entities $\left(\bar{\partial} \ttf\right)^{\sharp_\eta} \left(\varphi\right)$ and $\bar{\partial}^{*_\eta} \bar{\partial} \varphi$ in terms of the underlying \hmnm{} $q$ of the \krf{} $\eta$ and the local \hol{} coordinates $\left(z, w\right)$ on $X$ and then use integration by parts to obtain the answer as zero, just as done in Sz\'ekelyhidi \cite{Szekelyhidi:2014:eKintro}; Theorem 4.21. We show this same computation applicable for the integral over $\so$ present in (\ref{eq:logBFdef}):
\begin{equation}\label{eq:dert0intftetat}
\frac{d}{d t} \biggr\rvert_{t = 0} \left( \int\limits_{\so} \ttf_t \eta_t \right) = \int\limits_{\so} \left(\bar{\partial} \ttf\right)^{\sharp_\eta} \left(\varphi\right) \eta - \int\limits_{\so} \ttf \left(\bar{\partial}^{*_\eta} \bar{\partial} \varphi\right) \eta = \int\limits_{\so} \frac{\del \ttf}{\del \bar{z}} \frac{\del \varphi}{\del z} \sqrt{-1} d z \wedge d \bar{z} + \int\limits_{\so} \ttf \frac{\del^2 \varphi}{\del z \del \bar{z}} \sqrt{-1} d z \wedge d \bar{z} = 0
\end{equation}
since $\ideldb \varphi = \left(\Delta_\eta \varphi\right) \eta$ holds true on $\so$ from the same relation as in (\ref{eq:dert0etat2}). The variation of the integral over $\soo$ present in (\ref{eq:logBFdef}) is shown to be zero by the same computation as (\ref{eq:dert0intftetat}). \\
As far as the volume terms appearing in (\ref{eq:logBFdef}) are concerned, they are very well known to depend only on the \krcl{} of the smooth \krm{} \cite{Szekelyhidi:2014:eKintro}:
\begin{equation}\label{eq:volxsosookrcl}
\begin{gathered}
\int\limits_X \eta_t^2 = \Vol \left(X, \eta_t\right) = \left[\eta_t\right] \smile \left[\eta_t\right] = \left[\eta_t\right] \cdot \left[\eta_t\right] = \left(\pcsoo\right)^2 = \left(2 \pi\right)^2 m \left(m + 2\right) \\
\int\limits_{\so} \eta_t = \Vol \left(\so, \eta_t\right) = \left[\eta_t\right] \smile \so = \left[\eta_t\right] \cdot \so = \pcsoo \cdot \so = 2 \pi \\
\int\limits_{\soo} \eta_t = \Vol \left(\soo, \eta_t\right) = \left[\eta_t\right] \smile \soo = \left[\eta_t\right] \cdot \soo = \pcsoo \cdot \soo = 2 \pi \left(m + 1\right)
\end{gathered}
\end{equation}
where $\smile$ denotes the cup product, $\cdot$ denotes the intersection product, $S_0$, $S_\infty$ and $\mathsf{C}$ denote the Poincar\'e duals of these complex curves on the complex surface $X$ respectively which will be elements of the de Rham \cohoml{} space $H^{\left(1,1\right)} \left(X, \mathbb{R}\right)$ \cite{Barth:2004:CmpctCmplxSurf,Szekelyhidi:2014:eKintro} and we can use the intersection formulae (\ref{eq:IntersectForm}) to compute the volumes explicitly.
\end{proof} \par
\begin{motivation}
Theorem \ref{thm:logBFwelldefined} proves the invariance of the top $\log$ Bando-Futaki invariant only with respect to smooth \krm{s} coming from the fixed \krcl{} $\pcsoo$ on $X$. But the top $\log$ Bando-Futaki invariant is a concept meant for conical higher cscK metrics, and so we are supposed to check that it remains invariant even with respect to conical \krm{s} in $\pcsoo$. But for this, we first need to be sure that we are ``allowed to evaluate'' $\mathcal{F}_{\log; \bo, \boo} \left(w \frac{\partial}{\partial w}, \cdot\right)$ with respect to a conical \krm{} on $X$, because it is not readily clear why the integrals present in the expression (\ref{eq:logBFdef}) will make sense if we substitute a conical \krm{} $\omega$ in place of the smooth \krm{} $\eta$ over there. Doing this for general conical \krm{s} is out of hand, but as we are in a very nice situation of the Calabi ansatz on the minimal ruled surface $X$, we can consider only momentum-constructed conical \krm{s} $\omega$ belonging to the \krcl{} $\pcsoo$ and prove the following expected results for them:
\end{motivation} \par
\begin{lemma}\label{lem:volxsosooavghighScalconic}
Let $\omega$ be a momentum-constructed conical \krm{} on $X$ with cone angles $2 \pi \bo$ and $2 \pi \boo$ along $\so$ and $\soo$ respectively, given by the ansatz (\ref{eq:ansatzconesing}) and belonging to the \krcl{} $\pcsoo$. Then the volumes of $X$, $\so$ and $\soo$, computed with respect to the respective volume forms induced by $\omega$ on these \krmf{s}, are well-defined and are invariants of the \krcl{} $\pcsoo$ (just like in the case of a smooth \krm{}). Further the \textit{average higher scalar curvature} of $\omega$ on the whole of $X$, which is defined as $\lambda_0 \left(\omega\right) = 2 \left(2 \pi\right)^2 \frac{\int\limits_{X} c_2 \left(\omega\right)}{\int\limits_{X} \omega^2}$ where the top Chern current $c_2 \left(\omega\right)$ is given by the expression of currents (\ref{eq:Cherncurrentomega}) on $X$, is a \cohomll{} invariant of the \krcl{} $\pcsoo$ and satisfies the \cohomll{} equation (\ref{eq:avghighScal}) made applicable for $X$.
\end{lemma}
\begin{proof}
The volumes of $X$, $\so$ and $\soo$ with respect to the volume forms induced by $\omega$ on them are computed as follows:
\begin{equation}\label{eq:volxomega}
\Vol \left(X, \omega\right) = \int\limits_X \omega^2 = 2 \iint\limits_{\prj} \gamma \phi \ttp^* \omega_\Sigma \wedge \sqrt{-1} \frac{d w \wedge d \bar{w}}{\left\lvert w \right\rvert^2} = 2 \left(2 \pi\right)^2 \int\limits_1^{m + 1} \gamma d \gamma = \left(2 \pi\right)^2 m \left(m + 2\right)
\end{equation}
where we used the coordinate expression (\ref{eq:omega2conesing}) for $\omega^2$ and then evaluated the integral just like in (\ref{eq:intCherncalceta}) and (\ref{eq:intCherncalcomega}).
\begin{equation}\label{eq:volsoomega}
\Vol \left(\so, \omega\right) = \int\limits_{\so} \omega = \omega \wedge \left[\so\right] \left(1\right) = \mathtt{p}^* \omega_\Sigma \wedge \left[\so\right] \left(1\right) = \int\limits_{\so} \mathtt{p}^* \omega_\Sigma = \int\limits_{\Sigma} \omega_\Sigma = 2 \pi
\end{equation}
where we used the equations (\ref{eq:currentomegazeroBTprod}) and (\ref{eq:Cherncurrentomega'epsilon}) justifying the wedge product above. Similarly using (\ref{eq:currentomegainfinityBTprod}) instead of (\ref{eq:currentomegazeroBTprod}) we get:
\begin{equation}\label{eq:volsooomega}
\Vol \left(\soo, \omega\right) = \int\limits_{\soo} \omega = \left(m + 1\right) \int\limits_{\soo} \mathtt{p}^* \omega_\Sigma = \left(m + 1\right) \int\limits_{\Sigma} \omega_\Sigma = 2 \pi \left(m + 1\right)
\end{equation}
Comparing the values obtained for the volumes of $X$, $\so$ and $\soo$ in (\ref{eq:volxomega}), (\ref{eq:volsoomega}) and (\ref{eq:volsooomega}) respectively with those obtained in (\ref{eq:volxsosookrcl}) where the underlying \krm{} was smooth, we see that the volumes computed even with respect to the conical \krm{} $\omega$ turn out to be invariants of the \krcl{}. \\
With $\Vol \left(X, \omega\right)$ being clear from (\ref{eq:volxomega}), we can define and compute the \textit{average higher scalar curvature} of $\omega$ on $X$ as follows:
\begin{equation}\label{eq:avghighScalomega}
\lambda_0 \left(\omega\right) = 2 \left(2 \pi\right)^2 \frac{\int\limits_{X} c_2 \left(\omega\right)}{\int\limits_{X} \omega^2} = -\frac{8}{m \left(m + 2\right)}
\end{equation}
where the top Chern current $c_2 \left(\omega\right)$ is given by (\ref{eq:Cherncurrentomega}) and its integral over $X$ is computed in (\ref{eq:intCherncalcomega}). To check the \cohomll{} invariance of $\lambda_0 \left(\omega\right)$ calculated above in (\ref{eq:avghighScalomega}), let $\eta$ be a smooth \krm{} (momentum-constructed or otherwise) coming from the same \krcl{} $\pcsoo$, then the average higher scalar curvature of $\eta$, which is given by $\lambda_0 \left(\eta\right) = \frac{\int\limits_{X} \lambda \left(\eta\right) \eta^2}{\int\limits_{X} \eta^2}$, will satisfy the equation (\ref{eq:avghighScal}) applied to the surface $X$, as $\eta$ is smooth on $X$:
\begin{equation}\label{eq:avghighScaleta}
\lambda_0 \left(\eta\right) = \frac{2 \left(2 \pi\right)^2 c_2 \left(X\right)}{\left[ \eta \right]^2} = -\frac{8}{m \left(m + 2\right)}
\end{equation}
where the value of $\left[ \eta \right]^2 = \Vol \left(X, \eta\right)$ is given in (\ref{eq:volxsosookrcl}), and the top Chern class $c_2 \left(X\right)$ is numerically represented by $-4$ in the top-dimensional real \cohoml{} space $H^4 \left(X, \mathbb{R}\right)$ as can be observed from the values of the integrals computed in (\ref{eq:intCherncalceta}) and (\ref{eq:intCherncalcomega}). Thus from (\ref{eq:avghighScalomega}) and (\ref{eq:avghighScaleta}) we see that $\lambda_0 \left(\omega\right)$ is a \cohomll{} invariant of the \krcl{} $\pcsoo$ and satisfies the equation (\ref{eq:avghighScal}) on $X$.
\end{proof} \par
\begin{remark}\label{rem:lambda0lambda1omega}
The result obtained in Subsection \ref{subsec:CohomolInvcurrent} directly gave the second assertion of Lemma \ref{lem:volxsosooavghighScalconic}. But note that for the (momentum-constructed) conical \krm{} $\omega$, the higher scalar curvature $\lambda \left(\omega\right)$ as a smooth function is defined only on $\xsosoo$ though the expression (\ref{eq:Cherncurrentomega}) gives it as a current on the whole of $X$. If in Lemma \ref{lem:volxsosooavghighScalconic} we had taken the average of $\lambda \left(\omega\right)$ only on $\xsosoo$ instead of taking the average on the whole of $X$ as in (\ref{eq:avghighScalomega}), then the \textit{average higher scalar curvature} of $\omega$ on $\xsosoo$, which is given simply by $\lambda_1 \left(\omega\right) = \frac{\int\limits_{\xsosoo} \lambda \left(\omega\right) \omega^2}{\int\limits_{\xsosoo} \omega^2}$, will not be a de Rham \cohomll{} invariant, as the top Chern form $c_2 \left(\omega\right) \bigr\rvert_{X \smallsetminus \left(S_0 \cup S_\infty\right)}$ restricted to the non-compact $\xsosoo$, which is precisely given by the equation (\ref{eq:topChernXminus}), is not a \cohomll{} representative of the top Chern class $c_2 \left(X\right)$ (as we had remarked in Subsection \ref{subsec:CohomolInvcurrent}).
\begin{align}\label{eq:avghighScal1omega}
\lambda_1 \left(\omega\right) &= \frac{2 \int\limits_{\xsosoo} \mathtt{p}^* \omega_\Sigma \wedge \sqrt{-1} \frac{d w \wedge d \bar{w}}{\left\lvert w \right\rvert^2} \frac{\phi}{\gamma^2} \left( \gamma \left(\phi + 2 \gamma\right) \phi'' + \phi' \left(\phi' \gamma - \phi\right) \right)}{\left(2 \pi\right)^2 m \left(m + 2\right)} \nonumber \\
&= \frac{2 \int\limits_1^{m + 1} \frac{d}{d \gamma} \left( \left(\frac{\phi}{\gamma} + 2\right) \phi' \right) d \gamma}{m \left(m + 2\right)} = -\frac{4 \left(\beta_0+\beta_\infty\right)}{m \left(m + 2\right)}
\end{align}
where $\int\limits_{\xsosoo} \omega^2 = \Vol \left(\xsosoo, \omega\right) = \left(2 \pi\right)^2 m \left(m + 2\right)$ from (\ref{eq:volxomega}), the expression for $\lambda \left(\omega\right) \omega^2$ in terms of the momentum profile $\phi \left(\gamma\right)$ is given by (\ref{eq:Chernconesing}), and then the integral in (\ref{eq:avghighScal1omega}) is evaluated just by following the computation (\ref{eq:intCherncalcomega}). We can now directly see the equation (\ref{eq:lambda0lambda1}) which is supposed to relate the two quantities $\lambda_0 \left(\omega\right)$ and $\lambda_1 \left(\omega\right)$ in the special case of the momentum construction over here (by comparing the values of the respective averages obtained in (\ref{eq:avghighScalomega}) and (\ref{eq:avghighScal1omega}) and by retrieving the values of the volume terms from (\ref{eq:volxomega}), (\ref{eq:volsoomega}) and (\ref{eq:volsooomega})):
\begin{equation}\label{eq:lambda0lambda1omega}
\lambda_0 \left(\omega\right) = \lambda_1 \left(\omega\right) + 8 \pi \left(\bo - 1\right) \frac{\int\limits_{S_0} \omega}{\int\limits_X \omega^2} + \frac{8 \pi \left(\boo - 1\right)}{m + 1} \frac{\int\limits_{S_\infty} \omega}{\int\limits_X \omega^2}
\end{equation}
\end{remark} \par
\begin{theorem}[The Top $\log$ Bando-Futaki Invariant Evaluated for a Momentum-Constructed Conical K\"ahler Metric]\label{thm:logBFconicevalue}
Let $\omega$ be a momentum-constructed conical \krm{} on $X$ like in Lemma \ref{lem:volxsosooavghighScalconic}. Then $\mathcal{F}_{\log; \bo, \boo} \left(w \frac{\partial}{\partial w}, \omega\right)$ makes sense, i.e. all the integrals in the expression (\ref{eq:logBFdef}) evaluated with respect to $\omega$ are well-defined and finite. Further, the evaluation of (\ref{eq:logBFdef}) for $\omega$ is the following:
\begin{equation}\label{eq:logBFconicevalue}
\mathcal{F}_{\log; \bo, \boo} \left(w \frac{\partial}{\partial w}, \omega\right) = -\frac{1}{2 \left(2 \pi\right)^2} \int\limits_{\xsosoo} \tth \left(\lambda \left(\omega\right) - \lambda_1 \left(\omega\right)\right) \omega^2
\end{equation}
where $\tth : \xsosoo \to \R$ is the real \holy{} potential of $w \frac{\partial}{\partial w}$ computed with respect to $\omega$ by using (\ref{eq:holpotEulerVF}), $\lambda \left(\omega\right)$ as a smooth function on $\xsosoo$ is given by (\ref{eq:lambdaconesing}) and $\lambda_1 \left(\omega\right)$ is given by (\ref{eq:avghighScal1omega}). In particular if $\omega$ is the momentum-constructed conical higher cscK metric on $X$ yielded by Theorem \ref{thm:mainconesing} and Corollary \ref{cor:mainconesing}, then $\mathcal{F}_{\log; \bo, \boo} \left(w \frac{\partial}{\partial w}, \omega\right) = 0$.
\end{theorem}
\begin{proof}
First note a subtle technical point that the first term in the expression (\ref{eq:logBFdef}) computed with respect to the conically singular $\omega$ needs to be interpreted globally on the whole of $X$ and not just on $\xsosoo$, and so the quantity $\lambda \left(\omega\right) \omega^2$ appearing in (\ref{eq:logBFdef}) after substituting $\omega$ should be ``correctly read'' as $2 \left(2 \pi\right)^2 c_2 \left(\omega\right)$ with $c_2 \left(\omega\right)$ being given by the current expression (\ref{eq:Cherncurrentomega}) on $X$, and $\lambda_0 \left(\omega\right)$ will then be defined by (\ref{eq:avghighScalomega}). \\
With this, with the expressions (\ref{eq:volxomega}), (\ref{eq:volsoomega}) and (\ref{eq:volsooomega}) obtained in Lemma \ref{lem:volxsosooavghighScalconic} and using (\ref{eq:avghighScal1omega}) as well as (\ref{eq:lambda0lambda1omega}) from \textit{Remark} \ref{rem:lambda0lambda1omega}, we will ``try to evaluate'' the expression (\ref{eq:logBFdef}) with respect to $\omega$:
\begin{gather}
\mathcal{F}_{\log; \bo, \boo} \left(w \frac{\partial}{\partial w}, \omega\right) = - \int\limits_X \tth c_2 \left(\omega\right) + \frac{\lambda_0 \left(\omega\right)}{2 \left(2 \pi\right)^2} \int\limits_X \tth \omega^2 \label{eq:logBFomegaevalue} \\
+ \frac{\beta_0 - 1}{\pi} \int\limits_{S_0} \tth \omega - \frac{\beta_0 - 1}{\pi} \frac{\int\limits_{S_0} \omega}{\int\limits_X \omega^2} \int\limits_X \tth \omega^2 + \frac{\beta_\infty - 1}{\left(m + 1\right) \pi} \int\limits_{S_\infty} \tth \omega - \frac{\beta_\infty - 1}{\left(m + 1\right) \pi} \frac{\int\limits_{S_\infty} \omega}{\int\limits_X \omega^2} \int\limits_X \tth \omega^2 \nonumber \\
= -\frac{1}{2 \left(2 \pi\right)^2} \int\limits_X \tth \lambda \left(\omega\right) \omega^2 - \frac{\beta_0 - 1}{\pi} \omega \wedge \left[\so\right] \left(\tth\right) - \frac{\beta_\infty - 1}{\left(m + 1\right) \pi} \omega \wedge \left[\soo\right] \left(\tth\right) + \frac{\beta_0 - 1}{\pi} \int\limits_{S_0} \tth \omega + \frac{\beta_\infty - 1}{\left(m + 1\right) \pi} \int\limits_{S_\infty} \tth \omega \nonumber \\
- \frac{4}{\left(2 \pi\right)^2 m \left(m + 2\right)} \int\limits_X \tth \omega^2 - \frac{2 \left(\beta_0 - 1\right)}{\left(2 \pi\right)^2 m \left(m + 2\right)} \int\limits_X \tth \omega^2 - \frac{2 \left(\beta_\infty - 1\right)}{\left(2 \pi\right)^2 m \left(m + 2\right)} \int\limits_X \tth \omega^2 \nonumber \\
= -\frac{1}{2 \left(2 \pi\right)^2} \int\limits_X \tth \lambda \left(\omega\right) \omega^2 + \frac{\lambda_1 \left(\omega\right)}{2 \left(2 \pi\right)^2} \int\limits_X \tth \omega^2 = -\frac{1}{2 \left(2 \pi\right)^2} \int\limits_{\xsosoo} \tth \left(\lambda \left(\omega\right) - \lambda_1 \left(\omega\right)\right) \omega^2 \nonumber
\end{gather}
Since $\omega$ is constructed by the Calabi ansatz (\ref{eq:ansatzconesing}) with the convex function $f \left(s\right)$, we can use (\ref{eq:holpotEulerVF}) for the \hmnm{} $g$ associated with the \krf{} $\omega$ on $\xsosoo$ (since $\omega$ is smooth only on $\xsosoo$) to derive the \holy{} potential $\tth = 1 + f' \left(s\right)$ on $\xsosoo$, and then the limiting values of the variables involved in the momentum construction which are given in (\ref{eq:wgammasosoo}) or even more precisely the asymptotic expressions (\ref{eq:omegapolyhomo}) and (\ref{eq:f2sf1sfspolyhomotilde}) (or (\ref{eq:omegasemihomo}) and (\ref{eq:f2sf1sfssemihomotilde})) concerning the quantity $f' \left(s\right)$ will allow us to call $\tth = 1 + f' \left(s\right)$ as the real holomorphy potential of $w \frac{\partial}{\partial w}$ with respect to $\omega$ on the whole of $X$. \\
But $1 + f' \left(s\right)$ is smooth in the coordinates $\left(z, w\right)$ (or even in the coordinates $\left(z, \tilde{w}\right)$) only on $\xsosoo$, while near $\so$ and $\soo$ it is smooth only when considered as a function of the conical coordinates $w^{\bo}$ (or $\left\lvert w \right\rvert^{\bo - 1} w$) and $\tilde{w}^{\boo}$ (or $\left\lvert \tilde{w} \right\rvert^{\boo - 1} \tilde{w}$) respectively (see Definition \ref{def:coneKr2}), as was shown in Section \ref{sec:polyhomoConeSing} in the asymptotic expressions (\ref{eq:omegapolyhomo}) and (\ref{eq:omegasemihomo}). So the calculation (\ref{eq:logBFomegaevalue}) is valid only if we can rigorously make sense of the quantities $\omega \wedge \left[\so\right] \left(\tth\right) = \int\limits_{S_0} \tth \omega$ and $\omega \wedge \left[\soo\right] \left(\tth\right) = \int\limits_{S_\infty} \tth \omega$ for $\tth = 1 + f' \left(s\right)$, as currents are usually acted on test functions that are smooth everywhere. \\
The wedge products of currents $\omega \wedge \left[\so\right]$ and $\omega \wedge \left[\soo\right]$ are computed using Bedford-Taylor theory in (\ref{eq:currentomegazeroBTprod}) and (\ref{eq:currentomegainfinityBTprod}) respectively, and are also justified by taking smooth approximations in (\ref{eq:Cherncurrentomega'epsilon}). Since the limit of the wedge product in (\ref{eq:Cherncurrentomega'epsilon}) was verified by the limit of the integral in (\ref{eq:PLepsilon}), let us now take the test function $\varphi$ in (\ref{eq:PLepsilon}) to be smooth only away from $\so$ and $\soo$ with $\varphi$ being smooth only as a function of $\left\lvert w \right\rvert^{\bo}$ and $\left\lvert \tilde{w} \right\rvert^{\boo}$ near $\so$ and $\soo$ respectively (just like our function $1 + f' \left(s\right)$). Then we can proceed just as in (\ref{eq:PLepsilon}) only being careful that $\varphi$ is not the usual smooth test function over here:
\begin{align}\label{eq:PLepsilonnonsmooth}
&\omega \wedge \left[\so\right] \left(\varphi\right) = \lim\limits_{\epsilon \to 0} \frac{1}{2 \pi} \int\limits_X \varphi \omega_\epsilon \wedge \sqrt{-1} \deldb \ln \left(\left\lvert w \right\rvert^2 + \epsilon^2\right) \\
&= \lim\limits_{\epsilon \to 0} \frac{1}{2 \pi} \int\limits_X \varphi \left(1 + f_\epsilon' \left(s\right)\right) \frac{\epsilon^2}{\left(\left\lvert w \right\rvert^2 + \epsilon^2\right)^2} \mathtt{p}^* \omega_\Sigma \wedge \sqrt{-1} d w \wedge d \bar{w} \nonumber \\
&= \lim\limits_{\epsilon \to 0} \frac{1}{2 \pi} \int\limits_0^{2\pi} \int\limits_0^\infty \frac{2r \epsilon^2}{\left(r^2 + \epsilon^2\right)^2} \left( \int\limits_\Sigma \varphi \left(1 + f_\epsilon' \left(s\right)\right) \omega_\Sigma \right) dr d\theta \nonumber \\
&= \lim\limits_{\epsilon \to 0} \frac{1}{2 \pi} \int\limits_0^{2\pi} \left[ \left( -\frac{\epsilon^2}{r^2 + \epsilon^2} \int\limits_\Sigma \varphi \left(1 + f_\epsilon' \left(s\right)\right) \omega_\Sigma \right) \Biggr\rvert_0^\infty + \int\limits_0^\infty \frac{\epsilon^2}{r^2 + \epsilon^2} \frac{d}{dr} \left( \int\limits_\Sigma \varphi \left(1 + f_\epsilon' \left(s\right)\right) \omega_\Sigma \right) dr \right] d\theta \nonumber \\
&= \int\limits_{\so} \varphi \left(z, 0\right) \ttp^* \omega_\Sigma + \lim\limits_{\epsilon \to 0} \frac{1}{2 \pi} \int\limits_0^{2\pi} \int\limits_0^\infty \frac{\epsilon^2}{r^2 + \epsilon^2} \int\limits_\Sigma \frac{d}{dr} \left( \varphi \left(1 + f_\epsilon' \left(s\right)\right) \right) \omega_\Sigma dr d\theta \nonumber \\
&= \int\limits_{\so} \varphi \left(z, 0\right) \ttp^* \omega_\Sigma = \ttp^* \omega_\Sigma \wedge \left[\so\right] \left(\varphi\right) \nonumber
\end{align}
(\ref{eq:PLepsilonnonsmooth}) differs from (\ref{eq:PLepsilon}) at only one place, viz. for showing that the integral in the second term goes to zero as $\epsilon \to 0$, we have to note that here $\frac{d \varphi}{d r} \in O \left(r^{\beta_0 - 1}\right)$ as $r \to 0$ (because $\varphi$ itself is asymptotically of the order of $r^{\beta_0}$ near $r = 0$) and hence $\frac{d \varphi}{d r}$ is \intble{} (even if possibly unbounded) near $r = 0$ as long as $\bo > 0$, and everything else here is exactly the same as in (\ref{eq:PLepsilon}) and (\ref{eq:Cherncurrentomega'epsilon}), so dominated convergence theorem applies in this case as well to give us the expected answer. Similarly as done in (\ref{eq:PLepsilonnonsmooth}) above but with $\tilde{r}^{\boo} = \left\lvert \tilde{w} \right\rvert^{\boo}$ instead, we can show $\omega \wedge \left[\soo\right] \left(\varphi\right) = \left(m + 1\right) \int\limits_{\soo} \varphi \left(z, 0\right) \ttp^* \omega_\Sigma = \left(m + 1\right) \ttp^* \omega_\Sigma \wedge \left[\soo\right] \left(\varphi\right)$. We could have also viewed the wedge products $\omega \wedge \left[\so\right]$ and $\omega \wedge \left[\soo\right]$ as given by Bedford-Taylor theory but now acted on a non-smooth test function of the kind of this $\varphi$ in the following way: The expressions for these wedge products in (\ref{eq:Cherncurrentomega'}), (\ref{eq:currentomegazeroBTprod}) and (\ref{eq:currentomegainfinityBTprod}) imply that the conically singular positive $\left(1, 1\right)$-form $\omega$, given by local coordinate expressions of the form (\ref{eq:omegaconesing}) in $\left(z, w\right)$ and $\left(z, \tilde{w}\right)$ near $\so$ and $\soo$ respectively, is \intble{} over $\so$ and $\soo$, and $\left\lvert w \right\rvert^{\bo}$ and $\left\lvert \tilde{w} \right\rvert^{\boo}$ are bounded non-negative measurable functions near $w = 0$ and $\tilde{w} = 0$ respectively, thus implying that $\left\lvert w \right\rvert^{\bo} \omega$ and $\left\lvert \tilde{w} \right\rvert^{\boo} \omega$ should also be \intble{} over $\so$ and $\soo$ as follows:
\begin{equation}\label{eq:BTprodnonsmooth}
\begin{gathered}
\omega \wedge \left[\so\right] \left(\varphi\right) = \int\limits_{\so} \varphi \omega = \int\limits_{\so} \varphi \left(1 + f' \left(s\right)\right) \mathtt{p}^* \omega_\Sigma + \int\limits_{\so} \varphi \frac{f'' \left(s\right)}{\left\lvert w \right\rvert^2} \sqrt{-1} d w \wedge d \bar{w} \\
= \int\limits_{\so} \varphi \left(1 + f' \left(s\right)\right) \bigr\rvert_{w = 0} \mathtt{p}^* \omega_\Sigma = \int\limits_{\so} \varphi \left(z, 0\right) \mathtt{p}^* \omega_\Sigma = \ttp^* \omega_\Sigma \wedge \left[\so\right] \left(\varphi\right)
\end{gathered}
\end{equation}
where the test function $\varphi : X \to \R$ is smooth in $z$ throughout $X$, and is smooth in $w$ only away from $w = 0$, being (up to a constant) of the order of $\left\lvert w \right\rvert^{\bo}$ near $w = 0$, so that the first integral in (\ref{eq:BTprodnonsmooth}) happening with respect to the coordinate $z$ is clearly well-defined, and for the second integral in (\ref{eq:BTprodnonsmooth}) with respect to the coordinate $w$ we will have $\varphi \frac{f'' \left(s\right)}{\left\lvert w \right\rvert^2}$ up to a constant of the order of $\left\lvert w \right\rvert^{3\bo-2}$ near $w = 0$, as we had $\frac{f'' \left(s\right)}{\left\lvert w \right\rvert^2} \in O \left(\left\lvert w \right\rvert^{2 \beta_0 - 2}\right)$ as $w \to 0$ from the asymptotic expressions (\ref{eq:omegapolyhomo}) and (\ref{eq:omegasemihomo}), and then $\left\lvert w \right\rvert^{3\bo-2} \sqrt{-1} d w \wedge d \bar{w}$ will also be \intble{} over $\so$, as $\left\lvert w \right\rvert^{2\bo-2} \sqrt{-1} d w \wedge d \bar{w}$ was \intble{} over $\so$ by the Bedford-Taylor wedge product computed in (\ref{eq:currentomegazeroBTprod}). \\
So the computation of $\mathcal{F}_{\log; \bo, \boo} \left(w \frac{\partial}{\partial w}, \omega\right)$ done in (\ref{eq:logBFomegaevalue}) is indeed justified. The last assertion of Theorem \ref{thm:logBFconicevalue} follows trivially from Definition \ref{def:conehcscK} and the expression (\ref{eq:logBFconicevalue}), or in this case even from the expressions for $\lambda_1 \left(\omega\right)$ and $B$ derived in (\ref{eq:avghighScal1omega}) and (\ref{eq:BCConeSing}) respectively, noting that $B$ was the constant value of $\lambda \left(\omega\right)$ on $\xsosoo$ for the momentum-constructed conical higher cscK metric $\omega$.
\end{proof} \par
\begin{question}[Cohomological Invariance of the Top $\log$ Bando-Futaki Invariant for Momentum-Constructed Conical K\"ahler Metrics]\label{qstn:logBFconiccohomllinvar}
Let $\omega$ be a momentum-constructed conical \krm{} on $X$ with cone angles $2 \pi \bo$ and $2 \pi \boo$ along $\so$ and $\soo$ respectively, given by the ansatz (\ref{eq:ansatzconesing}) with the defining convex function being $f \left(s\right)$, and belonging to the \krcl{} $\pcsoo$. Let $\eta$ be a momentum-constructed smooth \krm{} on $X$, given by an ansatz of the form (\ref{eq:ansatzeta}) with the convex function being some $\rho \left(s\right)$, and belonging to the same \krcl{} $2 \pi \left(\sfc + m S_\infty\right)$. \\
Then is it true that $\mathcal{F}_{\log; \bo, \boo} \left(w \frac{\partial}{\partial w}, \omega\right) = \mathcal{F}_{\log; \bo, \boo} \left(w \frac{\partial}{\partial w}, \eta\right)$, meaning is the top $\log$ Bando-Futaki invariant going to be an invariant of the \krcl{} even when computed with respect to (momentum-constructed) conical \krm{s} coming from the \krcl{}? (Because Theorem \ref{thm:logBFwelldefined} gives this \invc{} only for smooth \krm{s} coming from the \krcl{} under consideration.) \\
In particular if this $\omega$ is taken to be the conical higher cscK metric constructed in Section \ref{sec:MomentConstructConehcscK}, then is it true that if we take any smooth \krm{} $\eta \in \pcsoo$ then we will be getting $\mathcal{F}_{\log; \bo, \boo} \left(w \frac{\partial}{\partial w}, \eta\right) = 0$? (Because Theorem \ref{thm:logBFconicevalue} gives this evaluation to be zero only when done with respect to the concerned conical higher cscK metric.)
\end{question} \par
In an attempt at answering \textit{Question} \ref{qstn:logBFconiccohomllinvar} we will evaluate the object $\mathcal{F}_{\log; \bo, \boo} \left(w \frac{\partial}{\partial w}, \cdot\right)$ at the two momentum-constructed metrics $\omega$ and $\eta$ and then compare the two values obtained. So we have $\gamma = 1 + f' \left(s\right) \in \left[1, m + 1\right]$ as the momentum variable and $\phi \left(\gamma\right) = f'' \left(s\right)$ as the momentum profile for the conical metric $\omega$, and similarly let $x = 1 + \rho' \left(s\right) \in \left[1, m + 1\right]$ be the momentum variable and $\psi \left(x\right) = \rho'' \left(s\right)$ be the momentum profile for the smooth metric $\eta$. Then all the properties of the momentum construction described in Subsection \ref{subsec:MomentConstructConeSing} hold here for $\omega$ as well as for $\eta$, except for the fact that the boundary conditions on the derivatives of their respective momentum profiles are different, viz. $\phi' (1) = \bo$, $\phi' (m + 1) = -\boo$ for $\omega$ and $\psi' \left(1\right) = 1$, $\psi' \left(m + 1\right) = -1$ for $\eta$ (as can be seen from (\ref{eq:BVPConeSing}) and (\ref{eq:ODEBVP0}) respectively). We just compute the integrals in the expression (\ref{eq:logBFdef}) in terms of $\phi\left(\gamma\right)$ and $\psi\left(x\right)$ individually, using the fact that the real holomorphy potential of $w \frac{\partial}{\partial w}$ with respect to a momentum-constructed \krm{} (conical or smooth) is equal to the momentum variable of the metric, i.e. $\tth = \gamma$ for $\omega$ and $\ttf = x$ for $\eta$ (see (\ref{eq:holpotEulerVF})):
\begin{align}\label{eq:logBFomegaphigamma}
&\mathcal{F}_{\log; \bo, \boo} \left(w \frac{\partial}{\partial w}, \omega\right) = -\frac{1}{2 \left(2 \pi\right)^2} \int\limits_{\xsosoo} \tth \lambda \left(\omega\right) \omega^2 + \frac{\lambda_1 \left(\omega\right)}{2 \left(2 \pi\right)^2} \int\limits_{\xsosoo} \tth \omega^2 \nonumber \\
&= -\frac{1}{\left(2 \pi\right)^2} \iint\limits_{\Sigma \times \left(\C \smallsetminus \left\lbrace 0 \right\rbrace\right)} \omega_\Sigma \wedge \sqrt{-1} \frac{d w \wedge d \bar{w}}{\left\lvert w \right\rvert^2} \frac{\phi}{\gamma} \left( \gamma \left(\phi + 2 \gamma\right) \phi'' + \phi' \left(\phi' \gamma - \phi\right) \right) \nonumber \\
&- \frac{4 \left(\beta_0 + \beta_\infty\right)}{\left(2 \pi\right)^2 m \left(m + 2\right)} \iint\limits_{\Sigma \times \left(\C \smallsetminus \left\lbrace 0 \right\rbrace\right)} \gamma^2 \phi \omega_\Sigma \wedge \sqrt{-1} \frac{d w \wedge d \bar{w}}{\left\lvert w \right\rvert^2} \nonumber \\
&= - \int\limits_1^{m + 1} \gamma \frac{d}{d \gamma} \left( \left(\frac{\phi}{\gamma} + 2\right) \phi' \right) d \gamma - \frac{4 \left(\beta_0 + \beta_\infty\right)}{m \left(m + 2\right)} \int\limits_1^{m + 1} \gamma^2 d \gamma \nonumber \\
&= 2 \left(\bo + \left(m + 1\right) \boo\right) + \int\limits_1^{m + 1} \frac{\phi \phi'}{\gamma} d \gamma + 2 \int\limits_1^{m + 1} \phi' d \gamma - \frac{4}{3} \frac{m^2 + 3 m + 3}{m + 2} \left(\beta_0 + \beta_\infty\right) \nonumber \\
&= 2 \left(\bo + \left(m + 1\right) \boo\right) + \frac{1}{2} \int\limits_1^{m + 1} \left(\frac{\phi}{\gamma}\right)^2 d \gamma - \frac{4}{3} \frac{m^2 + 3 m + 3}{m + 2} \left(\beta_0 + \beta_\infty\right)
\end{align}
where we used the expression (\ref{eq:logBFconicevalue}) for computing $\mathcal{F}_{\log; \bo, \boo} \left(w \frac{\partial}{\partial w}, \cdot\right)$ with respect to the conically singular $\omega$ (as proven in Theorem \ref{thm:logBFconicevalue}), the expression (\ref{eq:avghighScal1omega}) giving $\lambda_1 \left(\omega\right)$, the expressions (\ref{eq:lambdaconesing}) and (\ref{eq:omega2conesing}) for $\lambda \left(\omega\right)$ and $\omega^2$ respectively in terms of $\phi \left(\gamma\right)$, and then solved both the integrals above just like the ones in (\ref{eq:intCherncalceta}) and (\ref{eq:intCherncalcomega}).
\begin{align}\label{eq:logBFetapsix}
&\mathcal{F}_{\log; \bo, \boo} \left(w \frac{\partial}{\partial w}, \eta\right) = -\frac{1}{2 \left(2 \pi\right)^2} \int\limits_X \ttf \lambda \left(\eta\right) \eta^2 + \frac{\lambda_0 \left(\eta\right)}{2 \left(2 \pi\right)^2} \int\limits_X \ttf \eta^2 \nonumber \\
&+ \frac{\beta_0 - 1}{\pi} \int\limits_{S_0} \ttf \eta - \frac{\beta_0 - 1}{\pi} \frac{\int\limits_{S_0} \eta}{\int\limits_X \eta^2} \int\limits_X \ttf \eta^2 + \frac{\beta_\infty - 1}{\left(m + 1\right) \pi} \int\limits_{S_\infty} \ttf \eta - \frac{\beta_\infty - 1}{\left(m + 1\right) \pi} \frac{\int\limits_{S_\infty} \eta}{\int\limits_X \eta^2} \int\limits_X \ttf \eta^2 \nonumber \\
&= - \int\limits_1^{m + 1} x \frac{d}{d x} \left( \left(\frac{\psi}{x} + 2\right) \psi' \right) d x - \frac{8}{m \left(m + 2\right)} \int\limits_1^{m + 1} x^2 d x \nonumber \\
&+ 2 \left(\beta_0 - 1\right) - \frac{4 \left(\beta_0 - 1\right)}{m \left(m + 2\right)} \int\limits_1^{m + 1} x^2 d x + 2 \left(m + 1\right) \left(\beta_\infty - 1\right) - \frac{4 \left(\beta_\infty - 1\right)}{m \left(m + 2\right)} \int\limits_1^{m + 1} x^2 d x \nonumber \\
&= 2 \left(m + 2\right) + \frac{1}{2} \int\limits_1^{m + 1} \left(\frac{\psi}{x}\right)^2 d x - \frac{8}{3} \frac{m^2 + 3 m + 3}{m + 2} \nonumber \\
&+ 2 \left(\beta_0 - 1\right) - \frac{4 \left(\beta_0 - 1\right)}{3} \frac{m^2 + 3 m + 3}{m + 2} + 2 \left(m + 1\right) \left(\beta_\infty - 1\right) - \frac{4 \left(\beta_\infty - 1\right)}{3} \frac{m^2 + 3 m + 3}{m + 2} \nonumber \\
&= 2 \left(\bo + \left(m + 1\right) \boo\right) + \frac{1}{2} \int\limits_1^{m + 1} \left(\frac{\psi}{x}\right)^2 d x - \frac{4}{3} \frac{m^2 + 3 m + 3}{m + 2} \left(\beta_0 + \beta_\infty\right)
\end{align}
where we used the expression (\ref{eq:avghighScaleta}) giving $\lambda_0 \left(\eta\right)$, the expressions (\ref{eq:volxsosookrcl}) giving the volumes of $\so$, $\soo$ and $X$ (with respect to $\eta$), and considered the local expression of the form (\ref{eq:omegaconesing}) for the smooth metric $\eta$ in terms of $\rho \left(s\right)$ in the coordinates $\left(z, w\right)$ and $\left(z, \tilde{w}\right)$ near $\so$ and $\soo$ respectively, along with the limiting values at $\so$ and $\soo$ given in (\ref{eq:wgammasosoo}) to obtain the following integrals:
\begin{equation}\label{eq:intttfsosoo}
\begin{gathered}
\int\limits_{S_0} \ttf \eta = \int\limits_{S_0} \left( \left(1 + \rho' \left(s\right)\right)^2 \mathtt{p}^* \omega_\Sigma + \left(1 + \rho' \left(s\right)\right) \rho'' \left(s\right) \sqrt{-1} \frac{d w \wedge d \bar{w}}{\left\lvert w \right\rvert^2} \right) \\
= \int\limits_{S_0} \left(1 + \lim\limits_{s \to -\infty} \rho' \left(s\right)\right)^2 \mathtt{p}^* \omega_\Sigma + \int\limits_{S_0} \lim\limits_{s \to -\infty} \left( \left(1 + \rho' \left(s\right)\right) \frac{\rho'' \left(s\right)}{\left\lvert w \right\rvert^2} \right) \sqrt{-1} d w \wedge d \bar{w} = \int\limits_{\Sigma} \omega_\Sigma = 2 \pi \\
\int\limits_{S_\infty} \ttf \eta = \int\limits_{S_\infty} \left(1 + \lim\limits_{s \to \infty} \rho' \left(s\right)\right)^2 \mathtt{p}^* \omega_\Sigma + \int\limits_{S_\infty} \lim\limits_{s \to \infty} \left( \left(1 + \rho' \left(s\right)\right) \frac{\rho'' \left(s\right)}{\left\lvert \tilde{w} \right\rvert^2} \right) \sqrt{-1} d \tilde{w} \wedge d \bar{\tilde{w}} = 2 \pi \left(m + 1\right)^2
\end{gathered}
\end{equation}
and then simply calculated all the integrals in (\ref{eq:logBFetapsix}) similar to (\ref{eq:logBFomegaphigamma}). \par
Now looking at (\ref{eq:logBFomegaphigamma}) and (\ref{eq:logBFetapsix}) we can observe that the two evaluations will be equal if and only if the values of the integrals $\int\limits_1^{m + 1} \left(\frac{\phi}{\gamma}\right)^2 d \gamma$ and $\int\limits_1^{m + 1} \left(\frac{\psi}{x}\right)^2 d x$ are the same. But since Theorem \ref{thm:mainconesing} (and also Theorem \ref{thm:heKsmooth}) does not yield any explicit closed form expression for the momentum profile of the conical higher cscK metric (and the smooth higher \ext{} \krm{} respectively), it is impossible to explicitly determine these integrals in our case. This is quite unlike the case of conical cscK and smooth \ext{} \krm{s} studied in the setting of Calabi symmetry in the works of Hashimoto \cite{Hashimoto:2019:cscKConeSing} and T{\o}nnesen-Friedman \cite{Tonnesen:1998:eKminruledsurf}, Hwang-Singer \cite{Hwang:2002:MomentConstruct}, Sz\'ekelyhidi \cite{Szekelyhidi:2014:eKintro}; Section 4.4 respectively, where the momentum profile was explicitly determinable as a rational function of the momentum variable because the ODE over there was readily \intble{} unlike the ODE (\ref{eq:ODEConeSing}) (and also the ODE in (\ref{eq:ODEBVP0})). The author could not find any other method of indirectly proving the equality of these two integrals, and believes that this obstacle is occurring chiefly because of the non-\intbly{} of the left hand side of the higher cscK ODE (\ref{eq:ODEConeSing}) (as the $\log$ Futaki invariant of Hashimoto \cite{Hashimoto:2019:cscKConeSing} meant for conical cscK metrics was seen to be an \invt{} of the \krcl{} with momentum-constructed conical \krm{s} mainly due to its explicit evaluation). And if $\omega$ and $\eta$ are any momentum-constructed conical and smooth metrics respectively (not necessarily higher cscK and higher \ext{} \kr{}) then $\phi\left(\gamma\right)$ and $\psi\left(x\right)$ will be arbitrary smooth positive functions on $\cllml$ satisfying the required boundary conditions in which case it is hard to believe that $\int\limits_1^{m + 1} \left(\frac{\phi}{\gamma}\right)^2 d \gamma = \int\limits_1^{m + 1} \left(\frac{\psi}{x}\right)^2 d x$. But we do think that $\mathcal{F}_{\log; \bo, \boo} \left(w \frac{\partial}{\partial w}, \omega\right) = \mathcal{F}_{\log; \bo, \boo} \left(w \frac{\partial}{\partial w}, \eta\right)$ at least in the case when $\omega$ is the momentum-constructed conical higher cscK metric (in which case the common value of the two should be expected to be zero). \par
\begin{conjecture}[Vanishing of the Top $\log$ Bando-Futaki Invariant on the K\"ahler Class of a Momentum-Constructed Conical Higher cscK Metric]\label{conj:logBFconiccohomllinvar}
If there exists a conical higher cscK metric $\omega$ on $X$ with cone angles $2 \pi \bo$ and $2 \pi \boo$ along $\so$ and $\soo$ respectively, belonging to the \krcl{} $\pcsoo$ and yielded by the momentum construction as outlined in Subsection \ref{subsec:MomentConstructConeSing}, then given any smooth \krm{} $\eta$ on $X$ belonging to the same \krcl{} $2 \pi \left(\sfc + m S_\infty\right)$, we will have $\mathcal{F}_{\log; \bo, \boo} \left(w \frac{\partial}{\partial w}, \eta\right) = \mathcal{F}_{\log; \bo, \boo} \left(w \frac{\partial}{\partial w}, \omega\right) = 0$.
\end{conjecture} \par
Theorems \ref{thm:logBFwelldefined} and \ref{thm:logBFconicevalue} and Conjecture \ref{conj:logBFconiccohomllinvar} legitimize the object $\mathcal{F}_{\log; \bo, \boo} \left(w \frac{\partial}{\partial w}, \pcsoo\right)$ defined by the expression (\ref{eq:logBFdef}) as a Futaki-type invariant providing the right kind of obstruction to the existence of momentum-constructed conical higher cscK metrics with cone angles $2\pi\bo$ and $2\pi\boo$ along the divisors $\so$ and $\soo$ respectively in the \krcl{} $\pcsoo$ on the surface $X$. And Corollary \ref{cor:mainconesing} tells that in every \krcl{} $\pcsoo$ momentum-constructed conical higher cscK metrics always exist for some positive values of the cone angles $\bo, \boo$ depending on the parameter $m$ characterizing the \krcl{}. So we will now just physically compute the invariant $\mathcal{F}_{\log; \bo, \boo} \left(w \frac{\partial}{\partial w}, \cdot\right)$ with respect to the momentum-constructed smooth (non-higher cscK) higher \ext{} \krm{} $\eta$, which also exists in each \krcl{} $\pcsoo$ by Theorem \ref{thm:heKsmooth} (which was the main result of our previous paper \cite{Sompurkar:2023:heKsmooth}; Corollaries 2.3.2 and 4.1), and equate the expression obtained here with zero to derive a linear relationship given in terms of $m$ between the values of $\bo, \boo$ for which conical higher cscK metrics can be constructed in the given \krcl{} (assuming Conjecture \ref{conj:logBFconiccohomllinvar} to be true all the way). \par
Let $\eta$ be the momentum-constructed smooth higher \ext{} \krm{} on $X$, which is defined by the ansatz (\ref{eq:ansatz0}) (where the strictly convex smooth function is $\rho \left(s\right)$) and which belongs to the \krcl{} $2 \pi \left(\sfc + m S_\infty\right)$ (which we had seen earlier in Subsection \ref{subsec:Background}). Then its momentum profile $\psi \left(x\right) = \rho'' \left(s\right)$ satisfies the ODE boundary value problem (\ref{eq:ODEBVP0}) along with all the (boundary) conditions mentioned therein. Again referring to \cite{Sompurkar:2023:heKsmooth}; Subsection 2.2, we have the higher scalar curvature of $\eta$ given as a linear polynomial in the momentum variable $x = 1 + \rho' \left(s\right)$ with the coefficients being in terms of the constants $A, B, C$ appearing in the right hand side of the ODE in (\ref{eq:ODEBVP0}) precisely as $\lambda \left(\eta\right) = A x + B$ (compare this with the case of the conical higher cscK metric $\omega$ in this paper where $\lambda \left(\omega\right) = B$ in the ODE (\ref{eq:ODEConeSing}) as was seen in Subsection \ref{subsec:MomentConstructConeSing}). Now for computing $\mathcal{F}_{\log; \bo, \boo} \left(w \frac{\partial}{\partial w}, \eta\right)$ with respect to this smooth higher \ext{} \kr{} $\eta$, we simply have to follow the computation (\ref{eq:logBFetapsix}) where only one thing changes, viz. in the very first integral term we now have to substitute the linear polynomial expression $\lambda \left(\eta\right) = A x + B$ instead of the second-order fully non-linear differential expression (\ref{eq:lambdaconesing}) for $\lambda \left(\eta\right)$ given in terms of $\psi \left(x\right)$. Due to this we can now explicitly compute the first integral term of the invariant instead of leaving behind an integral of the type $\int\limits_1^{m + 1} \left(\frac{\psi}{x}\right)^2 d x$ which could not be solved further into a closed form.
\begin{align}\label{eq:logBFetapsixhighextKlr}
&\mathcal{F}_{\log; \bo, \boo} \left(w \frac{\partial}{\partial w}, \eta\right) = -\frac{1}{2 \left(2 \pi\right)^2} \int\limits_X \ttf \lambda \left(\eta\right) \eta^2 - \frac{8}{3} \frac{m^2 + 3 m + 3}{m + 2} \nonumber \\
&+ 2 \left(\beta_0 - 1\right) \left( 1 - \frac{2}{3} \frac{m^2 + 3 m + 3}{m + 2} \right) + 2 \left(\beta_\infty - 1\right) \left( m + 1 - \frac{2}{3} \frac{m^2 + 3 m + 3}{m + 2} \right) \nonumber \\
&= -\frac{1}{\left(2 \pi\right)^2} \iint\limits_{\prj} x \left(A x + B\right) x \psi \mathtt{p}^* \omega_\Sigma \wedge \sqrt{-1} \frac{d w \wedge d \bar{w}}{\left\lvert w \right\rvert^2} \nonumber \\
&- 2 \left(m + 2\right) + 2 \beta_0 \left( 1 - \frac{2}{3} \frac{m^2 + 3 m + 3}{m + 2} \right) + 2 \beta_\infty \left( m + 1 - \frac{2}{3} \frac{m^2 + 3 m + 3}{m + 2} \right) \nonumber \\
&= -\int\limits_1^{m + 1} \left(A x^3 + B x^2\right) d x - 2 \left(m + 2\right) - \frac{2 m \left(2 m + 3\right)}{3 \left(m + 2\right)} \beta_0 + \frac{2 m \left(m + 3\right)}{3 \left(m + 2\right)} \beta_\infty \nonumber \\
&= -\frac{m^3 \left(m^2 + 6 m + 6\right)}{12 \left(m + 1\right)^2} C \left(m\right) + \frac{m^2 \left(m + 2\right)^3}{6 \left(m + 1\right)^2} - \frac{2 m \left(2 m + 3\right)}{3 \left(m + 2\right)} \beta_0 + \frac{2 m \left(m + 3\right)}{3 \left(m + 2\right)} \beta_\infty
\end{align}
where we used the expressions for $A, B$ in terms of $C, m$ derived in \cite{Sompurkar:2023:heKsmooth}; Subsection 2.3, equation (2.3.2) (which have been reproduced in this paper in equation (\ref{eq:ABCm0})), and the unique value of the parameter $C = C \left(m\right)$ given by \cite{Sompurkar:2023:heKsmooth}; Theorem 2.3.2 (which is stated as Theorem \ref{thm:heKsmooth} in this paper) for which the ODE boundary value problem (\ref{eq:ODEBVP0}) has a solution, which then by \cite{Sompurkar:2023:heKsmooth}; Corollary 2.3.1 (which is Corollary \ref{cor:heKsmooth} in this paper) yields the smooth higher \ext{} \krm{} $\eta$ for each $m > 0$ i.e. in each \krcl{} $\pcsoo$. Note that the expressions (\ref{eq:ABCm0}) for $A, B$ from our previous paper \cite{Sompurkar:2023:heKsmooth}; Subsection 2.3 are the analogues of the expressions (\ref{eq:BCalphabeta0}) for $B, C$ in terms of the parameter $\alpha$ seen in Subsection \ref{subsec:Proof1} in this paper, and similarly the unique value of the parameter $C$ solving the ODE boundary value problem (\ref{eq:ODEBVP0}) for the momentum profile $\psi\left(x\right)$ of the smooth higher \ext{} \krm{} $\eta$ can be seen in analogy with the unique $\alpha = \alpha \left(m, \bo\right)$ given by Corollary \ref{cor:final} which solves the ODE boundary value problem (\ref{eq:ODEBVPConeSing}) in our conical higher cscK case of this paper. \par
We can now set the expression (\ref{eq:logBFetapsixhighextKlr}) to zero, to conjecturally see that for each $m > 0$ the positive values of the cone angles $\bo, \boo$, for which momentum-constructed conical higher cscK metrics are admitted in the \krcl{} $\pcsoo$ by Corollary \ref{cor:mainconesing}, are given by the locus of a straight line depending on the parameter $m$ associated with the underlying \krcl{}.
\begin{equation}\label{eq:boboomstrghtline}
\frac{2 \left(m + 3\right)}{m + 2} \beta_\infty - \frac{2 \left(2 m + 3\right)}{m + 2} \beta_0 = \frac{m^2 \left(m^2 + 6 m + 6\right)}{4 \left(m + 1\right)^2} C \left(m\right) - \frac{m \left(m + 2\right)^3}{2 \left(m + 1\right)^2}
\end{equation} \par
We have the following result saying that the vanishing of the top $\log$ Bando-Futaki invariant for a pair of values of $\bo, \boo$ is equivalent to the existence of a momentum-constructed conical higher cscK metric with these values of the cone angles at $\so, \soo$ respectively in the \krcl{} $\pcsoo$. This result follows directly from taking a combined view of Theorem \ref{thm:logBFconicevalue} and Conjecture \ref{conj:logBFconiccohomllinvar}, the equations (\ref{eq:logBFetapsixhighextKlr}) and (\ref{eq:boboomstrghtline}) and Corollary \ref{cor:mainconesing}. But note that as this result rests on the veracity of Conjecture \ref{conj:logBFconiccohomllinvar}, it is also conjectural in nature.
\begin{corollary}[Conjecture: The Linear Relationship between the Two Cone Angles]\label{cor:logBFiffvanish}
For each $m > 0$ there exists a conical higher cscK metric $\omega$ with cone angles $2 \pi \beta_0 > 0$ and $2 \pi \beta_\infty > 0$ along the divisors $S_0$ and $S_\infty$ of the minimal ruled surface $X = \prj$ respectively, satisfying the Calabi ansatz (\ref{eq:ansatzconesing}) and belonging to the K\"ahler class $2 \pi \left(\sfc + m S_\infty\right)$, if and only if $\mathcal{F}_{\log; \bo, \boo} \left(w \frac{\partial}{\partial w}, \pcsoo\right) = 0$, i.e. the top $\log$ Bando-Futaki invariant with these respective values of $\bo, \boo$ vanishes for all smooth \krm{s} coming from $\pcsoo$, if and only if the ordered pair $\left(\bo, \boo\right)$ satisfies the straight line equation (\ref{eq:boboomstrghtline}).
\end{corollary}
\begin{proof}
\textit{(Assuming Conjecture \ref{conj:logBFconiccohomllinvar})} By Theorem \ref{thm:logBFwelldefined} the vanishing of $\mathcal{F}_{\log; \bo, \boo} \left(w \frac{\partial}{\partial w}, \cdot\right)$ on the entire \krcl{} $\pcsoo$ is equivalent to the evaluation of $\mathcal{F}_{\log; \bo, \boo} \left(w \frac{\partial}{\partial w}, \cdot\right)$ being equal to zero for the smooth non-higher cscK higher \ext{} \krm{} $\eta$ provided by our previous paper \cite{Sompurkar:2023:heKsmooth}. And then from the computation (\ref{eq:logBFetapsixhighextKlr}) it will be equivalent to $\left(\bo, \boo\right)$ lying on the straight line given by the equation (\ref{eq:boboomstrghtline}). So the equivalence of the last two statements in Corollary \ref{cor:logBFiffvanish} is clear. \\
Conjecture \ref{conj:logBFconiccohomllinvar} is precisely saying that the existence of a momentum-constructed conical higher cscK metric in a given \krcl{} implies the vanishing of the invariant $\mathcal{F}_{\log; \bo, \boo} \left(w \frac{\partial}{\partial w}, \cdot\right)$ on the \krcl{}. So it remains to check the other way implication in this statement. \\
Let $m > 0$ be fixed and let the ordered pair $\left(\bo, \boo\right) \in \mathbb{R}_{> 0} \times \mathbb{R}_{> 0}$ lie on the straight line (\ref{eq:boboomstrghtline}) which depends on $m$. For this given $\bo$ there exists a unique $\boo' > 0$ afforded by Corollary \ref{cor:mainconesing} such that there exists a conical higher cscK metric $\omega$ in the \krcl{} $\pcsoo$ having $\bo, \boo'$ as the values of its cone angles along $\so, \soo$ respectively. By Conjecture \ref{conj:logBFconiccohomllinvar} the top $\log$ Bando-Futaki invariant for $\bo, \boo'$ vanishes on the \krcl{} $\pcsoo$, and so even the pair $\left(\bo, \boo'\right)$ satisfies the equation (\ref{eq:boboomstrghtline}). So we are left with both $\left(\bo, \boo\right)$ as well as $\left(\bo, \boo'\right)$ lying on the straight line (\ref{eq:boboomstrghtline}), and the line (\ref{eq:boboomstrghtline}) is clearly seen to not have slope $\infty$. So $\boo' = \boo$ and this means $\omega$ is the conical higher cscK metric that we were supposed to find for the pair $\left(\bo, \boo\right)$.
\end{proof}
\section*{Acknowledgements}
A good part of this research work was contained in the author's Ph.D. thesis \cite{Sompurkar:2024:heKsmoothhcscKconsingThes}, though many of the results were obtained in their entirety only later on. The author is greatly indebted to his former Ph.D. supervisor Prof. Vamsi Pritham Pingali and his former Ph.D. co-supervisor Prof. Ved V. Datar, without both of whose continuous guidance, instruction, support and encouragement throughout his Ph.D., this research work would not have been possible. The author thanks both of them for guiding him through some of their own research work (e.g. \cite{Datar:2014:CanonConeSing,Pingali:2018:heK}), for suggesting him this highly interesting research problem, for mentioning some key points and providing some good references needed to tackle this research problem, for having fruitful discussions with him regarding this research problem on multiple occasions and for giving him deep insights into the broader field of study of this research problem. Also thanks go to both of them for checking the author's research work, giving crucial feedback on it and suggesting essential improvements in the same. And finally the author highly appreciates the extremely important comments and suggestions made by the anonymous reviewer which helped in the overall improvement of this paper, and hence thanks them as well for their careful reading of the manuscript.
\end{document}